\documentclass[12pt]{amsart}
\usepackage{blkarray}
\usepackage{multicol}
\usepackage{mypreamble}
\usepackage{spverbatim}

\title[HMS on coordinate rings for a complex genus 2 curve]{Categorical mirror symmetry on cohomology for a complex genus 2 curve}
\author{Catherine Cannizzo}
\address{Simons Center for Geometry and Physics, State University of New York, Stony Brook, NY 11794-3636}
\date{\today}
\email{ccannizzo@scgp.stonybrook.edu}
\urladdr{https://cantab.net/users/catherine.cannizzo/}

\begin{document}

\begin{abstract}
Motivated by observations in physics, mirror symmetry is the concept  that certain manifolds come in pairs $X$ and $Y$ such that the complex geometry on $X$ mirrors the symplectic geometry on $Y$. It allows one to deduce symplectic information about $Y$ from known complex properties of $X$. Strominger-Yau-Zaslow \cite{syz} described how such pairs arise geometrically as torus fibrations with the same base and related fibers, known as SYZ mirror symmetry. Kontsevich \cite{hms} conjectured that a complex invariant on $X$ (the bounded derived category of coherent sheaves) should be equivalent to a symplectic invariant of $Y$ (the Fukaya category, see~\cite{fuk_intro}, \cite{fooo1}, \cite{new}, \cite{polyfold_lab}). This is known as homological mirror symmetry. In this project, we first use the construction of ``generalized SYZ mirrors" for hypersurfaces in toric varieties following Abouzaid-Auroux-Katzarkov \cite{AAK}, in order to obtain $X$ and $Y$ as manifolds. The complex manifold is the genus 2 curve $\Sigma_2$ (so of general type $c_1<0$) as a hypersurface in its Jacobian torus. Its generalized SYZ mirror is a Landau-Ginzburg model $(Y,v_0)$ equipped with a holomorphic function $v_0:Y \to \mb{C}$ which we put the structure of a symplectic fibration on. We then describe an embedding of a full subcategory of $D^bCoh(\Sigma_2)$ into a cohomological Fukaya-Seidel category of $Y$ as a symplectic fibration. While our fibration is one of the first nonexact, non-Lefschetz fibrations to be equipped with a Fukaya category, the main geometric idea in defining it is the same as in Seidel's construction for Fukaya categories of Lefschetz fibrations in \cite{seidel} and in Abouzaid-Seidel \cite{ab_seid}. 
\end{abstract}
 \dedicatory{I dedicate this paper to my mother and father, Catherine and John, for being the best role models, paving the way, and making everything in my life possible.}
\maketitle
Declarations of interest: none.
Highlights:
\begin{itemize}
    \item Homological mirror symmetry result for the genus 2 curve
    \item Proved fully faithful embedding on the cohomological level from complex side to sympectic side
\end{itemize}
\titlepage
\tableofcontents
\newpage
\listoffigures

\newpage

\section{Context and main result}

\subsection{Context} Progress in mirror symmetry began with compact Calabi-Yau manifolds ($c_1=0$). In particular, the geometric mirror for those of complex dimension three can be constructed from T-duality three times \cite{syz} by inverting the radius of each $S^1$ in a torus fiber to go from the A-model $\rightarrow$ B-model $\rightarrow$ A-model $\rightarrow$ B-model.  

For Fano manifolds with $c_1>0$, \cite{hori_vafa} describe a physical reason why a mirror should be a \emph{Landau-Ginzburg model}, which for mathematicians is a non-compact complex manifold $M$ equipped with a holomorphic function $W: M \to \mb{C}$ called a \emph{superpotential}. In \cite{cho_oh}, they explicitly compute the superpotential in the case of Fano toric varieties to be a weighted sum of discs according to their intersections with the toric divisors.

Homological mirror symmetry (HMS) \cite{hms} has been proven in the Calabi-Yau case \cite{sherid_CY}, \cite{seidel_quartic}, \cite{fuk_abel}. Proven examples in the Fano case include \cite{moh2}, \cite{ueda_del_pezzo}, \cite{auroux_weight_proj}, \cite{sherid_Fano}. In the case of general type ($c_1<0$), Landau-Ginzburg models are also natural candidates to satisfy HMS. The example of general type in this paper is a hypersurface of an abelian variety, based on work of \cite{AAK} for hypersurfaces of toric varieties. 

In this paper we consider homological mirror symmetry for a genus 2 curve as a hypersurface in an abelian variety based on \cite{AAK}, which was speculated in \cite{seidel_gen2_specul}. Its mirror Landau-Ginzburg model is one of the first non-exact (symplectic fibers are compact), non-monotone, non-Lefschetz symplectic fibrations to be equipped with a Fukaya category. The method follows that of \cite{seidel} (for Lefschetz fibrations), \cite{ushape} (introduction of U-shaped curves for non-compact Lagrangians), and \cite{ab_seid} (using categorical localization to define the morphism groups).  

We consider the B-model on the genus 2 curve. Seidel \cite{genus2} proves HMS when the genus 2 curve is on the A-model. One connection between his mirror and our mirror is that their superpotentials have the same critical locus given by three $\mb{P}^1$'s identifying their north poles to a point and their south poles to a point. This is known as the ``banana manifold."

\subsection{Main result} 

\begin{definition} A \emph{symplectic fibration} is a symplectic manifold $(Y,\omega)$ with a fibration such that fibers of the fibration are symplectic with respect to $\omega$.\end{definition}

\begin{theorem}[{\cite{me}}]\label{theorem: me} Let $V$ be the abelian variety $(\mb{C}^*)^2/\Gamma_B$ where $\Gamma_B:=\mb{Z}\left<\gamma', \gamma''\right>$ for
\begingroup \allowdisplaybreaks \begin{equation}
    \gamma':=\left( \begin{matrix} 2 \\ 1 \end{matrix} \right), \gamma'' := \left( \begin{matrix} 1 \\2 \end{matrix} \right)
\end{equation} \endgroup
acts on $(\mb{C}^*)^2$ by
\begingroup \allowdisplaybreaks \begin{equation}
    \begin{aligned}
    \mb{Z}^2 \times (\mb{C}^*)^2 \ni (\gamma_1, \gamma_2) \cdot (x_1,x_2) \mapsto (\tau^{-\gamma_1}x_1, \tau^{-\gamma_2}x_2) \in (\mb{C}^*)^2.
    \end{aligned}
\end{equation} \endgroup
for $\tau \in \mb{R}^+_{\ll1}$. Let $\mc{L} \to V$ be the ample line bundle $(\mb{C}^*)^2 \times \mb{C} / \Gamma_B$ where
\begingroup \allowdisplaybreaks \begin{equation}\label{eqn: mc{L}}
    \gamma \cdot (x_1,x_2,v) := (\gamma\cdot(x_1,x_2), x^{\tiny{\left( \begin{matrix} 2 & 1\\ 1 & 2\\ \end{matrix} \right)}^{-1}\gamma}\tau^{-\frac{1}{2}\gamma^t \tiny{ \left(\begin{matrix} 2 & 1 \\ 1 & 2 \end{matrix} \right)}^{-1}\gamma} v)
\end{equation} \endgroup
and with nonzero section $s: V \to \mc{L}$. Then $H:=s^{-1}(0)$ is a complex genus 2 curve and the following diagram commutes, with fully-faithful vertical embeddings corresponding to HMS on cohomological categories.
\begin{diagram}
D^b_{\mc{L}}Coh(V) & \rTo^{\iota^*} & D^b_{\mc{L}}Coh(H)\\
\dInto^{\mbox{HMS on $V$}} && \dInto_{\mbox{HMS on $H=\Sigma_2$}}\\
H^0Fuk(V^\vee) & \rTo^{\cup} & H^0 FS(Y,v_0)
\end{diagram}

\begin{compactitem}[\textbullet]
    \item where $V^\vee$ is the SYZ dual abelian variety to $V$,
    \item $(Y,v_0)$ is the $\Gamma_B$-quotient of a toric variety of infinite type and $v_0=xyz$ is a $\Gamma_B$-invariant product of the local toric coordinates,
\item $(Y,v_0)$ is a Landau-Ginzburg model which is equipped with the structure of a symplectic fibration via symplectic form $\omega$ defined in Definition \ref{defn of w}, 
\item $(Y,v_0)$ has generic fiber $V^\vee$ degenerating to the singular fiber $\mb{CP}^2(3)/\Gamma_B$,
\item $D^b_{\mc{L}}Coh$ denotes the full subcategory generated by $\{\mc{L}^i[n]\}_{i,n \in \mb{Z}}$,
\item $\iota^*$ denotes the restriction functor of line bundles to the hypersurface $H$,  
\item $\cup$ denotes the functor that parallel transports Lagrangians in fiber $V^\vee$ around U-shapes in the base of $v_0$ using the symplectic horizontal distribution $(TV^\vee)^\omega$, and 
\item FS stands for Fukaya-Seidel category and denotes the Fukaya category of the Landau-Ginzburg model $(Y,v_0)$.
\end{compactitem}
\end{theorem}

\begin{remark}[Relevance] This cohomology-level result already gives a lot of information. The product structure in Floer theory can be computed by counting triangles, and is mirror to the ring structure on $D^b_{\mc{L}|_H}Coh(H)$. Then since $\mathcal{L}|_H$ is the canonical bundle of the genus 2 curve, this determines the product structure on the canonical ring $\bigoplus_{i\geq 0} H^0(\Sigma_2, \mathcal{L}^i)$ of functions on $H=\Sigma_2$. Once we know this ring structure, we can describe embeddings of the genus 2 curve into projective space $\mb{CP}^{N-1}$ where $N=h^0(\Sigma_2,\mc{L}^k)$ for a very ample power of $\mc{L}$, using its sections. As a subvariety in projective space, functions on the genus 2 curve become polynomials, i.e.~restrictions of homogeneous polynomials on the ambient $\mb{CP}^{N-1}$. That is, degree $m$ homogeneous polynomials are sections of $\mc{O}(m)$ for positive integers $m$, which in turn are identified with sections of $\mc{L}|_{\Sigma_2}^{mk}$ via this embedding.
\end{remark}

\begin{remark}[Previous related work] HMS for abelian varieties of arbitrary dimension and quotient lattice was proven in \cite{fuk_abel} using more advanced machinery. We present a different argument for this particular case in the left vertical arrow of the Theorem \ref{theorem: me}. Seidel proved HMS with the A-model of the genus 2 surface \cite{genus2}, i.e.~the symplectic side. Seidel's complex mirror is a crepant resolution of $\mb{C}^3/\mb{Z}_5$, quotienting by rotation and resolving the orbifold singularity without changing the first Chern class. The critical locus of the superpotential in his paper and of the mirror here are the same. He also speculated in \cite{seidel_gen2_specul} HMS for the genus 2 curve on the complex side.
\end{remark}

\begin{remark}[Future directions] One future direction is to relate Seidel's genus 2 mirror to ours. Another is enhancing the theorem to $A_\infty$-functors, namely proving that higher order composition maps match in addition to objects, morphisms and composition. Powers of $\mc{L}$ split-generate the derived category so an $A_\infty$-enhancement of the result would allow us to extend the functors to iterated mapping cones and hence would give a HMS statement for the whole derived category.
\end{remark}

\subsection{Structure of the paper} 
In {\bf Section \ref{section: hms_abel}} we describe the fully-faithful embedding on abelian varieties, that is, the left vertical arrow of Theorem \ref{theorem: me}. This is Lemma \ref{functor_ok}. 

In {\bf Section \ref{section: toric_recap}} we construct $(Y,v_0)$ and a symplectic form defined in Equation \ref{def_s_form} on $Y$. The construction generalizes that in \cite{AAK} for mirrors of hypersurfaces of toric varieties, to the case of a hypersurface of an abelian variety. The idea is that a hypersurface may not admit a special Lagrangian torus fibration, as needed to use the SYZ mirror construction; however we may be able to consider it as sitting inside of something that does. A hypersurface $H=s^{-1}(0)$ of a toric variety $V$ has complex codimension 1, so blowing up $V$ along $H$ doesn't produce anything new. However, $X:=Bl_{H \times \{0\}} V\times \mb{C}$ works; with the holomorphic map and fibration given by projection to the last $\mb{C}$-coordinate, $y:Bl_{H \times \{0\}} V\times \mb{C} \to \mb{C}$, we find that $y^{-1}(0) \supset Crit(y)\cong H$. See Theorem \ref{theorem:Aside_blowup}. 

Via the log-norm map on $V$ and the moment map under the $S^1$-action on $\mb{C}$, \cite{AAK} construct a Lagrangian torus fibration on this blow-up by cut-and-paste operations on the moment polytope of $V \times \mb{C}$, for a suitably defined symplectic form. Morever, the hypersurface $H \subset V$ gives a degenerating family $s_{\tau}^{-1}(0)$ of tropical hypersurfaces under the $\log_{\tau}$-norm map, see Definition \ref{trop definition}; the tropical limit curve is the dual cell complex of some polyhedral decomposition $\mc{P}$. Using \cite{syz} one can then construct a dual Lagrangian torus fibration from a polytope whose charts encode wall-crossing behavior, as well as a toric degeneration via $\mc{P}$ (similar to the Gross-Siebert program, e.g. \cite{Gross} and \cite{Gross_Top}) which will ultimately be the superpotential on the mirror. See Equation \ref{eq:polytope_def_Y}. The difference in our setting is that we consider an \emph{abelian variety}, or the quotient of a toric variety by a lattice, so we must keep track of that lattice in the constructions. The mirror is hence a quotient of a toric variety of infinite type.  

The symplectic form is defined in pieces on the moment polytope in terms of the norms of the toric coordinates, and then glued together by analyzing the bump function derivatives to ensure the 2-form from the constructed K\"ahler potential remains non-degenerate.

In {\bf Section \ref{section: FS}} we equip the symplectic fibration $(Y,v_0)$ with a Fukaya-Seidel category, see Definition \ref{defn:fuk_cat}. The moduli spaces are defined using regular and generic choices. In the future, it may be possible to extend our definition to include non-regular choices, using e.g.~theory developed in \cite{fooo1}, \cite{fooo2}, \cite{fuk_lectures}, and \cite{kur_strs}. Resources for doing this using polyfolds (which may be equivalent) include \cite{j_li} and \cite{polyfold_lab}; for using neck-stretching there are the references \cite[\textsection 2.3]{oan_ciel} and \cite{ben_katrin}. Resources for self-gluing a curve to itself via ``connectors" to put a smooth structure on a moduli space near a multiply-covered curve include \href{https://floerhomology.wordpress.com/2014/07/14/gluing-a-flow-line-to-itself/}{Gluing a flow line to itself}, which is also discussed in \href{https://scholar.harvard.edu/files/gerig/files/note.pdf}{Obstruction Bundle Gluing}. 

In {\bf Section \ref{section:differential}} we compute the differential in this Fukaya-Seidel category. This uses several techniques. The rule that enables us to do so is the Leibniz rule stated in Lemma \ref{leib_lemma}; we can find the differential on all Lagrangians in the category by calculating it for only a subset of Lagrangians. One Lagrangian in this subset is the preimage of a moment map value, and discs bounded by this Lagrangian can be counted by existing theory. Since we defined the differential in Section \ref{section: FS} using generic choices, and we want to compute the differential using this specific Lagrangian and $J$ given by multiplication by $i$ (which is not regular), we prove Lemma \ref{lemma: seidel_htpy_cob} which builds a cobordism between the generic- and specific-choice moduli spaces. In particular, in the cobordism we vary both the almost complex structure and the Lagrangian boundary conditions.

At this point, we have now reduced the calculation of the differential to an open Gromov-Witten invariant of curves bounded by the one Lagrangian which is the preimage of a moment map value, and one marked point; the almost complex structure is multiplication by $i$. The disk-only configurations can be counted using \cite{cho_oh}, and the disk-and-sphere configurations can be counted using \cite{chan} and \cite{kl}. These are the only possible configurations by Corollary \ref{exclude_cor}, where we exclude all other configurations. As we sum over all possible relative homology classes, we see that these moduli spaces are all isomorphic via the lattice group action coming from the abelian variety and can hence compute the intermediary differential with the moment map Lagrangian in Lemma \ref{lemma:final_diffl_computation}. From here we can at last compute the differential more generally in Lemma \ref{lem: the real final diffl calcn}.

In {\bf Section \ref{section: proof}} we prove the fully-faithful embedding result for the genus 2 curve on the cohomological level, which is the right vertical arrow in Theorem \ref{theorem: me}, using that the differential is proportional to the theta function.

\subsection{Acknowledgements} I would first like to thank my thesis advisor Denis Auroux for the immensely helpful mathematical advice, discussions, and support on this project. This project had a lot of moving parts and I benefited from the expertise of many in discussions during conferences and research talks. I thank Mohammed Abouzaid, Melissa Liu, Katrin Wehrheim, Kenji Fukaya, Mark McLean, Charles Doran, Alexander Polishchuk, Sheel Ganatra, Heather Lee, Sara Venkatesh, Haniya Azam, Zack Sylvan, Jingyu Zhao, Roberta Guadagni, Weiwei Wu, Wolfgang Schmaltz, Zhengyi Zhou, Benjamin Filippenko, Hiro Lee Tanaka, Andrew Hanlon, and Rodrigo Barbosa for fruitful mathematical discussions. I thank the referee for helpful and thorough comments. I also thank the Fields Institute for hosting me as a Long Term Visitor during their thematic program on Homological Mirror Symmetry in Fall 2019, which resulted in several potential collaborations. This work was partially supported by NSF grants DMS-1264662, DMS-1406274, and DMS-1702049, and by a Simons Foundation grant (\#\,385573, Simons Collaboration on Homological Mirror Symmetry). Lastly, P.~Taylor's \href{https://www.paultaylor.eu/diagrams/}{commutative diagrams} package was used in this paper. 

\section{Left arrow of main theorem: fully-faithful embedding $D^b_{\mc{L}}Coh(V) \into H^0FS(V^\vee)$}\label{section: hms_abel}
\subsection{The symplectic side}

 We define the action-angle coordinates corresponding to a $T^2$-action on the abelian surface. All Lagrangians will be expressed in these coordinates. We will find that the image of the moment map is the same as the toric polytope, an instance of Delzant's theorem e.g.~see \cite{symp_intro}. Note that a usual moment map would land in $\mb{R}^n$ where $n$ is the dimension of the torus, but here the moment map will land in $\mb{R}^2/\Gamma_B$ instead. This is known as a \emph{quasi-Hamiltonian} action. \label{caveat}

\begin{claim}[Symplectic coordinates]\label{claim:action_angle_fiber}
 Let $V$ be the abelian variety $(\mb{C}^*)^2/\Gamma_B$ as above and $V^\vee$ the SYZ mirror abelian variety with complex coordinates $x$ and $y$. Consider the standard positive rotation $T^2$-action i.e.~$(e^{2\pi i \alpha_1}x,e^{2\pi i \alpha_2}y)$ for $(\alpha_1,\alpha_2) \in T^2$ acting on $(x,y)$. Let $(\xi_1,\xi_2,\theta_1,\theta_2)$ denote the action-angle coordinates, so $\xi_1, \xi_2$ are the quasi-moment map coordinates for the above quasi-Hamiltonian $T^2$-action with respect to a symplectic form (which will be the restriction to a fiber of a symplectic form on $Y$, defined below). Then $\xi_i = \log_\tau |x_i|$ for $(x_1,x_2) \in V=(\mb{C}^*)^2/\Gamma_B$, hence $\gamma \in \Gamma_B$ acts on $(\xi_1,\xi_2)$ by translating in the negative direction $-\gamma$. Lastly $\theta_1:=\arg(x)$ and $\theta_2:=\arg(y)$.
\end{claim}

\begin{proof}
Since the Lagrangian torus fibration on $X$ is special with respect to the $n$-form $\Omega = d \log x_1 \wedge d \log x_2 \wedge dy$ for $n=3$, we have integral affine structures on the bases of both fibrations, for $X$ and for $Y$. The complex affine structure on the B-model corresponds to the symplectic affine structure on the A-model. In the construction of SYZ mirror symmetry, e.g.~cf \cite{t_duality}, the complex affine structure on one side (the $\log|\cdot |$ of the complex coordinates) is mirror to the symplectic affine structure on the other side (the moment map). Hence $\xi_i = \log_\tau |x_i|$. Since we rotate the complex coordinates $x$ and $y$, their angles are the remainder of the action-angle coordinates. So the symplectic form on $V^\vee$ in these coordinates is $d \xi \wedge d\theta$. The statement about the $\Gamma_B$-action follows from the $\Gamma_B$-action on $V$, namely $\gamma \cdot (x_1,x_2) = (\tau^{-\gamma_1}x_1, \tau^{-\gamma_2}x_2)$. Thus taking the logarithm base $\tau$ implies $\gamma$ acts additively in the negative direction.
\end{proof}

\begin{definition}[Setting up notation]\label{defn:Bside_notation}
Let $T_B:=\mb{R}^2/\Gamma_B$ where $\gamma$ acts by negative translation $-\gamma$, and $T_F:=\mb{R}^2/\mb{Z}^2$. Thus $V^\vee = {T_B} \times {T_F} \ni (\xi_1,\xi_2,\theta_1,\theta_2)$ with symplectic form $d\xi_1 \wedge d\theta_1 + d\xi_2 \wedge d\theta_2$. Furthermore, let $\la$ denote the linear map on $\mb{R}^2$ in standard bases given by $\left( \begin{matrix} 2 & 1 \\ 1 & 2 \end{matrix}\right)^{-1}=\left( \begin{matrix} \frac{2}{3} & -\frac{1}{3} \\ -\frac{1}{3} & \frac{2}{3} \end{matrix}\right)$. In particular, $\la(\gamma')=(1,0)$ and $\la(\gamma'') = (0,1)$. For ease of notation to follow we also define $\kappa(\gamma) :=  -\frac{1}{2} \left<\gamma, \la(\gamma) \right>$.

\end{definition}

\begin{remark}[Intuition for choice of Lagrangians] HMS for abelian varieties was previously known in more generality by Fukaya \cite{fuk_abel}. In his paper, he uses Family Floer theory to define a line bundle $\mc{E}$ by requiring $\mc{E}_p= \Ext(\mc{E}, \mc{O}_p)$ to be defined as the corresponding hom set on the mirror side, i.e.~$\mb{C} \times (L_p \cap L)$ where $L$ is a linear Lagrangian on the torus, $L_p$ is vertical (infinite slope), and this intersection has one point. He then constructs a holomorphic structure on this line bundle.

Here we will write explicit linear Lagrangians first, which will be two-dimensional analogues of those considered in \cite{zp}. In \cite{zp} they consider a square depicting $(T^2, \int_{T^2} \omega = a)$ as mirror to elliptic curve with complex structure $\mb{C}/\mb{Z} + ia \mb{Z}$, on which lines of slope $k$ on the square are mirror to $\mc{L}^k$ powers of a degree 1 line bundle on the mirror elliptic curve. Similarly our linear Lagrangians denoted $\ell_i$ will be mirror to $\mc{L}^i$ for the $\mc{L}$ defined in (\ref{eqn: mc{L}}).
\end{remark}

\begin{remark}[Intuition for choice of line bundles] We will see how we arrived at the definition of the holomorphic line bundle $\mc{L} \to V$ in Theorem \ref{theorem: me}. Its transition functions on different $\gamma$ translates are given by:
$$(\gamma,x) \mapsto x^{\la(\gamma)} \tau^{\kappa(\gamma)}$$
where $\la \in \hom(\Gamma_B, \mb{Z}^2) = \hom (\Gamma_B, \Gamma_F^*)$ corresponds to the first Chern class and is ${\left( \begin{matrix} 2 & 1\\ 1 & 2\\ \end{matrix} \right)}^{-1}$ in Theorem \ref{theorem: me}. This is because the first Chern class arises as follows:
\begin{allowdisplaybreaks}
    \begin{align*}
        H^2(V; \mb{Z}) = H^2(T_B \times T_F) & \cong \oplus_i H^i(T_B) \otimes H^{2-i}(T_F)\\
        \therefore c_1(\mc{L}) \in H^2(V; \mb{Z}) \cap H^{1,1}(V) \implies c_1(\mc{L}) & \in H^1(T_B; \mb{Z})\otimes H^1(T_F; \mb{Z})\\
        &=\hom(\Gamma_B, \mb{Z}) \otimes \hom(\Gamma_F, \mb{Z})\\
        &=\Gamma_B^* \otimes \Gamma_F^*\stepcounter{equation}\tag{\theequation}\\
        &=\hom(\Gamma_B, \Gamma_F^* \cong \mb{Z}^2))\\
        \implies c_1(\mc{L}) & \in \hom(\Gamma_B, \mb{Z}^2)
    \end{align*}
\end{allowdisplaybreaks}
by Corollary \ref{cor:v_vs_vs+} (which holds for our particular choice of complex structure on $V$ determined by $\Gamma_B$) and Claim \ref{claim: holo_abel_str}.
\end{remark}

\begin{lemma}\label{lem:fuk_subcat}  The following defines a full subcategory of $Fuk(V^\vee)$. The objects are
\begingroup \allowdisplaybreaks \begin{equation}
    \ell_k :=\{(\xi_1, \xi_2,\theta_1,\theta_2) \in V^\vee \mid (\theta_1,\theta_2) \equiv -k \left( \begin{matrix} 2 & 1 \\ 1 & 2 \end{matrix}\right) ^{-1}\left(\begin{matrix} \xi_1 \\ \xi_2 \end{matrix} \right)~\mod \mb{Z}^2 \}
\end{equation} \endgroup

The morphisms $HF(\ell_i, \ell_j)$ have rank $(i-j)^2$ for $i\neq j$ and $HF(\ell_i, \ell_i) \cong H(T^2)$. The multiplicative structure for $CF(\ell_j, \ell_k) \otimes CF(\ell_i, \ell_j) \to CF(\ell_i, \ell_k)$ is
$$ \left<\mu^2(p_1,p_2), q\right> = \sum_{\gamma_A \in \Gamma_B} \tau^{- \frac{l}{l'l''}\cdot  \kappa\left(\frac{l''}{l}\gamma_{e,l} + \gamma_A\right)}$$
summing over possible intersection points of $\ell_i \cap \ell_j$, where $l'=j-i$, $l''=k-j$, and $l=k-i$.

\end{lemma}

\begin{proof} \mbox{}
\begin{center}\fbox{\bf Objects}
\end{center} 

The definition of $\ell_k$ is well-defined because the minus sign ensures that it's well-defined as a graph modulo group action; $\Gamma_B$ acts negatively on $(\xi_1,\xi_2)$, so in the theta coordinates it becomes the standard positive additive $\mb{Z}^2$ action in the angular coordinates. 

Secondly, the $\ell_k$ are Lagrangian. Given a path 
$$p(t):=(\xi_1(t), \xi_2(t), -k\la(\xi(t))_1,-k\la(\xi(t))_2):(-\eps,\eps) \to V^\vee$$
consider the vector $\frac{d}{dt}|_{t=0} p(t)$ tangent to $\ell_k$. It's of the form
$$(c_1, c_2, -k\la(c)_1, -k\la(c)_2) \equiv c_1 \dd_{\xi_1} + c_2 \dd_{\xi_2} -k\la(c)_1 \dd_{\theta_1} - k\la(c)_2 \dd_{\theta_2}$$
Hence the tangent bundle $T\ell_k$ is spanned by the vectors with $c=(1,0)$ and $c=(0,1)$:
$$T\ell_k = \mb{R}\left<\dd_{\xi_1} -\frac{2k}{3}\dd_{\theta_1} +\frac{k}{3}\dd_{\theta_2}, \dd_{\xi_2}+\frac{k}{3}\dd_{\theta_1} -\frac{2k}{3} \dd_{\theta_2}\right>=:\mb{R}\left<X_{1,0}, X_{0,1}\right>$$ 

The symmetry of the matrix representing $\la$ implies that $d\xi\wedge d\theta$ of these two vectors is zero. 
\begin{equation}
\begin{aligned}
    \omega\left(X_{1,0}, X_{0,1}\right)&=\sum_{i=1,2}d\xi_i(X_{1,0})d\theta_i(X_{0,1}) - d\theta_i(X_{1,0})d\xi_i(X_{0,1})\\
    &=\left(1\cdot \frac{k}{3} - \frac{-2k}{3}\cdot 0 \right) + \left(0 \cdot \frac{-2k}{3} - \frac{k}{3}\cdot 1 \right)\\
    &=(k/3)-(k/3)=0
\end{aligned}
\end{equation}
and $\omega(X_{1,0}, X_{1,0}) = 0 = \omega(X_{0,1}, X_{0,1})$ by skew-symmetry of $\omega$. Thus
$$\omega|_{\ell_k} \equiv 0$$
and the $\ell_k$ are Lagrangian. Note that an alternative proof would be to show $\ell_k$ is Hamiltonian isotopic to $\ell_0$. 

\begin{center}
\fbox{{\bf Morphisms: geometric set-up}}    
\end{center}

A 4-torus is aspherical and there are no bigons between two of these linear Lagrangians on a 4-torus, which prohibits sphere bubbling and strip-breaking respectively. The latter statement follows because two planes in $\mb{R}^4$ which intersect in a finite number of points can only intersect in one point, and strip-breaking occurs on strips between two Lagrangians. These linear Lagrangians also do not bound discs, which excludes the remaining limiting behavior in Gromov compactness for discs with Lagrangian boundary condition, namely disc bubbling. 

\begin{center}
    \fbox{{\bf Transversely intersecting Lagrangians}}
\end{center}

Thus $HF(\ell_i, \ell_j) = CF(\ell_i, \ell_j)$ since the differential is zero. So Floer cohomology is freely generated by intersection points of $\ell_i$ and $\ell_j$.
 {\small \begingroup \allowdisplaybreaks \begin{equation}\label{eq:abel_symp_basis}
    \begin{aligned}
    \ell_i \cap \ell_j & =  \{(\xi_1, \xi_2,\theta_1,\theta_2) \in T_B \times T_F \mid (\theta_1,\theta_2) \equiv -i\la(\xi) \equiv -j\la(\xi) \mod \mb{Z}^2\}\\
    \iff \xi & \in (j-i)^{-1}\Gamma_B/\Gamma_B\\
    \therefore |\ell_i \cap \ell_j| & = (j-i)^2
    \end{aligned}
\end{equation} \endgroup
}
\begin{center}
    \fbox{{\bf Notation to be used throughout}}
\end{center} 

Notationally, we write a generic intersection point as
$$\left(\frac{\gamma_{i \cap j}}{j-i}, -i\la\left(\frac{\gamma_{i \cap j}}{j-i} \right) \right) \in \ell_i \cap \ell_j$$
where $\gamma_{i\cap j} = e_1 \gamma' + e_2 \gamma''$ for $0 \leq e_1,e_2 < (j-i)^2$ (since the lattice $\Gamma_B$ has index $(j-i)^2$ in $(j-i)^{-1}\Gamma_B$). Letting $l:=j-i$ and $1 \leq e \leq l^2$ index the possible choices of $(e_1,e_2)$ in $\gamma_{i \cap j}$, we write the collection of all the intersection points as
$$\left\{ \left(\frac{\gamma_{e,l}}{l}, -i\la\left(\frac{\gamma_{e,l}}{l} \right) \right) \right\}_{0 \leq e < (j-i)^2}$$
(Note that we could've taken $j$ instead of $i$ as the coefficient on the $\la$, and otherwise use the same notation.)

\begin{center}
    \fbox{{\bf Non-transversely intersecting Lagrangians}}
\end{center} 

The above holds when $i \neq j$. Since Lagrangians are half-dimensional, if they intersect transversely then their intersection is 0-dimensional and we have a discrete set of points. In the case $i=j$, we perturb $\ell_i$ by a Hamiltonian as in standard Floer theory. That is, we introduce a Hamiltonian function $H$ and flow one Lagrangian along the symplectically dual vector field $X_H$ to $dH$, until the two Lagrangians intersect transversely, c.f.~the introductory paper to Fukaya categories \cite{fuk_intro}. Furthermore, we want to take care how we perturb around the intersection points that are already transverse. We pick a 1-form $\phi$ defined to be zero near boundary punctures on the domain curve $\mb{D}$ where there is no problem and the intersection is transverse, and nonzero near punctures where the Lagrangians do not intersect transversally. Since we have moved one Lagrangian, we keep track of how this affects the Cauchy-Riemann equation. Holomorphic maps $u$ should satisfy the modified Cauchy-Riemann equation
\begin{equation}\label{eq:perturb_CR}
    (du - X_H \circ \phi)^{0,1} = 0
\end{equation}
with a modified asymptotic condition: at a puncture (the preimage in $\mb{D}$ of a perturbed intersection point) the map $u$ converges to a trajectory of $X_H$ from $\ell_i$ to $\ell_j$. So that the elements of hom sets are again intersection points, we can instead flow the image of $u$ along $X_H$ to obtain a Cauchy-Riemann equation $(du_H)^{0,1}$ with modified almost complex structure and boundary conditions. If we take the Hamiltonian to be a Morse function, then intersection points of the 0-section of $T^*\ell_i$ and the graph of $df$ are critical points of $f: T^2 \to T^2$. By Morse theory, $HF(\ell_i, \ell_i) \cong H(T^2)$.

\begin{center}
    \fbox{{\bf Existence of regularity, moduli spaces, independence of choices}}\label{page:tri_J0_reg}
\end{center} 

The standard complex structure $J_0 = i$ is regular. This is because the linearized $\ol{\dd}_{J_0}$ operator at a holomorphic map $u\in \pi_2(V^\vee, \cup_{i \in I} \ell_i=:L)$ for some index set $I$ is $\ol{\dd}$ on $\bigwedge^{(0,1)}((\mb{D}, \dd \mb{D}), (u^*TV^\vee,u^*TL))\subset \Omega^{(0,1)}(\mb{D})$ where smoothness follows from considering smooth maps $u$, by elliptic regularity. However, a Riemann surface has trivial $H^{2,0}$ since there are no $(2,0)$ forms (so no $(0,2)$ forms either). Thus $0=H^{(1,0)}(\mb{D}) = \ker(\ol{\dd})/\im(\ol{\dd})=\bigwedge^{1,0}(\mb{D})/\im(\ol{\dd})$ hence the image of the $\ol{\dd}$-operator is surjective. 

In particular, moduli spaces of $k$-pointed $i$-holomorphic discs are smooth orbifolds, which we can take the 0-dimensional part of and count. The structure map $\mu^k$ which inputs $k$ intersection points and outputs one intersection point, counts holomorphic polygons with vertices mapping to those intersection points with boundary on the corresponding intersecting Lagrangians. Furthermore:
 
\begin{lemma}\label{lem:regJ_torus} There exists a dense set $\mc{J}_{reg} \subset \mc{J}(V^\vee,\omega)$ of $\omega$-compatible almost complex structures $J$ such that, for all $J$-holomorphic maps $u: \mb{D} \to V^\vee$ with suitable Lagrangian boundary condition, the linearized $\ol{\dd}$-operator $D_u$ is surjective.\end{lemma}

We postpone the proof to our discussion of regularity below in more generality in Section \ref{section:q_invc_choices} for quasi-invariance of the Fukaya category on regular choices. However here we can make the stronger statement that the two $\mu^2$ counts are equal on the nose.

\begin{lemma} The $\mu_2$ for two regular almost-complex structures $J_1$ and $J_2$ on $V^\vee$ are equal. \end{lemma}

\begin{proof}[Outline] The proof will rely on arguments from the proof of the previous Lemma \ref{lem:regJ_torus}. The argument will be similar, but the Fredholm problem will have an additional $[0,1]$ factor in the Banach bundle setup. So we will obtain a 1-dimensional manifold.  There is no other boundary expected because 1) sphere bubbling cannot happen as $\pi_2(T^4) = 0$, furthermore 2) strip breaking would break off a bigon but there are no bigons between two linear Lagrangians on a torus (all intersection points have the same index since the Lagrangians have constant slope, whereas a broken strip would have intersection points with indices differing by two) and 3) disc bubbling doesn't occur because Lagrangians don't bound discs on a torus (since $\pi_2$ is preserved upon taking the universal cover of the torus which is $\mb{R}^4$ so has no $\pi_2$). Since the signed boundary of a 1-dimensional manifold is zero, we find that $\# \mc{M}(p_1,p_2,p_3,[u],J_1) = \# \mc{M}(p_1,p_2,p_3,[u],J_2)$ for regular $J_1$ and $J_2$. (These moduli spaces will be defined in Section \ref{defn}.) The existence of a dense set of regular paths is similar to the proof for the existence of regular $J$.
\end{proof}
 Since we consider $H^0Fuk(V^\vee)$ only here, it remains to compute the multiplication $\mu^2$. 

\begin{center}
\fbox{{\bf Counting triangles}}
\end{center}

We compute $\mu^2: CF(\ell_j, \ell_k) \otimes CF(\ell_i, \ell_j) \to CF(\ell_i, \ell_k)$. This will count $J$-holomorphic triangles between points $p_1,p_2,q$ as in Figure \ref{fig: triangle}, weighted by area. We can compute their area as they wrap around the abelian variety, by lifting to the universal cover.

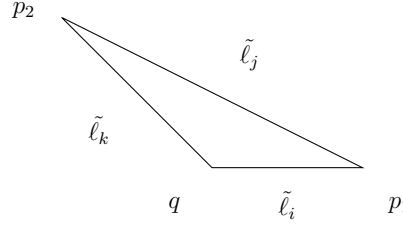
\begin{figure}[h]
\begin{tikzpicture}[every node/.style={inner sep=0,outer sep=0, scale = 0.8}]

\draw (-2.5,0.5) node (v1) {} -- (-0.5,-1.5) -- (1.5,-1.5) -- (v1) -- cycle;
\node at (-1,-2) {$q$};
\node at (2,-2) {$p_1$};
\node at (-3,0.6) {$p_2$};
\node at (0.5,-2) {$\tilde{\ell_i}$};
\node at (0,0) {$\tilde{\ell_j}$};
\node at (-2,-1) {$\tilde{\ell_k}$};
\end{tikzpicture}
    \caption{A triangle in $V^\vee$ contributing to $\mu^2$, viewed in $\xi_1,\xi_2$ plane in the universal cover $\mb{R}^4$}
    \label{fig: triangle}
\end{figure}

Use subscript $\xi$ and $\theta$ to denote the $(\xi_1,\xi_2)$ coordinates and $(\theta_1, \theta_2)$ coordinates respectively of a point in the universal cover $\mb{R}^4$ of $V^\vee$. Choose $k>j>i$. Fix $q=\left< \frac{\gamma_{k\cap i}}{k-i},-\frac{k}{k-i}\la(\gamma_{k\cap i}) \right>$ to be one of the $(k-i)^2$ intersection points in the fundamental domain. In particular, the sum of the three vectors around the triangle must be zero. Let the $\xi$ coordinates of these vectors be $\xi, \xi', \xi''$ respectively. Then setting their sum, and the sum of $\theta$-coordinates equal to zero:
\begin{equation}
    \begin{aligned}
    \xi + \xi'+\xi''& = 0\\
         i\xi +j\xi' -k(\xi+\xi')& =0\\
        \therefore \xi'& = -\frac{l}{l''}\xi\\
        \xi'' & = -\frac{l'}{l''}\xi
    \end{aligned}
\end{equation}

Now we apply the constraint that $p_1 \in \ell_i \cap \ell_j$ and $p_2 \in \ell_j \cap \ell_k$. The first constraint is equivalent to $-j\la({p_1}_{\xi}) \equiv {p_1}_\theta \pmod{\mb{Z}^2} = {p_1}_\theta + \la(\gamma_A)$ for some $\gamma_A \in \Gamma_B$, where the $A$ indicates we are on the symplectic side.
    \begin{allowdisplaybreaks}
    \begin{align*}  
        p_1 =({p_1}_\xi, {p_1}_\theta) & = \left(\frac{\gamma_{e,l}}{l} + \xi, - \la\left(k\frac{\gamma_{e,l}}{l} + i\xi\right) \right)\\
        \therefore -j\la({p_1}_\xi) & \equiv {p_1}_\theta\\
        \implies j\frac{\gamma_{e,l}}{l} + j\xi & = k \frac{\gamma_{e,l}}{l} + i \xi + \gamma_A \stepcounter{equation}\tag{\theequation}\label{eq:defn_xi}\\
        \implies \xi & =  \frac{l''}{ll'}\gamma_{e,l} + \frac{\gamma_A}{l'}
    \end{align*}
        \end{allowdisplaybreaks}

Let $\xi_0 :=\frac{l''}{ll'}\gamma_{e,l} + \frac{\gamma_A}{l'}$. To compute the area of the triangle, recall that the symplectic form in action-angle coordinates $\xi, \theta$ is $\omega = d \xi \wedge d \theta$. In particular, in the plane spanned by $\vec{u}:=\left<\xi_0, 0 \right>$ and $\vec{v}:=\left<0, \la(\xi_0)\right>$, with respect to these vectors we can write $\vec{qp_1}=\vec{u} - i \vec{v}$ and $\vec{qp_2}=\frac{l'}{l''}\vec{u} - k \cdot \frac{l'}{l''}\vec{v}$ so under the parametrization $\Psi: (a,b) \mapsto a\vec{u} + b\vec{v}$ we find that parallelogram $\vec{qp_1} \times \vec{qp_2}$ is $\Psi((1,-i) \times (\frac{l'}{l''}, -k \frac{l'}{l''}))$ hence
{\small \begingroup \allowdisplaybreaks 
\begin{align*}
    area(\Delta_{p_1p_2q}) &= \frac{1}{2}\int_{\vec{qp_1} \times \vec{qp_2}} d \xi \wedge d \theta \\
    & = \frac{1}{2}\left| (1,-i) \times (\frac{l'}{l''}, -k \frac{l'}{l''})\right| \int_{[0,1]^2}\Psi^* (d \xi \wedge d\theta)\\
    & = \frac{1}{2}[-kl'/l'' +il'/l'']\int_{[0,1]^2} d(a\vec{u} + b \vec{v})_\xi \wedge d(a\vec{u} + b \vec{v})_\theta\\
    &= -\frac{1}{2} \left(\frac{ll'}{l''} \right)\xi_0\cdot \la(\xi_0)\int_{[0,1]^2} da \wedge db \stepcounter{equation}\tag{\theequation}\label{eqn:compn_triangles}\\
    \xi \cdot \la(\xi) & = \left<\frac{l''}{ll'}\gamma_{e,l} + \frac{\gamma_A}{l'}, \la\left(\frac{l''}{ll'}\gamma_{e,l} + \frac{\gamma_A}{l'}\right) \right> = -\frac{2}{l'^2} \kappa\left(\frac{l''}{l}\gamma_{e,l} + \gamma_A\right)\\
    \therefore area(\Delta_{p_1p_2q}) & =  \frac{l}{l'l''}\cdot  \kappa\left(\frac{l''}{l}\gamma_{e,l} + \gamma_A\right)\\
    \implies \left<\mu^2(p_1,p_2), q\right> &= \sum_{\gamma_A \in \Gamma_B} \tau^{- \frac{l}{l'l''}\cdot  \kappa\left(\frac{l''}{l}\gamma_{e,l} + \gamma_A\right)}= \left< \mu^2(p_{e'',l''}, p_{e',l'}), p_{e,l}\right>
    \end{align*}
\endgroup}

\begin{center}
\fbox{{\bf Count of triangles computes $\mu_2$}}    
\end{center}

This concludes the triangle count computation. It remains to prove that these triangles are holomorphic for some regular $J$. In the basis $\vec{u},\vec{v}$ above, we can construct the standard $J$.

\begin{claim}\label{holo_tri_count} Let
$$J :=\begin{blockarray}{ccc}
& \xi & \theta \\
\begin{block}{c&(c&c)}
 \xi & 0 & -\left(\begin{matrix} 2 & 1 \\ 1 & 2\end{matrix}\right)  \\
 \theta & \left(\begin{matrix} 2 & 1 \\ 1 & 2\end{matrix}\right)^{-1} & 0 \\
\end{block}
\end{blockarray} \mbox{ }$$ Then (i) $J$ is a compatible almost complex structure, (ii) the triangles described above bounded by $\ell_i,\ell_j,\ell_k$ are $J$-holomorphic, and (iii) $J$ is regular. \end{claim}

\begin{proof}[Proof of Claim \ref{holo_tri_count}] 

%

(i) Let $M_{\la} := \left( \begin{matrix} 2& 1\\ 1 & 2 \end{matrix}\right)$. Then \begin{allowdisplaybreaks}
\begin{align*}
J^2 &=\left( \begin{matrix} 0 & -M_\la \\ M_\la^{-1} & 0 \end{matrix}\right)\left( \begin{matrix} 0 & -M_\la \\ M_\la^{-1} & 0 \end{matrix}\right) =\left(\begin{matrix}-I & 0 \\ 0 & -I \end{matrix} \right)\\
\omega(\cdot, J \cdot) &=(d\xi \wedge d\theta) (\cdot, J \cdot)\stepcounter{equation}\tag{\theequation} \\
&=\left( \begin{matrix}0 & I \\ -I & 0 \end {matrix} \right)\left(\begin{matrix} 0 & -\left(\begin{matrix} 2 & 1 \\ 1 & 2\end{matrix}\right)\\
\left(\begin{matrix} 2 & 1 \\ 1 & 2\end{matrix}\right)^{-1} & 0 \end{matrix} \right) = \left(\begin{matrix} \left(\begin{matrix} 2 & 1 \\ 1 & 2\end{matrix}\right)^{-1} & 0\\
0 & \left(\begin{matrix} 2 & 1 \\ 1 & 2\end{matrix}\right)^{-1}  \end{matrix} \right) >0
\end{align*}
\end{allowdisplaybreaks}

Thus $J^2 = -\bm{1}$ and $J$ is compatible with $\omega$, namely $\omega(\cdot, J \cdot)$ is a metric since the matrix above is positive definite.\newline

(ii) Recall the discussion before Equation (\ref{eqn:compn_triangles}). Taking the universal cover of $V^\vee$, we split up the resulting linear space into a product of two 2-planes. Let $P$ be the plane spanned by $\vec{u}=\left<\xi_0,0 \right>$ and $\vec{v} = \left<0,\la(\xi_0)\right>$ for the choice of $\xi_0$ right below Equation (\ref{eq:defn_xi}). Let $P^\omega$ be the symplectic orthogonal complement. In particular, $P$ and $P^\omega$ are $J$-holomorphic planes since $\la = \left(\begin{matrix} 2 & 1 \\ 1 & 2\end{matrix}\right)^{-1}$. We can view the universal cover as $P \times P^\omega \cong \mb{R}^4$ with respect to basis vectors $\vec{u}, \vec{v}$ for $P$ and $\vec{u^\omega}$, $\vec{v^\omega}$ for $P^\omega$. In this basis, we have a $J$-holomorphic disc in $P$ where $J$ is a complex structure (as every almost complex structure is integrable in two dimensions) with the specified Lagrangian boundary condition. \newline

(iii) All the triangles described above are regular, and they are the only discs that appear in the $\mu_2$ count as follows. The Lagrangians $\tilde{\ell}_i, \tilde{\ell}_j, \tilde{\ell}_k$ decompose as products of a straight line of slope $-i$ (respectively $-j$,$-k$) in $P$, and a straight line of the same slope in $P^\omega$. In $P$ we obtain the triangles previously considered and projected to $P^\omega$ the straight lines all meet in a single point. Therefore the maps in the moduli spaces for $\mu^2$, $u:\mb{D}^2 \to V^\vee$, have a triangle image in $P$ and are constant maps in $P^\omega$ at the triple intersection point of the three Lagrangians. (Namely, the boundary condition of the projection of the 3-punctured disc to the latter plane must be constant.) By the Riemann mapping theorem the discs we considered are unique. The discs are regular by the same argument on page \pageref{page:tri_J0_reg} that $J_0$ is regular.
\end{proof}

%

%


This concludes the proof of Lemma \ref{lem:fuk_subcat} that the linear Lagrangians and their morphisms define a full subcategory of the cohomological Fukaya category, and that the multiplicative structure is as stated in the lemma. \end{proof}

\subsection{The complex side}

We now set up the notation describing the abelian variety and its bounded derived category of coherent sheaves, needed to show that the map in the main Theorem \ref{theorem: me}, $D_{\mc{L}}^bCoh(V) \into H^0 Fuk(V^\vee)$ given by $\mc{L}^i \mapsto \ell_i$ defines a fully-faithful functor. 


It will be easier to compute cohomology of $V$ if we take log so the lattice acts additively. Taking the natural logarithm of $(x_1,x_2) \in V$, which is a locally holomorphic map, where we set $|x_i|=\tau^{\xi_i}$, we obtain coordinates
\begingroup \allowdisplaybreaks \begin{equation}
(\xi_1 \log \tau  + 2 \pi i \arg(x_1), \xi_2 \log \tau + 2\pi i \arg(x_2)) \in \mb{C}^2/(\log \tau) \Gamma_B + 2\pi i \mb{Z}^2  
\label{equation: V+}
\end{equation} \endgroup

Define $\Gamma_F = \mb{Z}^2$ then $V_+:= \mb{C}^2/ \log \tau\Gamma_B + 2\pi i \Gamma_F$. Then the lattice $\Gamma_B$ now acts by subtraction on $\xi$ and $\Gamma_F$ acts by addition on $\frac{1}{2\pi i}\arg(x)$. Let $\Gamma:=(\log \tau) \Gamma_B + 2\pi i \mb{Z}^2$. Note that topologically we can express $V_+$ as a product of tori:
$$V_+ \cong \mb{R}^2_{\xi_1,\xi_2}/\Gamma_B \times \mb{R}^2_{\theta_{x_1}, \theta_{x_2}}/\Gamma_F = T_B \times T_F$$
where $\theta_{x_i} := \frac{1}{2\pi} \arg(x_i)$. (However, as an abelian variety, $V \cong V_+$ are not product elliptic curves.) Namely, we can identify $V$ and $V_+\cong T_B\times T_F$ by taking $\log$ to go from $V$ to $V_+$ and then $(\frac{1}{\log \tau} \Re(\cdot), \frac{1}{2\pi} \Im(\cdot ))$ to get from $V_+$ to $T_B \times T_F$, see Equation (\ref{eq: add_and_mult_coords}) below. So $V$ is diffeomorphic to $T_B \times T_F$, allowing us to compute its cohomology and choice for $c_1(\mc{L})$ in terms of $T_B \times T_F$, however they are not the same as complex manifolds.
\begin{equation}\label{eq: add_and_mult_coords}
(x_1,x_2)=(\tau^{\xi_1}e^{2\pi i \theta_{x_1}},x_1=\tau^{\xi_2}e^{2\pi i \theta_{x_2}}) \mapsto \ \ \ (\xi_j, \theta_{x_j})= \left(\frac{1}{\log \tau} \log |x_j|, \frac{1}{2\pi i} (\log x_j-\log |x_j|) \right)
\end{equation}

The following collection of facts and definitions allows us to prove fully-faithfulness of the functor on a chosen basis of hom groups on the A- and B-sides.




\begin{claim}[{\cite[\textsection 1]{pol_bk}}] Since $H^n(V; \mb{Z}) \cong \wedge^n \Hom(\Gamma, \mb{Z})$, complex line bundles on $V$ are topologically classified by their first Chern class, which is equivalent to a skew-symmetric bilinear form $E: \Gamma \times \Gamma \to \mb{Z}$.
\end{claim}




\begin{claim}[{\cite[Appendix B]{cx_abel}}] Holomorphic line bundles on $V_+$ are classified by 
\begingroup \allowdisplaybreaks \begin{equation}
H^1(\pi_1(V_+); H^0(\mc{O}^*_{\tilde{V}_+})) \cong H^1(V_+,\mc{O}^*)
\label{equation: automor}
\end{equation} \endgroup
where $\tilde{V}_+= \mb{C}^2$.
\end{claim}

\begin{claim}[{\cite[Lemma 1.2.5]{huy}, \cite[Theorem 1.4.1]{cx_abel}}]\label{claim: holo_abel_str} $$H^{1,1}(V_+) \cong \Hom_{\mb{C}}(\mb{C}^2,\mb{C}) \otimes \Hom_{\ol{\mb{C}}}(\mb{C}^2,\mb{C})$$
\end{claim}


\begin{claim}[C.f~{\cite[Proposition 2.1.6]{cx_abel}}]  A complex line bundle $\mc{L}$ with first Chern class $c_1(\mc{L})=E$ admits a holomorphic structure if and only if $H(\cdot,\cdot):= E(i( \cdot), \cdot) + i E(\cdot, \cdot)$ is Hermitian. 
\label{cor: param_lines}
\end{claim}





\begin{theorem}[{\cite[Theorem 1.3]{pol_bk}} and {\cite[Appell-Humbert Theorem 2.2.3]{cx_abel}}]\label{main_cx_lemma} The Picard group of $V_+$ can be classified by the following set of pairs.
\begingroup \allowdisplaybreaks \begin{equation}
    \begin{aligned}
    Pic(V_+)\cong & \{(H,\alpha) \mid H: \mb{C}^2 \times \mb{C}^2\to \mb{C} \text{ Hermitian}, E(\Gamma,\Gamma) \subseteq \mb{Z}, \alpha: \Gamma \to U(1)\\
    &\alpha(\gamma + \tilde{\gamma}) = \exp(\pi i E(\gamma, \tilde{\gamma})) \alpha(\gamma) \alpha(\tilde{\gamma}), \text{where } E = \im H \}
    \end{aligned}
\end{equation} \endgroup
\end{theorem}

\begin{cor}\label{cor:v_vs_vs+}  Holomorphic line bundles on $V$ are in one-to-one correspondence with pairs $(H,\alpha)\in Pic(V_+)$ (see Theorem \ref{main_cx_lemma}) where $H$ can be represented as a real integral symmetric $2 \times 2$ matrix, under the pullback by $\exp$.
\end{cor}


\begin{proof}  Recall $E:=\im H$ for some hermitian form $H = \left( \begin{matrix} a & b \\ \ol{b} & a \end{matrix} \right)$ so $a \in \mb{R}$. 
\begingroup \allowdisplaybreaks \begin{equation}
\begin{aligned}
& (\begin{matrix} \gamma_1 + i w_1 & \gamma_2 + iw_2\end{matrix}) \left( \begin{matrix} a & b+ic \\ b-ic & d \end{matrix} \right) \left( \begin{matrix} \tilde{\gamma}_1 + iv_1 \\ \tilde{\gamma}_2 + iv_2 \end{matrix} \right)
\end{aligned}
\end{equation} \endgroup
We require that $E$ is trivial in the $T_F$ directions, because under $\exp$ the $T_F$ directions are already quotiented by in $(\mb{C}^*)^2$. These are the purely imaginary vectors in $\mb{C}^2$ by Equation \ref{equation: V+}, so we want $E(iw, iv) = 0$. We find that 
\begin{align*}
E(iw,iv) &= \im H(iw, iv) = \im(aw_1v_1 + bw_1v_2+icw_1v_2+bw_2v_1 + dw_2v_2 - icv_1w_2)\\
& = c(w_1v_2 - v_1 w_2)=0 \; \forall w,v \in \mb{Z}^2 \therefore c = 0
\end{align*}
so $H$ is real hence symmetric. We can express $H$ on $\mb{R}^4$ as 
$H =\begin{blockarray}{ccc}
& \Gamma_B & \Gamma_F \\
\begin{block}{c&(c&c)}
 \Gamma_B & 0 & A  \\
 \Gamma_F & A & 0 \\
\end{block}
\end{blockarray} \mbox{ }$ where $A$ is real symmetric. Note that the information of $H$ is that of a homomorphism $\Gamma_B \to \Gamma_F$.
\end{proof}


Hom groups on the B-side have bases represented by sections of powers of $\mc{L}$.

\begin{claim}\label{defn:line_bundle} $s:(\gamma_\xi,x) \mapsto x^{\la(\gamma_\xi)} \tau^{\kappa(\gamma_\xi)}$ is a factor of automorphy for $\mc{L}$ where recall $\la$ and $\kappa$ are defined in Definition \ref{defn:Bside_notation}.
\end{claim}

\begin{proof} We show its pullback to $V_+$ is a factor of automorphy which is trivial in the $\Gamma_F$ directions. 
\begin{center}
    \begin{tikzpicture}
\draw[<-] (-0.5,-0.5) node (v1) {} .. controls (0,-1) and (0,-1.5) .. (-0.5,-2);
\node (v3) at (-1,0) {$\mathcal{L}$};
\node (v5) at (-1,-2) {$V$};
\node (v2) at (-3,0) {$\exp^* \mathcal{L}$};
\node (v4) at (-3,-2) {$V_+$};
\draw [->] (v2) edge (v3);
\draw [->] (v4) edge (v5);
\draw [->] (v2) edge (v4);
\draw [->] (v3) edge (v5);
\draw [->](-3.5,-2) .. controls (-4,-1.5) and (-4,-1) .. (-3.5,-0.5);
\node at (-4.5,-1) {$\exp^*s$};
\node at (0.1,-1) {$s$};
\end{tikzpicture}
\end{center}
We work in holomorphic coordinates on $V$ and $V_+$. Let $v:=(\log \tau) \xi + 2\pi i \theta$ be the coordinate on $V_+$, using the notation above. Let $\gamma =  \gamma_\xi + i \gamma_\theta$. Then

\begin{equation}
    \begin{aligned}
        \exp^*s : (\gamma, v)& \mapsto  \exp( (\la(\gamma_\xi) \log \tau) \xi +2\pi i \la(\gamma_\xi) \theta+ \kappa(\gamma_\xi) \log \tau ) \\
        &= \exp(\la(\Re \gamma) \cdot v + \kappa(\Re \gamma) \log \tau)
    \end{aligned}
\end{equation}
for a linear term in $\gamma$ and a quadratic term in $\gamma$, which matches Theorem \ref{main_cx_lemma} with $\alpha \equiv 1$ (else $\alpha$ would contribute a linear term to $\kappa$.) This is because recall from Theorem \ref{main_cx_lemma} that a pair $(H,\alpha)$ gives rise to a factor of automorphy on $V_+$ by 
\begin{equation}
    (\gamma, v)  \mapsto \alpha(\gamma) \exp(\pi H(v,\gamma) + \frac{\pi}{2}H(\gamma,\gamma))
\end{equation}

By the proof of Corollary \ref{cor:v_vs_vs+}, we found $H=iE$ where $E$ is a real 2 by 2 symmetric matrix. Thus we may define $H$ in terms of the map $\la: \Gamma_B \to \Gamma_F$ so that:
\begingroup \allowdisplaybreaks \begin{equation}
    \begin{aligned}
(\gamma, v) & \mapsto \alpha(\gamma) \exp(\pi H(v,\gamma) + \frac{\pi}{2}H(\gamma,\gamma))\\
&= \exp(\left< v, \la(\gamma_\xi)\right> +\kappa(\gamma_\xi)\log \tau)\\
& =  \exp(\left< (\log \tau) \xi + 2\pi i \theta, \la(\gamma_\xi)\right> +\kappa(\gamma_\xi)\log \tau)\\
&=x^{\la(\gamma_\xi)}\tau^{\kappa(\gamma_\xi)}
    \end{aligned}
\end{equation} \endgroup
So the condition that the line bundle on $V_+$ passes to one on $V$ is the condition that $H=iE$, since under exponentiation the $\Gamma_F$ action becomes multiplication by $e^{2\pi i n}=1$ for some $n \in \mb{Z}$.

\end{proof}

We now use $\gamma$ instead of $\gamma_\xi$ to denote group elements of $\Gamma_B$.

\begin{claim} Sections of holomorphic lines bundles on $V$ are functions on $(\mb{C}^*)^2$ with the periodicity property
$$s(\gamma\cdot x) = \tau^{\kappa(\gamma)}x^{\la(\gamma)}s(x)$$
so have a Fourier expansion.
 \end{claim}

\begin{proof}
A section $s:V \to \mc{L}$ must have the same transition functions as the line bundle, by considering the Cartier data. 
$$s(\gamma \cdot x)/s(x) = \tau^{\kappa(\gamma)} x^{\la(\gamma)}$$

\end{proof}

\begin{cor}\label{theta_basis} Let $\mc{L}$ be the line bundle defined above in Claim \ref{defn:line_bundle}. Then using the notation from the proof of Lemma \ref{lem:fuk_subcat}, $H^0(V,\mc{L}^{\otimes l})$ has the following basis of sections:
\begingroup \allowdisplaybreaks \begin{equation}
    s_{e,l}:= \sum_\gamma \tau^{-l \kappa(\gamma + \frac{\gamma_{e,l}}{l})} x^{-l\la(\gamma) - \la(\gamma_{e,l})}
\end{equation} \endgroup
where $\gamma_{e,l} = {e_1}\gamma' + {e_2} \gamma'', \;\; 0 \leq e_1,e_2 < l$. So $h^0(\mc{L}^l) =l^2$, e.g.~$\mc{L}$ is a degree 1 line bundle.
\end{cor}

\begin{proof}
Tensoring the line bundle $l$ times means we multiply the transition function of $\mc{L}$ by $l$ times. In particular, the exponents now add in $\la$ and $\kappa$. So equivalently, we could scale the lattice $\Gamma_B$ to $l \Gamma_B$ and note that the quotient has $l^2$ lattice points. If we think of the parallelogram in $\Gamma_B$ of length $l$ in the $\gamma'$ and $\gamma''$ directions, then unique lattice points index the sections. So the functions in the statement of this Corollary are $l^2$ linearly independent sections with the same transition functions as $\mc{L}^{\otimes l}$.
\end{proof}

\subsection{The fully faithful functor}\label{section:fff_T4} Now we show that $D^b_{\mc{L}}Coh(V)\ni \mc{L}^{\otimes k} \mapsto \ell_k \in H^*Fuk(V^\vee)$ is a functor, by showing it respects composition on elements of a basis.

\begin{center}
\fbox{{\bf Basis on complex side}}    
\end{center}

Recall that $D^b_{\mc{L}} Coh(V)$ is defined in Theorem \ref{theorem: me} to be generated by powers of $\mc{L}$, and for $j>i$ $\Hom(\mc{L}^i, \mc{L}^j) \cong H^0(\mc{O}, \mc{L}^{j-i})$ (see \cite{zp} for the case of line bundles on an elliptic curve). Set $\tilde{l}:=j-i$, $\tilde{\tilde{l}}:=k-j$, and $l:=\tilde{l} + \tilde{\tilde{l}}=k-i$.  Recall that Corollary \ref{theta_basis} gives a basis of sections of $H^0(V,\mc{L}^{\otimes l})$:
$$s_{e,l}:= \sum_\gamma \tau^{-l \kappa(\gamma + \frac{\gamma_{e,l}}{l})} x^{-l\la(\gamma) - \la(\gamma_{e,l})}$$
where $\gamma_{e,l} = {e_1}\gamma' + {e_2} \gamma'', \;\; 0 \leq e_1,e_2 < l$.

\begin{center}
\fbox{\bf Basis on symplectic side}    
\end{center} 

On the symplectic side, we consider a basis of $\Hom_{V^\vee}(\ell_i,\ell_j) = \bigoplus_{p \in \ell_i \cap \ell_j} \mb{C}\cdot p$ given by the $(j-i)^2 = l^2$ intersection points of Equation \ref{eq:abel_symp_basis}:
\begin{equation}\label{def:symp_basis_fiber}
\begin{aligned}
    p_{e,l} := \left(\frac{\gamma_{e,l}}{l}, -k \la \left( \frac{\gamma_{e,l}}{l} \right) \right)
\end{aligned}
\end{equation}
where again $\gamma_{e,l} = {e_1}\gamma' + {e_2} \gamma'', \;\; 0 \leq e_1,e_2 < l$. These morphism groups have the same dimension as vector spaces. It remains to show there is a functor.

\begin{remark}This section is a bit notationally heavy so we collect the notations in this remark: 
\begin{compactitem}[\textbullet]
\item $\gamma'$ and $\gamma''$ form a basis for $\Gamma_B$
\item $e$ indexes the intersection points of two Lagrangians
\item One tilde corresponds to $j-i = \tilde{l}$, two tildes corresponds to $k-j=\tilde{\tilde{l}}$, and no tilde corresponds to the indexing lattice element once we've multiplied together and are considering $l=k-i$.
\item $s$ denotes sections and $p$ denotes intersection points
\end{compactitem}
\end{remark}

\begin{example} E.g.~for $i=0, j=1$, and $k=2$ we have $\mc{O}, \mc{L}, \mc{L}^2$ and $\ell_0, \ell_1, \ell_2$ with maps between objects and between homs as follows:
\begin{align*}
& hom(\mc{O},\mc{L}) = H^0(V,\mc{L}) \ni s_{1,1}  \mapsto p_{1,1}\in hom(\ell_0,\ell_1) = \bigoplus_{p \in \ell_0 \cap \ell_1} \mb{C}\cdot p=\mb{C}\cdot p_{1,1}\\
& hom(\mc{O},\mc{L}^2)=H^0(V,\mc{L}^2) \ni (s_{1,2},s_{2,2},s_{3,2},s_{4,2})  \mapsto (p_{1,2},p_{2,2},p_{3,2},p_{4,2}) \in \bigoplus_{p \in \ell_0 \cap \ell_2} \mb{C} \cdot p, \;\; |\ell_0 \cap \ell_2| = 4\\
& hom(\mc{L},\mc{L}^2)\cong hom(\mc{O}, \mc{L}^* \otimes \mc{L}^2) \cong hom(\mc{O}, \mc{L}) = H^0(V,\mc{L}) \ni s_{1,1}  \mapsto p_{1,1} \in hom(\ell_1,\ell_2) = \mb{C} \cdot p_{1,1}
\end{align*}
\end{example}

In particular we define the map to send units to units. The statement that this map respects composition is the following.

\begin{lemma}\label{functor_ok} The left vertical map of Theorem \ref{theorem: me}, $D^b_{\mc{L}}Coh(V)\ni \mc{L}^{\otimes k} \mapsto \ell_k \in H^*Fuk(V^\vee)$, respects composition so is a functor. Namely, for $s_{e,l}$ and $p_{e,l}$ bases defined as above:
\begin{equation}
    \begin{aligned}
    HMS(s_{e,l}) &=p_{e,l}\\
    HMS(s_{\tilde{\tilde{e}},\tilde{\tilde{l}}} \cdot s_{\tilde{e},\tilde{l}}) &=  HMS(s_{\tilde{\tilde{e}},\tilde{\tilde{l}}}) \cdot HMS(s_{\tilde{e},\tilde{l}})  =  p_{\tilde{\tilde{e}},\tilde{\tilde{l}}} \cdot p_{\tilde{e},\tilde{l}} \\
\iff s_{\tilde{\tilde{e}},\tilde{\tilde{l}}} \cdot s_{\tilde{e},\tilde{l}}&=\sum_{e \in \mb{Z}^2/l\mb{Z}^2} C_e \cdot s_{e,l}\\
        p_{\tilde{\tilde{e}},\tilde{\tilde{l}}} \cdot p_{\tilde{e},\tilde{l}} & = \sum_{e \in \mb{Z}^2/l\mb{Z}^2} C_e \cdot p_{e,l}
    \end{aligned}
\end{equation}
for the same $C_e$.
\end{lemma}

\begin{proof} First note that $C_e$ was computed above in the count of triangles in Equation (\ref{eqn:compn_triangles}). On the other hand, multiplying the theta functions gives
\begin{equation}
\begin{aligned}
s_{\tilde{\tilde{e}},\tilde{\tilde{l}}} \cdot s_{\tilde{e},\tilde{l}} & =  \sum_{\tilde{\gamma},\tilde{\tilde{\gamma}}} \tau^{-\tilde{l} \kappa(\tilde{\gamma} + \frac{\gamma_{\tilde{e},\tilde{l}}}{\tilde{l}}) - \tilde{\tilde{l}}\kappa(\tilde{\tilde{\gamma}} + \frac{\gamma_{\tilde{\tilde{e}},\tilde{\tilde{l}}}}{\tilde{\tilde{l}}})}x^{-\la(\tilde{l}\tilde{\gamma} + \gamma_{\tilde{e},\tilde{l}} + \tilde{\tilde{l}}\tilde{\tilde{\gamma}} + \gamma_{\tilde{\tilde{e}},\tilde{\tilde{l}}})}
\end{aligned}\label{eq:compare_multn_counts}
\end{equation}

We want to find new variables $(\gamma,\gamma_A)$ to sum over so that we can factor out the bases $s_{e,l}$. In particular, $-\la(l\gamma+ \gamma_{e,l})$ must be the exponent on $x$. So we want this product to equal $\sum_e C_e \sum_\gamma \tau^{-l\kappa(\gamma + \frac{\gamma_{e,l}}{l})} x^{-\la(l\gamma + \gamma_{e,l})} =\sum_e C_e s_{e,l}$. The other factor that sums over $\gamma_A$ will arise from counting triangles on the A-side, hence the subscript $A$. Define
\begin{equation}\label{eq:new_var1}
    \begin{aligned}
        l\gamma + \gamma_{e,l}& :=  \tilde{l}\tilde{\gamma} + \gamma_{\tilde{e},\tilde{l}} + \tilde{\tilde{l}}\tilde{\tilde{\gamma}} + \gamma_{\tilde{\tilde{e}},\tilde{\tilde{l}}}
    \end{aligned}
\end{equation}
If we sum over $\gamma$ and $1 \leq e \leq l^2$, we will obtain some of the lattice $\Gamma_B\times \Gamma_B \ni (\tilde{\gamma},\tilde{\tilde{\gamma}})$. Given $\gamma$ and $e$, there are multiple corresponding solutions in $(\tilde{\gamma}, \tilde{\tilde{\gamma}})$. We need another variable. We do a weighted version of the change of coordinates $(u,v) \mapsto ((u+v)/2, (u-v)/2)$. Namely let $u:= \gamma + \frac{\gamma_{e,l}}{l} = \frac{1}{l}(\tilde{l}\tilde{\gamma} + \gamma_{\tilde{e},\tilde{l}} + \tilde{\tilde{l}}\tilde{\tilde{\gamma}} + \gamma_{\tilde{\tilde{e}},\tilde{\tilde{l}}})$. We want to find $v$ such that 
\begin{equation}
\begin{aligned}
    u+c_1v & = \tilde{\gamma} +\frac{\gamma_{\tilde{e},\tilde{l}}}{\tilde{l}}\\
    u-c_2v &= \tilde{\tilde{\gamma}}+\frac{\gamma_{\tilde{\tilde{e}}, \tilde{\tilde{l}}}}{\tilde{\tilde{l}}}
    \end{aligned}
\end{equation}
for some constant $c_1$ and $c_2$ such that the $v$ terms cancel when we multiply the first equation by $\tilde{l}$ and add it to the second equation multiplied by $\tilde{\tilde{l}}$. In other words, $c_1\tilde{l} - c_2 \tilde{\tilde{l}}=0$. So take $c_1=\tilde{\tilde{l}}$ and $c_2 = \tilde{l}$. Then we can simplify the exponent on $\tau$ in Equation \ref{eq:compare_multn_counts}:
\begin{equation}\label{eq:main_transformn}
    \begin{aligned}
        \tilde{l}\kappa(u+\tilde{\tilde{l}}v) + \tilde{\tilde{l}} \kappa(u-\tilde{l}v) & = l \kappa (u) + \tilde{l}\tilde{\tilde{l}} \cdot l \kappa(v)
    \end{aligned}
\end{equation}
since $lv = \tilde{\gamma} +\frac{\gamma_{\tilde{e},\tilde{l}}}{\tilde{l}} - \tilde{\tilde{\gamma}}-\frac{\gamma_{\tilde{\tilde{e}}, \tilde{\tilde{l}}}}{\tilde{\tilde{l}}}$. Thus we now can factor out $l\kappa(u)$ as needed to obtain $s_{e,l}$ when summing over $\gamma$. On the other hand, recall from Equation (\ref{eqn:compn_triangles}) that
$$p_{\tilde{e},\tilde{l}} \cdot p_{\tilde{\tilde{e}}, \tilde{\tilde{l}}} = \sum_ e \sum_{\gamma_A \in \Gamma_B} \tau^{- \frac{l}{\tilde{l}\tilde{\tilde{l}}}\cdot  \kappa\left(\frac{\tilde{\tilde{l}}}{l}\gamma_{e,l} + \gamma_A\right)}\cdot p_{e,l}$$
That is $C_e = \sum_{\gamma_A \in \Gamma_B} \tau^{- \frac{l}{\tilde{l}\tilde{\tilde{l}}}\cdot  \kappa\left(\frac{\tilde{\tilde{l}}}{l}\gamma_{e,l} + \gamma_A\right)}$. So in order for the functor to respect composition, we would like this to be the coefficient on the $s_{e,l}$ as well. Comparing exponents on $\tau$ implies:
\begin{equation}\label{eqn:rotate_coords_functor}
    \frac{l}{\tilde{l}\tilde{\tilde{l}}}\cdot  \kappa\left(\frac{\tilde{\tilde{l}}}{l}\gamma_{e,l} + \gamma_A\right)= \tilde{l}\tilde{\tilde{l}} \cdot l \kappa(v)
\end{equation}
In other words, multiplying by $\tilde{l}\tilde{\tilde{l}}$ and equating the arguments of $l\kappa$:
\begin{equation}
        \frac{\tilde{\tilde{l}}}{l}\gamma_{e,l} + \gamma_A = \tilde{l}\tilde{\tilde{l}} \cdot \frac{1}{l}\left(\tilde{\gamma} + \frac{\gamma_{\tilde{e},\tilde{l}}}{\tilde{l}} - \tilde{\tilde{\gamma}} - \frac{\gamma_{\tilde{\tilde{e}}, \tilde{\tilde{l}}}}{\tilde{\tilde{l}}} \right) \iff \gamma_A = \tilde{l}\tilde{\tilde{l}} \cdot \frac{1}{l}\left(\tilde{\gamma} + \frac{\gamma_{\tilde{e},\tilde{l}}}{\tilde{l}} - \tilde{\tilde{\gamma}} - \frac{\gamma_{\tilde{\tilde{e}}, \tilde{\tilde{l}}}}{\tilde{\tilde{l}}} \right) -  \frac{\tilde{\tilde{l}}}{l}\gamma_{e,l} 
\end{equation}

Recall that $l\gamma + \gamma_{e,l} =\tilde{l}\tilde{\gamma} + \gamma_{\tilde{e},\tilde{l}} + \tilde{\tilde{l}}\tilde{\tilde{\gamma}} + \gamma_{\tilde{\tilde{e}},\tilde{\tilde{l}}} $. Thus simplifying we find that:
\begin{equation}
    \begin{aligned}
        \gamma_A & = \frac{\tilde{l}\tilde{\tilde{l}}}{l}\left(\tilde{\gamma} + \frac{\gamma_{\tilde{e},\tilde{l}}}{\tilde{l}}  - \frac{1}{\tilde{\tilde{l}}} \left(l \gamma + \gamma_{e,l} - \tilde{l}\tilde{\gamma} - \gamma_{\tilde{e}, \tilde{l}} \right) \right)-\frac{\tilde{\tilde{l}}}{l}\gamma_{e,l} \\
        & = \frac{\tilde{l}\tilde{\tilde{l}}}{l}\left(\tilde{\gamma} (1+\tilde{l}/\tilde{\tilde{l}}) + \gamma_{\tilde{e},\tilde{l}} (1/\tilde{l} + 1/\tilde{\tilde{l}})-(l/\tilde{\tilde{l}})\gamma \right) -\gamma_{e,l}(\tilde{l}+\tilde{\tilde{l}})/l\\
        &=\tilde{l}(\tilde{\gamma}-\gamma) + \gamma_{\tilde{e},\tilde{l}} -\gamma_{e,l} \in \Gamma_B
    \end{aligned}
\end{equation}
so $\gamma_A \in \Gamma_B$ as we would like. Hence using Equation (\ref{eqn:rotate_coords_functor}):

\begin{equation}\label{eq:theta_prod}
\begin{aligned}
s_{\tilde{e},\tilde{l}} \cdot s_{\tilde{\tilde{e}},\tilde{\tilde{l}}} & = \sum_{\tilde{\gamma},\tilde{\tilde{\gamma}}} \tau^{-\tilde{l} \kappa(\tilde{\gamma} + \frac{\gamma_{\tilde{e},\tilde{l}}}{\tilde{l}}) - \tilde{\tilde{l}}\kappa(\tilde{\tilde{\gamma}} + \frac{\gamma_{\tilde{\tilde{e}},\tilde{\tilde{l}}}}{\tilde{\tilde{l}}})}x^{-\la(\tilde{l}\tilde{\gamma} + \gamma_{\tilde{e},\tilde{l}} + \tilde{\tilde{l}}\tilde{\tilde{\gamma}} + \gamma_{\tilde{\tilde{e}},\tilde{\tilde{l}}})}\\
& =  \sum_e\sum_{\gamma_A} \tau^{-\frac{l}{\tilde{l}\tilde{\tilde{l}}} \kappa\left( \frac{\tilde{\tilde{l}}}{l}\gamma_{e,l} + \gamma_A\right)} \sum_{\gamma} \tau^{-l \kappa({\gamma} + \frac{\gamma_{e,l}}{l}) }x^{-\la(l{\gamma} + \gamma_{e,l}) }\\
& =\sum_e \left( \sum_{\gamma_A \in \Gamma_B} \tau^{-\frac{l}{\tilde{l}\tilde{\tilde{l}}} \kappa\left( \frac{\tilde{\tilde{l}}}{l}\gamma_{e,l} + \gamma_A\right)}\right) s_{e,l}
\end{aligned}
\end{equation}

So we see that the two coefficients on the basis elements agree between multiplication of sections and of intersection points, hence composition is respected, and we do indeed have a functor.
\end{proof}

This completes the proof of the left vertical arrow in the main Theorem \ref{theorem: me}. We proceed to the proof that the right vertical arrow is a fully-faithful embedding. First we need to define the symplectic fibration $(Y,v_0)$.

\section{Construction of symplectic fibration on $(Y,v_0)$}\label{section: Y_Aside}\label{section: toric_recap}

The definition of the symplectic form on the genus 2 SYZ mirror $(Y,v_0)$ arises from standard algebro-geometric results on toric varieties. The subsection \cite[\textsection 3.1]{liu_hms_toric} provides a concise summary ($M$ encodes the algebra and $N$ encodes the geometry). Other references include \cite[p 59, p128]{cls}, \cite[Chapter 1]{fulton}, and \cite{Gu94}. We have elements of $M := \Hom_\mb{Z}(N, \mb{Z})$ called \emph{weight vectors} which exponentiate to functions on the toric variety, called \emph{toric monomials} or \emph{characters}. Given a polytope $\Delta$, we can construct a dual fan which prescribes the charts and transition functions defining the toric variety (\cite[p 76]{cls}). We use the symplectic form obtained on $\mb{CP}^2$ in the following manner for our later calculation defining a symplectic fibration on $(Y,v_0)$. 

\begin{lemma}[Definition of toric symplectic form, {\cite{cls}[Proposition 4.3.3]},{\cite{cls}[Proposition 6.1.1]}, and {\cite{huy}[Example 4.1.2]}]\label{toric_metric}  Let polytope $\Delta:=\{m \in M_\mb{R} \mid \left<m, u_i\right> \geq -a_i\}\subseteq M_\mb{R}$. Then this defines basepoint-free line bundle $\mc{O}(D_\Delta)$ where $D_\Delta=\sum_{F \mbox{\tiny{ facet}}} a_F D_F$ is a weighted sum of the toric divisors. Since its sections are $\Gamma(Y_\Delta,\mc{O}(D_\Delta)) = \bigoplus_{m \in \Delta \cap M}  \mb{C} \cdot \chi^m$, we can define a {K\"ahler form} by
$$\omega_\Delta:= \frac{i}{2\pi} \ol{\dd}\dd h, \qquad h = \frac{1}{\sum_{i=1}^s |\chi^{m_i}|^2}$$
where $|\cdot|$ refers to the standard norm in $\mb{C}$ and $s=|\Delta \cap M|$.
\end{lemma}


\begin{example}[Complex $\mb{CP}^2$] \label{ex:cp2fan} The fan $\Sigma$ has three cones from three rays $\rho_1=\mb{R}_+f_1,\rho_2=\mb{R}_+f_2, \rho_3=\mb{R}_+ \cdot (-f_1-f_2)$ as in the left diagram of Figure \ref{cp2}, and dual cones $\sigma_1^\vee =span(e_1,e_2)$, $\sigma_2^\vee  = span(-e_1, -e_1+e_2)$, $\sigma_3^\vee  = span(-e_2, e_1-e_2)$ depicted in color on the right for $e_1, e_2$ standard basis vectors on $M$ and $f_1,f_2$ standard basis vectors on $N$. Each chart is a copy of $\mb{C}^2$. The choice of generator in each case gives a complex coordinate on the chart, e.g.~$U_{\sigma_1}  = \Spec \mb{C}[\chi^{1,0}, \chi^{0,1}] \cong \mb{C}^2$. If $\tau = \sigma_1 \cap \sigma_2$, then we have a coordinate change from inverting $\chi^{1,0}$ since $\mb{C}[S_\tau] = \mb{C}[\chi^{\pm (1,0)}, \chi^{(0,1)}] =  \mb{C}[\chi^{\pm (1,0)}, \chi^{(-1,1)}]$. (Each choice of coordinates on $U_\tau$ identifies it with $\mb{C}^* \times \mb{C}$ included as the identity map into $\mb{C}^2$ in the two charts $U_{\sigma_i}$.) This coordinate change is the transition map.
$$g_{12}(\chi^{(1,0)}, \chi^{(0,1)}) := \left( \frac{1}{\chi^{(1,0)}}, \frac{\chi^{(0,1)}}{\chi^{(1,0)}}\right)$$
This recovers how we think about $\mb{CP}^2\ni [z_0:z_1:z_2]$ with $\chi^{(1,0)} =z_1/z_0$ and $\chi^{(0,1)}= z_2/z_0$. 
\begin{figure}[h]
\includegraphics[scale=0.2]{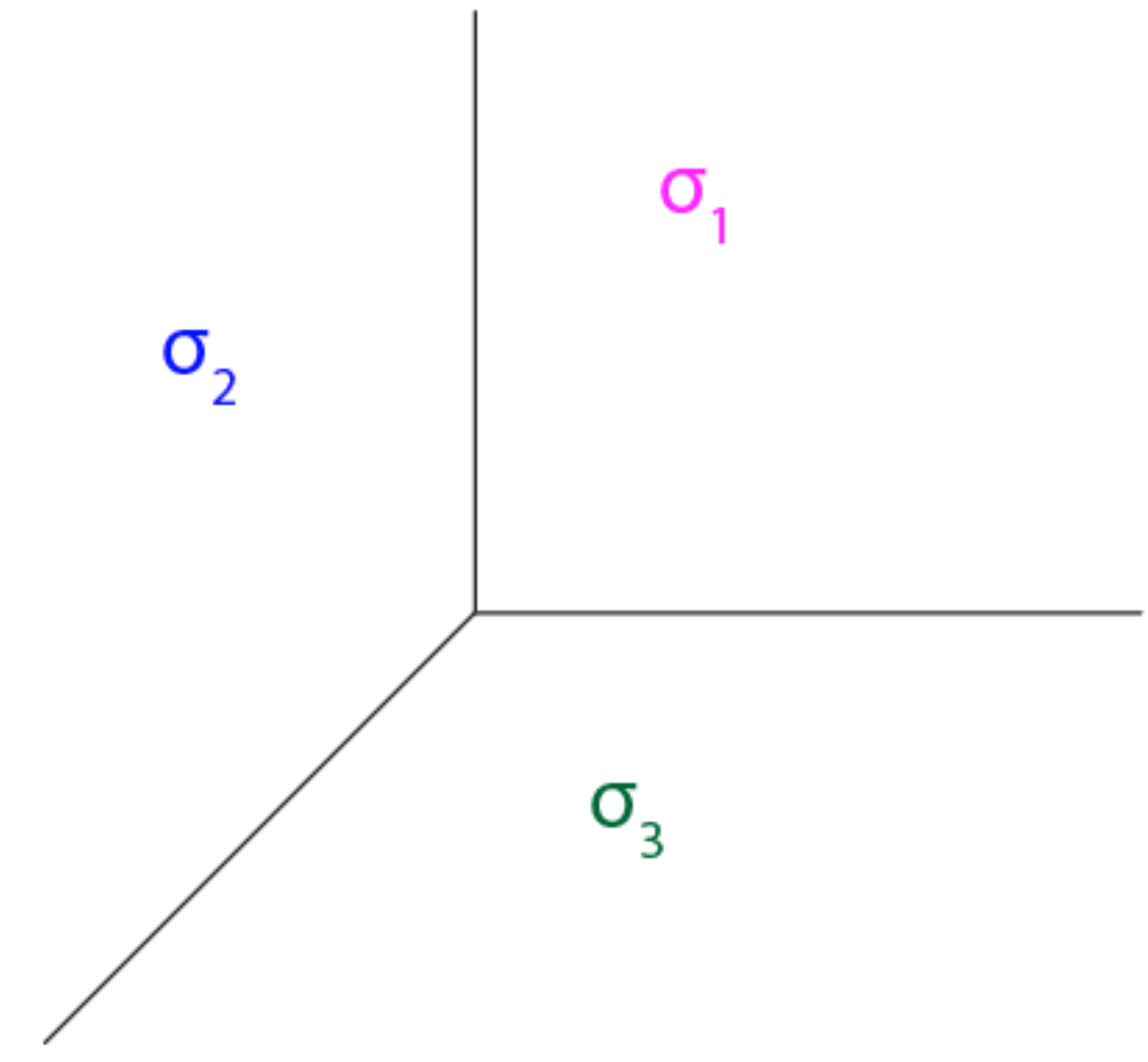}\hspace{1in} \includegraphics[scale=0.18]{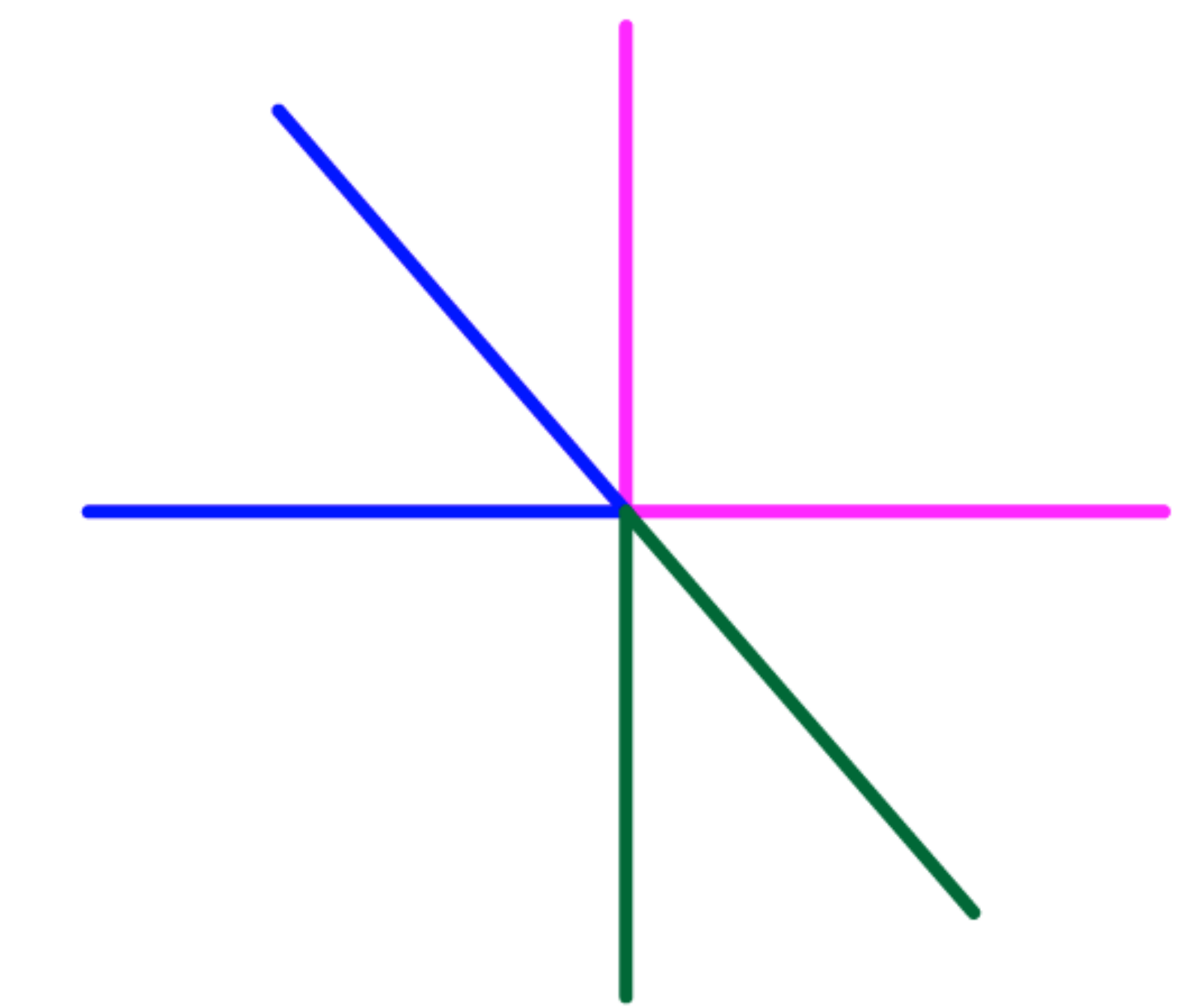}
\caption{$\mb{CP}^2$ example}
\label{cp2}
\end{figure}

Dually, the toric polytope for $\mb{CP}^2$ is a triangle, suppose delineated by the two coordinate axes (say $m_1=0,m_2=0$) and $m_1+m_2 = -1$. The line bundle described by this polytope is $\mc{O}(\{z_0=0\})$ with sections given by the integral vertices of the polytope $\chi^{(0,0)}=1, \chi^{(-1,0)}, \chi^{(0,-1)}$. Thus $\omega = \frac{i}{2\pi} \ol{\dd}\dd \log( 1 + |z_1/z_0|^2 + |z_2/z_0|^2)$, which recovers the Fubini-Study form.
\end{example}

%

\subsection{Finding a Lagrangian torus fibration}

We would like to construct an SYZ mirror to $H=\Sigma_2$ the genus 2 curve. The required input to do this is a special Lagrangian torus fibration on $H$.  Finding a special Lagrangian torus fibration is a hard problem. Guadagni's thesis \cite{roberta} finds Lagrangian torus fibrations on central fibers of toric degenerations, which will be the setting of the mirror $Y$ in our case. However there is not an obvious Lagrangian torus fibration on, or toric degeneration to, $\Sigma_2$. 

What one can do, as in Abouzaid-Auroux-Katzarkov \cite{AAK}, is embed $\Sigma_2$ in an abelian variety. We then take the trivial fibration over $\mb{C}$ with the abelian variety as a fiber, and blow-up the copy of $H$ over 0. The resulting fibration has $H$ as a critical locus in the central fiber. It also admits another fibration which is a Lagrangian torus fibration, by taking the moment map before the blow-up (a \emph{toric} Lagrangian torus fibration), and then keeping track of the blow-up in the base as in \cite{4from2} to obtain a \emph{non-toric} Lagrangian torus fibration. Note that \cite{AAK} did this process for hypersurfaces of toric varieties, and Seidel speculated that this could be done on hypersurfaces of abelian varieties \cite{seidel_gen2_specul}, which we do here. 

The SYZ construction \cite{syz} produces a candidate mirror complex manifold by prescribing dual fibers over the same base. The points of a dual fiber are parametrized by unitary flat connections on the trivial line bundle on the original fiber. This process is also discussed in \cite{t_duality}. The SYZ construction inverts the radius of each $S^1$ on a torus fiber, known as T-duality, and passes between the A- and B-models. In particular on a Calabi-Yau 3-fold as in \cite{syz}, SYZ mirror symmetry is $T$-duality three times.

{\bf Toric Lagrangian torus fibration.} A toric variety with its corresponding symplectic form as in Corollary \ref{toric_metric}, has a natural Hamiltonian $T^n$ action given by rotation on the dense $(\mb{C}^*)^n$, which extends to the full toric variety. Here is an example with $\mb{CP}^2$.

\begin{example}[Symplectic $\mb{CP}^2$]\label{ex:cp2}
Consider the complex projective plane with the Fubini-Study form: $(\mb{CP}^2, \omega_{FS} = \frac{i}{2\pi} \dd \ol{\dd} \log(\sum_{i=0}^2 |z_i|^2))$ where points are denoted $[z_0:z_1:z_2]$. There is a well-defined Hamiltonian $T^2$-action where $
(\theta_1,\theta_2) \in \mb{R}^2/ \mb{Z}^2 = T^2$ acts on $\mb{CP}^2$ by rotation:
$\rho(\alpha_1,\alpha_2)[z_0:z_1:z_2]  = [z_0:e^{2\pi i\alpha_1}z_1 : e^{2 \pi i \alpha_2} z_2]$. This is Hamiltonian with Hamiltonian functions defining the moment map coordinates $\mu_i$. For local $\mb{CP}^2$ coordinates $(x_1,x_2)$ we can define $\theta_i:=\arg(x_i)$ and the infinitesimal rotation action is:
\begingroup \allowdisplaybreaks \begin{equation}
\begin{aligned}
X_j &: = d \rho\left( \frac{\dd}{\dd \alpha_j} \right) = 2\pi iz_j \frac{\dd}{\dd z_j}\\
\implies \iota_{X_i} \omega_{FS} & = d \mu_i \text{ where } \mu_i := -\frac{|z_i|^2}{|z_0|^2 + |z_1|^2 + |z_2|^2}\\
\implies \omega_{FS}  &= d \mu_1 \wedge d \theta_1 + d \mu_2 \wedge d \theta_2
\end{aligned}
\end{equation} \endgroup
The last line is true more generally, that in action-angle coordinates $\omega = d \mu \wedge d\theta$, e.g.~see \cite[Theorem 1.3.4]{da_silva_toric_notes} and set $f_i=\mu_i$. The contraction with the vector field rotating coordinates gives $d\mu_i$ more generally. Thus the moment map here $\mu:= (\mu_1,\mu_2): \mb{CP}^2 \to \mb{R}^2 \cong (Lie(T^2), [,]=0)$ is given by 
$$\mu = \displaystyle{\left(-\frac{|z_1|^2}{|z_0|^2 + |z_1|^2 + |z_2|^2},-\frac{|z_2|^2}{|z_0|^2 + |z_1|^2 + |z_2|^2} \right)}$$
Its image can be seen in Figure \ref{fig:cp2_polytope}, where the diagonal edge follows from adding $\mu_1+\mu_2$ and allowing the coordinates to vary:

\begin{figure}[h]
\begin{center}
\includegraphics[scale=0.5]{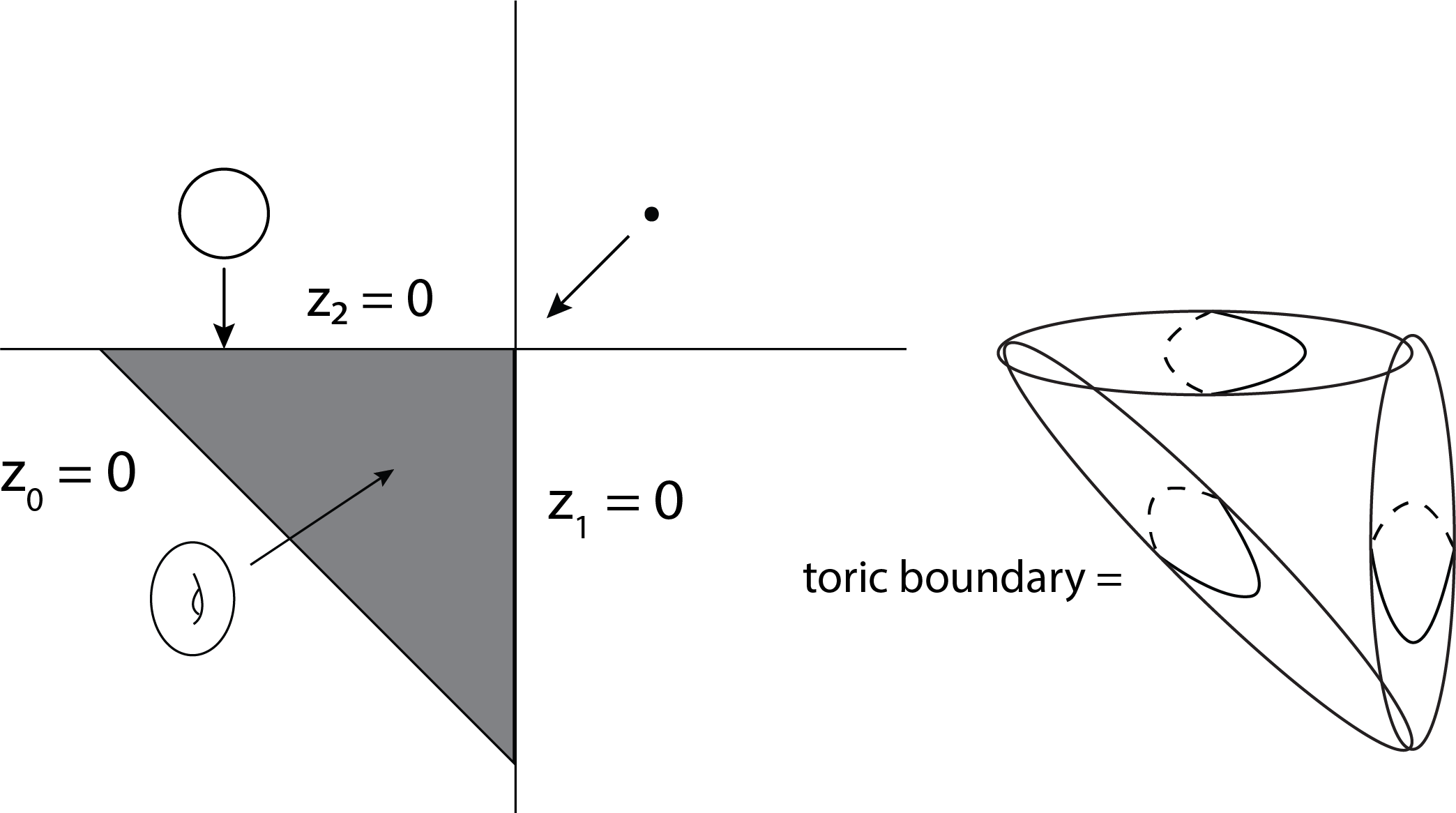}

\caption{Moment map gives Lagrangian torus fibration: $\mb{CP}^2$ example}
\label{fig:cp2_polytope}
\end{center}
\end{figure}

The moment map gives a Lagrangian torus fibration because the moment map is a function of the norms, as is the K\"ahler potential. Thus $\omega_{FS}|_{\mu^{-1}(pt)}=0$. We can read off the geometry of the fibration from Figure \ref{fig:cp2_polytope}. When both local complex coordinates are non-zero, the preimage is $T^2$ by rotating under the $T^2$-action. When one coordinate goes to zero, we only have the other coordinate to rotate, so the fiber is an $S^1$. And when both coordinates are zero in each of the three local $\mb{C}^2$ charts we obtain a point.
\end{example}

The above Lagrangian torus fibration arose from a moment map, so is called \emph{toric}. If a Lagrangian torus fibration is not from a moment map, it's called \emph{non-toric}. Note that $\mb{CP}^2$ blown up at $[1:0:0]$ corresponds to removing a small triangle on the base, which gives a quadrilateral. This can still be the base of a Lagrangian torus fibration by taking $T^2$'s above interior points, $S^1$ on the edges and points at the vertices. Note that we get $\mb{CP}^1 \times \mb{CP}^1$ which is again toric. However, we could have blown up at a point interior in the toric divisor, such as $[1:1:0]$. The result cuts a triangle out of the base again, but we introduce monodromy around the top vertex by gluing via the Dehn twist. See Figure \ref{fig:blowups_base} and also \cite[Fig 2]{AAK}.

\begin{figure}[h]
    \centering
    \includegraphics{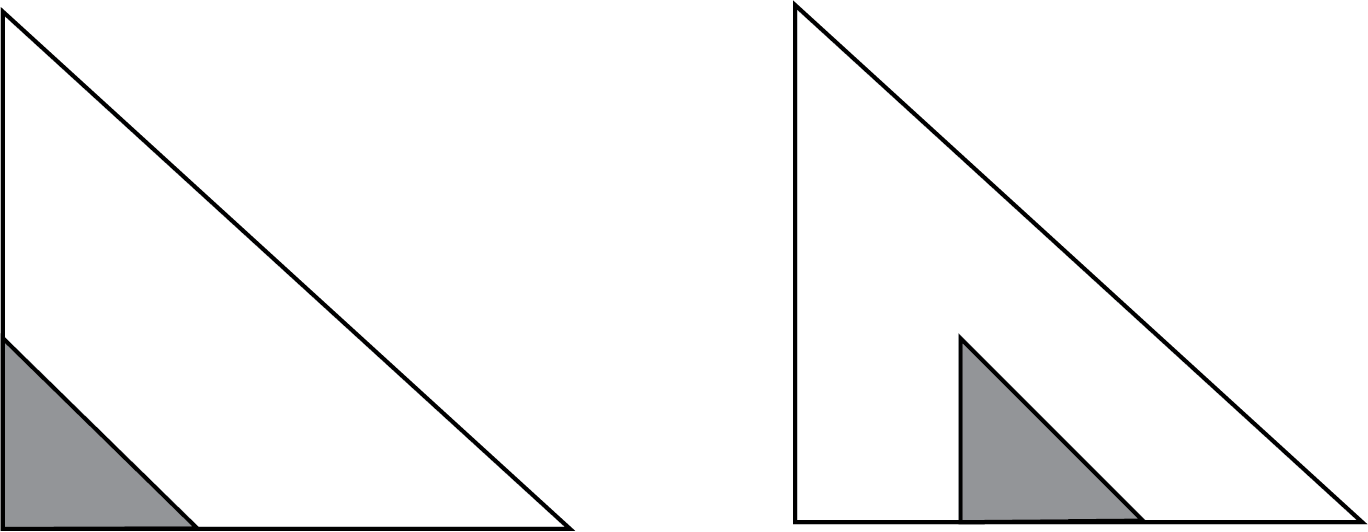}
    \caption{Base of exterior and interior blow-up on $\mb{CP}^2$}
    \label{fig:blowups_base}
\end{figure}

These describe Lagrangian torus fibrations over a base with a symplectic affine structure. Taking the Legendre transform, we obtain a Lagrangian torus fibration over a base with a complex affine structure, namely $\log |\cdot|$. Indeed SYZ sends the symplectic affine structure to the complex affine structure on the mirror base. In other words, the torus fibration on the mirror is $\log |\cdot |$ over the interior of the base polytope.

\begin{theorem}[Construction of Lagrangian torus fibration from {\cite[\textsection 4]{AAK}}]\label{theorem:Aside_blowup} $$X:= \mbox{Bl}_{H \times \{0\}} (V \times \mb{C})$$ admits a Lagrangian torus fibration. Furthermore, invariants on $X$ are related to $H$:
\begingroup \allowdisplaybreaks \begin{equation}
\begin{aligned}
    Fuk(X)& \cong Fuk(H)\\
D^bCoh(X) &= \left<D^bCoh(V \times \mb{C}), D^b Coh(H)\right>
\end{aligned}
\end{equation} \endgroup
where the angle brackets denote semi-orthogonal decomposition.
\end{theorem}

An example of this, and a non-toric Lagrangian torus fibration, is when $H=(1,0) \subset \mb{CP}^1$ defined by $s(x) = x-1$. This gives mirror symmetry for the point $H$, and the mirror is a Lefschetz fibration generated by one thimble. Namely if $\dim_{\mb{C}}V = 1$, then the zero fiber of $y: X \to \mb{C}$ involves a normal crossings divisor of the form $yz=0$, which for dimensional reasons is a Lefschetz singularity. Hence Seidel's Fukaya category for Lefschetz fibrations \cite{seidel} can be used, where $H=$ a point and $Fuk(X)$ is generated by a thimble. 

\begin{claim} $H = Crit(y)$ is the critical locus of the Bott-Morse fibration $y:X \to \mb{C}$. \end{claim}
\begin{proof} The zero fiber is the union of the proper transform of $V$, namely $\tilde{V} := \overline{p^{-1}(V\backslash H \times \{0\})}$ and the exceptional divisor. The normal bundle $\mc{N}_{V \times \mb{C}/H \times 0} = \mc{L} \oplus \mc{O}$. Then $H=\Sigma_2$ in the blow-up in $\mb{P}(\mc{L} \oplus \mc{O})$ is the intersection of two divisors in a normal crossing singularity. This intersection forms the critical locus of a Morse-Bott fibration given by $y: (x,y,(u:v)) \mapsto y$. Let $p$ be the blow-up map:
$$p: X \to V \times \mb{C}, \;\; (x,y,(u:v)) \mapsto (x,y)$$ 
Geometrically $\tilde{V}$ is a copy of $V$, i.e.~the closure of the part of $V$ away from $H$ in the blow-up which fills in the rest of the $V$ copy. Note that the closure adds in only a point for dimension reasons as we approach each point in $H$, so the closure adds back in a copy of $H$. The exceptional divisor is a $\mb{P}^1$-bundle over $H$ if we consider projection to $V\times \mb{C}$ and then take the ``zero-th" level at $V \times \{0\}$. Let $s^p$ be a section of this $\mb{P}^1$-bundle.
\begin{align*}
E & := \{(x,0,(u:v))\mid x \in H\}=p^{-1}(H \times \{0\})\\
p|_E& : E \to H \times \{0\}\\
s^p &: H \to E, \;\; s^p(x):=(x,0,(1:0))
\end{align*}
Now we can see $H$ as the critical locus of the $y$ fibration, as the fixed point set of the $S^1$-action that rotates $y$, namely $(x, e^{i\theta}y, (e^{-i\theta}u:v))$.
\begin{align*}
y^{-1}(0) &= \{(x,0,(u:v))\mid s(x)v=0\cdot u\}\\
&= \{(x,0,(1:0))\} \cup \{(x,0,(u:v))\mid s(x)=0\}\\
& = \tilde{V} \cup E
\end{align*}
Thus
$$\tilde{V} \cap E=s^p(H)  \cong  H = Crit(y)$$
since $x \in E$ implies $s(x)=0$ and $x \in \tilde{V} \implies (u:v)=(1:0)$.
\end{proof}

\begin{proof}[Proof overview of Theorem \ref{theorem:Aside_blowup} from {\cite{AAK}}] Let $x=(x_1,x_2)\in V, y \in \mb{C}$ and $s: V \to \mc{L}$ define the hypersurface which here is the theta divisor. Then, the blow-up projectivizes the normal bundle to $H \times \{0\}$ namely $\mc{L} \oplus \mc{O}$ by the adjunction formula.
\begingroup \allowdisplaybreaks \begin{equation}
\begin{aligned}
X&= \mbox{Bl}_{H \times \{0\}} V \times \mb{C} \\
&=\overline{\mbox{graph}[s(x):y]} \\
&= \overline{ \{(x,y,[s(x):y])\in (\mb{P}(\mc{L} \oplus \mc{O})\to V\times \mb{C}) \}}\\
&= \{(x,y,[u:v]) \mid s(x)v = yu\} \subset \mb{P}(\mc{L} \oplus \mc{O})
\end{aligned}
\end{equation} \endgroup
a subset of the $\mb{P}^1$-bundle $\mb{P}(\mc{L} \oplus \mc{O})$ on $V\times \mb{C}$. The torus fibration on $V\times \mb{C}$ is $(\log_\tau |\cdot |, \mu_{S^1})$ where $\mu_{S^1}$ is the moment map from the Hamiltonian $S^1$-action that rotates the $y$ complex coordinate. The base of this fibration is $B = T_B^2 \times \mb{R}_+$ because of quotienting by the $\Gamma_B$-action in the first two coordinates, which scales $|x_i|$. We keep track of the blow-up in the base, as above in the case of interior blow-up, to obtain a Lagrangian torus fibration on $X$ for a suitable symplectic form $\omega$ constructed in \cite{AAK}. It is symplectomorphic to the pullback of the canonical toric K\"ahler form $p^*\omega_{V\times\mb{C}}$ away from $E$ and controls symplectic area near the exceptional divisor.  

For the relations between invariants, \cite[Corollary 7.8]{AAK} states that $Fuk(H)$ is equivalent to a Fukaya category $F_s(X,y)$, because Lagrangians in $H$ can be parallel transported from the central fiber to obtain non-compact Lagrangians in $X$ admissible with respect to the superpotential $y$. On the complex side, [Orlov, \cite{Der_cats}] implies there is a semi-orthogonal decomposition of $D^bCoh(X)$ into $D^b(V \times \mb{C})$ and $D^bCoh(H)$. 
\end{proof}


%
%
%
%


\subsection{Background needed to define the generalized SYZ mirror} We have local charts for the mirror over open sets in the base $B$, as described in \cite{t_duality}, which are glued across walls in the base $B$. A wall occurs in the base over which the Lagrangian fibers are singular. The gluing information can be encoded in a polytope by \cite{AAK}:
$$\Delta_{\tilde{Y}} := \{(\xi_1,\xi_2,\eta) \in \mb{R}^3 \mid \eta \geq \Trop(s)(\xi_1,\xi_2)\}$$
where the tropicalization of a function describes how it tends to infinity as its variables tend to infinity, as a function of the direction $\xi$ we let the variable norms $|x_i|:=\tau^{\xi_i}$ go to infinity. Mathematically:

\begin{definition}[Tropicalization] Let $f(x) = \sum_{a \in A\subset \mb{Z}^n}c_ax^a \tau^{\rho(a)}$. Let $|x_i| = \tau^{\xi_i}$. Then 
$$f(x/|x|, \xi) = \sum_a c_a\left(\frac{x}{|x|}\right)^a \tau^{\rho(a) + \left<a,\xi\right>}$$

The \emph{tropicalization of $f$} is:
$$\Trop(f)(\xi) := -\min_{a \in A} \left<a,\xi\right> + \rho(a)$$
\label{trop definition}
\end{definition}

As $\tau \to 0$, the leading order term in $f$ has exponent $-\Trop(f)(\xi)$. The vanishing of $f$ limits to a tropical curve given by those $\xi \in \mb{R}^n$ where two terms can cancel in $f$, namely where two different $a\in A$ give the same minimum for $\xi$.

\begin{example}[{\cite{AAK}[\textsection 9.1]}]\label{pants_ex} Suppose $H \subset V$ is the pair of pants $f(x_1,x_2):=1+x_1+x_2 = 0$ in $V = (\mb{C}^*)^2$. This is a pair of pants because $x_1 \in \mb{C}^*\backslash\{-1\}$ and a cylinder minus a point is a pair of pants. Then $\rho \equiv 0$ and $A = \{(0,0), (1,0), (0,1)\}$ hence $\Trop(f)(\xi_1,\xi_2) = \max\{0, \xi_1,\xi_2\}$. If $\xi_1,\xi_2 < 0 $ then 0 is the maximum, if $\xi_1 > \xi_2 >0$ then $\xi_1$ is the maximum and if $\xi_2 > \xi_1 > 0$ then $\xi_2$ is the maximum. So the vanishing of $\Trop(f)$ is the left diagram in Figure \ref{fig:pop_ex}, and the moment polytope is $\Delta:=\{(\xi,\eta) \mid \eta \geq \Trop(f)(\xi)\} \subset \mb{R}^{n+1}$ depicted in the right diagram of Figure \ref{fig:pop_ex} projected to $(\xi_1,\xi_2)$ coordinates with the $\eta$-coordinate coming out of the page.

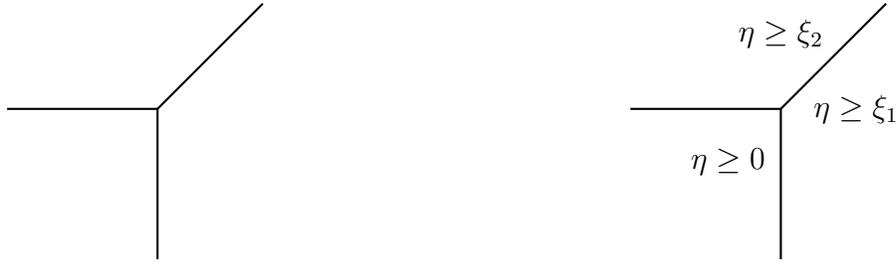
\begin{figure}[h]
\begin{multicols}{2}
\begin{tikzpicture}
\draw[thick] (0,0) -- (-2,0); \draw[thick] (0,0) -- (0,-2); \draw[thick] (0,0) -- (1.4,1.4);
\end{tikzpicture}

\begin{tikzpicture}
\draw[thick] (0,0) -- (-2,0); \draw[thick] (0,0) -- (0,-2); \draw[thick] (0,0) -- (1.4,1.4);
\node at (1,0) {$\eta \geq \xi_1$}; \node at (0,1) {$\eta \geq \xi_2$}; \node at (-.7,-.7) {$\eta \geq 0$}; 
\end{tikzpicture}
\end{multicols}
\caption{L) $\Trop(1+x_1+x_2)=0$\qquad \qquad R) Moment polytope in $\mb{R}^3$}
\label{fig:pop_ex}
\end{figure}    

The corresponding toric variety is $\Spec \mb{C}[x,y,z] = \mb{C}^3$ where $x$, $y$ and $z$ are the three toric coordinates arising from the toric monomials with weight vectors given by primitive generators of the three edges. The superpotential is $v_0 = xyz$, giving the expected pair of pants mirror $(\mb{C}^3, xyz)$.
\end{example}

Tropicalizing the infinite series given by the theta function at first sight seems hard. In fact, it satisfies a periodicity property which allows us to see the tropicalization as a honeycomb shape when projected to $(\xi_1,\xi_2)$ coordinates.

\begin{claim} The tropicalization of the theta function $\vp:=\mbox{Trop } s$ satisfies the following periodicity property
\begingroup \allowdisplaybreaks \begin{equation}
\vp(\xi+\tilde{\gamma}) =  \vp(\xi) -\kappa(\tilde{\gamma}) + \left<\xi,\la(\tilde{\gamma})\right>
\label{phi_periodicity}
\end{equation} \endgroup
\end{claim}

\begin{proof}Recall from the definition in Claim \ref{defn:line_bundle}
$$s(x) = \sum_{\gamma \in \Gamma_B} \tau^{-\kappa(\gamma)} x^{-\la(\gamma)}= \sum_{\gamma \in \Gamma_B} \tau^{\frac{1}{2}\left<\gamma,\la(\gamma)\right>} x^{-\la(\gamma)}$$
Since $|x_i|=\tau^{\xi_i}$, and letting $\tau \to 0$ we see that the leading term is the minimum exponent or the maximum of the negated exponent, namely
\begingroup \allowdisplaybreaks \begin{equation}\label{trop_eq}
\varphi(\xi):=\Trop(s)(\xi) = \max_{\gamma} \kappa(\gamma)+\left<\xi,\la(\gamma)\right>  
\end{equation} \endgroup
Since $\kappa$ is a negative definite quadratic form of degree 2 and $\la$ is positive of degree 1, this should have a maximum. For example, $\vp(0)=0$. We have the following periodicity property.
\begingroup \allowdisplaybreaks \begin{equation}\label{eq:phi_periodicity}
    \begin{aligned}
\varphi(\xi + \tilde{\gamma}) &= \max_\gamma  \kappa(\gamma)+ \left<\xi + \tilde{\gamma}, \la(\gamma)\right>  = \max_\gamma  \kappa(\gamma) + \left<\xi, \la(\gamma)\right>  + \left<\tilde{\gamma},\la(\gamma)\right>\\
\kappa(\gamma-\tilde{\gamma}) & = \kappa(\gamma) + \kappa(\tilde{\gamma}) + \left<\gamma, \la(\tilde{\gamma})\right>\\
\implies \varphi(\xi + \tilde{\gamma}) &=\max_\gamma \left<\xi,\la(\gamma)\right> + [\kappa(\gamma - \tilde{\gamma}) -\kappa(\tilde{\gamma})]\\
 & = \left(\max_\gamma \kappa(\gamma - \tilde{\gamma})+ \left<\xi,\la(\gamma-\tilde{\gamma})\right> \right) - \kappa(\tilde{\gamma})+\left<\xi,\la(\tilde{\gamma})\right>\\
 \implies \vp(\xi+\tilde{\gamma}) &=  \vp(\xi) -\kappa(\tilde{\gamma}) + \left<\xi,\la(\tilde{\gamma})\right>
    \end{aligned}
\end{equation} 
\endgroup

\end{proof}

\begin{claim} The vanishing set $V(\mbox{Trop } s) \subset \mb{R}_{\xi_1,\xi_2}^2$ is a honeycomb shape that is a tiling by hexagons.\end{claim}

\begin{proof} 
Let $\xi = 0$. Then 
$$\Trop(s)(0) = \max_{\gamma\in \Gamma_B} \kappa(\gamma) \leq 0 \implies \Trop(s)(0)=0$$
since $\kappa$ is negative definite so its maximum is achieved when $\gamma=0$. We know that $\Trop(s)(\xi_1,\xi_2)$ is a piecewise linear function, so let $F_{0,0}$ denote the piece that is identically zero.

In order to prove that the projection of $\Delta_{\tilde{Y}}$ to the first two coordinates is a tiling of hexagons (equivalent to the statement of the claim), we will proceed as follows. We determine where adjacent hyperplanes intersect, by finding the equations of their lines of intersection. This will produce the hexagonal shape.

Fix $(\xi_1,\xi_2,\eta)\in F_{0,0}$. Add $i_1\gamma'+i_2\gamma''$ for $(i_1,i_2) \in \{(\pm 1,0), (0, \pm 1), \pm(1,-1) \}$. These six choices will give rise to the hexagonal shape. Recall $\gamma'$ and $\gamma''$ are the generators for $\Gamma_B$. By Equation (\ref{eq:phi_periodicity}) and that $\vp(\xi)=0$:
\begingroup \allowdisplaybreaks \begin{equation}
\begin{aligned}
\eta=\vp(\xi+i_1\gamma' +i_2\gamma'' + 0) &=\vp(\xi) + \left<\xi, \la(i_1\gamma'+i_2\gamma'')\right> -\kappa(i_1\gamma'+i_2\gamma'')\\
&=i_1\xi_1 + i_2\xi_2 + i_1^2+i_2^2+i_1i_2 \\
&= i_1\xi_1+i_2\xi_2+ 1
\end{aligned}\label{always1}
\end{equation} \endgroup

On the other hand, the equation of the plane when $(\xi, \eta) \in F_{i_1,i_2}$ is the set of $\xi$ and $\eta=\vp(\xi)$ such that  
$$\eta = i_1(\xi_1-2i_1-i_2)+i_2(\xi_2-i_1-2i_2)+1=i_1\xi_1+i_2\xi_2-1$$
(We relabeled $\xi$ from Equation (\ref{always1}) so it lies in $F_{i_1,i_2}$ and not $F_{0,0}$.) So points $(\xi_1,\xi_2,\eta)$ on the intersection of the two planes must satisfy both, hence:
$$i_1\xi_1+i_2\xi_2=1$$
We get the shape enclosed by the lines $\xi_i=\pm 1$ for $i=1,2$, a box, and $\xi_1-\xi_2=\pm 1$ which is the line $\xi_1=\xi_2$ shifted up and down by 1. This means $F_{0,0}$ is a hexagon in the $\eta=0$ plane. 
\begin{figure}[H]
\centering
\begin{tikzpicture}[scale = 2.5]
\draw[shift = {(-1,-1)}] (0,0) -- (1,0) -- (2,1) -- (2,2) -- (1,2) -- (0,1)--(0,0);
\node [above] at (0,0) {$(0,0)$}; \node [above] at (-0.5,-1) {$\xi_2 = -1$}; \node [right] at (.5, -.5) {$\xi_2=\xi_1-1$}; \node[right] at (1,.5) {$\xi_1=1$};

\end{tikzpicture}
\caption{The (0,0) tile delimited by the tropical curve}
\end{figure}
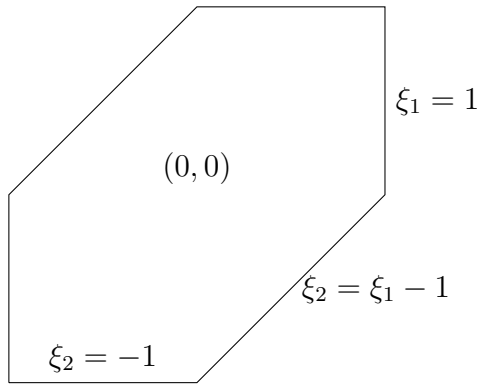

Now pick any $(\xi_1,\xi_2) \in \mb{R}^2$. Choose $\gamma$ such that $(\xi_1-\gamma_1,\xi_2-\gamma_2) \in F_{0,0}$, which we can do since the hexagon is the same size as the fundamental domain for the $\Gamma_B$-action by subtraction. Let $\la(\gamma) =: (m_1,m_2)^t$. Then again by Equation (\ref{eq:phi_periodicity}) and using that $\vp(\xi-\gamma)=0$:
\begingroup \allowdisplaybreaks \begin{equation}\label{facet_eq}
\begin{aligned}
\eta=\vp(\xi) &=\vp(\xi-\gamma) + \left<\xi-\gamma, \la(\gamma)\right> -\kappa(\gamma)\\
&=0 + \xi_1 m_1+\xi_2m_2 -(m_1^2+m_1m_2+m_2^2)\\
\iff 0 &=\left<\left(\begin{matrix} m_1\\m_2\\-1 \end{matrix}\right), \left(\begin{matrix} \xi_1 \\ \xi_2 \\ \eta \end{matrix} \right) \right> - (m_1^2+m_1m^2+m_2^2)\\
\iff 0 & = \left<\left(\begin{matrix} \la(\gamma) \\-1 \end{matrix}\right), \left(\begin{matrix} \xi \\ \eta \end{matrix} \right) \right> +\kappa(\gamma)
\end{aligned}
\end{equation} \endgroup
Then using Equation (\ref{facet_eq}), we find the intersection of the $(m_1,m_2)$ plane with the $(m_1+i_1, m_2+i_2)$ plane, i.e.~those $(\xi,\eta)$ satisfying both plane equations.
\begingroup \allowdisplaybreaks 
\begin{align*}
 \Delta \eta = 0 &=\vp(\xi + (i_1,i_2) \cdot (\gamma', \gamma'')) - \vp(\xi)\\
 &=\left<\la(\gamma)+(i_1,i_2), \xi \right> + \kappa(\gamma+(i_1,i_2)\cdot (\gamma',\gamma''))-[\left<\la(\gamma), \xi \right> + \kappa(\gamma)]\\
 &=i_1\xi_1 + i_2 \xi_2 -\left< \gamma,(i_1,i_2)\right> -\frac{1}{2}\left<i_1\gamma'+i_2\gamma'', (i_1,i_2)\right>\stepcounter{equation}\tag{\theequation}\\
 &=i_1\xi_1 + i_2 \xi_2 - \left<(2m_1+m_2, m_1+2m_2),(i_1,i_2)\right> - \frac{1}{2}\left<(2i_1+i_2, i_1+2i_2), (i_1, i_2)\right>\\
 &=i_1\xi_1 + i_2\xi_2 - i_1(2m_1+m_2) - i_2(m_1+2m_2)-1\\
 \implies  \xi_1 & = 2m_1+m_2 \pm 1\\
 \xi_2 & = m_1+2m_2 \pm 1\\
 \xi_1-\xi_2 &= \pm1+m_1-m_2
\end{align*}
\endgroup

\begin{figure}[H]
\centering
\begin{tikzpicture}[scale = 2.5]
\draw[shift = {(-1,-1)}] (0,0) -- (1,0) -- (2,1) -- (2,2) -- (1,2) -- (0,1)--(0,0);
\node [above] at (0,0) {$m_1\gamma'+m_2\gamma''$}; \node [below] at (-0.5,-1) {$\xi_2 = m_1+2m_2-1$}; \node [right] at (.5, -.5) {$\xi_2=\xi_1-1-m_1+m_2$}; \node[right] at (1,.5) {$\xi_1=2m_1+m_2+1$};

\end{tikzpicture}
\caption{The $(m_1, m_2)$ tile delimited by the tropical curve}
\end{figure}
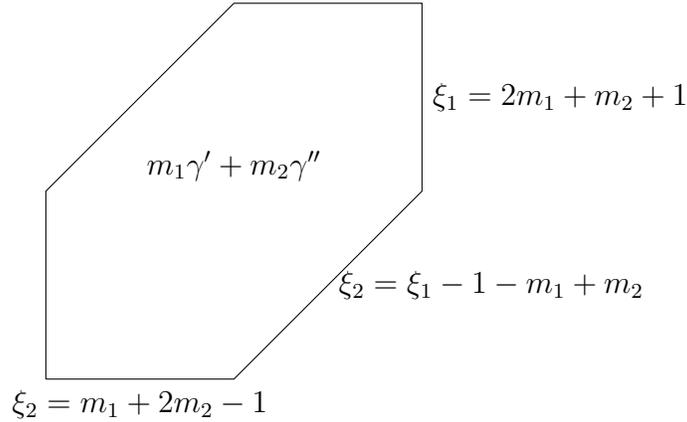
\end{proof}

\subsection{The definition of $(Y,v_0)$} 

The smooth manifold $Y$ is constructed as a {portion} of a toric variety $\tilde{Y}$ quotiented by $\Gamma_B$ acting properly discontinuously via holomorphic maps. The defining polytope is
\begingroup \allowdisplaybreaks
\begin{align*}
\Delta_{\tilde{Y}} &:= \{(\xi_1,\xi_2,\eta) \in \mb{R}^3 \mid \eta \geq \Trop(s)(\xi)\}\\
\Delta_Y &:= (\Delta_{\tilde{Y}})|_{\eta \leq T^l}/\Gamma_B \mbox{ where}\\
\Trop(s)(\xi) & := \max_{\gamma} \kappa(\gamma)+\left<\xi,\la(\gamma)\right>\stepcounter{equation}\tag{\theequation}\label{eq:polytope_def_Y}  \\
\gamma \cdot (\xi_1,\xi_2, \eta) &:= (\xi_1 + \gamma_1, \xi_2 + \gamma_2, \eta -\kappa(\gamma) +\left<\xi, \lambda(\gamma)\right>)\\
 \because \Trop(s)(\xi+\gamma) & = \Trop(s)(\xi)-\kappa(\gamma) +\left<\xi, \lambda(\gamma)\right>
\end{align*}
\endgroup
where the last line is Equation \ref{phi_periodicity}. The parameter $T\ll 1$ is the Novikov parameter on the genus 2 curve, hence the complex parameter on $Y$. The polytope $\Delta_{\tilde{Y}}$ is illustrated in Figure \ref{coords in tiling}, where $\eta$ is bounded below by the expression in the center of the tile, and comes out of the page. The superpotential $v_0$ is the toric degeneration given by the monomial $v_0=xyz$ in toric coordinates which vanishes to order 1 on each toric strata, hence the $T^4$ fibers degenerate over 0 to the toric strata. The three toric coordinates $(x,y,z)$ correspond to the three edges on the polytope $\Delta_{\tilde{Y}}$ from the lower left vertex of the $(0,0)$ hexagon. So $\eta$ is the weight vector $(0,0,1)$ which gives the toric monomial $v_0=xyz$. The fact that $v_0$ is well-defined under the $\Gamma_B$-quotient will be proven below. In particular, the action on complex coordinates only gives a nontrivial action if we restrict $v_0=xyz$ to be very small. This is why the condition $\eta \leq T^l$ is imposed. 

More specifically, recall that vertices of $\Delta_{\tilde{Y}}$ correspond to $\mb{C}^3$ charts. For example, consider the $(x,y,z)$ coordinates and the green vertices to the right in Figure \ref{coords in tiling}, call them $(x',y',z')$. The way in which we glue these two charts is encoded by the edge they are connected by, as follows. Recall that the normal $\nu_n$ to the $(n_1,n_2)$ tile is 
\begin{equation} 
\nu_n:= \left(\begin{matrix} -n_1\\-n_2\\ 1 \end{matrix}\right)
\end{equation}
The two vertices, taking the normals to the three facets at the vertex, gives us two cones spanned by the convex hull of the following rays:
\begin{allowdisplaybreaks}
    \begin{align*}
    \sigma_1 = \left<\left(\begin{matrix}0\\0\\1 \end{matrix} \right), \left(\begin{matrix} 1\\0\\1 \end{matrix}\right),\left(\begin{matrix}0\\1\\1 \end{matrix} \right)\right> & \implies U_{\sigma_1}  = \Spec \mb{C}[x,y,z]\\
    \sigma_2 = \left<\left(\begin{matrix}0\\0\\1 \end{matrix} \right), \left(\begin{matrix} -1\\1\\1 \end{matrix}\right),\left(\begin{matrix}0\\1\\1 \end{matrix} \right)\right> & \implies U_{\sigma_2}  = \Spec \mb{C}[x^{-1}, xy, zy^{-1}]\stepcounter{equation}\tag{\theequation}\\
    \tau:=\sigma_1 \cap \sigma_2 = \left< \left( \begin{matrix}0\\0\\1 \end{matrix}\right), \left( \begin{matrix} 0\\1\\1\end{matrix} \right) \right> & \implies U_\tau = \Spec \mb{C}[\mb{Z}^3 \cap \tau^\vee] = \Spec \mb{C}[x^{\pm}, y, zy^{-1}]
    \end{align*}
\end{allowdisplaybreaks}

Thus in coordinates on $\mb{C}^* \times \mb{C}^* \times \mb{C}$, which is the overlap of the two charts
\begin{equation}
    \begin{aligned}
    \phi_1=\bm{1} & : \mb{C}^3 \ni (x,y,z)  \mapsto (x,y,z)\\
    \phi_2 & : \mb{C}^3 \ni (x,y,z) \mapsto (x^{-1}, xy, zy^{-1})
    \end{aligned}
\end{equation}
we find that identifying the $U_\tau \subset U_{\sigma_i}$ for each $i$, we obtain the transition map:
\begin{equation}
\begin{aligned}
    \phi_{12} & : \mb{C}^* \times \mb{C}^* \times \mb{C} \curvearrowright\\
    (x,y,z) & \mapsto (x^{-1}, xy, zy^{-1})
    \end{aligned}
\end{equation}
Thus this gluing gives $\mb{P}^1$ in the first component. In particular, gluing all the toric charts will contain the dense $(\mb{C}^*)^3$ but infinitely many toric divisors coming from the $\mb{CP}^2(3)$ glued along $\mb{P}^1$'s. In the fibration $v_0=xyz$, we have this infinite $\mb{Z}^2$-chain of toric divisors over 0, and since the coordinates in each chart preserve $v_0$, a generic fiber is $(\mb{C}^*)^2$. So we have non-compact fibers and a non-compact base for the fibration $v_0$. When we take the quotient, the base stays the same but the fibers become compact.

\begin{definition}[{\cite{AAK}[Definition 1.2]}] $(Y,v_0)$ is the \emph{generalized SYZ mirror} of $H = \Sigma_2$.\end{definition}

\begin{remark} Note that one can apply SYZ in the reverse direction by starting with a Lagrangian torus fibration on $Y$ minus a divisor to recover $X$ as its complex mirror, see \cite[\textsection 8]{AAK} or \cite{cll}.
\end{remark}

We now discuss the complex coordinates on $Y$.

\subsection{Definition of complex coordinates on $\tilde{Y}/\Gamma_B$} 

\begin{center}
\fbox{\bf Reason for 1-parameter family}    
\end{center}

We can strengthen our result to be not just between two manifolds, but between two families of manifolds. Namely a family of genus 2 curves parametrized by $\tau$ and a family of symplectic manifolds parametrized by $\tau$. Symmetrically, we can also allow $T$ to parametrize a symplectic structure on the genus 2 curve which is mirror to a complex structure parametrized by $T$ on $Y$, which is what we define in this subsection.  

More specifically, we view $\tau$ or $T$ parametrizing the complex structure via scaling the lattice we quotient by, so multiplication by $i$ gets scaled in a 1-parameter family. The way in which they parametrize the symplectic structure as Novikov parameters is by symplectically weighting counts of discs in homology class $\beta$ by $\tau^{-\omega(\beta)} = e^{-(\log \tau) \omega (\beta) }$, i.e.~the symplectic form is scaled to $(\log \tau) \omega$ as $\tau \to 0$. 

So the upshot is that in this paper we consider mirror symmetry between 1-parameter families; $\tau$ is the complex parameter on the genus 2 curve/Novikov parameter on $Y$, and $T$ will be the complex parameter on $Y$/Novikov parameter on the genus 2 curve. Although the complex structure on $Y$ won't affect its symplectic geometry, the K\"ahler potential will be defined in terms of the complex coordinates, necessarily in a way invariant under $T$. This is why we need to define the complex structure on $Y$. (So adding in the $T$ doesn't extend our results to more manifolds in the direction we are considering, because we consider $Y$ as a symplectic manifold and $T$ varies the complex manifold. But if we ever wanted to consider mirror symmetry in the other direction, swapping the A- and B-models, we would have 1-parameter families in that direction too. So it's a bit stronger result to include the $T$ here.)

\begin{center}
\fbox{\bf Properties defining the symplectic form}    
\end{center}

We would like $\omega$ to have some nice properties that allow us to compute parallel transport and monodromy later on. (However, there are other symplectic forms one can equip $(Y,v_0)$ with.) These properties will enforce the way in which the $\Gamma_B$-lattice acts on the complex coordinates in terms of $T$. We will define a K\"ahler potential for $\omega$ as a function of the norms of the complex coordinates. That way it is a function of the moment map coordinates as well by Legendre transform. In order to adapt Seidel's Fukaya category for Lefschetz fibrations, we also would like $v_0:Y \to \mb{C}$ to be a symplectic fibration. That is, $\omega_{V^\vee}$ is a symplectic form on the fiber. Over zero we have a singular fiber, so we only require $\omega$ to restrict to a symplectic form away from the zero fiber. In a neighborhood of the zero fiber, this corresponds to a neighborhood of the facets in $\Delta_{\tilde{Y}}$. Away from a vertex, we take the limit of the symplectic form as $v_0 \to 0$. Near a $\mb{C}^3$-vertex, we take the standard $\omega_{\mb{C}^3}$ on $\mb{C}^3$. In between the two, we interpolate with bump functions. Checking non-degeneracy in these regions is the bulk of the calculation for $\omega$, to which readers may refer to Appendix B. We proceed to describe the symplectic form in a fiber so the limit to the zero fiber is well-defined.

\begin{center}
\fbox{\bf Properties defining the symplectic form on a $v_0$-fiber}    
\end{center}

The central fiber of $v_0$ is the $\Gamma_B$-quotient of the toric variety with moment polytope given by the hexagon in Figure \ref{hexagon}. This comes equipped with a symplectic form as described in the theory of Section \ref{section: toric_recap} with the toric variety $\mb{CP}^2(3 \mbox{ points})$, i.e.~blown-up at three points. As we move away from $v_0=0$ but still near a vertex of the polytope, the toric variety is locally modelled on the $(\mb{C}^3,xyz)$ picture that was the local model of Example \ref{pants_ex}.  So we will use bump functions to interpolate between the toric symplectic form of $\mb{CP}^2(3 \mbox{ points})$ and the standard form on $\mb{C}^3$. Away from $v_0=0$ and away from the $\mb{C}^3$ vertex we use the toric K\"ahler form $\omega_{\mb{CP}^2(3)}$.
\begin{figure}
\begin{center}
\includegraphics[scale=0.3]{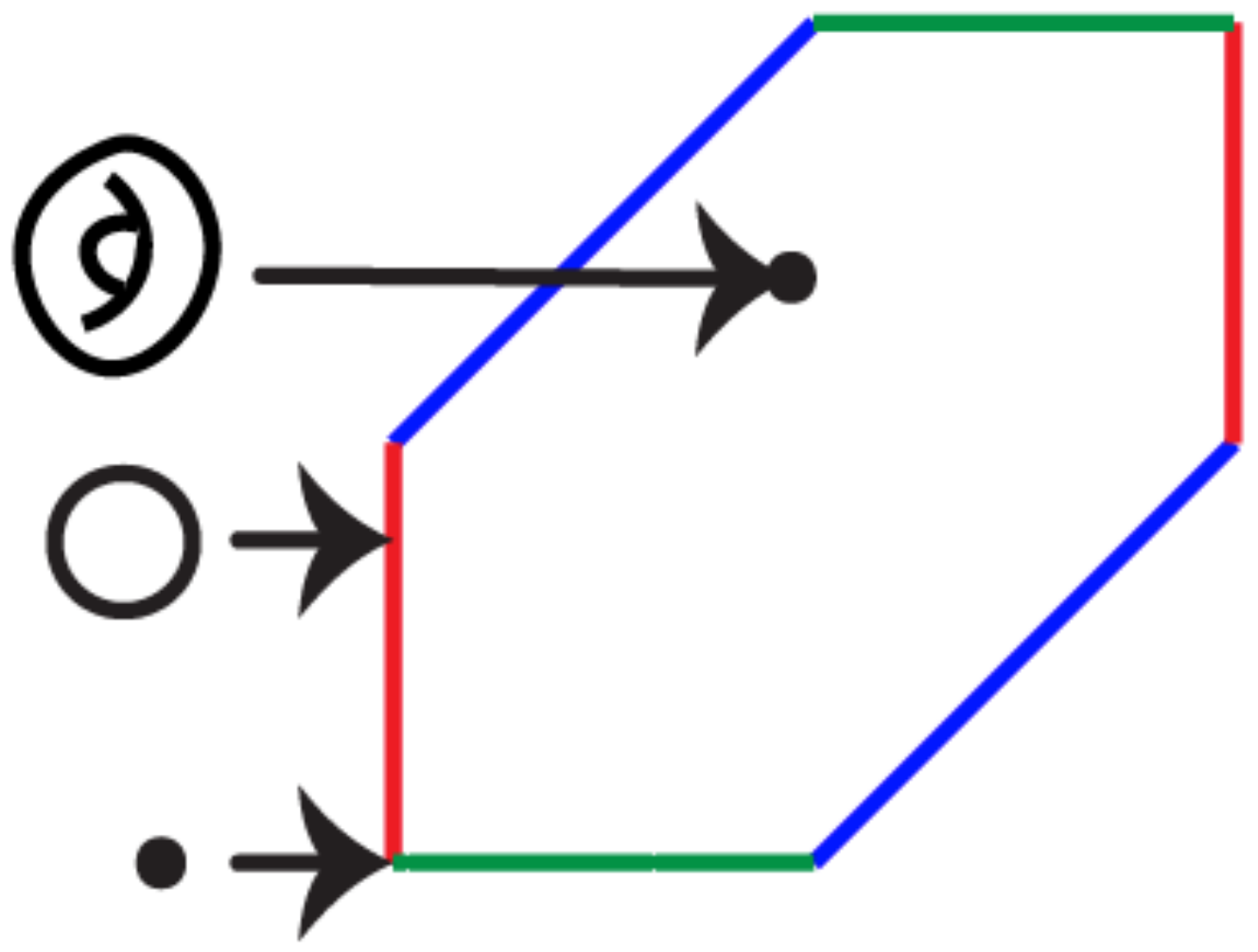}
\caption{Moment polytope for central fiber of $(Y,v_0)$ when $H=\Sigma_2$.}
\label{hexagon}
\end{center}
\end{figure}

\begin{example}[{\bf Illustrative example:~symplectic form on mirror to $H=pt$}] We illustrate this in one dimension down, where the polytope is two dimensional. Let $H$ be a point inside an elliptic curve. Then the polytope $\Delta_{\tilde{Y}}$ is two-dimensional, so we can draw it. 
\begin{figure}[h]
\begin{center}
\includegraphics[scale=0.2]{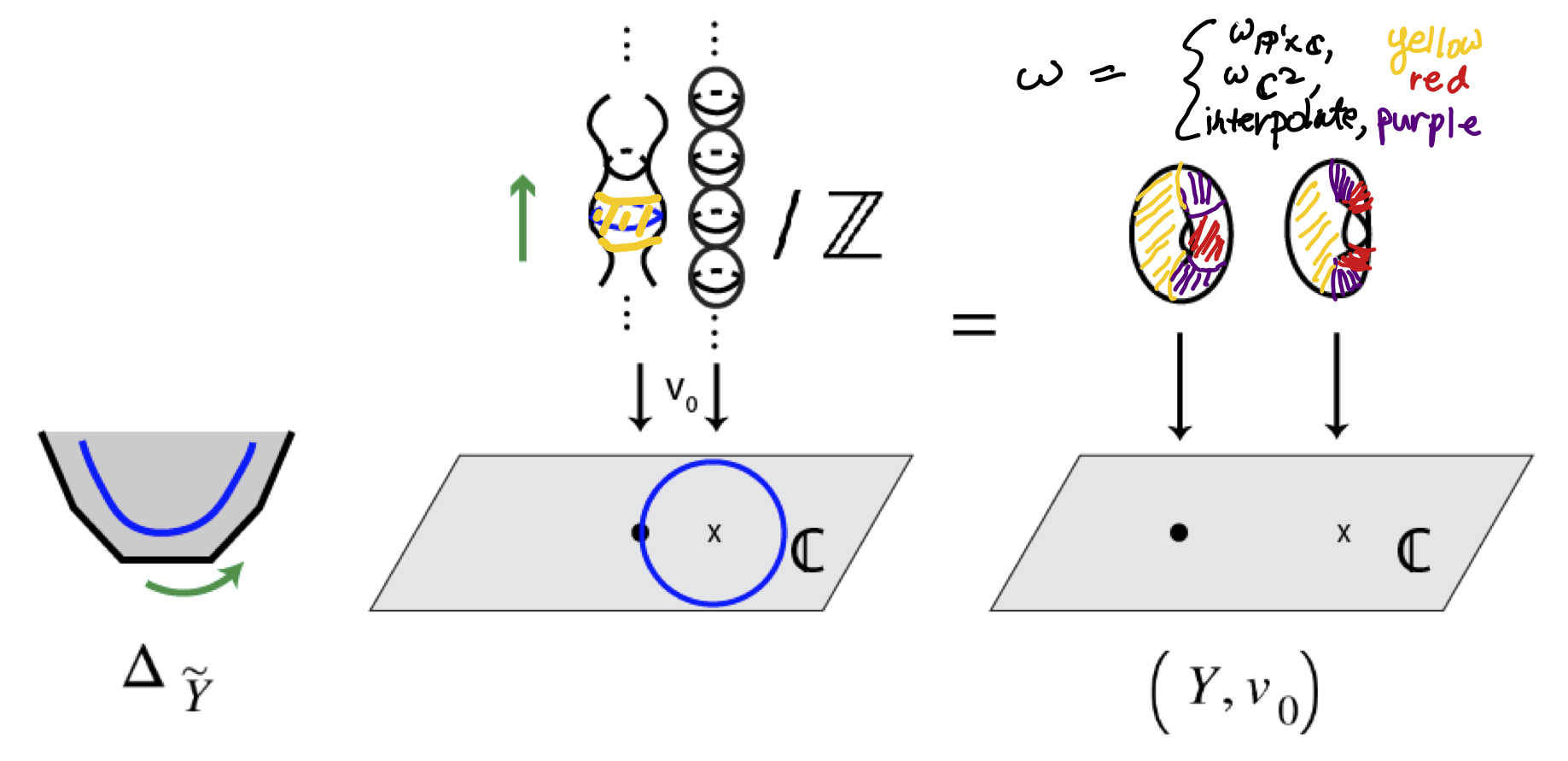}
\caption{In one dimension lower, the boundary of $\Delta_{\tilde{Y}}$ is the moment map image of a string of $\mb{P}^1$'s. In the polytope, $|v_0|$ increases in the $(0,1)$ direction. In the fibration $v_0$, $|v_0|$ is the radius of the circle in the base.}
\label{fig:ypic}
\end{center}
\end{figure}
 
A fiber of $v_0: \tilde{Y} \to \mb{C}$ is topologically a cylinder with belts increasingly  pinched, $\mb{Z}$-periodically, as $|v_0| \to 0$, degenerating to a string of $\mb{P}^1$'s over the central fiber. Around the widest parts of the cylinder, it looks like a portion of a sphere, and $\mb{P}^1$ comes canonically equipped with the Fubini-Study form. So on (fiber neighborhood of $v_0$-fiber blue circle) $\times \mb{C}$ the symplectic form is a product of that on the base and on the fiber.  

On the other hand, neighborhoods of the vertices of the polytope give $\mb{C}^2$ charts where the local picture of $v_0=xy:\tilde{Y} \to \mb{C}$ is the Lefschetz fibration $\mb{C}^2 \to \mb{C}$, $(x,y) \mapsto xy$ where cylinders degenerate to a double cone over zero. In particular, $\mb{C}^2$ comes canonically equipped with the standard form $\omega_{std}$ with K\"ahler potential $|x|^2 + |y|^2$. When the toric coordinates $x, y$ are very small (these coordinates correspond to the weight vectors $\left<1,0\right>$ and $\left<-1,1\right>$ in $\Delta_{\tilde{Y}}$), the Fubini-Study form and the standard form are approximately the same by a Taylor expansion of log. For example when $Tr_x$ is very small, we have $\log(1+(Tr_x)^2) \approx (Tr_x)^2$. So a symplectic form can be constructed by interpolating between these two K\"ahler forms. $\square$
\end{example}

Next we do the same one dimension up, for $H=\Sigma_2$ inside an abelian surface.

\begin{figure}[h]
\centering
\begin{tikzpicture}[scale = 2.5]
\draw[help lines, dotted] (-2,-2) grid (2,2);
\draw (0,0) -- (1,0) -- (2,1) -- (2,2) -- (1,2) -- (0,1)--(0,0);
\draw[shift={(-2,-1)}] (0,0) -- (1,0) -- (2,1) -- (2,2) -- (1,2) -- (0,1)--(0,0);
\draw[shift={(-1,-2)}] (0,0) -- (1,0) -- (2,1) -- (2,2) -- (1,2) -- (0,1)--(0,0);
\node[above, red] at (1/2,0) {$x$}; \node[left, red] at (-.2,-.1) {$\displaystyle{z=\frac{v_0}{xy}}$};\node[left, red] at (0,1/3) {$y$};
\draw[->, ultra thick, red] (0,0) -- (.3,0); \draw[->, ultra thick, red] (0,0) -- (-.2,-.2);
\draw[->,ultra thick, red] (0,0) -- (0,1/4);

\draw[->, ultra thick, red] (2,1) -- (1.8,.8); \node[right,red] at (1.5, .6) {$T^{-3}v_0^{-1}z$};
\draw[->,ultra thick, shift = {(2,1)}, red] (0,0) -- (0,1/4); \node[right,red] at (2,1.3) {$y$};
\draw[->, ultra thick, shift ={(2,1)},red] (0,0) -- (.3,0); \node[right,red] at (2.3,1) {$T^{3}v_0 x$};

\draw[->, ultra thick, red] (1,2) -- (.8,1.8); \node[right,red] at (.9,1.8) {$T^{-3}v_0^{-1}z$};
\draw[->,ultra thick, shift = {(1,2)}, red] (0,0) -- (0,1/4); \node[above,red] at (1,2.3) {$T^{3}v_0y$};
\draw[->, ultra thick, shift ={(1,2)},red] (0,0) -- (.3,0); \node[right,red] at (1.3,2.1){$x$};

\draw[->,ultra thick,teal] (0,1) -- (0,.8); \node[right,teal] at (0,.8) {$T^{-2} y^{-1}$};
\draw[->,ultra thick, teal] (0,1) -- (.2, 1.2); \node[right,teal] at (.2,1.15) {$Tv_0z^{-1}=Txy$};
\draw[->, ultra thick,teal] (0,1) -- (-.2,1); \node[above, teal] at (-.4,1.05) {$Tv_0x^{-1}=Tyz$};

\draw[->,ultra thick,teal, shift = {(2,1)}] (0,1) -- (0,.8); \node[right,teal] at (2,1.8) {$T^{-2}y^{-1} $};
\draw[->,ultra thick, teal, shift = {(2,1)}] (0,1) -- (.2, 1.2); \node[right,teal] at (2.2,2.1) {$T^{4}v_0^2z^{-1}$};
\draw[->, ultra thick,teal, shift = {(2,1)}] (0,1) -- (-.2,1); \node[above, teal] at (2,2.2) {$ T^{-2}x^{-1}$};

\draw[->,ultra thick,teal, shift = {(1,-1)}] (0,1) -- (0,.8); \node[right,teal] at (1,-.2) {$Tv_0y^{-1}$};
\draw[->,ultra thick, teal, shift = {(1,-1)}] (0,1) -- (.2, 1.2); \node[right,teal] at (1.2,.1) {$Tv_0z^{-1}=Txy$};
\draw[->, ultra thick,teal, shift = {(1,-1)}] (0,1) -- (-.2,1); \node[below, teal] at (.6,-.1) {$ T^{-2} x^{-1}$};

\draw[magenta, dotted, shift = {(-.5,-.5)}] (-1.5,-1.5) -- (.5,-.5) -- (1.5,1.5) -- (-.5,.5)--(-1.5,-1.5);
\draw[magenta, dotted, shift = {(1.5,.5)}] (-1.5,-1.5) -- (.5,-.5) -- (1.5,1.5) -- (-.5,.5)--(-1.5,-1.5);
\draw[magenta, dotted, shift = {(.5,1.5)}] (-1.5,-1.5) -- (.5,-.5) -- (1.5,1.5) -- (-.5,.5)--(-1.5,-1.5);

\draw[->,ultra thick, teal] (-1,-1) -- (-.8,-.8); \node[right,teal] at (-.8,-.8) {$T^{-2}z^{-1} $};
\draw[->, ultra thick, teal] (-1,-1) -- (-1.2,-1); \node[above,teal] at (-1.4,-1) {$Tv_0x^{-1}$};
\draw[->, ultra thick, teal] (-1,-1) -- (-1,-1.2); \node[right,teal] at (-1,-1.2) {$Tv_0y^{-1}$};

\node at (0,-1) {$-\xi_2-1$}; \node[above] at (-1,0) {$-\xi_1-1$}; \node at (1,1) {$0$};

\end{tikzpicture}
\caption{Depiction of 3D $\Delta_{\tilde{Y}}$. Coordinates respect $\Gamma_B$-action, see below; magenta parallelogram = fundamental domain. Vertices = $\mb{C}^3$ charts. Coordinate transitions, see Lemma \ref{symm_lemma}. Expressions in the center of tiles indicate e.g.~$\eta \geq \vp(\xi) = -\xi_1-1$ over that tile, so $\mbox{tile}_{(0,0)}$ is given by $(\xi_1,\xi_2) \in \{ (\xi_1,\xi_2) \mid -1 \leq \xi_1,\xi_2, -\xi_1+\xi_2 \leq 1\} $.
}
\label{coords in tiling}
\end{figure}
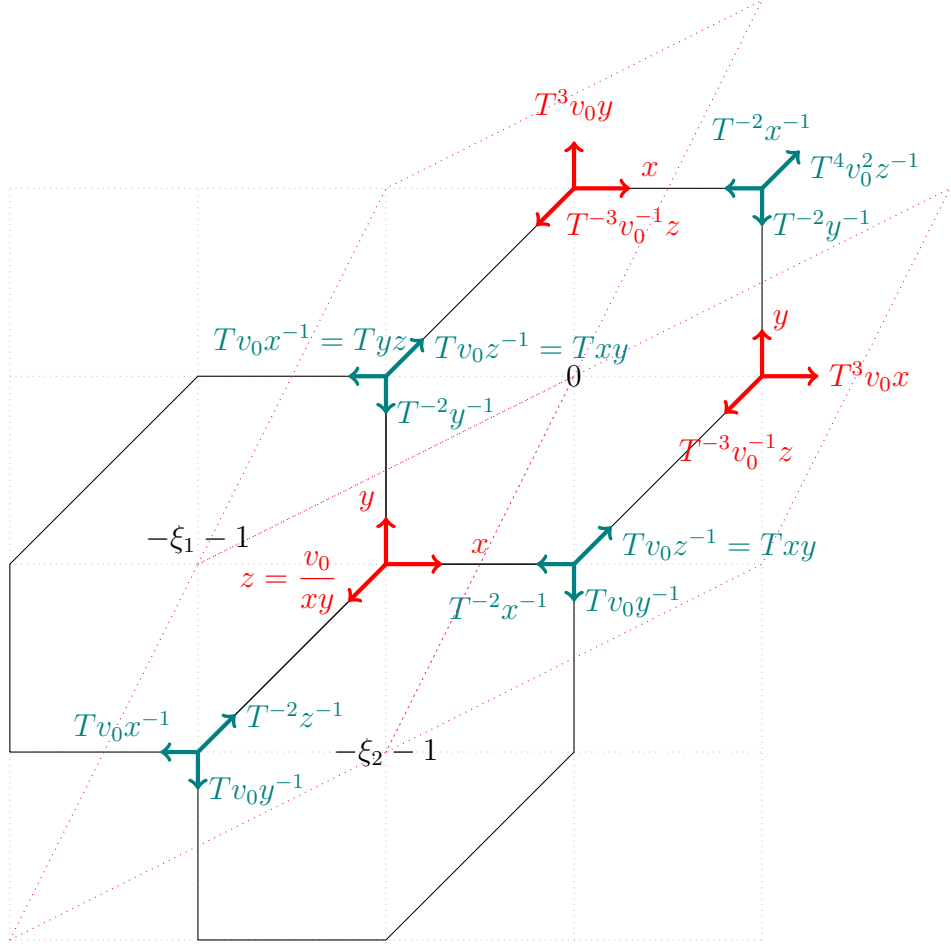

Let $r_x=|x|$ and the same for $y$ and $z$, and let $0<T<<1$ be a small number i.e.~the Novikov parameter on the genus 2 curve. The symplectic form will interpolate between $(\mb{C}^3,\omega_{std})$ with K\"ahler potential $\frac{1}{3}T^2(r_x^2+r_y^2+r_z^2)$ on $\mb{C}^3$-charts at each vertex in Figure \ref{coords in tiling}, and the toric K\"ahler form for $\mb{CP}^2(\mbox{3 points})$, the blow-up at three points, induced by the hexagon in Figure \ref{coords in tiling}. Recall from Corollary \ref{toric_metric} that such a potential is the logarithm of the sum of squares of sections corresponding to lattice points. 

 Toric geometry determines the complex coordinates up to scaling. We don't want to change the symplectic form so we scale the K\"ahler potential as well to cancel out the scaling of the complex coordinates. This reflects the phenomenon that under mirror symmetry, varying the complex structure on $Y$ doesn't change the symplectic structure. In particular, recall that introducing $T$ allows us to make a statement about mirror symmetry that is between families if we were to flip the A and B sides (analogous to the role of $\tau$ in the current direction). Since we are free to scale the sections by a scalar multiple, we may scale in a way that allows us to factor their sum and write the K\"ahler potential as:
\begin{equation}\label{eq:cp2_3_potl_def}
\begin{aligned}
g_{xy}:=\log(1+|T^a x|^2)(1+|T^b y|^2)(1+|T^c xy|^2)
\end{aligned}
\end{equation}
from a to-be-determined choice of $(a,b,c)$. The three dimensional moment polytope $\Delta_{\tilde{Y}}$ in Figure \ref{coords in tiling} has a $\mb{Z}/3$ symmetry if the $\mb{P}^1$ along each axis is defined to have the same symplectic area under $\omega$, say equal to 1. This symmetry will enforce how $\Gamma_B$ acts on the local complex coordinates.  

\begin{center}
    \fbox{\bf Above properties define a choice of complex structure on $Y$}
\end{center}

\begin{definition}[Complex coordinates on $\tilde{Y}$]\label{def:complex_coordinates} We define the following coordinate charts on the polytope $\Delta_{\tilde{Y}}$. Let $g^k$ index the chart obtained by rotating $k$ vertices clockwise from the lower-left vertex of the (0,0) hexagon. Then the coordinate charts are:
\begin{itemize}
\item $U_{\underline{0}, g^0}: (x,y,z)$
\item $U_{\underline{0},g}: (Tv_0x^{-1}, T^{-2}y^{-1}, Tv_0z^{-1})=:(x'',y'',z'')$
\item $U_{\underline{0},g^2}: (x,T^3v_0y, T^{-3}v_0^{-1}z)$
\item $U_{\underline{0},g^3}: (T^{-2}x^{-1}, T^{-2}y^{-1}, T^4v_0^2z^{-1})$
\item $U_{\underline{0},g^4}: (T^3v_0x, y, T^{-3}v_0^{-1}z)$
\item $U_{\underline{0},g^5}: (T^{-2}x^{-1}, Tv_0y^{-1}, Tv_0z^{-1})=:(x',y',z')$
\end{itemize}

Furthermore, we extend this definition to all the coordinate charts on $\tilde{Y}$ by symmetry. E.g.~note that going along the $z$ axis we get to the $g^{-1}$ vertex in the $(-1,0)$ tile and may define $U_{(-1,0),g^{-1}}: (Tv_0x^{-1}, Tv_0 y^{-1}, T^{-2}z^{-1})=:(x''',y''',z''')$. In the new coordinate system denoted $(x'',y'',z'')$ in Figure \ref{coords in tiling}, they will have the same small values as $(x,y,z)$ in the original chart. It is the same idea to obtain the coordinate charts at all the other vertices.
\end{definition}

\begin{definition}[Definition of complex structure on $Y$ by $\Gamma_B$-action]\label{gammab_action} The complex structure on $\tilde{Y}$ defines a complex structure on $Y$ as follows. We define the group action to be
\begin{equation}
    \begin{aligned}
    (-\gamma') \cdot (x,y,z)& := (T^3v_0x, y, T^{-3}v_0^{-1}z) \\
    (-\gamma'')\cdot (x,y,z) & :=(x,T^3v_0y, T^{-3}v_0^{-1}z)
    \end{aligned}
\end{equation}

The convention is that moving up and right in Figure \ref{coords in tiling} is negative since the powers of $T<<1$ are positive, hence decreasing values in the coordinates corresponds with moving in the negative direction of the group action. That is, the actions of $\gamma'$ and $\gamma''$ map the coordinates $(x,y,z)$ to the charts centered at $(-2,-1)$ and $(-1,-2)$ respectively in Figure \ref{coords in tiling}. Then $\Gamma_B$ acts properly discontinuously on the restriction of $\tilde{Y}$ to small $|v_0|$ and $Y$ is a well-defined complex manifold. In particular we obtain a product complex structure.

\end{definition}

The reason we choose this group action is explained in the proof of Claim \ref{claim:omega_descends}.

\begin{claim}\label{claim:omega_descends}
The symplectic forms for $\mb{CP}^2(3)$ (i.e.~$\omega$ on a fiber) and $\mb{C}^3$ (i.e.~$\omega$ in a neighborhood of the origin in $\mb{C}^3$) are invariant under the $\Gamma_B$ action so descend to forms on the corresponding neighborhoods in $Y$, (shaded in Figure \ref{fig:ypic} in the one dimension lower case.)
\end{claim}

\begin{proof}
The interior of each $n$th hexagonal tile in $\Delta_{\tilde{Y}}$ is identified under the $\Gamma_B$-action on the polytope, so on the K\"ahler potential for $\mb{CP}^2(3)$ we would like $\gamma^*g_{xy}$ and $g_{xy}$ to differ by an element in the kernel of $\dd \ol{\dd}$ e.g.~a harmonic function such as $|x|^2=x \ol{x}$. One way to guarantee these conditions is if $\gamma', \gamma''$ arise from a $\mb{Z}/6$-action on the hexagon. Let $G$ denote the $\mb{Z}/6$ rotation action in the first two moment map coordinates on the hexagons in the honeycomb tiling in $\Delta_{\tilde{Y}}$. Suppose that $g$ acts by 
\begin{equation}\label{gp_action}
(x,y) \mapsto (T^\alpha y^{-1}, T^\beta xy)
\end{equation}
for some $\alpha$ and $\beta$, so that the action on the $z$ coordinate is $g \cdot z = T^{-\alpha-\beta}yz$, determined by $v_0=xyz$ gluing to a global function so must be preserved under $G$. Thus:

{ \begin{equation}
 \begin{aligned}
&g^*\log(1+|T^ax|^2)(1+|T^by|^2)(1+|T^cxy|^2) \\
&= \log(1+|T^{a+\alpha}y^{-1}|^2)(1+|T^{b+\beta}xy|^2)(1+|T^{c +\alpha+\beta}x|^2)\\
g_{xy} & = \log(1+|T^ax|^2)(1+|T^by|^2)(1+|T^cxy|^2)
\end{aligned}
\end{equation}}

Comparing coefficients on $x$, $y$ and $xy$, in order for the $g_{xy}$ and $g^* g_{xy}$ to differ by a harmonic function we want $a=c+\alpha+\beta$, $b=-a-\alpha$ and $c=b+\beta$. Since we have five unknowns and only three equations, there are multiple solutions. We make a choice so we will be able to define a symplectic form, and fix $a=b=1$. Then the rest is determined: $\alpha=-2$ and $c=3-\beta=1+\beta$ hence $\beta=1$ and $c=2$. We find that 
\begin{equation}\label{eq:exponents_Kahl_potl}
    \begin{aligned}
    g_{xy}&:= \log(1+|Tx|^2)(1+|Ty|^2)(1+|T^2 xy|^2) \\
    g \cdot (x,y) &= (T^{-2}y^{-1}, Txy)
    \end{aligned}
\end{equation}
and we see that $g^2$ and $g^{-2}$ do indeed give the group action defined in Definition \ref{gammab_action}. 
\end{proof}

Now we can see where the choice of complex coordinates on $\tilde{Y}$ arose from. 

\begin{lemma}\label{symm_lemma} The action of $G$ defines the complex coordinate charts given above in Lemma \ref{lem:complex_coordinates}.
\end{lemma}
\begin{proof}
The $g^k$ allows us to index the charts, but the coordinates are a permutation of the coordinates $g^k \cdot (x,y,z)$ so the subscript is mainly for indexing. Namely, we re-label the coordinates in each $\mb{C}^3$ chart to match with the $(x,y,z)$. Since $g\cdot(x,y) = (T^{-2}y^{-1}, Txy)$, applying the group action twice we find
$$(x,y) \sim (T^{-2}y^{-1}, Txy) \sim (T^{-2}(T^{-1}x^{-1}y^{-1}), T(T^{-2}T^{-1})(Txy)) = (T^{-3}x^{-1}y^{-1}, x)$$
Since $\Gamma_B$ fixes $v_0 = xyz$ so we may rewrite the transformed $y$-coordinate as $T^3v_0y$ to obtain the $\gamma''$ action, after suitable permutation $\sigma_{g^2}$:
\begin{equation} 
\begin{aligned} 
\sigma_{g^2} \cdot g^2\cdot (x,y,z)&  =(x, T^3v_0y, T^{-3}v_0^{-1}z)\\
\implies  -\gamma''\cdot (x,y,z)& := (x, T^3v_0y, T^{-3}v_0^{-1}z)
\end{aligned}
\end{equation}
(Note that $\gamma'$ and $\gamma''$ increase the coordinate norms, analogous to the $\Gamma_B$-action increasing norms by $\tau^{-\gamma}$ on the mirror side.) The $\gamma'$ calculation is similar; since $g^{-1} \cdot (x,y) = (Txy, T^{-2} x^{-1})$ we obtain $g^{-1} \cdot (x,y,z) = (Tv_0z^{-1}, T^{-2}x^{-1}, Tv_0y^{-1})$. Then:
\begin{equation}
\begin{aligned}
-\gamma' \cdot (x,y,z) & = \sigma_{g^{-2}} \circ g^{-2} \cdot (x,y,z)\\
&= \sigma_{g^{-2}} \circ g^{-1}\cdot (Tv_0z^{-1}, T^{-2}x^{-1}, Tv_0y^{-1})\\
&= (T^{-3}v_0^{-1}x, y, T^3 v_0 z)
\end{aligned}
\end{equation}

For example, to find the coordinate system $(x''', y''', z''')$, we first move down and left by $\gamma'$ action to be in the $(0,-1)$ tile, then rotate by $g^{-1}$ and finally apply a suitable permutation.
\end{proof}

\begin{cor}
$\omega_{\mb{CP}^2(3)}$ defines a symplectic form in each $v_0$ fiber on a neighborhood given by the preimage of an open set in the interior of the hexagon at fixed $|v_0|$, (corresponding to the yellow shading of Figure \ref{fig:ypic}.) This open set is defined in Figure \ref{delineation_pic} in terms of the coordinates introduced beginning with Equation \ref{phi_def}.
\end{cor}
\begin{proof}
We are justified in defining it just on a fiber, and taking the compactification as $v_0 \to 0$ we'll obtain the fiber over 0. Fix $v_0 \neq 0$. A fiber has two complex coordinates $(x,y)$, and in particular $|v_0|$ is also fixed so in the moment map this corresponds to fixing a certain height above the base of the infinite bowl. Note that all $\xi_1,\xi_2,\eta$ will vary, but $\eta$ is a function of $(\xi_1,\xi_2)$ since all points lie on a surface. We restrict $|v_0| < T^l$ for some large power of $l$, where $T<<1$. (In other words, up to rescaling, this A-side is either non-compact from the base, or compact from boundary on the base of $v_0$.) Namely, for $|v_0| = T^l$ with $T<<1$ and $l$ sufficiently small, $\Gamma_B$ acts properly discontinuously and holomorphically so the quotient is a well-defined complex manifold $Y$.
\end{proof}

\begin{claim}[$\Gamma_B$-action on symplectic coordinates]\label{gamma_acts}  $\gamma \in \Gamma_B$ acts by $\gamma \cdot (\xi, \eta) = (\xi-\gamma, \eta - \kappa(\gamma) + \left<\xi,\la(\gamma)\right>)$.
\end{claim}

\begin{proof} Recall that the $\Gamma_B$-action on $V$ was given by $x \mapsto \tau^{-\gamma}x$. Since $\tau^{\xi} = |x|$, we find that $\Gamma_B$ acts on $\xi$ by negative addition. For the statement about $\eta$, recall that $\eta$ takes values $\eta \geq \vp(\xi)$, so since $\vp(\xi+\gamma) = \vp(\xi) -\kappa(\gamma) + \left<\xi,\la(\gamma)\right>$ 
$$\gamma \cdot \eta = \eta - \kappa(\gamma) + \left<\xi,\la(\gamma)\right>$$
when $\gamma \in \Gamma_B$.\end{proof}

The additional piece of information we need to check when taking an SYZ mirror in this setting of a $\Gamma_B$-action, is that the group action respects the Lagrangian torus fibration.

\begin{lemma}\label{lem:complex_coordinates}
The $\Gamma_B$-action commutes with the moment map, i.e.~$\forall \gamma \in \Gamma_B$, 
$$(\xi_1,\xi_2)(\gamma \cdot (x,y,z)) = \gamma \cdot (\xi_1,\xi_2)(x,y,z)$$
\end{lemma}

\begin{proof} The right-hand side is $\xi-\gamma$. In order to determine the left-hand side, we need to compute how the moment map changes as a function of the complex coordinates. To do this, we let $F$ denote the local K\"ahler potential for the symplectic form. So $F$ is interpolating between the three toric $\mb{CP}^2(3)$ potentials around a vertex. The change in moment map value can be calculated by integrating the symplectic form to compute the area of a disc, as follows.

\begin{claim}\label{area_cyl_2dirns} Consider a disc $D \subset Y$, whose lift $\tilde{D} \subset \tilde{Y}$ is invariant under the action of an $S^1$ subgroup of $T^3$. Denote the corresponding moment map by $\mu$. In particular the boundary of $\tilde{D}$ is an $S^1$-orbit $S^1.(x,y,z)$, while its center is a fixed point $(x_0,y_0,z_0)$. Then the symplectic area of the disc $D$ is equal to $\mu(x,y,z) - \mu(x_0,y_0,z_0)$.\end{claim}

\begin{proof} We claim $\mu(x,y,z)- \mu(x_0,y_0,z_0) = \int_{\tilde{D}} \omega$. The $\mb{CP}^1$ moment map image is a segment over which we can draw an elongated sphere and map a value in the segment to the area of the part traced out in the sphere. The lift $\tilde{D}$ has boundary component given by an orbit, which we can think of as the integral flow of the vector field generated by the infinitesimal action, call it $X^{\#}$. Then the integral over $\tilde{D}$ involves integrating over this $X^{\#}$ and the line from $(x,y, z)$ to $(x_0,y_0,z_0)$. Call this line $C$. Then we can write the integral as
$$\int_C\iota_{X^\#} \omega = \int_C d \mu = \mu(x,y,z) - \mu(x_0,y_0,z_0) $$
\end{proof}
So in order to compute the change in moment map coordinates, we go ahead and compute area as follows.

\begin{claim}\label{claim:sympl_area_1} The symplectic area of the $\mb{P}^1$ along each of the $x$, $y$ and $z$ axes is 1. \end{claim}

\begin{proof} In the simplified case of $\mb{P}^1$, if the gluing is normally $x \sim x^{-1}$, the analogue here would be to define a gluing $x \sim K x^{-1}$ for some constant $K$ (so a $\mb{Z}$-action). We do the computation along the $y$-axis, and the others will be the same because we will impose $x \leftrightarrow y \leftrightarrow z$ symmetry on the symplectic form. Recall $\mb{P}^1$ has an open covering $U_0, U_\infty$ and charts $\phi_0$ and $\phi_1$ sending $[z_0:z_1]$ to $z_1/z_0$ and $z_0/z_1$ respectively. We want to split up the integration over $\mb{P}^1$ into these two charts, but only the portion of the chart up to where they intersect (else we integrate over too much). So we have $y$ is the coordinate on $\phi_0(U_0) \cong \mb{C}$ then it is $T^{\alpha}y^{-1}$ on $\phi_\infty(U_\infty)$ and $|y| = |T^\alpha y^{-1}| \implies |y|=T^{\alpha/2}$. Let $C_{T^\alpha/2}$ denote this circle of complex radius $T^{\alpha/2}$. Then let $D_0$ be the disc in $\phi_0(U_0)$ and $D_\infty$ the corresponding disc in the other one. 

Let $F$ denote the local K\"ahler potential for $\omega$ and $F_0, F_\infty$ be the K\"ahler potential in the two charts $U_0$ and $U_\infty$ on the $y$-axis $\mb{P}^1$. Then for K\"ahler potential $F$ given above from the toric K\"ahler form on $\mb{CP}^2(3)$:
\begin{allowdisplaybreaks}
\begin{align*}
\int_{\mb{P}^1} \frac{i}{2\pi} \dd\ol{\dd} F &=\frac{i}{2\pi} \left[\int_{\phi_0^{-1}(D_0)} d(\ol{\dd} F) +\int_{-\phi_\infty^{-1}(D_\infty)} d(\ol{\dd} F)\right] \\
& = \frac{i}{2\pi} \left[\int_{\dd D_0} \ol{\dd} (\phi_0^{-1})^*F - \int_{\dd D_\infty} \ol{\dd} (\phi_\infty^{-1})^*F\right]\\
& = \frac{i}{2\pi} \int_{C_{T^{\alpha/2}}} \ol{\dd}(F_0 - F_\infty)\\
& = \frac{i}{2\pi}\int_{C_{T^{\alpha/2}}} \ol{\dd} \log(|T^{\alpha/2}y|^2) = \frac{i}{2\pi}\int_{C_{T^{\alpha/2}}} \frac{T^{\alpha/2} y d\ol{y}}{|T^{\alpha/2}y|^2}\\
& = \frac{i}{2\pi}\int_0^{2\pi} e^{i \theta} (-i)e^{-i\theta} d\theta = 1
\end{align*}
\end{allowdisplaybreaks}
%
\end{proof}

{\bf Commutes in $\xi_1,\xi_2$.} Thus when gluing across the $z$-axis the coordinate chart $U_{(-1,0),g^{-1}}$ with coordinates $(x''',y''',z''')$ to the main coordinate chart $U_{0,g^0}$ with coordinates $(x,y,z)$, the K\"ahler potential transforms by $F'''_V = F_V - \log(|Tz|^2)$. Thus this discrepancy between the local K\"ahler potentials implies that the value of $\xi_1$ is modified by adding the constant $-1$, and similarly for $\xi_2$. Similarly, when gluing across the $x$-axis the main chart to the chart $U_{0,g^5}$ with coordinates $(x',y',z')$ the local K\"ahler potentials differ by $-\log(|Tx|^2)$, which modifies $(\xi_1,\xi_2)$ by $(1,0)$. Combining the two transformations amounts to the action of $\gamma'$, namely changing coordinates from $(x''',y''',z''')$ to $(x',y',z')$. This action modifies $(\xi_1,\xi_2)$ by adding $(2,1)$, which is exactly $\gamma'$. Similarly for $\gamma''$. This completes the proof that the $T^2$-moment map is $\Gamma_B$-equivariant.
\end{proof}  

We will need the following claim later when we define the full symplectic form.

\begin{claim}\label{claim:monotonic_mom_map} The moment map coordinates $(\xi_1,\xi_2)$ are monotonic increasing functions of $r_x$ and $r_y$ in a $v_0=xyz$ fiber.
\end{claim}

\begin{proof}Recall the action-angle coordinates $(\xi,\theta)$ of Claim \ref{claim:action_angle_fiber} from the fiber-wise action $\rho$:
$$\rho: T^2 \times V^\vee \ni (\alpha_1,\alpha_2)\cdot (x,y,z) \mapsto (e^{2i\pi \alpha_1} x, e^{2i\pi \alpha_2} y,e^{2i\pi (-\alpha_1-\alpha_2)}z ) \in  V^\vee$$
Let $X_i:= d\rho(\dd_{\alpha_i})$. If $\iota_{X_i} \omega|_{V^\vee}$ were exact, say $dH_i$ for some function $H_i$ (known as the Hamiltonian), then the torus action would be called a \emph{Hamiltonian group action} and $(H_1,H_2)$ would be the moment map. This leads us to the caveat at the start of Section \ref{caveat}. In the setting here, $\omega$ is complicated so computing $\iota_{X_i} \omega|_{V^\vee}$ is involved. However, the 1-forms $\iota_{X_i} \omega$ are closed, and hence locally exact, so there exist locally defined functions $\xi_1, \xi_2$ so that $\iota_{X_i}\omega = d\xi_i$. Notationally we are assuming $\xi_\eta \equiv \eta$. Globally, as seen above, the first two $\xi_i$ are $\Gamma_B$-periodic and their differentials pass to 1-forms on a torus fiber $\mb{R}^4/\Gamma$. The infinitesimal action is expressed by the pushforwards $d\rho(\dd_{\alpha_i}) $ for $i \in \{1,2\}$. These vector fields are
\begingroup \allowdisplaybreaks \begin{equation}
\begin{aligned}
X_1=\frac{\dd}{\dd \theta_1}  = (ix, 0, -iz), \qquad
X_2=\frac{\dd}{\dd \theta_2} = (0, iy, -iz)
\end{aligned}
\end{equation} \endgroup

Then the action coordinates can be expressed locally in terms of the K\"ahler potential:
\begingroup \allowdisplaybreaks
\begin{align*}
\iota_{X_i} \omega|_{V^\vee} &= d\xi_i \implies \omega|_{V^\vee} = d\xi \wedge d\theta \\
\implies \iota_{X_i} \left(\frac{i}{2\pi} \dd \ol{\dd} F\right) & = \frac{i}{2\pi} d\iota_{X_i} {\dd} F = d \xi_i\\
\therefore \xi_1 & := i/2\pi \partial F ( 2\pi i x \partial_x - 2\pi i z \partial_z)+const \\
\therefore \xi_1 & = -\frac{1}{2}\left(\frac{\dd F}{\dd \log |x|} - \frac{\dd F}{\dd \log|z|}\right)+const\\
\therefore \xi_2 & = -\frac{1}{2}\left(\frac{\dd F}{\dd \log |y|} - \frac{\dd F}{\dd \log|z|}\right) +const
\end{align*}
\endgroup
using $\dd \ol{\dd} = -\ol{\dd} \dd = -d \dd$, the conversion to polar coordinates from \textsection \ref{convert_polar} and that $F$ is preserved by rotating $x,y,z$ as it is a function of their norms, hence the Lie derivative $\mc{L}_{X_i}\ol{\dd} F = 0$ and also $\dd_{\theta_x}F = 0$. The calculation for $\xi_2$ is similar. This calculation of $\xi_i$ is up to additive constants. The upshot is that the moment map $(\xi_1,\xi_2)$ provides periodic action-coordinates which are monotonic increasing in $|x|$ and $|y|$ for fixed $v_0$ because of how we defined $F$. (But recall the caveat, we are calling it a moment map but it takes values in a torus instead of affine space, so we are expanding the definition of moment map here to allow the $\xi_i$ to be periodic multivalued functions for a \emph{quasi-Hamiltonian action}.) We see that $V^\vee$ is symplectomorphic to $(T_B \times T_F, \omega_{std})$. 
\end{proof}

Note that an orbit is precisely the kernel of $(d\xi_1,d\xi_2)$ so that a preimage of a moment map value is a $T^2$-orbit. Said another way, the torus action preserves the moment map. Or said yet another way, the tangent space to a fiber of the moment map is $\dd_{\theta_1},\dd_{\theta_2}$.

{\bf How to view $\eta$ as a moment map coordinate.} Let $\theta_\eta=\arg(v_0)=\arg(xyz)$. We can't extend the above to a $T^3$-action. If we were to rotate $v_0$ too, then $x$ and $y$ would change in the process as well, due to monodromy. Because $T^3$ has trivial bracket on its Lie group, we should rotate one variable while fixing the others. When $v_0 \not \in \mb{R}_+$ the angle variables $\theta_1$ and $\theta_2$ on $V^\vee$ are only well-defined up to additive constants, as they jump by $\arg(v_0) = \theta_\eta$ under the action of the generators of $\Gamma_B$ due to the transformation rules for the complex coordinates $x$ and $y$, see Corollary \ref{gammab_action}. Thus we can only define an infinitesimal $T^3$ action generated locally by sufficiently small rotations $\alpha_i$,  $P (\alpha_1, \alpha_2, \alpha_\eta) \cdot (x,y,z) = (e^{2i\pi \alpha_1} x, e^{2i\pi \alpha_2} y, e^{2i\pi (\alpha_\eta-\alpha_1-\alpha_2)} z)$. The infinitesimal action on $v_0$ is expressed by the pushforward $dP(\dd_{\alpha_\eta}) $:
\begingroup \allowdisplaybreaks \begin{equation}
\begin{aligned}
\frac{\dd}{\dd \theta_\eta}  = (0,0,iz)
\end{aligned}
\end{equation} \endgroup

This isn't global as it transforms non-trivially under $\Gamma_B$. (Note that, notationally, $\xi_\eta$ is the same thing as $\eta$.) The upshot is that upon integrating we find that 
$$\eta = \frac{1}{2} \cdot  \frac{\dd F}{\dd \log|z|} + const$$

Though the action coordinates $(\xi_1,\xi_2,\eta)$ are globally well-defined on the universal cover $\tilde{Y}$, on which the $T^3$-action discussed above is well-defined and Hamiltonian, $\eta$ is only defined locally on $Y$. This coordinate corresponds to complex coordinates on $X$ via the $S^1$-action rotating the coordinate $y\in \mb{C}$ in $Bl_{H \times \{0\}} V \times \mb{C}$. That action has moment map $\mu_X(\bm{x},y)$ given in \cite[Equation (4.1)]{AAK}\label{page_mu}, and then $|y|$ is determined by the choice of $\eta$ since $X$ and $Y$ have the same base as Lagrangian torus fibrations, with base coordinates given by $(\xi_1,\xi_2,\eta)$.

%
%
%

\begin{definition}[Superpotential] The \emph{superpotential $v_0$} is the holomorphic function $Y \to \mb{C}$
\begingroup \allowdisplaybreaks \begin{equation}
v_0(x,y,z) :=xyz
\end{equation} \endgroup
which is well-defined as a global function on $Y$ because $\gamma^*v_0 = v_0$ for all $\gamma \in \Gamma_B$.
\end{definition}

\begin{example}
To see why the fibers are complex tori with complex coordinates $x$ and $y$, e.g. consider just $x \sim 2x$ on $\mb{C}^*$: if we identify the unit circle with the circle of twice the radius, we've just formed an elliptic curve or 2-torus. Now we compactify: if we have $x \sim T^n x$ for all $n$ then taking $n\to \pm \infty$, we see that 0 and $\infty$ are identified to give the pinched torus.
\end{example}

\begin{remark}[Mirror to non-standard symplectic form on mirror] Note that the symmetry properties required for $\omega$ told us what complex structure was needed (it produces a product complex structure), which by mirror symmetry will correspond to a specific symplectic form on A-model of $X$. The complex structure is a product here on $x$ and $y$. In particular, the $\Gamma_B$-action on complex coordinates here is different than that described in \cite[\textsection 10.2]{AAK}. The complex structure there is mirror to the standard K\"ahler form and is defined by:
\begingroup \allowdisplaybreaks \begin{equation}
 \bm{v}^m \sim v_0^{\left<\la(\gamma),m\right>}T^{\left<\gamma,m\right>}\bm{v}^m
 \label{aak vars}
 \end{equation} \endgroup
where they use complex coordinates $\bm{v} = (v_1,v_2)$. Thus setting $(x,y,z)=(v_1^{-1}, v_2^{-1}, v_0v_1v_2)$, their complex structure arises from the following $\Gamma_B$-action:
\begin{align*}
\gamma'\cdot (x,y,z) &= (T^{-2}v_0^{-1}x, T^{-1}y, T^{3}v_0 z)\\
\gamma''\cdot (x,y,z) &= (T^{-1}x, T^{-2}v_0^{-1}y, T^{3}v_0 z)
\end{align*}
So our complex structure is mirror to a non-standard K\"ahler form on the genus 2 curve.
\end{remark}

\begin{remark}[Terminology] The toric variety $\tilde{Y}$ is referred to as a \emph{toric variety of infinite-type} by \cite{kl}, because of the infinitely many facets, where the neighborhood of the toric divisor there is the same as our restriction to $|v_0|$ small, i.e.~$\eta$ small. 
\end{remark}

%
%
%
%

\subsection{The symplectic form}\label{section:the_symp_form}
Now we define the symplectic form, first on a fiber, as a function of the norms $r_x, r_y, r_z$ viewed in the moment polytope. See Figure \ref{interpolate}. We are starting with a polytope, which we want to be the moment map image with respect to some symplectic form that we construct. In particular, we have found a symplectic form so that the boundary $\mb{P}^1$'s of the hexagon have length 1 i.e.~symplectic area 1. This follows from the change in K\"ahler potential under gluing across coordinate axes, see Claim \ref{gamma_acts}. 

The three tiles adjacent to the main corner define toric coordinates $x,y$ and $x,z$ and $y,z$ respectively. Recall the three $\mb{CP}^2(\text{3 points})$ potentials are denoted $g_{xy},g_{xz}, g_{yz}$ by Equations (\ref{eq:cp2_3_potl_def}) and (\ref{eq:exponents_Kahl_potl}). In between we interpolate between the two potentials on either side. In the Roman-numbered regions, all three potentials are at play and all of $r_x,r_y,r_z$ are very small. They are small because in a fiber we fix $v_0$ and we've restricted to small $v_0$ in the definition of $\Delta_Y$ in Equation (\ref{eq:polytope_def_Y}). So if $|v_0| = T^{l}$, then $Tr_*$ are each of the form $T^{l_*}$ where $l_x+l_y+l_z=l$. Furthermore in the Roman-numbered sections around the vertex, because all three toric coordinates go to zero at the vertex, near the vertex their norms are still small by continuity. Geometrically, the toric coordinates are small in the corresponding region of a torus fiber close to the zero fiber.  
\begin{center} 
\begin{figure}
\begin{multicols}{2}
\includegraphics[scale=0.3]{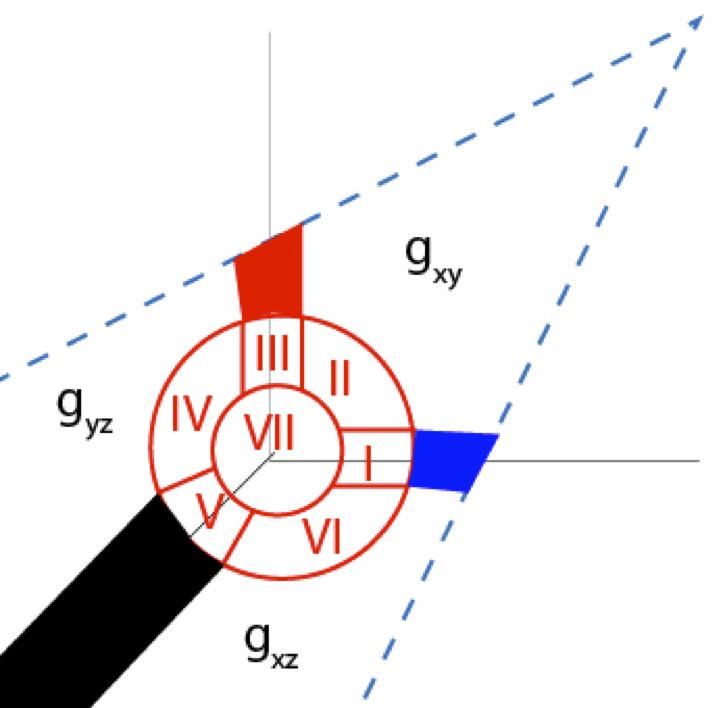}

\includegraphics[scale=0.2]{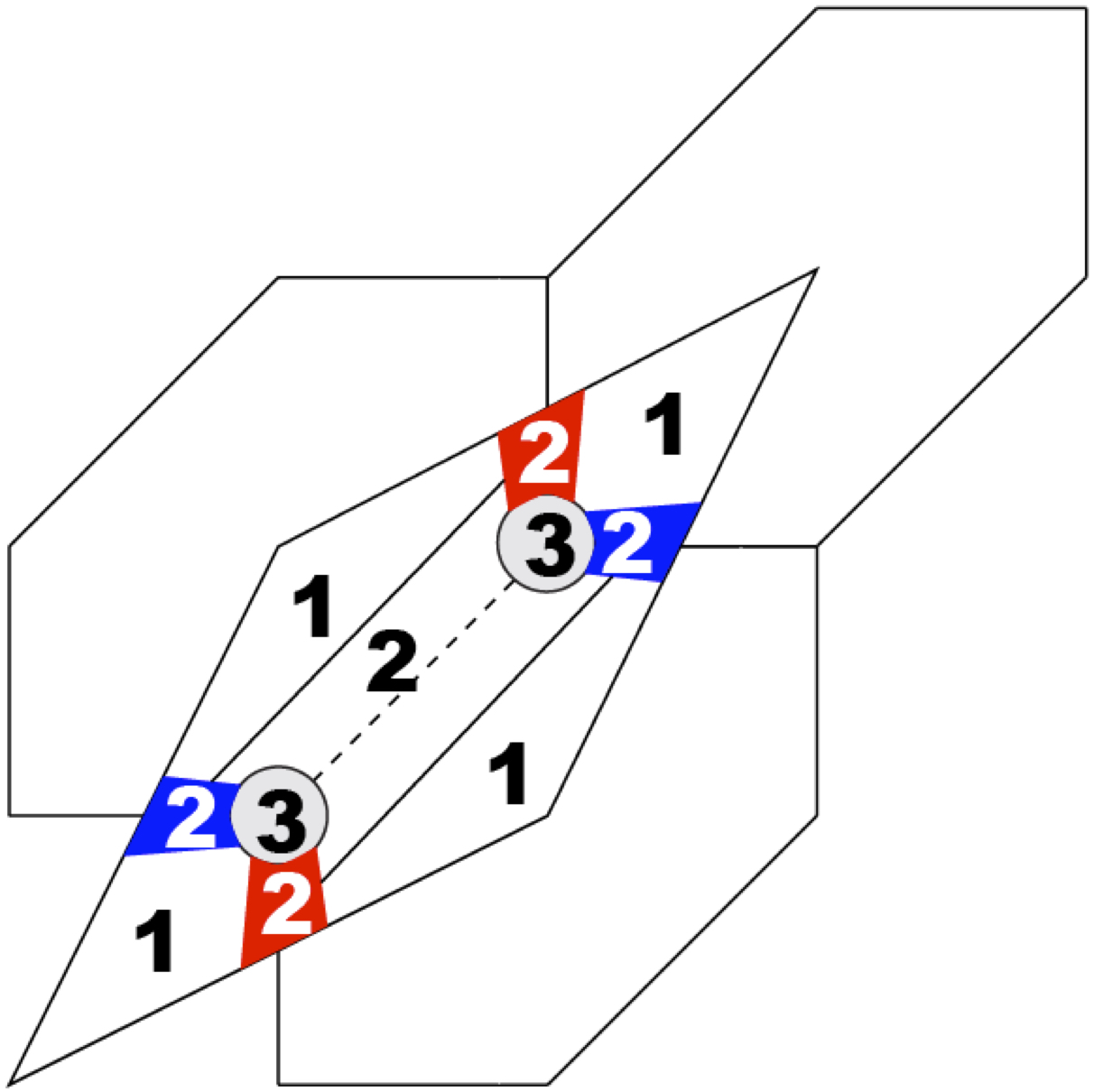}
\end{multicols}
\caption{L) Regions near a vertex, R) Number of regions interpolated between}
\label{interpolate}
\end{figure}
\end{center}
\begin{center}
    \fbox{\bf Definition of new radial and angular coordinates $d$ and $\theta$, definition of $\omega$}
\end{center}

We introduce local real radial and angular coordinates $d$ and $\theta$ to define delineations on the moment polytope around the vertex. Their subscripts indicate which region of Figure \ref{interpolate} they are defined in. This will allow us to interpolate between the three K\"ahler potentials $g_{xy}, g_{xz}, g_{yz}$ as functions of these local coordinates. We fix $|v_0| (= |xyz|) = T^l$ for $T<<1$ and $l$ a large positive constant, so that in the red regions denoted by Roman numerals of Figure \ref{interpolate} we have $r_x,r_y,r_z<<1$ and can use approximations to simplify the radial and angular coordinates. We start with
\begin{equation}\label{eq:interpolate}
F = \alpha_1 g_{xy} + \alpha_2 g_{xz} + (1-\alpha_1-\alpha_2)g_{yz};\;\; 1/3\leq \alpha_1,\alpha_2; \;\; \alpha_1+\alpha_2 \leq 1
\end{equation}
where Region VII is where $\alpha_1 = \alpha_2 = 1/3$ and $F  = \frac{1}{3}(g_{xy} + g_{xz} + g_{yz}) \approx \frac{2}{3}((Tr_x)^2 + (Tr_y)^2 + (Tr_z)^2)$ via the log approximation and in the other regions $\alpha_1,\alpha_2$ interpolate between the three K\"ahler potentials.

We want $d$ and $\theta$ to locally around a vertex be approximated by Figure \ref{angular} where we can read off how $r_x, r_y$ and $r_z$ compare to each other. It is hard to define them as functions of the moment map coordinates, but we use that the moment map coordinates are monotonic increasing functions in the norm coordinates. This is proven in Claim \ref{claim:monotonic_mom_map}. In other words, $r_x$ increases in the $(1,0,0)$ direction, $r_y$ increases in the $(0,1,0)$ direction on the polytope and $r_z$ increases in the $(-1,-1,1)$ direction. These motivate the radial variable definition. Hence, as shown in Figure \ref{angular} we see that e.g.~$T^2(r_y^2-r_z^2)$ increases in the upward vertical direction. This motivates the angular variable definition. Namely our choices for $d$ and $\theta$ approximate to the expressions in Figure \ref{angular}.  
 
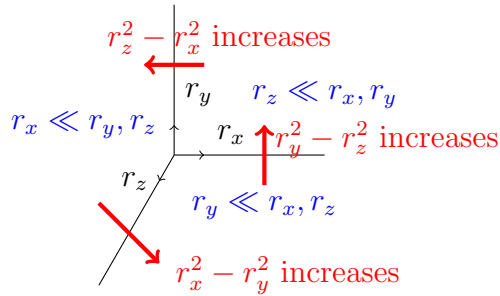
\begin{figure}[h]
\centering
\begin{tikzpicture}[scale = 0.4]
\draw[->] (0,0) -- (1,0); \draw (1,0) -- (5,0);
\draw[->] (0,0) -- (0,1); \draw (0,1) -- (0,5);
\draw[->](0,0) --(-1/2, -0.866); \draw (-1/2, -0.866) -- (-5/2, -4.33);
\node[blue] at (5,2) {$r_z \ll r_x, r_y$};
\node[blue] at (-3,1) {$r_x \ll r_y, r_z$};
\node[blue] at (3,-1.7) {$r_y \ll r_x, r_z$};
\node [above] at (1.9,0) {$r_x$}; \node [right] at (0,2) {$r_y$}; \node [left] at (-1/2, -0.866) {$r_z$};

\draw[->, ultra thick, red] (3,-1) -- (3,1); \node [right,red] at (3,1/2) {$r_y^2-r_z^2$ increases};
\draw[->, ultra thick, red]  (1,3) -- (-1,3); \node [above,red] at (1.5,3) {$r_z^2-r_x^2$ increases};
\draw[->, ultra thick, red] (-5/2,-1.598) -- (-1/2,-3.598); \node [right,red] at (-0.3,-4) {$r_x^2-r_y^2$ increases};
\end{tikzpicture}
\caption{How the three angular directions vary for $r_xr_yr_z$ constant on a fiber}\label{angular}
\end{figure}

We define functions $\phi_x, \phi_y, \phi_z$ for expressions we will use often in the coming definitions.

\begin{allowdisplaybreaks}
\begin{align*}
    \phi_x(x,y,z) & := \log_T\frac{1+|Tx|^2}{1+|T^2yz|^2}\\
\phi_y(x,y,z) & := \log_T \frac{1+|Ty|^2}{1+|T^2xz|^2}\stepcounter{equation}\tag{\theequation}\label{phi_def}\\
\phi_z(x,y,z) & := \log_T\frac{1+|Tz|^2}{1+|T^2xy|^2}
\end{align*}
\end{allowdisplaybreaks}

\begin{definition}[$d$ and $\theta$ coordinates and their approximations for $r_x,r_y,r_z<<1$]
\begin{allowdisplaybreaks}
\begin{align*}
d_I  := & \phi_x - \frac{1}{2}(\phi_y + \phi_z) = \log_T\left(\frac{1+|Tx|^2}{1+|T^{2}v_0x^{-1}|^2}/ \sqrt{\frac{1+|Ty|^2}{1+|T^{2}v_0y^{-1}|^2} \cdot\frac{1+|Tz|^2}{1+|T^{2}v_0z^{-1}|^2}}\right)\\
\approx & (Tr_x)^2   - \frac{1}{2} \left( (Tr_y)^2 +(Tr_z)^2  \right) \stepcounter{equation}\tag{\theequation}\label{polar_coords}
 \\
\theta_I & := \phi_y - \phi_z = \log_T\left(\frac{1+|Ty|^2}{1+|Tz|^2} \cdot \frac{1+|T^2xy|^2}{1+|T^2xz|^2}    \right) { \approx   (Tr_y)^2  -(Tr_z)^2 } \\
\\
\hline
\end{align*}
\end{allowdisplaybreaks}
\begingroup \allowdisplaybreaks 
\begin{align*}
\textcolor{black}{d_{IIA}} &:= \phi_x - \frac{1}{2}(\phi_y + \phi_z) + \frac{3}{2}\alpha_6 (\theta_{II}) \cdot \phi_y  \approx \textcolor{black}{T^2[r_x^2 - \frac{1}{2}(r_y^2+r_z^2)+\frac{3}{2}\alpha_6(\theta_{II}) \cdot r_y^2] }\stepcounter{equation}\tag{\theequation}\label{polar_coords}
\\
d_{IIB} & := \phi_x + \phi_y - \frac{1}{2}\phi_z \\
d_{IIC} &: =\phi_y - \frac{1}{2}(\phi_x + \phi_z)  +\frac{3}{2} \alpha_6(-\theta_{II}) \cdot \phi_x\\
\theta_{II} &:= \log_T r_y - \log_T r_x\\
\end{align*} \endgroup
where $\alpha_6$ will be a cut-off function of the angular direction. By symmetry, we define

\begin{multicols}{2}
\begin{allowdisplaybreaks}

\begin{align*}
d_{III} & := \phi_y - \frac{1}{2}(\phi_x + \phi_z)\stepcounter{equation}\tag{\theequation}\\
\theta_{III} & := \phi_z - \phi_x
\end{align*}

\begin{equation*}
\begin{aligned}
d_{V} &: =\phi_z - \frac{1}{2}(\phi_x+\phi_y)\\
\theta_{V} & :=  \phi_x-\phi_y
\end{aligned}
\end{equation*} 
\end{allowdisplaybreaks}

\end{multicols}

\begin{figure}
\begin{center}
 \includegraphics[scale=0.5]{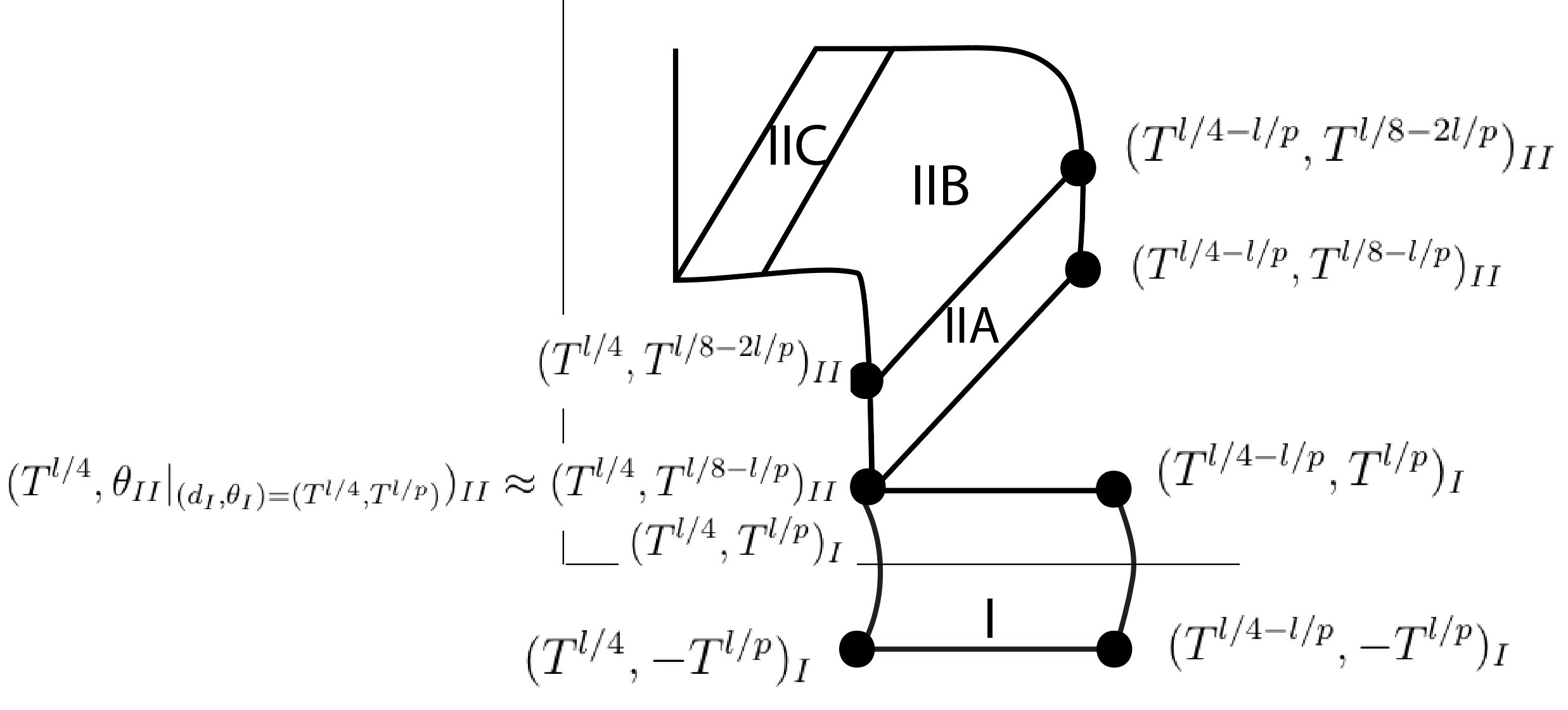}
\end{center}
\caption{Delineating regions in coordinates $(d_I, \theta_I)$ and $(d_{II}, \theta_{II})$}
\label{delineation_pic}
\end{figure}

In Figure \ref{delineation_pic}, we define regions $I$ and $IIA$ of the polytope in terms of $(d_I,\theta_I)$ and $(d_{II}, \theta_{II})$ coordinates (indicated as $(,)_I$ or $(,)_{II}$ respectively) using the approximations for $d$ and $\theta$. This defines the rest of the regions by symmetry and completes the definition of the $d$ and $\theta$ coordinates.
\end{definition}

Figure \ref{delineation_pic} will allow us to estimate how much $r_x, r_y$ and $r_z$ vary in each of the regions and show there is enough space to bound the second derivatives of the bump functions used. Namely, in order to go from the values of the bump functions at either end of a region, the function has space to grow sufficiently gradually that the slope and rate of change of slope can be made small and taking second order derivatives in the definition of the symplectic form will not have contributions from the $\alpha_i$ making it degenerate. This is proven in Appendix A. Now we define the symplectic form in terms of $r_x, r_y, r_z$ in these regions.

\begin{definition}[{\bf Definition of symplectic form on $Y$}] We set $\omega|_{V^\vee} = \frac{i}{2\pi}F$ where $F$ is defined locally as follows in terms of the coordinates in Equation \ref{polar_coords} and 
\begingroup \allowdisplaybreaks \begin{equation}
\begin{aligned}
g_{yz} & =  \log(1+|Ty|^2)(1+|Tz|^2)(1+|T^2yz|^2)\\
g_{xz} & =  \log(1+|Tx|^2)(1+|Tz|^2)(1+|T^2xz|^2)\\
g_{xy} & =  \log(1+|Tx|^2)(1+|Ty|^2)(1+|T^2xy|^2)
\end{aligned}
\label{f_fns}
\end{equation} \endgroup

We introduce new bump functions $\alpha_3,\alpha_4,\alpha_5,\alpha_6$ of the new variables $(d,\theta)$ as follows:
$$ \frac{2}{3} \leq \alpha_3(d_I)  =\alpha_1+\alpha_2 \leq 1,\; -\frac{1}{2}\leq \alpha_4(\theta_I) \leq \frac{1}{2},\; 0 \leq \alpha_5(d_I) \leq 1, \; 0 \leq \alpha_6(\theta_{II}) \leq 1$$
$$\alpha_4(\theta_I)\cdot \alpha_5(d_I) =\frac{1}{2}(\alpha_1-\alpha_2)$$
These bump functions are smooth, increasing as functions of the specified variable, and near the ends of their domain of definition they are constant at the bounds given. We also require that $\alpha_4$ is an odd function. See the subsection below entitled {\bf Motivation for new bump functions defined in $\omega$} for an explanation of the properties of these bump functions. Now the definition is as follows, noting that $g_{xz}-g_{yz} = \phi_x - \phi_y$ and similarly permuting $(x,y,z)$:

\begin{allowdisplaybreaks}
{\small\begin{align*}
 \mbox{Regions $g_{*\bullet}$ in Fig \ref{interpolate}: } F&=g_{xy}, F=g_{yz}, F=g_{xz} \mbox{ respectively}\\
\mbox{Region I: } F &= g_{yz} + \alpha_3(d_I)d_I + \alpha_4(\theta_I)\alpha_5(d_I)\theta_I\\
\mbox{Region IIA: } F & = g_{yz} - \alpha_6(\theta_{II}) \phi_y + \alpha_3(d_{IIA}) d_{IIA}+ \frac{1}{2} \alpha_5(d_{IIA}) (\phi_y-\phi_z-\alpha_6(\theta_{II})\phi_y) \\
\mbox{Region IIB: } F & = (g_{yz} - \phi_y) + \alpha_3(d_{IIB})d_{IIB} - \frac{1}{2} \alpha_5(d_{IIB}) \phi_z \stepcounter{equation}\tag{\theequation}\label{def_s_form}\\ 
\mbox{Region IIC: } F &= g_{xz} - \alpha_6(-\theta_{II}) \phi_x + \alpha_3(d_{IIC}) d_{IIC} + \frac{1}{2} \alpha_5(d_{IIC}) (\phi_x-\phi_z-\alpha_6(-\theta_{II})\phi_x)\\
\mbox{Region VII: } F& =  \frac{1}{3}(g_{xy} + g_{xz}+g_{yz})
\end{align*}}
\end{allowdisplaybreaks}

These formulas match at the boundaries, which allows us to define the rest of the regions III -- VI similarly to I and II by symmetry, via permuting the coordinates $(x,y,z)$. For example, one can check that the formula for region IIA agrees with that for region VII when $\alpha_3 = \frac{2}{3}$ and $\alpha_5 = 0$; with that for region I when $\alpha_6 = 0$ and $\alpha_4 = \frac{1}{2}$; with that for region IIB when $\alpha_6 = 1$; and with $g_{xy}$ when $\alpha_3 = \alpha_5 = 1$. The calculation is similar for the other regions.

Along the coordinate axes (namely the regions shaded red, blue, and black) we interpolate between the relevant $g_{*\bullet}$'s using the same formulas as in regions I, III, and V, with $\alpha_3 \equiv 1$. E.g.~along the $r_x$-axis (blue region) the formula is $F = \frac{1}{2}(g_{xy}+g_{xz}) + \alpha_4(\theta_I)(g_{xy}-g_{xz})$ and similarly for the other edges. 

Finally by adding a term proportional to $|xyz|^2$ on the base, with sufficiently large constant of proportionality, we obtain a non-degenerate form $\omega$ on $Y$. This completes the definition of the symplectic form.
\label{defn of w}
\end{definition}

\begin{center}
\fbox{\bf Motivation for new bump functions in $\omega$} \label{motivation_w}    
\end{center}

Note that the set $\{\alpha_1,\alpha_2, 1-\alpha_1 -\alpha_2\}$ is asymmetric but all three should be treated symmetrically. In other words, if we rotate $(\alpha_1,\alpha_2)$ thought of as a vector, by $\pi/4$, we get something proportional to
$$\left(\begin{matrix} 1 & -1\\ 1 & 1 \end{matrix}\right)\left(\begin{matrix} \alpha_1 \\ \alpha_2 \end{matrix} \right) = \left(\begin{matrix}\alpha_1-\alpha_2\\\alpha_1+\alpha_2\end{matrix}\right)$$
We accordingly rearrange terms of the initial expression of $F$ in terms of $\alpha_1-\alpha_2$ and $\alpha_1+\alpha_2$. 
\begin{allowdisplaybreaks}
\begin{align*}
F & = \alpha_1 g_{xy} + \alpha_2g_{xz} +(1-\alpha_1-\alpha_2)g_{yz}\\
& = \alpha_1\log(1+|Tx|^2)(1+|Ty|^2)(1+|T^2xy|^2) \\
&+ \alpha_2\log(1+|Tz|^2)(1+|Tx|^2)(1+|T^{2}xz|^2) \\
& + (1-\alpha_1-\alpha_2) \log(1+|Tz|^2)(1+|Ty|^2)(1+|T^{2}yz|^2)\stepcounter{equation}\tag{\theequation}\label{eq:F_rotate}\\
& = g_{yz} +(\alpha_1+\alpha_2)\phi_x-\left(\frac{\alpha_1+\alpha_2}{2} - \frac{\alpha_1-\alpha_2}{2}\right)\cdot  \phi_y-\left(\frac{\alpha_1+\alpha_2}{2} + \frac{\alpha_1-\alpha_2}{2}\right)\cdot \phi_z \\
& = g_{yz} +(\alpha_1+\alpha_2)(\phi_x - \frac{1}{2}(\phi_y+\phi_z)) + \frac{1}{2}(\alpha_1-\alpha_2)(\phi_y - \phi_z)
\end{align*}
\end{allowdisplaybreaks}

We want bump functions to be multiplied by the variable they are a function of because $\frac{d\alpha}{d(\log_T \mu)}=\alpha'(\mu) \cdot \mu \approx \frac{\Delta \alpha}{\Delta \log_T \mu}$, so we can find estimates on terms containing $\alpha'(\mu) \cdot \mu$ when $\mu$ is the argument of $\alpha$. In particular, when $\mu$ is $d\approx (Tr_*)^2$ or $\theta \approx (Tr_*)^2$ we can estimate the $\Log_T$-derivative of $\alpha(\mu)$ as $\frac{O(1)}{l}$. Recall from Equation (\ref{polar_coords}) that $d_I:=\phi_x - \frac{1}{2}(\phi_y+\phi_z) \approx (Tr_y)^2 - (Tr_z)^2$ is a radial direction in region $I$ and $\theta_I:=\phi_y-\phi_z \approx (Tr_y)^2 - (Tr_z)^2$ an angular direction in region I. So define 
\begin{equation}
\alpha_3 = \alpha_1+\alpha_2, \;\; 2/3 \leq \alpha_3 \leq 1
\end{equation}
to be a bump function of $d_I$ going from 2/3 to 1 in region I, by Equation (\ref{eq:interpolate}).

We use $\frac{1}{2}(\alpha_1-\alpha_2)$ to define a bump function $\alpha_4$ that varies in the angular direction $\theta_I$. The range of $\alpha_1-\alpha_2$ depends on $d_I$. At the start of region I we have $\alpha_1=\alpha_2 = 1/3$ so $\alpha_1-\alpha_2$ goes from 0 to 0 as we trace out the angle $\theta_I$. However at the end of region I we go from $(\alpha_1,\alpha_2)=(0,1)$ at the bottom to $(\alpha_1,\alpha_2) = (1,0)$ at the top, so that $\frac{1}{2}(\alpha_1-\alpha_2)$ goes between $-1/2$ and $1/2$. Thus we multiply $\alpha_4(\theta_I)$ by another bump function $\alpha_5(d_I)$ of the radial direction that goes from 0 to 1, so we scale the interval that $\alpha_4$ varies in. We define $\alpha_4$ to be an odd function for symmetry reasons. Let 
\begin{equation}
    \alpha_4(\theta_I) \alpha_5(d_I) = \frac{1}{2}(\alpha_1-\alpha_2), \;\; -1/2 \leq \alpha_4 \leq 1/2, \;\; 0 \leq \alpha_5 \leq 1
\end{equation}
Now we can rewrite Equation (\ref{eq:F_rotate}) for $F$ as:
\begin{equation}
F = g_{yz} +  \alpha_3(d_I) \cdot d_I + \alpha_4(\theta_I)\cdot \alpha_5(d_I)\cdot \theta_I
\end{equation}
We can similarly define $F$ in regions III and IV, and then region II will interpolate between regions I and III. This is where the defining equations for $\omega|_{V^\vee}$ in Definition \ref{defn of w} came from.

\subsection{Leading order terms in $\omega$}\label{convert_polar} We convert $\dd\ol{\dd}F$ to polar coordinates where calculations are easier, using the real transformation $(r_x,\theta_x) \leftrightarrow (x_1,x_2)$, where $x = x_1 + ix_2 = r_x e^{i\theta_x}$ and similarly for $y$. Recall (e.g.~\cite{huy})
\begin{equation}
\dd_x = \frac{1}{2}(\dd_{x_1} - i \dd_{x_2}), \dd_{\ol{x}} = \frac{1}{2}(\dd_{x_1} + i \dd_{x_2})    
\end{equation} 
By the Chain Rule, since $r_x^2 = x_1^2 + x_2^2$,  $\theta_x = \tan^{-1}(x_2/x_1)$, and $\arctan(t)' = 1/(1+t^2)$:
\begin{equation}
\begin{aligned}
    \dd_{x_i} &= \dd_{x_i}(r_x) \dd_{r_x} + \dd_{x_i}(\theta_x) \dd_{\theta_x}= \frac{x_i}{r_x} \dd_{r_x} + \frac{\dd_{x_i}(x_2/x_1)}{1+(x_2/x_1)^2} \dd_{\theta_x}=\frac{x_i}{r_x} \dd_{r_x} \mp \frac{x_{i+1}}{x_1^2+x_2^2} \dd_{\theta_x}
    \end{aligned}
\end{equation}
Hence plugging in for $\dd_{x_i}$ and $\ol{\dd_{x_i}}$:
\begin{equation}
\begin{aligned}
\frac{\dd}{\dd x} & = \frac{1}{2}\left( \frac{\dd}{\dd x_1} - i \frac{\dd}{\dd x_2} \right)=\frac{1}{2}\left( \frac{x_1-ix_2}{r_x} \frac{\dd}{\dd r_x} + \frac{-x_2-ix_1}{r_x^2}\frac{\dd}{\dd \theta_x}\right)\\
& = \frac{1}{2}\left(e^{-i\theta_x}\dd_{r_x} - \frac{ie^{-i \theta_x}}{r_x} \right)\\
\frac{\dd}{\dd \ol{x}}  & = \frac{1}{2}\left(e^{i\theta_x}\dd_{r_x} + \frac{ie^{i \theta_x}}{r_x} \right)
\end{aligned}
\end{equation}

We also need to rewrite the differentials $dx=dx_1 + i dx_2$ and $d\ol{x}$ in terms of polar coordinates. Since $d$ is complex linear, and using Euler's formula $e^{i\theta} = \cos \theta + i \sin \theta$:
\begin{equation}
\begin{aligned}
dx &= d(r_x \cos \theta_x) + i d(r_x \sin \theta_x)=e^{i\theta_x} dr_x + ir_xe^{i\theta_x}d\theta_x\\
d\ol{x} & = e^{-i\theta_x}dr_x - ir_xe^{-i\theta_x}d\theta_x
\end{aligned}
\end{equation}
Similarly for $y$. Now we can convert $\dd \ol{\dd} F$ into polar coordinates.
{\small \begin{allowdisplaybreaks}
\begin{align*}
i\dd \ol{\dd} F & = \sum_{i,j = 1}^3 \frac{\dd^2 F}{\dd z_i \dd \ol{z_j}} dz_i \wedge d\ol{z_j}\\
& =i \sum_{i,j} (\frac{1}{2}e^{-i\theta_{z_i}}\left[\dd_{r_{z_i}} - i/r_{z_i} \dd_{\theta_{z_i}} \right]) (\frac{1}{2}e^{i\theta_{z_j}}\left[\dd_{r_{z_j}} + i/r_{z_j} \dd_{\theta_{z_j}} \right])(F) (e^{i\theta_i}(dr_i + ir_id\theta_i))\\
& \wedge (e^{-i\theta_j}(dr_j - ir_jd\theta_j))\\
& = i\frac{1}{4} \sum_{i,j} e^{- i\theta_i} \left[e^{i\theta_j}\dd^2_{r_ir_j}F - \frac{i}{r_i} \dd_{r_j}F \delta_{ij} i e^{i \theta_j}\right](e^{i\theta_i}(dr_i + ir_id\theta_i))\wedge (e^{-i\theta_j}(dr_j - ir_jd\theta_j))\\
& =\frac{i}{4}\left[ \sum_{i} (\dd^2_{r_i}F + \frac{1}{r_i}\dd_{r_i}F)(dr_i + ir_id\theta_i)\wedge(dr_i - ir_id\theta_i)\right]\\
& +\frac{i}{4} \left[\sum_{i\neq j} (\dd^2_{r_ir_j}F)(dr_i + ir_id\theta_i)\wedge(dr_j - ir_jd\theta_j)\right]
\end{align*}
\end{allowdisplaybreaks}}

A K\"ahler form is compatible with its complex structure by definition (e.g.~\cite[Definition 1.2.13]{huy}), so $J:=i$ given by multiplication by $i$ in the toric coordinates, is $\omega$-compatible. That is, $\omega(\cdot, J \cdot)$ is a metric, which we want to express in polar coordinates to facilitate calculations below. This is a 6 by 6 block diagonal matrix with the $r$-derivatives block and the $\theta$-derivatives block. Recall the complex structure acts on real tangent vectors by $\dd_{x_1} \mapsto \dd_{x_2}$ and $\dd_{x_2} \mapsto -\dd_{x_1}$. Again by the Chain rule:
%
\begin{equation}
\begin{aligned}
\dd_{r_1} &= \frac{\dd x_1}{\dd r_1} \dd_{x_1} + \frac{\dd x_2}{\dd r_1} \dd_{x_2} = \cos \theta_1 \dd_{x_1} + \sin \theta_1 \dd_{x_2}\\
\dd_{\theta_1} & = -r_1 \sin \theta_1 \dd_{x_1} + r_1 \cos \theta_1 \dd_{x_2}\\
J(\dd_{x_1}) &= \dd_{x_2},  J(\dd_{x_2}) = - \dd_{x_1} \implies J(\dd_{r_1}) = \frac{1}{r_1} \dd_{\theta_1}, J(\dd_{\theta_1})  = -r_1 \dd_{r_1}
\end{aligned}
\end{equation}
Hence the entries along the diagonal in the $r$ block will be
\begin{align*}
g_{ii}& = \omega(\dd_{r_i},J\dd_{r_i}) = \omega(\dd_{r_i},\frac{1}{r_i} \dd_{\theta_i}) = \frac{1}{r_i} \omega(\dd_{r_i},\dd_{\theta_i}) = \frac{1}{2}(\dd^2_{r_i}F + \frac{1}{r_i}\dd_{r_i}F)
\end{align*}
because we pick up the $dr_i \wedge d\theta_i$ term, i.e.~$\frac{2}{4}r_i dr_i \wedge d\theta_i$ times the $F$ derivative term. Similarly
\begin{align*}
g_{ij} & = \omega(\dd_{r_i}, J\dd_{r_j}) =\frac{1}{r_j} \omega(\dd_{r_i}, \dd_{\theta_j}) = \frac{1}{2} (\dd^2_{r_i r_j} F)
\end{align*}

{\bf Dominant terms for metric Region I.} The metric in polar coordinates is:
\vspace{-25pt}
{\small  \begin{multicols}{2}
\mbox{}\vfill
\begin{align*}
\left(\begin{matrix} \dd^2_{r_x}F + \frac{1}{r_x} \dd_{r_x}F & \dd^2_{r_xr_y}F & \dd^2_{r_xr_z}F\\
\dd^2_{r_xr_y}F & \dd^2_{r_y}F + \frac{1}{r_y} \dd_{r_y}F & \dd^2_{r_yr_z}F \\
\dd^2_{r_xr_z} F& \dd^2_{r_yr_z}F & \dd^2_{r_z}F + \frac{1}{r_z} \dd_{r_z}F
\end{matrix}\right)
\end{align*}
\mbox{}\vfill
\columnbreak
\begin{align*}
F& =f+  \alpha_3 d+\alpha_4\cdot \alpha_5 \theta\\
d & \approx T^2[r_x^2   - \frac{1}{2} \left( r_y^2 + r_z^2  \right)] \\
\theta & \approx T^2[  r_y^2 -r_z^2 ]  \\
f &\approx T^2[ r_y^2+r_z^2 ]
\end{align*}
\end{multicols}
}
Note that because everywhere $r_x$ appears is as $r_x^2$, applying $\dd^2_{r_xr_x}$ is the same as applying $\dd_{r_x}/r_x$. Here are the terms that do not involve derivatives of the $\alpha_i$, where for ease of notation subscript $x$ means $\dd_{r_x}$:

\begin{align*}
& \left(\begin{matrix} 2(f_{xx} + \alpha_3d_{xx} +\alpha_4\alpha_5 \theta_{xx})  & f_{xy} + \alpha_3d_{xy} +\alpha_4\alpha_5 \theta_{xy}  &f_{xz} + \alpha_3d_{xz} +\alpha_4\alpha_5 \theta_{xz}\\
`` & 2(f_{yy} + \alpha_3d_{yy} +\alpha_4\alpha_5 \theta_{yy})& f_{yz} + \alpha_3d_{yz} +\alpha_4\alpha_5 \theta_{yz}\\
`` & `` & 2(f_{zz} + \alpha_3d_{zz} +\alpha_4\alpha_5 \theta_{zz})
\end{matrix}\right)  \\
\approx\; &
T^2 \left(\begin{matrix} 4\alpha_3 &0&0\\
0 & 4 - 2\alpha_3+4\alpha_4\alpha_5 &0\\
0 & 0 & 4-2 \alpha_3 -4\alpha_4\alpha_5 
\end{matrix}\right)
\end{align*}

The approximation follows from the estimates on bump function derivatives in Appendix B. One of the coordinates does go to zero as the bump functions reach their bounds. However, the one that goes to zero in the $xy$-plane is the $z$ term and similarly in the $xz$-plane it's the $y$ term. Since we add a term $|xyz|^2$ for the base, this will ensure positive definiteness away from the zero fiber.

{\bf Dominant terms in Regions III and V.} In region I, $r_x$ was the dominating coordinate. In region III, $r_y$ will dominate and in region V, $r_z$ will dominate. So we take the analogous data for half regions of III and V, by modeling I.  

{\bf Dominant terms in Region II.} The expressions for the K\"ahler potentials in regions III and V thus differ from that in region I by a permutation of the coordinates $x, y, z$ and all the estimates above carry through under this permutation. Thus we patch together the $d$ coordinate across regions II, IV and VI. This uses another bump function $\alpha_6$ going from 0 to 1 as we increase a suitable $\theta_{II}$-coordinate. In particular, functions of $d$ become functions of $d\circ \alpha_6$. We have
\begin{allowdisplaybreaks}
\begin{align*}
F & \approx T^2[(r_y^2 + r_z^2 - \alpha_6(\theta_{II})\cdot r_y^2) +\alpha_3(d_{IIA})\cdot (r_x^2 - \frac{1}{2}(r_y^2+r_z^2) +\frac{3}{2} \alpha_6(\theta_{II}) \cdot r_y^2) \\
& + \frac{1}{2}\alpha_5(d_{IIA}) \cdot(r_y^2 - r_z^2 - \alpha_6(\theta_{II}) \cdot r_y^2)]
\end{align*}
\end{allowdisplaybreaks} 
since
\begin{allowdisplaybreaks}
\begin{align*}
g_{yz} &\approx T^2[r_y^2 + r_z^2]\\
d_{IIA} & \approx T^2[r_x^2 - \frac{1}{2}(r_y^2 + r_z^2)] +\frac{3}{2} \alpha_6(\log r_y - \log r_x) \cdot (Tr_y)^2\\
\theta_{II} & =\log(r_y/r_x)
\end{align*}
\end{allowdisplaybreaks}
The terms not involving derivatives of the bump functions will form the nondegenerate part of the metric on region II. {Off-diagonal terms $\dd^2_{r_\bullet r_\star}$ for $\star \neq \bullet$} are zero because derivatives of non-bump functions means differentiating $r_*^2$ for some $*$. On the diagonal terms we get $\frac{1}{r_*}\dd_{r_*} + \dd^2_{r_* r_*}=2\dd^2_{r_* r_*}$ which, applied to $(Tr_*)^2$ is $4T^2$. So in the $(*,*)$ entry of the matrix, the leading terms are $4T^2$ times the coefficients on $r_*^2$.
\begin{align*}
x:\;  & T^2(4\alpha_3 )  \geq T^2 8/3\\
y: \; & T^2(4-4\alpha_6 - 2\alpha_3 + 6\alpha_3\alpha_6 + 2\alpha_5 - 2\alpha_5\alpha_6)= T^2[4+(2-2\alpha_6)(\alpha_5-\alpha_3) + 4 \alpha_6(\alpha_3-1) ]\\
&\geq T^2 [4 + 2(0-2/3) + 4 (-1/3)] = T^2\cdot \frac{4}{3}
\\
z: \; & T^2(4-2\alpha_3-2\alpha_5)
\end{align*}


Note that when $\alpha_3 = \alpha_5 = 1$ the $z$ term goes to zero. However, it is bounded below in a region where $\alpha_3, \alpha_5$ are bounded away from 1. In the region where it goes to 1, we add a term to $F$ from the base, i.e.~$|xyz|^2$, to maintain nondegeneracy. Again because $xyz$ is bounded below in the region where we add it, we can take its partial derivatives and the result will be positive definite. Region IIC, IVA, IVC and VIA, VIC are the same after permuting $x \leftrightarrow y \leftrightarrow z$. For example, swap $r_x \leftrightarrow r_y$ to get to IIC, and the subscripts I are replaced with subscripts III.

The characteristics in region IIB which we did not have in regions IIA and C are 1) $r_x$ and $r_y$ go from $r_x>>r_y$ to $r_y>>r_x$, passing through $r_x = r_y$ and 2) $\alpha_6 \equiv 1$. All of $r_x,r_y,r_z$ are still small so we still have an approximation for the K\"ahler potential. The calculation for the negligible terms is given in Appendix B.

{\bf Dominant terms in the remainder of $\mb{C}^3$ patch.} See Figure \ref{fig:remainder} for the regions left to consider.
\begin{figure}
 \includegraphics[scale=0.45]{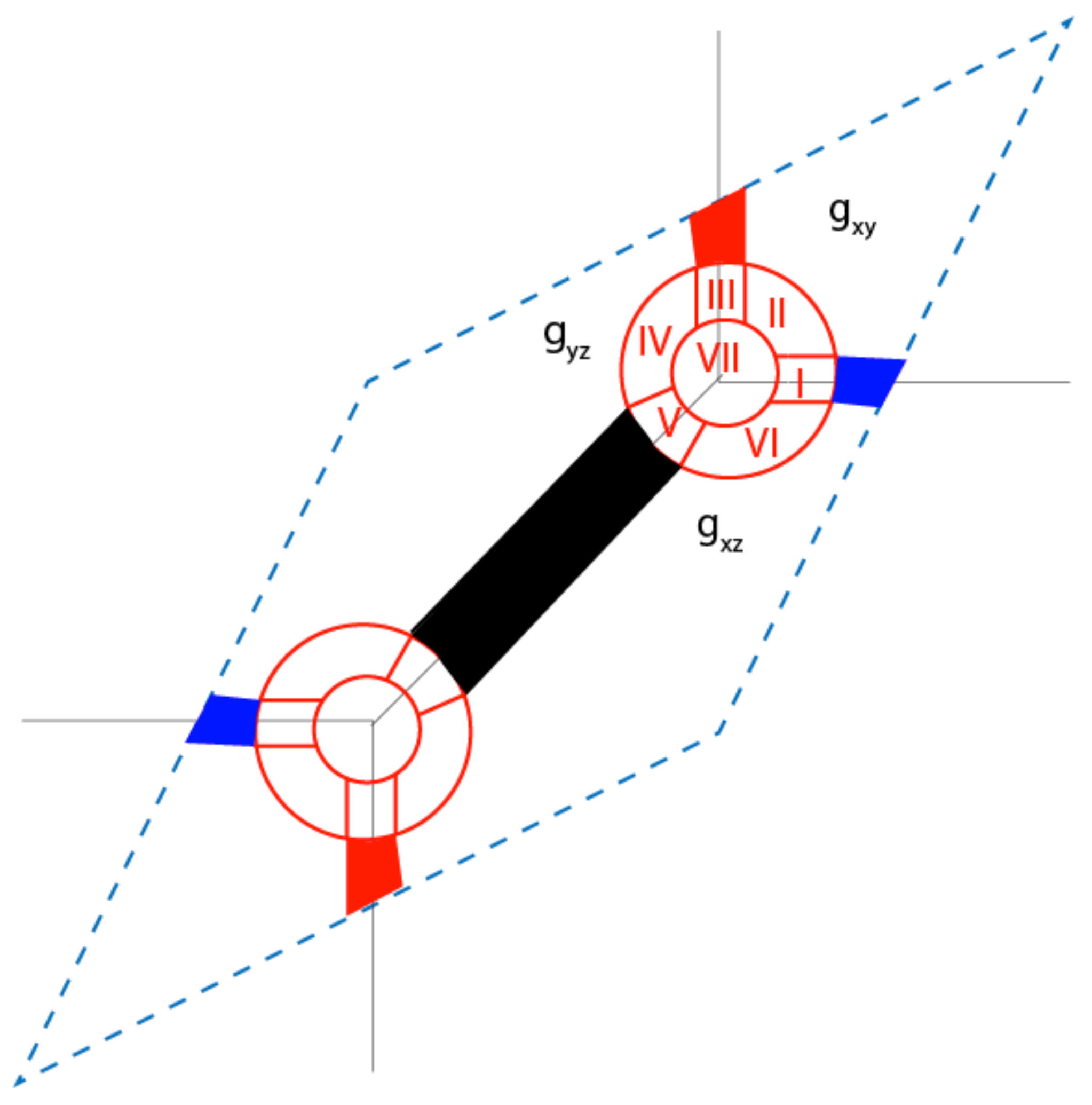}
 \caption{Delineated fundamental domain}
 \label{fig:remainder}
 \end{figure}
In the region between region I and region IIA the only bump functions at play are $\alpha_3$ and $\alpha_5$ and they are allowed to vary in the same amount of space described in Appendix A, where nondegeneracy was checked in Appendix B. Here still $r_x>>r_y, r_z$ so those estimates still apply.

What remains to be checked is that the symplectic form glues positive-definitely along the $z$ axis, then the other axes will follow by symmetry, and in the remaining regions $\omega$ agrees with the standard K\"ahler form of the blowup $\mb{CP}^2(3)$. In the black region along the $z$-axis we no longer have dependence on the $d_V$ coordinate because $\alpha_3, \alpha_5 \equiv 1$. We still have an angular coordinate that allows us to interpolate between $g_{xz}$ and $g_{yz}$ in the unprimed coordinates, or $g_{x'''z'''}$ and $g_{y'''z'''}$ in the tripled primed coordinates at the lower left vertex. So we need to check we have positive definiteness when only $\alpha_4$ is at play and $r_z$ is large. The proof that the bump function derivative terms can be made arbitrarily small is again in Appendix B. 

With the Appendices, this proves the following:

\begin{lemma} The $\omega$ defined in Definition \ref{defn of w} on $Y$ is non-degenerate and puts the structure of a symplectic fibration on $v_0:Y \to \mb{C}$.
\end{lemma}

\section{Donaldson-Fukaya-Seidel type category of linear Lagrangians in $(Y,v_0)$}\label{section: FS}

\subsection{Context and definition}
Donaldson introduced the pair of pants product and in this paper we only define and compute the differential and ring structure for the morphism groups. \cite{fooo1} explain obstructions to defining an $A_\infty$-category on a symplectic manifold which record how Lagrangians intersect upon perturbations, and is the reason we exclude the infinite-sloped linear Lagrangian parallel transported around a circle in the base in our subcategory. (Such a Lagrangian bounds $i$-holomorphic discs.) Seidel adapted the definition of a Fukaya category to the case of symplectic Lefschetz fibrations, which we adapt to the symplectic fibration $(Y,v_0)$. Note we've constructed $(Y,\omega)$ so that $v_0: Y \to \mb{C}$ is a symplectic fibration because we started by constructing $\omega$ in a fiber in Chapter \ref{section:the_symp_form}.

The notation ``Donaldson-Fukaya-Seidel (DFS)- type category of linear Lagrangians in $(Y,v_0)$" means that we only consider a subset of possible objects, and their collection forms a category that will be a subcategory of the full Fukaya-Seidel category of $(Y,v_0)$ once such a category is defined. In particular, it is expected this subcategory would split-generate the full category. The definition of linear Lagrangians in a fiber $V^\vee$ of $Y$ was inspired by the rational slope Lagrangians considered in \cite{zp} on categorical mirror symmetry for the elliptic curve. Because of the linearity of Lagrangians considered, they have a lift that allows for Maslov grading given by their slope, and they have a Spin structure as well. We obtain Lagrangians in the total space by parallel transporting these Lagrangians in the fiber over an arc in the base with respect to the horizontal distribution induced from the symplectic form. 

\begin{definition}\label{defn:par_transp} The \emph{symplectic horizontal distribution} of a symplectic fibration $\pi: Y \to C$ to a base manifold $C$ is $H \subset TY$ such that if $F$ is a generic fiber of $\pi$ then $H = TF^\omega$ is the $\omega$-complement, i.e.
$$\omega(H, TF) = 0$$

Given two points $p_0, p_1 \in C$ and a path $\gamma: I \to C$ between them (i.e.~$\gamma(0) = p_0$ and $\gamma(1) = p_1$), the \emph{parallel transport map} is a symplectomorphism $$\Phi: F_{p_0} \to F_{p_1}$$ defined as follows: given $x \in F_{p_0}$, we set $\Phi(x)$ to be $\tilde{\gamma}(1)$ where $\tilde{\gamma}:I \to Y$ such that $\pi \circ \tilde{\gamma} = \gamma$, $d \pi(\tilde{\gamma}' ) = \gamma'$ and $\tilde{\gamma}' $ is in the horizontal distribution.\end{definition}

\begin{claim}\label{claim:par_trans_symplecto} By standard theory, this last condition implies $\Phi$ is a symplectomorphism.\end{claim}

\begin{proof} There is a unique horizontal vector field $X_H$ on $\pi^{-1}(\gamma(I))$ with flow $\phi_H$ by existence and uniqueness of differential equation solutions and that horizontal implies there is no component of the vector field in the fiber direction. Then since $d\Phi$ is the identity on vectors in $H$ we have $\omega(d\Phi(X_H), d\Phi(v)) = \omega(X_H, v)$ for $v \in H$. Also, $H$ is $\omega$-perpendicular to $TF$, which is a condition also preserved under parallel transport: when $v \in TF$ is transported infinitesimally in the parallel direction, it must still be in $TF$, otherwise any component in $H$ could be reverse parallel-transported to a vector component in $H$ at the original fiber, contradiction. So $\omega(X_H, v) = 0 = \omega(d\Phi(X_h), d\Phi(v))$ and $\Phi^*\omega = \omega$ since $TY = TF \oplus H$ in regions where we parallel transport. 
\end{proof}

\begin{cor} $\Phi(\ell_i)$ is Lagrangian in $V^\vee$ with respect to $\omega|_{V^\vee}$.\end{cor}

\begin{claim}
$\Phi$ fixes $\xi_1,\xi_2$.
\end{claim}

\begin{proof}Let $\rho$ be the quasi-Hamiltonian $T^2$-action rotating coordinates $(x,y)\in V^\vee$ by angles $(\alpha_1,\alpha_2)$. Let $X_H$ be the horizontal vector field with flow $\phi_H^t$. Then 
\begin{allowdisplaybreaks}
\begin{align*}
    \frac{d}{dt}(\xi_i \circ \phi_H^t)|_{t=0}&=d\xi_i(X_H)\\ &=\iota_{d\rho(\dd_{\alpha_i})}\omega(X_H)\\
    &=\omega(\dd_{\theta_i},X_H)=0 \because X_H \perp^\omega \dd_{\theta_i} \in TV^\vee\\
    \therefore \xi\circ \phi^t_H &= \xi
\end{align*}
\end{allowdisplaybreaks}
\end{proof}

\begin{cor}\label{claim:par_fix_xi}
Parallel transport $\Phi$ is of the form
$$(\xi_1,\xi_2,\theta_1,\theta_2) \mapsto (\xi_1,\xi_2,\theta_1 + f_1(\xi), \theta_2 + f_2(\xi))$$
for some functions $f,g$ depending on $\xi$ but not the angular coordinates.
\end{cor}

\begin{proof} Let $\pi=v_0$. We know by the previous result that parallel transport does not affect $\xi_1, \xi_2$. We are stating further that $f$ and $g$ are independent of the angles, so we can express monodromy as a graph of a function $T_B \to T_F$. So again let $\rho$ be the $T^2$-action on a fiber, namely addition on the $\theta$ coordinates. Since $\omega$ is a function of the norms only, $\rho$ is a symplectomorphism and preserves that $TV^\vee \perp^\omega H$ and acts fiber-wise. Hence
$$ 0 = {\rho}^*\omega(TV^\vee,H)  = \omega(\rho_*TV^\vee, \rho_*H) = \omega(TV^\vee,\rho_*H) \therefore \rho_*H = H$$
Furthermore since $\pi=v_0$ and $\rho$ preserves fibers, we have $\pi \circ \rho = \pi$ and
$$ d\pi(\rho_* X_{hor})  = d(\pi \circ \rho)(X_{hor}) = d\pi(X_{hor}) \therefore \rho_*(X_{hor}) \circ \rho^{-1} = X_{hor}$$
because $\rho_* X_{hor} \in H$ and has the same horizontal component as $X_{hor}\in H$. Integrating both sides of the last equality we obtain
$$d\rho\left(\frac{d}{dt}\phi_H^t\right)\circ \rho^{-1}  = \frac{d}{dt}\left(\rho \circ \phi_H^t\right)\circ \rho^{-1} = \frac{d}{dt}(\phi_H^t)
\therefore \rho \circ \phi_H^t  = \phi_H^t \circ \rho $$

In other words, we get the same answer whether we rotate the coordinates $\theta \mapsto \theta+\alpha$ and then transport by adding $f(\theta+\alpha)$, or transport and then rotate. So parallel transport doesn't depend on the angular coordinates. Namely
$$\theta_i+\alpha_i+tf_i(\xi,\theta+\alpha)= \theta_i+\alpha_i+tf_i(\xi,\theta)\therefore f_i(\xi,\theta+\alpha) = f_i(\xi,\theta) \forall \alpha
$$
Thus $f_1,f_2$ are independent of $\theta_1,\theta_2$. 
\end{proof}


Now we can define the DFS-type category we consider in this paper. Objects are U-shaped Lagrangians as in the appendix of \cite{ab_sm} and morphisms are defined via categorical localization as in \cite{ab_seid}, \cite[Chapter 4]{math257b}.

\begin{definition}[Definition of category on A-side]\mbox{}\label{defn:fuk_cat}  

\begin{center}\fbox{\bf Objects} 
\end{center}

Recall the definition of $\ell_k=\{(\xi_1, \xi_2, -k \left( \begin{matrix} 2 & 1 \\ 1 & 2 \end{matrix}\right) ^{-1}\left(\begin{matrix} \xi_1 \\ \xi_2 \end{matrix} \right) \}_{(\xi_1, \xi_2) \in T_B}$ in Lemma \ref{lem:fuk_subcat} for Lagrangians in $V^\vee$. Also define 
$$t_x:=\{(\log_\tau |x|, \theta)\}_{\theta \in [0,1)^2} \subseteq V^\vee.$$
In $Y$, define these Lagrangians to exist in the $v_0$-fiber over $-1$. Let $\bigcup_{\gamma}$ denote parallel transport with respect to $(TV^\vee)^\omega$ over a curve $\gamma \subset \mb{C}$ in the base of $v_0$. Let $\gamma: \mb{R} \to \mb{C}$ be a smooth curve so that $\gamma(0) = -1$ and $\gamma(\mb{R})$ traces out a U-shape in the base given outside a compact set by rays in the right half plane, i.e.~lines $re^{i \theta_0}$ for a fixed $\theta_0 \in [-\pi, \pi)$, as in Figure \ref{u_shape}. Then define 
\begin{equation}
\begin{aligned}
    L_k&:=\cup_\gamma \ell_k\\
    T_x&:= \cup_\gamma t_x
    \end{aligned}
\end{equation}
to be fibered Lagrangians which generate all objects in the DFS-type category. In particular, we don't allow the rays to wind around the origin in the compact set. I.e.~in the language of stops we have a stop at $-\infty$ along the real line in $\mb{C}$.


\begin{figure}
\begin{center}
\includegraphics[scale=0.5]{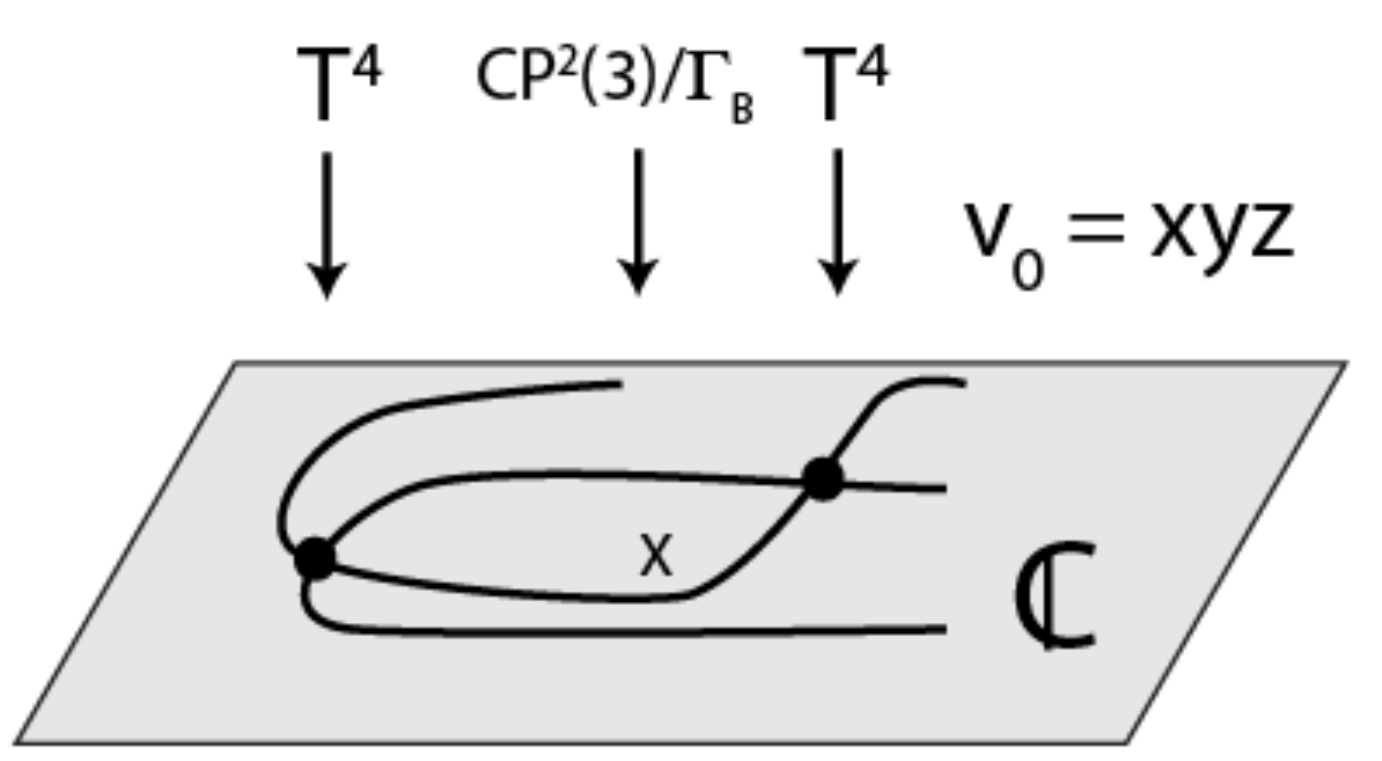}
\end{center}
\caption{$L_i$ = parallel transport $\ell_i$ around U-shaped curve in base of $v_0$}
\label{u_shape}
\end{figure}

\begin{center}\fbox{\bf Floer complexes}
\end{center}

This is based on \cite[\textsection 4.2.2]{math257b}. The condition for rays to lie in the right half-plane ensure morphisms will be well-defined. We first define a directed category. Let $K$ and $L$ be objects. Then first set
  \begingroup \allowdisplaybreaks \begin{equation}
    CF_{directed}(K,L)=
    \begin{cases}
      CF(K,L), & \text{if}\ K>L \\
      \mb{C}\cdot e_L, & \text{if}\ K=L\\
      0, & \text{else}
    \end{cases}
  \end{equation} \endgroup
where the ordering $K > L$ denotes that, outside a compact set, all the rays of $K$ lie above all the rays of $L$ and $e_L$ is defined as follows. If $K=L$, then push the rays of $K$ to lie above the rays of $L$ as in Figure \ref{u_shape}. Denoting this deformation of $K$ as $L^\epsilon$, this defines the \emph{quasi-unit} $e_L \in CF(L^\epsilon, L)$ as the count of discs with those Lagrangians as boundary conditions. We then localize at these $e_L$, in other words we set them to be isomorphisms. Recall that formally inverting morphisms involves taking equivalence classes of \emph{roofs}
$$K \leftarrow K^\eps \to L$$
This is the definition of an arrow $K \to L$ when $K<L$, namely push up to $K^\eps>L$ so that the first arrow of the roof is the quasi-unit $e_K$, and then $K^\eps \to L$ is defined as above. 

\begin{center}\fbox{\bf Morphisms}
\end{center}

The morphisms in this category are obtained by quotienting by the differential on the hom spaces just defined. The differential ($\mu^1$), composition ($\mu^2$), and higher order $\mu^k$ count bigons, triangles, and $(k+1)$-gons respectively, with the usual Lagrangian boundary conditions from these $L_i$, and are $J$-holomorphic for some regular $J$. This is theory from Chapter 12 of \cite{seidel}, contents of which are listed \href{https://drive.google.com/open?id=10D7IKn9mZ6q8ApeASBcodrVQUb-cWris}{here}. 

However, for self-Floer cohomology groups, Hamiltonian perturbation is only used in the fiber, as in the {\bf Non-transversely intersecting Lagrangians} section of the proof of Lemma \ref{lem:fuk_subcat}. Not in the total space, because we would need to find a Hamiltonian vector field tangent to the fibers of $v_0$, so the projection of a Hamiltonian-perturbed holomorphic curve (i.e.~one that satisfies the perturbed Cauchy-Riemann equation (\ref{eq:perturb_CR})) under $v_0$ is still a holomorphic polygon. Since the symplectic form is not a product of one on the base and fiber, we instead use the categorical localization method of \cite{ab_seid}, see also \cite{math257b}. For each Lagrangian in the fiber, there are many different objects of $FS$ obtained by parallel transport along U-shapes that go to infinity in slightly different directions, but these will be isomorphic to each other after localizing.

\begin{center}
\fbox{\bf Moduli spaces}    
\end{center}

This concludes the definition of the DFS-type category used on the A-side $(Y,v_0)$, modulo the definition of the moduli spaces, which is the remainder of this chapter after we compute monodromy.



%

\end{definition}

\begin{remark} If we include $t_x$ parallel transported in a circle around the base of $v_0:Y \to \mb{C}$ as a Lagrangian in the subcategory we are considering, then this Lagrangian bounds nonconstant holomorphic discs. A future direction is to incorporate and define $M^0$ for the category containing this Lagrangian. Note that $M^0$ is a degree 2 operation, and the Lagrangian is only $\mb{Z}/2$-graded. In this paper, we do not include it in the subcategory being considered. We do still want to count the discs and $J_0$-spheres, but they will show up only in the $c$ in $M^2(c,-)$ considered as a map on Floer groups in Section \ref{section:cobord}. 
\end{remark}



\begin{lemma} For Lagrangians in correct position, morphisms and compositions in the localized Fukaya-Seidel category coincide with those in the directed category. Namely, if $K>L$, then $\hom_{FS}(K,L) \simeq CF^*(K,L)$, and if $L_0 > ... > L_k$, then $\mu^k$ in $FS(Y,v_0)$ is the same as in the directed category. Namely it is given by counting $J$-holomorphic discs. \end{lemma}

\begin{proof}[Proof references] This is unpublished work of Abouzaid-Seidel, e.g.~see \cite[Proposition 119, 120]{math257b}, and Abouzaid-Auroux. \end{proof}

\begin{lemma} The composition of roofs is a roof.\end{lemma}

\begin{proof}[Proof references] This is a consequence of the naturality of quasi-units with respect to continuation maps, see Abouzaid-Seidel and Abouzaid-Auroux. \end{proof}




\begin{definition} The \emph{Donaldson-Fukaya category} is obtained by passing to cohomology on the morphism chain complexes. \end{definition}



\begin{remark}[Not exact] One difference of our fibration $v_0: Y \to \mb{C}$ to those in \cite{seidel} is that we are in the non-exact case, because fibers are compact tori. In particular, we will have sphere bubbles which would have been excluded. Furthermore, we do not have a \emph{horizontal boundary} since fibers are tori, but we do have a \emph{vertical boundary} by taking the preimage of the boundary of the base. (Recall that the polytope locally describing $Y$ has the restriction $\eta \leq T^\ell$ for $T \ll 1$, so equivalently $|v_0|\ll 1$ in the base.) 
\end{remark}

\begin{remark}[Not Lefschetz]
In \cite{seidel} the Fukaya category of a Lefschetz fibration has Lagrangians given by thimbles obtained by parallel transporting a sphere to the singular point in the Morse singular fiber where it gets pinched to a point. In our situation, the degenerate fiber over 0 is not a Lefschetz singularity (i.e.~modeled on $\sum_i z_i^2$) but instead the fiber $T^4$ degenerates by collapsing a family of $S^1$'s, and also collapsing a $T^2$ in two points. This produces singular Lagrangians. So instead, we go around this singular fiber in a U-shape as in \cite{ushape}.
\end{remark}

\begin{remark}[Not monotone]
Lastly, a compact symplectic manifold $(P,\omega)$ is \emph{monotone} if $\int c_1(P) = \alpha \int \omega: \pi_2(P) \to \mb{R}$ for some $\alpha>0$. Our setting is not monotone: taking $P:=Y$, which is Calabi-Yau, we have $c_1(Y) = 0$ however $[\omega] \neq 0$ so $\alpha$ would have to be zero, contradiction. Furthermore, the Lagrangians are not monotone, which we now define.
\end{remark}

\begin{definition}[Maslov index of a map] Given $u: (\mb{D},\dd \mb{D}) \to (M,L)$ we first trivialize $u^*TM$ over the contractible disc. Then the \emph{Maslov index} counts the rotation of $u^*TL$ around the boundary in this trivialized pullback.\end{definition}

\begin{definition} A \emph{monotone} Lagrangian $L$ in symplectic manifold $(M,\omega)$ is such that $[\omega] = k \cdot [\mu(u)]$ for all $u \in \pi_2(M,L)$, for some constant $k$. Namely, the areas of discs are proportional to the Maslov indices of those discs.\end{definition}

\begin{example} An example can be found in Oh \cite{oh}, who shows the Clifford torus in $(\mb{CP}^n, \omega_{FS})$ is a monotone Lagrangian submanifold. In this thesis, all discs considered have Maslov index 2, but the areas vary as prescribed by a formula of \cite{cho_oh} which we will elaborate on later.
\end{example}

%
%

Now that we've indicated how this set-up differs from those currently in the literature, we give an outline of the subsections of this chapter. In Section \ref{mon} we discuss monodromy around the central fiber in $v_0:Y \to \mb{C}$ because it is used to find intersections of the Lagrangians $L_k$ we defined in the total space. Then Section \ref{moduli} sets up the background needed to define the structure maps and proves the moduli spaces involved have the required conditions to be put into the definition. Finally we show the definition of the DFS-type category is independent of choices, in Section \ref{defn}.

\subsection{Monodromy}\label{mon}

\begin{definition} \emph{Monodromy} is the parallel transport map $\Phi$ of Definition \ref{defn:par_transp} around a loop which goes once around a singularity in the base.\end{definition}

\begin{claim}\label{claim:mon_ind_angle} If $\psi(t)(v_0) = e^{2\pi i t} v_0$ then the symplectic horizontal lift of $d\psi(d/dt)$ is of the form 
$$X_{hor} = \dd/\dd \theta_\eta + f_1(\xi, \eta)\dd/\dd \theta_1 + f_2(\xi, \eta) \dd/\dd \theta_2$$
where $f_i$ are now functions of $\xi_1,\xi_2$, and $\eta$.

\end{claim}

\begin{proof}[Proof of Claim \ref{claim:mon_ind_angle}] We saw above in Corollary \ref{claim:par_fix_xi} that $X_{hor}$ does not involve $\dd/\dd \xi_i$. We now show it also preserves $\eta$ because we are considering parallel transport around a circle with fixed $|v_0|$. Let $w$ be the coordinate on the base $\mb{C}$ of $\pi=v_0$. In particular, $X_{hor}$ is defined by the property 
\begin{equation}\label{eqn:horiz_vec} d\pi (X_{hor}) = \dd/\dd \theta_w  \end{equation}

With respect to the action-angle coordinates  $\pi:(\xi_1,\xi_2,\eta,\theta_1,\theta_2,\theta_\eta) \to (|w|,\theta_w):=(|v_0|,\theta_\eta)$ and noting that $|v_0|$ is a function of $(\xi_1,\xi_2,\eta)$ only, Equation (\ref{eqn:horiz_vec}) can be expressed as:
\begin{align*}
& \left( \begin{matrix} \dd |v_0|/\dd \xi_1 & \dd |v_0|/\dd \xi_2 & \dd |v_0|/\dd \eta & 0 & 0 & 0\\ 0 & 0 & 0 & 0 & 0 & 1\end{matrix}\right) \left ( \begin{matrix} 0\\0\\a \\ f_1 \\ f_2 \\b \end{matrix}\right)=\left(\begin{matrix} a \dd |v_0|/\dd \eta\\   b   \end{matrix}  \right)=\left(\begin{matrix} 0 \\ 1 \end{matrix} \right)
\end{align*}
thus $b = 1$ and $a$ must be zero as $\eta$ depends on $|v_0|$. So the horizontal lift is of the form:
$$X_{hor} = \dd/\dd \theta_\eta + f_1 \dd/\dd \theta_1 + f_2 \dd /\dd \theta_2$$

Also we saw above in Claim \ref{claim:par_fix_xi} that for a fixed fiber, the $f_i$ do not depend on $\theta_1$ and $\theta_2$. A similar argument shows they are independent of $\theta_\eta$, i.e.~$\rho_\la \circ  \phi_H^t= \phi_H^t \circ \rho_\la$ for $\la \in S^1$ defining the rotation action on the $v_0$ coordinate. This can be seen by replacing $\rho$ with $\rho_\la$ in the proof of Claim \ref{claim:par_fix_xi}. 
\end{proof}

\begin{example}[One dimension down, 2D local case] 

The case of $\mb{C}^2$ with symplectic fibration $(x,y) \mapsto xy$ is the setting of a Lefschetz fibration, with singular fiber given by two copies of $\mb{C}$ from $x=0$ or $y=0$. The monodromy is a Dehn twist about the $S^1$ given by the belt of the cylindrical fibers. In that case $f(x,y) = xy$ and $S^1$-action is $(e^{i\theta}x, e^{-i\theta}y)$. The holomorphic vector field corresponding to this is $iz_1 \dd_{z_1} - i z_2 \dd_{z_2}$, whose contraction with $\omega = \frac{i}{2}( dz_1 \wedge d\ol{z_1} + dz_2 \wedge d{\ol z_2})$ gives Hamiltonian vector field $-\frac{1}{2}(|z_1|^2 - |z_2|^2)$. We have a new set of coordinates on the two dimensional fiber: the moment map coordinate $\mu$ and the angle coordinate of the action $\theta$. As we approach $xy=0$, the orbit at $\mu$-height 0, namely $|x|=|y|$, goes to zero. That's one way to see how we get the picture of a cylinder with the belt pinching to zero.  
\end{example}

Recall from Claim \ref{claim:mon_ind_angle} that we can describe parallel transport by the graph of a function $T_B \to T_F$ by adding $(f_1(\xi,\eta), f_2(\xi,\eta),1)$ to the angular coordinates $(\theta_1,\theta_2,\theta_\eta)$. The main result of this section is the following Lemma.

\begin{lemma}[Monodromy, see Figure \ref{monodr_pic}]\label{monodr} Over the parallelogram-shaped fundamental domain of the torus $T_B$, the monodromy for a loop around the base is given by $(f_1(\xi),f_2(\xi))$ as follows. They equal $(0,0)$ for $\xi$ in the upper right corner where $r_z^{-1} >> r_x^{-1},r_y^{-1}$, then $(0,1)$ on the right where $r_y^{-1}$ is the largest, $(1,0)$ on the left where $r_x^{-1}$ largest and thus $(1,1)$ in the bottom left corner. 
\end{lemma}

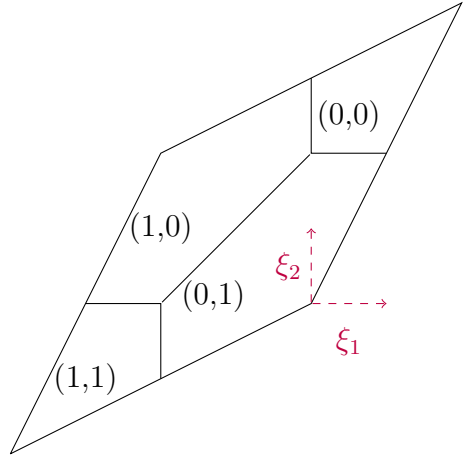
\begin{figure}
\begin{center}
\begin{tikzpicture}[every node/.style={inner sep=0,outer sep=0}]

\draw (-1,-0.5)  -- (1,3.5) -- (5,5.5) -- (3,1.5) node (v3) {} -- (-1,-0.5);

\draw (3,4.5) -- (3,3.5) node (v2) {} -- (4,3.5);
\draw (0,1.5) -- (1,1.5) node (v1) {} -- (1,0.5);
\draw  (v1) edge (v2);

\node at (1.7,1.6) {(0,1)};
\node at (1,2.5) {(1,0)};
\node at (3.5,4) {(0,0)};
\node at (0,0.5) {(1,1)};
\draw[->, purple, dashed] (v3) -- (3,2.5);
\draw[->, purple, dashed] (v3) -- (4,1.5);
\node[purple] at (3.5,1) {$\xi_1$};
\node[purple] at (2.7,2) {$\xi_2$};
\end{tikzpicture}
\caption{Monodromy in fiber, thought of as a section over the parallelogram $(\xi_1,\xi_2) \mapsto (f_1(\xi_1,\xi_2),f_2(\xi_1,\xi_2))$.}\label{monodr_pic}
\end{center}
\end{figure}

\begin{proof}[Proof of Lemma \ref{monodr}]\label{periodicity} First we show the result holds for $\xi$ in the $\mb{C}^3$ patch. Then we show it holds in the middle region of the hexagon where the K\"ahler potential is that of $\mb{CP}^2(3)$. Then we conclude the result by showing the contributions in between are negligible.  

\begin{center}
\fbox{\bf Parallel transport in $\mb{C}^3$-patch with standard metric}  
\end{center}

We want to find a horizontal lift of the angular vector field $\dd/\dd \theta_w$, or with respect to complex coordinates $\dd/\dd w$ this would be the vector field on the complex plane whose value at $w$ is $iw$. Namely $d\pi (X_{hor})= iw$. We can explicitly compute real vectors in the horizontal distribution $H$ in the $\mb{C}^3$-patch. Then we can find a vector parallel to $iw\dd/\dd w$ and scale it suitably so it projects to $iw \dd/\dd w$ and is not just parallel to it.

\begin{claim}[Finding the horizontal subspace]\label{claim:horiz}  $X_{\Re(v_0)}, X_{\Im(v_0)}$ generate the horizontal distribution $H$.\end{claim}

\begin{proof} The kernel of $d_p v_0$ when $v_0=v_0$ is also the tangent space of a $v_0$-fiber. This allows us to conclude the following. Let $p \in v_0^{-1}(c)$ for some $c \in \mb{C}$ and $|c| \ll T^l$ by Definition \ref{eq:polytope_def_Y}. Then $\ker(d_p v_0) = T_p(v_0^{-1}(c))$. On the other hand $\ker(dv_0) = \ker(d\Re(v_0)) \cap \ker(d\Im(v_0))$. This now allows us to find generators for $H$.
\begin{align*}
 H=\ker(dv_0)^\omega & = \left[ \ker(d \Re(v_0)) \cap  \ker(d\Im(v_0))\right]^\omega\\
& = \ker(d \Re(v_0))^\omega + \ker(d\Im(v_0))^\omega  \\
& \supseteq \mb{R}\cdot X_{\Re(v_0)} + \mb{R}\cdot X_{\Im(v_0)} 
\end{align*}
where the last line follows from the general notation of $X_f$ as $\omega$-dual to $df$ and that $\omega(X_f,X_f)=0$ for alternating form $\omega$. Thus since $H$ is rank two, we have equality in the last line.
\end{proof}

\begin{claim}\label{claim:H_cx} The horizontal distribution $H$ is a complex subspace, i.e.~invariant under multiplication by $i$.
\end{claim}

\begin{proof} Note $v_0=\Re(v_0) + i\Im(v_0)$. If we consider $v_0$ as a coordinate (rather than a function) then since $J=i$ on vector fields sends $\dd_{\Re(v_0)} \mapsto \dd_{\Im(v_0)}$ and $\dd_{\Im(v_0)} \mapsto -\dd_{\Re(v_0)}$, it does the transpose on the dual differential forms: $J \circ d\Re(v_0) = -d\Im(v_0)$. Thus using the duality from the metric and that the gradient is related to the Hamiltonian vector field (thinking of $v_0$ now as a function) by $-J$:
\begingroup \allowdisplaybreaks \begin{equation}
\begin{aligned}
J \grad(\Re(v_0)) &= -\grad(\Im(v_0))\\
\therefore  -JX_{\Re(v_0)} &=\grad (\Re(v_0))=J\grad(\Im(v_0))=- X_{\Im(v_0)}\\
\therefore JX_{\Re(v_0)}&= X_{\Im(v_0)}
\end{aligned}
\end{equation} \endgroup
This concludes the proof by Claim \ref{claim:horiz}.
\end{proof}

\begin{cor}
$X_{hor} = \frac{i}{g}\nabla_{g_0} |v_0|^2$ up to some scalar function $g$ to be determined. 
\end{cor}

\begin{proof}
First, $\nabla |v_0|^2 = \nabla(\Re(v_0)^2 + \Im(v_0)^2) \in H$ by Claim \ref{claim:H_cx}. In particular, from calculus we know that the gradient of a real function in two variables is perpendicular to its level sets in the plane. Here, that means that $d\pi(\nabla |v_0|^2) \perp \psi(t)(|v_0|)$ (in the notation of Claim \ref{claim:mon_ind_angle}) for $0\leq t < 2\pi$. Multiplication by $i$ turns a vector perpendicular to a circle to a vector tangent to the circle. Mathematically, since $\pi=v_0$ (now thought of as a function) is holomorphic, its derivative commutes with $J$ and we obtain the following statement of two vectors being parallel:
\begin{equation}
    Jd\pi(\nabla |v_0|^2) = d\pi (J\nabla |v_0|^2) \parallel d\psi(d/dt)
\end{equation}
while $X_{hor}$ has the property that $d\pi(X_{hor}) = d\psi(d/dt)$. Hence since both $X_{hor}$ and $J\nabla |v_0|^2$ are in $H$, they must be proportional to each other by some scalar function $g$.
\end{proof}

Now we can compute the monodromy in the $\mb{C}^3$-patch. 
\begin{allowdisplaybreaks}
\begin{align*}
dv_0 &= yz dx + xz dy + xy dz\\
\therefore \nabla_{g_0} |v_0|^2 &= 2\left<x|yz|^2, y|xz|^2, z|xy|^2\right>\\
\therefore X_H &= \frac{2i}{g}\left<x|yz|^2, y|xz|^2, z|xy|^2\right> \mbox{ s.t. } dv_0(X_H) = iv_0 \dd_{\eta} 
\end{align*}
\end{allowdisplaybreaks}
This last condition allows us to solve for the scalar function $g$.
 \begin{allowdisplaybreaks}
\begin{align*}
\therefore \frac{2i}{g}xyz(|yz|^2 + |xz|^2 + |xy|^2) &= i(xyz) \implies g ={2(|yz|^2 + |xz|^2 + |xy|^2) }\\
\therefore X_H &= \frac{i}{|yz|^2 + |xz|^2 + |xy|^2}\left<x|yz|^2, y|xz|^2, z|xy|^2\right>
\end{align*}
\end{allowdisplaybreaks}
So integrating, we obtain the following parallel transport map over the $|v_0|$ circle in the base.
 \begin{allowdisplaybreaks}
\begin{align*}
\therefore \phi_\theta^{H}(x,y,z,\theta) &=(e^{i\theta\frac{|y|^2|z|^2}{|y|^2|z|^2 + |x|^2|z|^2+|x|^2|y|^2}}x,e^{i\theta \frac{|x|^2|z|^2}{|y|^2|z|^2 + |x|^2|z|^2+|x|^2|y|^2}}y,e^{i \theta\frac{ |x|^2|y|^2}{|y|^2|z|^2 + |x|^2|z|^2+|x|^2|y|^2}}z)\\
 &=(e^{i\theta\frac{|x|^{-2}}{|x|^{-2} + |y|^{-2}+|z|^{-2}}}x,e^{i\theta \frac{|y|^{-2}}{|x|^{-2} + |y|^{-2}+|z|^{-2}}}y,e^{i \theta\frac{|z|^{-2}}{|x|^{-2} + |y|^{-2}+|z|^{-2}}}z)\stepcounter{equation}\tag{\theequation}\label{eqn:par_transport}
\end{align*}
\end{allowdisplaybreaks}
In the last step we divide by $|v_0|^2$ to more easily see the Dehn twisting behavior. This gives the monodromy result, Lemma \ref{monodr} at the start of this section, in the local model near the origin of $\mb{C}^3$. Namely, as we move counterclockwise across the $x$ axis of the hexagon then $|y|\ll |z|$ becomes $|z|\ll |y|$ and we see that $\theta_1$ stays fixed at a small angle but $\theta_2$ rotates from large to small, by 1. Hence $(\theta_1,\theta_2)$ changes from $(0,0)$ to $(0,1)$ as we move from the $(0,0)$ tile then down to the $(0,-1)$ tile. The other regions are similar.

Note that at first glance, there is more dependence on $x$ once we divide by $g$, which may make a harder ODE to solve for finding the flow. However, recall that symplectic parallel transport preserves the moment map coordinates $\xi_1,\xi_2$, and $\eta$. Thus, since these are monotonic increasing functions in $r_x,r_y, r_xr_yr_z$, we find that norms must also remain constant along the circle. So since we have that $g=|v_0|^2/|\nabla |v_0|^2|^2$ is a function of the norms, we find that this is constant and doesn't contribute to the flow when we integrate.

\begin{center}
\fbox{\bf Away from the $\mb{C}^3$- and $\mb{CP}^2(3)$-patches}    
\end{center}

Elsewhere in $Y$, we may not be able to directly compute $\nabla_{g} |v_0|^2$. All we know from Appendix B is that $g=g_1 + g_\eps$ where $g_1$ is diagonal in polar coordinates and $g_\eps$ can be made to have arbitrarily small entries. Let $V_1$ be $d|v_0|^2$ with respect to polar coordinates, so it will be zero in the last three angular coordinates, let $G_1+G_\eps$ denote the metric, and $V_2$ the vector $\nabla_g |v_0|^2$, all of them with respect to polar coordinates as in Section \ref{convert_polar}. Then if $\bm{O}(\cdot )$ denotes a matrix with entries on the order given in parentheses:
\begin{equation}
\begin{aligned}
    d|v_0|^2 &= (g_1+g_\eps)(\nabla_g |v_0|^2, \cdot)\\
    \implies V_1^t & = V_2^t(G_1 + G_\eps)\\
    \implies V_2 & = (G_1^t + G_\eps^t)^{-1}V_1\\
    &=(\bm{O}(1) + \bm{O}(1/l))^{-1}V_1\\
    \therefore V_2 & \approx (G_1^t)^{-1}V_1
\end{aligned}
\end{equation}
where $(G_1^t)^{-1}$ is the leading order terms of the metric defined in Definition \ref{defn of w}. It is diagonal by Section \ref{convert_polar} and consists of constants because all bump functions are functions of the norms $r_x,r_y,r_z$ which are fixed for $(\xi_1,\xi_2,\eta)$ fixed. Thus the parallel transport map will be similar to that in Equation \ref{eqn:par_transport}, but $|x|^{-2}, |y|^{-2}, |z|^{-2}$ will be scaled by certain constants. 

\begin{center}
\fbox{\bf In the $\mb{CP}^2(3)$ hexagons}    
\end{center}

Recall that in the upper right corner of the fundamental domain parallelogram of Figure \ref{coords in tiling} where $r_z \ll r_x, r_y$, the K\"ahler potential is a sum of the $\mb{CP}^2(3)$ toric potential $g_{xy}$, which only involves the complex coordinates $x$ and $y$ and does not involve $z$, and a term proportional to $|xyz|^2$. The latter is constant on the $v_0$-circle. Hence horizontal vectors are those whose $x$ and $y$ components vanish, namely
$$H=span\left(\frac{\dd}{\dd \Re z}, \frac{\dd}{\dd\Im z}\right)$$

Thus in this upper right corner region, parallel transport varies $z$ while fixing $x$ and $y$. It follows that the monodromy preserves $\theta_1=\arg(x)$ and $\theta_2=\arg(y)$, so $(f_1(\xi,\eta),f_2(\xi,\eta))=(0,0)$ for $(\xi,\eta)$ in the upper right corner, remain the same as near the $\mb{C}^3$-patch.

By the same argument, in the region where the fiberwise potential is $g_{yz}$, the horizontal distribution is parallel to the $x$ coordinate axis, so parallel transport varies $x$ while keeping $y$ and $z$ constant. In particular the angular parallel transport vector field is $(ix, 0, 0)$ so the monodromy increases $\theta_1$ at unit rate while keeping $\theta_2$ constant, and $(f_1,f_2)=(1,0)$. Similarly where we have $g_{xz}$, angular parallel transport is $(0,iy,0)$, monodromy increases $\theta_2$ while keeping $\theta_1$ constant, and $(f_1,f_2)=(0,1)$.

Finally, in the remaining fourth hexagonal tile of Figure \ref{monodr_pic}, the values of $(f_1,f_2)$ are again integer constants, determined by using the change of coordinate transformations described in Definition \ref{def:complex_coordinates}. In particular at the lower-left corner of the parallelogram, in terms of the coordinates $(x''',y''',z''')$ the parallel transport only varies $z'''$ while keeping $\arg(x''')$ and $\arg(y''')$ constant. However, because $x''' = T v_0 x^{-1}$, fixing $\arg(x''')$ implies $\arg(v_0) - \arg(x)$ remains constant. Thus varying $\arg(v_0)$  around the unit circle also varies $\arg(x)$ and so $\theta_1$ increases by 1. Similarly for $y$, and we find $(f_1,f_2)=(1,1)$. This concludes the calculation of monodromy, by $\Gamma_B$-invariance of $\omega$ as proven in Claim \ref{claim:omega_descends}.
\end{proof}

We need the following main result in order to compute the differential of the Fukaya category. Now that we know the monodromy we can prove it. The place where the monodromy is used is bold-faced in the proof below.

\begin{lemma} The parallel transported $\phi^H_{2\pi}(\ell_i)$ is Hamiltonian isotopic to $\ell_{i+1}$.\label{lem:ham_isotop}\end{lemma}

\begin{proof} We construct an isotopy $h_t:V^\vee \to V^\vee$ in the fiber in coordinates $(\xi_1,\xi_2,\theta_1,\theta_2)$. On a fiber, $\eta$ is a function of $\xi_1,\xi_2$ so doesn't show up in the notation until the end of the proof, when we consider maps on the total space $Y$. We want $h_t$ to map $\ell_1$ to $\phi_H^{2\pi}(\ell_0)$ where $\phi_H^{2\pi}$ is the monodromy. To prove $h_t$ is a Hamiltonian isotopy, i.e.~$\iota_{\frac{d}{dt}h_t}\omega = dH_t$ for some smooth function $H_t$, a classical result of Banyaga (cf \cite[Theorem 3.3.2]{symp_intro} or \cite[Equation (6)]{pasc_notes}) implies that it suffices to show that the flux of $\omega$ through cylinders traced out by generators of $H_1(V^\vee)$ under $h_t$, is zero. That is:
$$\left<\int_t \iota_{X_t}\omega ,[\gamma]\right> =: \left<\mbox{Flux}(h_t),[\gamma]\right> = \int_{h_t(\gamma)} \omega =0$$
where $X_t:=\frac{d}{dt} h_t$ and the second equality follows by an argument similar to the proof of Claim \ref{area_cyl_2dirns}, where area corresponds to integrating over the height direction $X_t$ and angular direction $\gamma'(t)$ on the cylinder. The isotopy $h_t$ interpolates linearly between $\phi_H^{2\pi}(\ell_0)$ and $\ell_{1}$. Recall that the angular coordinates on $\ell_{1}$ are given by $-\la(\xi)$ which we denote in components as $-(\la_1(\xi), \la_2(\xi))$. Then define the isotopy on the fiber as:
$$h_t(\xi_1,\xi_2,\theta_1,\theta_2):=(\xi_1,\xi_2,\theta_1 + t(f_1(\xi) +\la_1(\xi)), \theta_2 + t(f_2(\xi) + \la_2(\xi))):V^\vee \to V^\vee$$

This is well-defined modulo $\Gamma_B$ in the first two coordinates. \emph{ The isotopy is also well-defined modulo $\mb{Z}^2$ in the second two coordinates by our monodromy computation.} Lemma \ref{monodr} implies that $$(f_1,f_2)(\xi+\gamma) = (f_1,f_2)(\xi) - \la(\gamma)$$
therefore 
$$(f_1,f_2)(\xi+\gamma) + \la(\xi+\gamma)=(f_1,f_2)(\xi) - \la(\gamma)+\la(\xi+\gamma)= (f_1,f_2)(\xi) +\la(\xi)$$
for $\forall \gamma \in \Gamma_B$ so $h_t$ indeed descends to an isotopy on $V^\vee$. 

Note that $H_1(V^\vee)$ has rank four since $V^\vee \cong T_B \times T_F$. If we let $\gamma$ be the loop generated by one of the two angular directions $(0,0,1,0)$, then e.g.~$\gamma(s)=(0,0,s,0)$ and $h_t(\gamma(s))=(0,0,s,0)$ for all $t$ because $\xi=0$ and $f_i(0)=0$ as illustrated in Figure \ref{monodr_pic}. So the integral of $\omega$ over this cylinder of height zero is zero. Now we let 
$$\gamma(s) = (2s,s,0,0), \;\; -\frac{1}{2} \leq s \leq \frac{1}{2}$$
The case of $\gamma(s)=(s,2s,0,0)$ will be similar.
\begin{equation}
\begin{aligned}
\int_{h_t(\gamma)} \omega =\int_\gamma \iota_{X_t} \omega & = \int_\gamma (d\xi_1 \wedge d\theta_1 + d\xi_2 \wedge d\theta_2)( (f_1-\la_1)\dd_{\theta_1}+(f_2-\la_2)\dd_{\theta_2}, -)\\
& = \int_\gamma (f_1-\la_1)d\xi_1 + (f_2-\la_2)d \xi_2\\
& = \int_{s=-1/2}^{1/2}  2f(2s,s)ds + g(2s,s) ds
\end{aligned}
\end{equation}
where we used that $\la$ being a linear map implies anti-symmetry across zero, so the integral from $0$ to $1/2$ cancels the integral from $-1/2$ to 0. It remains to show that what remains is zero, namely that we have anti-symmetry across zero in $f_1$ and $f_2$. It suffices to show that $f_i(-\xi) = -f_i(\xi)$. We do this as follows. Consider the map on $Y$ given by:
$$\phi_{-}: (\xi, \eta, \theta, \theta_\eta) \mapsto (-\xi,\eta, -\theta, \theta_\eta)$$ 
It is a symplectomorphism because $d\xi \wedge d\theta + d\eta \wedge d \theta_\eta \mapsto d(-\xi) \wedge d(-\theta) + d\eta \wedge d \theta_\eta= d \xi \wedge d\theta + d\eta \wedge d \theta_\eta$. It remains to prove that it is fiber-preserving. By monotonicity of $\xi_1,\xi_2,\eta$ on the coordinate norms, the map $\phi_-$ on complex coordinates is
$$(x,y,z) \to (T^{-2} x^{-1}, T^{-2} y^{-1}, T^4 v_0^2 z^{-1})$$
because $\arg(x^{-1})=-\arg(x)$ and $\log |x|$ becomes $-\log|x|$ up to a constant in $T$, and similarly for $|y|$ and $|z|$. That this gives the map $\phi_-$ up to additive constants then follows by Claim \ref{claim:monotonic_mom_map}. It preserves the polytope $\Delta_{\tilde{Y}}$ as seen from the coordinates in Figure \ref{coords in tiling}. Since Claim \ref{claim:monotonic_mom_map} describes $\xi_1,\xi_2,\eta$ only up to additive constants we have a bit more work to do to show $\phi_-$ preserves a fiber. Suppose $\phi_-$ is defined by:
$$(\xi_1, \xi_2,\eta) \mapsto (-\xi_1 + c_1, -\xi_2 + c_2, \eta + c)$$

Since the polytope $\Delta_{\tilde{Y}}$ is preserved, this must map the origin to itself. Hence the constants are zero. Recall from Claim \ref{claim:mon_ind_angle} that $X_{hor} = \frac{\dd}{\dd \theta_\eta} + f_1(\xi, \eta) \frac{\dd}{\dd \theta_1} + f_2(\xi, \eta) \frac{\dd}{\dd \theta_2}$. Also $X_{hor}$ is preserved by any fiber-preserving symplectomorphism of $Y$ as in the proof of Corollary \ref{claim:par_fix_xi}. Since $\phi_-$ is one such, we have: 
\begin{align*}
X_{hor}=(\phi_{-})_*(X_{hor}) &=  X_{hor}\circ \phi_{-}\\
\implies (\phi_{-})_*\left(\frac{\dd}{\dd \theta_\eta} + f_1(\xi, \eta) \frac{\dd}{\dd \theta_1} + f_2(\xi, \eta) \frac{\dd}{\dd \theta_2}\right)&= \frac{\dd}{\dd \theta_\eta} - f_1\circ \phi_-(\xi, \eta) \frac{\dd}{\dd \theta_1} - f_2\circ \phi_-(\xi, \eta) \frac{\dd}{\dd \theta_2}\\
\implies \frac{\dd}{\dd \theta_\eta} + f_1(\xi, \eta) \frac{\dd}{\dd \theta_1} + f_2(\xi, \eta) \frac{\dd}{\dd \theta_2} &= \frac{\dd}{\dd \theta_\eta} - f_1(-\xi, \eta) \frac{\dd}{\dd \theta_1} - f_2(-\xi, \eta) \frac{\dd}{\dd \theta_2}\\
\therefore f_i(-\xi,\eta) &= -f_i(\xi,\eta)
\end{align*}
for $i=1,2$. So the flux of $\omega$ through $h_t(\gamma)$ is zero on generators of $H_1(V^\vee)$. By running this argument repeatedly, we can see that $(\phi_H^{2\pi})^k(\ell_0)$ is Hamiltonian isotopic to $\ell_{k+1}$ for any $k$. Hence $\phi_H^{2\pi}(\ell_i)$ is Hamiltonian isotopic to $(\phi_H^{2\pi})^{i+1}(\ell_0)$ which is Hamiltonian isotopic to $\ell_{i+1}$.
\end{proof}

\subsection{Setup for defining moduli spaces}\label{moduli}

Now that we have computed the monodromy, the next step will be computing the differential between two intersection points in the base. This will involve moduli space considerations. More generally, the Fukaya category is an $A_\infty$-category meaning in addition to objects and morphisms, there are structure maps on $k$ morphisms for any natural number $k$ that satisfy $A_\infty$-relations. These can be thought of as higher order associativity relations on the morphisms, hence the $A$ in $A_\infty$, and in the case of symplectic fibrations they involve counting pseudo-holomorphic discs which project to polygons in the base. The remainder of this chapter will set up the theory to define these counts. The word ``moduli" indicates we look at a set of objects ``modulo" an equivalence relation, and the ``space" refers to equipping this set with a topology. We start by defining the structure on the domains of the pseudo-holomorphic curves we will want to count. These are the \emph{source curves}. The following terminology was learned from \cite{seidel}.

\begin{definition}[Domains] A \emph{punctured boundary Riemann surface $S$} is the data of a compact, connected, Riemann surface with boundary and with punctures on its boundary, as well as the assignment of a Lagrangian $L_i$ to the $i$th component of $\dd S$. We further \emph{``rigidify"} by adding extra structure to $S$; denote punctures as ``positive" or ``negative" and define \emph{strip-like ends} via embeddings $\eps: (-\infty,0]_s \times [0,1]_t \to \mb{D}$ or $\eps: [0, \infty)_s \times [0,1]_t \to \mb{D}$ for the negative and positive punctures $\zeta^\pm$ respectively, such that $\lim_{s \to \pm \infty} \eps(s,t) = \zeta^\pm$. This is called a \emph{Riemann surface with strip-like ends}. \end{definition}

\begin{remark} This ``rigidifies" because any operation on $S$ must preserve the additional data of strip-like ends, thus placing further restrictions. The strips provide a nice set of coordinates near the punctures (namely $s$ and $t$) and give a straightforward way to pre-glue two sections by identifying linearly in the $(s,t)$ coordinates. See \cite[\textsection (8i) and \textsection (9k)]{seidel}. \end{remark}

\begin{example} In this thesis, $S$ will be one of the unit disc $\mb{D}$, two discs glued together at a point on their boundary with two punctures on one disc, or a disc with a marked boundary point identified at one point to a configuration of spheres in the central fiber. An example of a strip-like end we use later on to glue two discs is $(-\infty,0] \times [0,1] \ni (s,t) \mapsto \eps^-(s,t):=\frac{e^{-\pi(s+it)}+i}{e^{-\pi(s+it)}-i} = \frac{z+i}{z-i} \circ e^{-z} \circ \pi \cdot (s+it)\in \mb{D} \backslash\{1\}$. Note that $-\infty$ is the puncture which would map to $\zeta^-:=1$ in the disc. See Figure \ref{str_le}.\end{example}

\begin{figure}
\begin{center}
\includegraphics[scale=0.3]{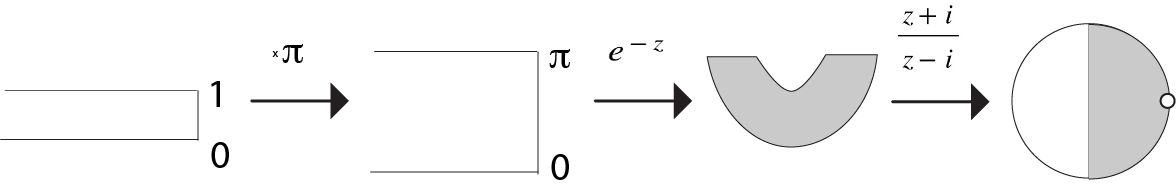}
\caption{Example of strip-like end}
\label{str_le}
\end{center}
\end{figure}

 Next we find the 2-homology of $Y$. This is where the pseudo-holomorphic discs map to in the target. \begin{lemma}[Homology of $Y$] $$H_2(Y)\cong H_2(\mb{CP}^2(3)/\Gamma_B)  \cong \mb{Z}^4$$ where all homology classes will be over $\mb{Z}$.
\end{lemma}

\begin{proof} Note that $Y$ deformation retracts onto the central fiber over the contractible base, determined by parallel transport in the inward radial direction. Therefore the homology of $Y$ is the homology of the central fiber, a degeneration of a $T^4$ fiber. 

%

To get $H_2(\mb{CP}^2(3)/\sim)$ we may use Mayer-Vietoris as follows. Cut out an open set projecting to a disc at the center of the hexagon in the hexagonal picture, which then retracts onto the boundary. This gives two open sets, which in the moment map picture are a disc and its complement, intersecting in an $S^1$.

Let $B:=D^2 \times T^2$ correspond to the disc in the moment polytope. Let $\tilde{A}$ be the complement, a string of 6 $\mb{CP}^1$'s in a circle and $A:=\tilde{A}/\Gamma_B$ the banana manifold. In particular, $\tilde{A}\cap B$ retracts onto $S^1 \times T^2 = T^3$. Then we have 

\begin{align*}
& 0=H_3(\tilde{A}) \oplus H_3(B) \to H_3(Z) = 0\\
&\to H_2(\tilde{A} \cap B)=\mb{Z}^3 \xrightarrow{(*)} H_2(\tilde{A}) \oplus H_2(B) = \mb{Z}^6 \oplus \mb{Z} \to H_2(\mb{CP}^2(3)) \to 0\\
& H_1(\tilde{A} \cap B) = \mb{Z}^3 \xrightarrow{\cong} H_1(\tilde{A}) \oplus H_1(B) = \mb{Z} \oplus \mb{Z}^2 \to H_1(\mb{CP}^2(3)) = 0
\end{align*}
where {$(*) = \left( \begin{matrix} 1 & 0 & 0 \\ 0 & 1 & 0 \\ -1 & 1 & 0 \\ -1 & 0 & 0 \\ 0 & -1 & 0\\ 1 & -1 & 0 \\ 0 & 0 & 1\end{matrix} \right)$} and the second to last map is the 3-by-3 identity matrix with respect to the three generators of $H_1(\tilde{A} \cap B) = H_1(S^1 \times T^2)$ (call them $b$ for the $S^1$ in the moment polytope and $T_1$, $T_2$ for the two loops in the 2-torus fiber). Then passing to the quotient $\mb{CP}^2(3)/\sim$ we find that
\begin{equation}
\begin{aligned}
& 0=H_3(A) \oplus H_3(B) \to H_3(\mb{CP}^2(3)/\sim) \\
&\to H_2(A \cap B)=\mb{Z}^3 \xrightarrow{(**)} H_2(A) \oplus H_2(B) = \mb{Z}^3 \oplus \mb{Z}\\
& \to \textcolor{red}{H_2(\mb{CP}^2(3)/\sim) = \mb{Z}^4} \to  \mb{Z} \cdot b \subset H_1(A \cap B) = \mb{Z}^3\\
&  \xrightarrow{(***) } H_1(A) \oplus H_1(B) = \mb{Z}^2 \oplus \mb{Z}^2 \to H_1(\mb{CP}^2(3)/\sim) = \mb{Z}^2
\end{aligned}
\end{equation}
where $(**) = \left( \begin{matrix} 0 & 0 & 0 \\ 0 & 0 & 0 \\ 0 & 0 & 0 \\ 0 & 0 & 1\end{matrix} \right)$ and $(***) =\left(\begin{matrix} 0 & 0 & 0 \\ 0 & 0 & 0 \\ 0 & 1 & 0\\ 0 & 0 & 1 \end{matrix}\right)$. Note that the final term is consistent with $Y$ being the quotient of its universal cover $\tilde{Y}$ by $\Gamma_B = \mb{Z}^2$, where $\pi_1(Y) = \pi_1(CP^2(3)/\sim)$ should be $\mb{Z}^2$. So ultimately we look at the projection of the 6 $\mb{P}^1$'s mapping to the three glued $\mb{P}^1$'s. Again we find that the homology is $\mb{Z}^4$.
\end{proof}


Now that we know what the second homology class of $Y$ is, we are better equipped to classify the possible choices for $\beta=[u]$ for pseudo-holomorphic maps $u$. We proceed to define the class of almost complex structures. Then we can define $J$-holomorphic curves. The almost complex structures we consider are compatible with the $\omega$ defined above and equal the standard $J_0$ induced from the complex toric coordinates near the boundary. Say $J=J_0$ outside of the open set 
\begin{equation}\label{defn: of_U_set}
U:= v_0^{-1}\left(B\left(\frac{1}{2}\right)\right),    
\end{equation} 
the preimage of a disc radius $1/2$ about the origin. In particular, U-shaped Lagrangians lie outside of $U$. We will denote this set as $\mc{J}_\omega(Y, U)$. It is non-empty because it contains $J_0$, and contractible by the same argument as for the set of all $\omega$-compatible almost complex structures. 

\begin{definition}[Maps]\label{defn:maps} Let $J \in \mc{J}_\omega(Y, U)$. A \emph{pseudo-holomorphic map} is a $J$-holomorphic map $u:(S,\dd S) \to (Y, \sqcup_{i\in \pi_0(\dd S)} L_i)$. A \emph{$J$/pseudo-holomorphic curve} is the image of such a map. We require
$$\lim_{s \to \pm \infty} u(\eps_{ij}(s,t)) \in L_i \cap L_j$$
where $\eps_{ij}$ is a strip-like end attached at an intersection point $L_i \cap L_j$ on the boundary.
\end{definition}

\subsection{Moduli spaces for a fibration} 
\begin{definition}[Section-like maps of a symplectic fibration]
We can think of the previous setup as a section of a trivial fibration with fixed Lagrangian boundary condition. For non-trivial fibrations, the Lagrangian boundary condition now consists of fibered Lagrangians. These are the ones described above obtained from parallel transporting a linear Lagrangian in a fiber over a U-shaped arc.  

Let $J_0$ be multiplication by $i$ in the toric coordinates. Given a $J_0$-holomorphic map to the total space $u: S \to Y$, we can compose with the holomorphic projection $v_0 \circ u: S \to \mb{C}$ to obtain a biholomorphism onto its image by the Riemann mapping theorem. In particular, if we identify $S$ with $v_0 \circ u(S)$, then $v_0 \circ u = \bm{1}_{S}$ and $u$ is actually a holomorphic section of $v_0$ which projects to an embedded holomorphic polygon in the base. 

On the other hand for generic $J$ close to $J_0$, whose existence we prove in the next section, pseudo-holomorphic maps $u$ are sections only outside $U$, where $J=J_0$. They still have algebraic intersection number 1 with fibers of $v_0$ in the region where $J_0$ has been perturbed to $J$. We will refer to these as \emph{section-like maps}.
\end{definition}

\begin{example} Let $t_x$ be the preimage of a moment map value $(c_1,c_2) \in T_B$, intersected with a fiber. So $t_x=\{c_1,c_2,\theta_1,\theta_2\}_{\theta_i \in [0,2\pi)}$ which in particular is invariant under parallel transport because that map rotates the angles.  Since $\ell_i \cap t_x$ is the one point of $\ell_i$ with $(\xi_1,\xi_2) = (c_1,c_2)$, any pseudo-holomorphic section $u$ must limit to that one point over the corresponding puncture in the base.
\end{example}

\begin{center}
    \fbox{\bf Count curves, not parametrizations:~stable maps quotient by reparametrization}
\end{center}

We have now discussed the Lagrangian boundary condition, the homology class, and the almost complex structure. So we next quotient by automorphisms of the domain. In this thesis, we will be concerned with stable maps to $Y$ where the domain is a disc with one boundary marked point, or a disc union sphere with one boundary marked point on the disc and one interior marked point on the disc and sphere each where they meet. These are stable maps (in particular they have finite reparametrization action) but not stable domains. So stabilization plays the role of quotienting by the automorphism group. One reference for the topology of the moduli space of pseudo-holomorphic maps is \cite{fuk_lectures}. 

\begin{definition}[Stable map, {see \cite[\textsection 9]{seidel}}] A \emph{stable map} is given by a tree of pseudo-holomorphic maps, where each vertex $\alpha$ is a sphere bubble, except for the vertex corresponding to original disc curve. The \emph{stable} refers to the absence of continuous families of nontrivial automorphisms; in particular if we fix three points on every constant component there are no nontrivial automorphisms. A tree encodes information for how to glue, where interior edges are assigned a gluing length, and semi-infinite edges at either end give the resulting marked points of the final glued disc. We have a family of discs, parametrized in the base by possible cyclic configurations of these points. We also have a gluing parameter for each interior edge where a sphere bubble is glued. More interior edges increases the codimension of the moduli space of that configuration in the moduli space of a disc. We can encode the possible degenerations in the \emph{Stasheff associahedron}.
\end{definition}
\begin{center}
   \fbox{{\bf Topology and compactification on set of curves}} 
\end{center}
Stable maps are defined in Chapter 5 of \cite{jholbk} in the \underline{case of spheres}, where they also prove compactification. The compactification of the moduli space of unparametrized curves (i.e.~equivalence classes) is a union over possible bubble trees of unparametrized curves, c.f.~{\cite[Equation (5.1.5)]{jholbk}}. The reason is that these are the possible limit configurations. Following {\cite[Chapter 5]{jholbk}}, if the derivatives in a fixed homology class are bounded in $W^{1,p}$ for some $p>2$, then we have a Bolzano-Weierstrass type result that any sequence has a convergent subsequence. However, in the borderline $p=2$ case, as is the case here, the limit may not be in original class of curves. The moduli space of \emph{stable} maps exhibits \emph{Gromov convergence}; a limit of a sequence of $J$-holomorphic curves is either in the original moduli space or not. If it is not, then in the limit we could have sphere bubbling. 

Gromov convergence to a stable map means: for each vertex $\alpha$ of the tree we have a family of reparametrizations $\phi^\nu_\alpha$ such that if we remove the bubbling points, then the sequence of curves converges uniformly on compact subsets to the main component. To obtain $C^\infty_{\text{loc}}$ convergence to the main component we use the Bolzano-Weierstrass theorem on the energy. And the reason the main component is fully defined on domain $S$ (even though we removed the bubble points) is because of the removable singularity property of pseudo-holomorphic maps. To find the other bubbled-off components, reparametrize the source curve by $z/R_n$, where $R_n \to \infty$, and include a sequence of points tending to the bubble point. Energies of the main component plus that of the bubbles should add up to the original energy since energy remains the same in the limit.
\begin{example}[A sphere bubble]
For example $\mb{CP}^1 \to \mb{CP}^2$ given by $[x:y] \mapsto [x^2: \eps y^2: xy]$ parametrizes the holomorphic curve in $\mb{CP}^2$ given by $ab = \eps c^2$. As $\eps \to 0$, we only see the curve $[x:0:y]$. However we can reparametrize and obtain a different limit; the image is actually two spheres, $ab=0$ and not just $b=0$, because of rescaling in $\mb{CP}^2$. In $ab = \eps c^2$, let $\eps \to 0$ to obtain $ab=0$.
\end{example}
Conversely, any such configuration can be viewed as a limit of $J$-holomorphic curves by pregluing the maps, i.e.~pasting them together (which may not give something $J$-holomorphic, hence it is called pregluing). However it is suitable as a compactification of a topological space. 

Here we are in the \underline{case of stable discs}. The paper \cite{Fra} proves compactness for pseudo-holomorphic discs in the sense of Chapter 4 of \cite{jholbk} (i.e.~not involving marked points). Lectures 5--7 of \cite{fuk_lectures} discuss the topology, Hausdorffness, and that sequential compactness implies compactness when the limit configurations are included. Limits of disc curves include two additional possible configurations to sphere bubbling: disc bubbling and strip-breaking. 

\begin{example}[Gromov compactness in our setting]\label{ex:compactness} For fixed $\beta$ and Lagrangian boundary condition, there is no disc bubbling or strip breaking under reparametrization. Only sphere bubbling cannot be excluded. Also, there are no multiply covered discs. \end{example}

\begin{proof}[Proof of example] The projection of a hypothetical disc bubble under $v_0$, if non-constant, would satisfy the open mapping principle outside of $U$ because it is a $J_0$-holomorphic section there. A disc bubble would have boundary in a single Lagrangian. However, all Lagrangians considered project to a U-shaped curve which does not enclose any bounded region of the complex plane. So any such disc would have to lie entirely in a fiber. However, linear Lagrangians in tori have zero relative $\pi_2$ i.e.~they don't bound discs. So there can be no disc bubbling in a fiber and hence no disc bubbling in general.  

Similarly if a strip breaks, any component that breaks off must either lie entirely within a fiber of $v_0$, or the degeneration must be visible in the projection to $v_0$. Namely, the polygonal region of the complex plane over which the section-like map projects also decomposes into a union of polygons with boundary on the given arcs. Since Maslov indices add, and all discs have intersection 1 with the central fiber, there is no room for a Maslov zero disc since it would have to be a bigon in a fiber (as it can't pass through the central fiber) and linear Lagrangian in tori cannot bound bigons. 


If we did have multiple covers, we may require that $J$ depend on $z \in S$, in which case we replace $J$ with $J_z$. The text \cite{jholbk} lays this foundation for Lagrangian boundary conditions case which is described in \cite{fuk_intro} and will involve an $\mb{R}$-family of $J$'s on a strip. We don't need to do that here since we are considering section-like discs, which are injective near their boundary. They are \emph{somewhere injective} and only wrap once around the boundary.\end{proof}

So in summary we define the resulting moduli space by taking a collection of maps, quotienting by reparametrization, and then compactifying. For Hausdorffness and compactness of this space, see \cite[\textsection 5.3]{ccliu}.

\begin{center}
\fbox{{\bf Dimension and Maslov class}}     
\end{center}
\begin{definition}The \emph{virtual dimension} of the moduli space of section-like maps is given by the Maslov index of the chosen $\beta$ homology class, minus twice the number of edges in the stable tree. That is, each bubbled off sphere reduces dimension by 2. Thus we expect to have a \emph{pseudocycle}, i.e.~the image of the boundary of the moduli space has codimension at least two in the total space. This can be achieved in the \emph{semipositive} case, which means there are no $J$-holomorphic spheres of negative Chern number for generic $J$ -- this is our setting because $Y$ is Calabi-Yau. 
\end{definition}

\begin{remark}[Analogue of dimension in Morse theory] With Morse-Smale data we know the dimension of the moduli space by counting eigenvalues. We require transverse intersection of the unstable manifold of $x_{-}$ with the stable manifold of $x_+$. We have a projection map from the unstable manifold to directions perpendicular to the stable manifold. On the other hand, the index of a Fredholm section is computed from the \emph{spectral flow} of a loop of symmetric matrices, see \cite[\textsection 3.2]{wendl}. \end{remark}


Given a Fredholm section, we can compute the expected dimension of the manifold of parametrized curves, and then the moduli space of unparametrized curves will be three less from quotienting by $\Aut(\mb{D})$. A Spin structure on the Lagrangian determines an orientation on the moduli space of parametrized curves, so quotienting by $\Aut(\mb{D})$ induces an orientation on the moduli space of unparametrized curves. \label{dimension} 
%
%

\begin{example}[Deligne-Mumford moduli space of domains,{see {\cite[\textsection 5, \textsection 6]{jholbk}} and \cite[Lectures1--4]{fuk_lectures}}] A {stable disc} is a {stable map} to a point. The Deligne-Mumford space $\mc{M}_{0,d+1}$ arises from considering degenerations of stable discs with the additional data of $d+1$ boundary marked points, e.g.~ possible limits under reparametrization in Gromov compactness, see \cite[\textsection 5.5]{jholbk}. The real dimension for genus $g$ curves with $l$ marked points is $6g-6+2l$.
\end{example}

In general we can compute dimension by considering the Maslov class of the disc, which in our setting can be computed as an intersection number with toric divisors. The following is based on \cite{t_duality} and \cite{cho_oh}. 

\begin{definition}[Maslov class of a Lagrangian and of a disc] Suppose $2c_1(M)=0$, so the square of the anticanonical bundle is trivializable by some section $s$. We have a map from $LGr$ to the unit bundle of $K^{-2}$ by taking $\det^2$ of a basis for each Lagrangian. We can identify that unit bundle with $S^1$ using the trivializing section $s$ mentioned above. The upshot is that we get a map from $LGr \to S^1$.  The ``Maslov class" of the Lagrangian is the pullback of $[S^1]$, i.e.~a homology class in $LGr$. If it is zero (meaning we can lift to $\mb{R}$), then we can define $\mu(\beta)$ as the evaluation of this homology class on $\beta$. This measures the obstruction to extending the square of this normalized section on $L$ to one on a disc representing 2-homology class $\beta$.
\end{definition}


\begin{definition}\label{sLag}
A special Lagrangian $L$ is defined to be one such that if $\Omega$ trivializes $K_{X\backslash D}$ on the complement of an anticanonical divisor $D$, then $\Omega|_L = e^{-i\phi}dvol_L$ for some constant phase $\phi$.
\end{definition}

\begin{lemma}[{\cite[\textsection 3.1]{t_duality}}]\label{maslov_intersec} Let $(X,\omega,J)$ be a smooth, compact and K\"ahler manifold. Let $\Omega \in \mc{M}^0(X,({T^*}^{(1,0)}X)^n)$ be a global meromorphic $n$-form, with poles along an anti-canonical divisor $D$, e.g.~using $\log$ coordinates. In other words, $\Omega^{-1}$ is a nonzero holomorphic section of the anti-canonical bundle on $X \backslash D$. Let $L$ be a special Lagrangian submanifold in $X \backslash D$. Let $\beta \in \pi_2(X,L)$ be nonzero. Then $\mu(\beta)$ is twice the algebraic intersection number $\beta \cdot [D]$, where $\mu(\beta)$ denotes the Maslov class.\end{lemma}

\begin{proof}[Proof from {{\cite[\textsection 3.1]{t_duality}}}]
The tangent space to $L$ is real since being Lagrangian is defined by $(TL)^\perp = JTL$ with respect to $\omega(-,J-)$. Taking a real basis gives a nonvanishing section of $K_X^{-1}|_L$ which we can scale to unit length. Since we've normalized, this section is independent of choice of basis. In particular, its square also trivializes the square of the anticanonical bundle over $L$. Since $L$ is special Lagrangian, $\Omega^{-1}$ which defines the divisor of $D$ is equal to this volume form on $L$, up to a constant phase factor, by Definition \ref{sLag}. The Maslov class of $\beta$ is then $\deg(\Omega^{-2}|_\beta)=2D \cdot\beta$.
\end{proof}

\begin{claim}[{\cite{cho_oh}, \cite{t_duality}}]\label{one_disc_only} Consider the moduli space of $J_0$-holomorphic discs with boundary in $L=\bigcup_{circle}t_x$ for a circle around the origin in the base of $v_0$. This is a product torus in a toric variety, and $\beta$ is a class of Maslov index 2, i.e. by Lemma \ref{maslov_intersec} the class of a disc intersecting the central fiber once. Then we claim that there is only one $J_0$-holomorphic disc in each $\beta$ class.
\end{claim}
\begin{proof}


Using Lemma \ref{maslov_intersec}, we can interpret the dimension geometrically from intersection numbers. Since we're considering $J_0$-sections $u$, they pass once through the central fiber at 0, which is also the divisor $D$ in our setting.  $D$ is the union of toric divisors in $v_0^{-1}(0)$. They intersect $D$ transversely once, hence they have Maslov index 2 by the above result. There are no nontrivial Maslov zero discs by the fact that linear Lagrangians in the fiber do not bound discs. If we require one boundary marked point to map to a given point of $L$, this cuts down the dimension of the moduli space of holomorphic curves by $n$, but we add one dimension from the choice of marked point. (The collection of discs with boundary on $\bigcup_{circle} t_x$ and a point constraint on the boundary is a zero dimensional family, otherwise we could rotate the disc by the $T^3$ action and obtain a family of discs.) Since the real dimension of the Lagrangian $n=3$, $\mu(\beta)=2$, we subtract three dimensions from fixing three complex-valued boundary points (only the identity automorphism on a disc fixes 3 boundary points so this quotients by the automorphism group), and then including a marked point on the Lagrangian boundary, we find that the virtual dimension of the moduli space of holomorphic curves in class $\beta$ intersecting the toric divisor $D$ once is 
$$\ind(D_u) - 3+1= n \chi(D^2)+ \mu(\beta)-2=3+ 2 (\beta \cap [D])-2=3$$ 
by the index theorem in Fredholm theory, where $\chi$ denotes the Euler characteristic and $n$ is the real dimension of the Lagrangian, see \cite[Theorem C.1.10(ii)]{jholbk}. Thus, if we constrain the evaluation map at the marked point to lie at a particular point in the 3-dimensional Lagrangian, this will cut down the dimension to 0. This results in an expected dimension of zero. We now show that there is only one disc in each homology class.

%
%

%


The linear fiber Lagrangian $t_x$ can be thought of as corresponding to the skyscraper sheaf under mirror symmetry. Recall that the definition of $\ell_k$ involved rotating an amount $2\pi k$ along the two angle directions as we traverse one loop in each of the base $\xi_i$ moment map directions. Rotating only the angle directions and not in the base defines $t_x$, namely, we fix the moment map coordinates and let the angles vary. This gives the preimage of a moment map coordinate $A: = (a_1,a_2,a_3)=(\xi_1,\xi_2,\eta)$ and we let $|z|:= \tau^{A}$ denote the exponentiated coordinates. The reason for choosing the letter $A$ is that the formula for counting such discs is discussed in a paper of \cite{cll} and that notation follows theirs. 

Because each disc considered intersects the central fiber only once, its lift to the universal cover $\tilde{Y}$ can intersect only one of the toric divisors of $\tilde{Y}$. Which divisor it intersects is determined by the class $\beta$. The image of the disc is contained in the union of the open stratum of $\tilde{Y}$ and the open stratum of the component of $D$ that the disc intersects, which by standard toric geometry is the image of a toric chart $\mb{C}^* \times \mb{C}^* \times \mb{C}$ inside $\tilde{Y}$.

Thus we think of the disc as mapping to the chart $\mb{C}^* \times \mb{C}^* \times \mb{C}$. Then a disc with boundary in $S^1(r_1) \times S^1(r_2) \times S^1(r_3)$ implies it is constant in the first two components (by the maximum principle) and we obtain a disc in the last $v_0$ coordinate. And this is the only disc by the Riemann mapping theorem. It cannot be multiply covered because it must be Maslov index 2 by the previous paragraph. 

So the discs we count correspond to selecting a point $A$ in the moment polytope and drawing a line to a facet in that polytope. Geometrically, fixing $A$ implies each disc has the same $T^3 \cong S^1(r_1) \times S^1(r_2) \times S^1(r_3)$ Lagrangian boundary condition, namely the moment map preimage of $A$. As we allow $|v_0| \to 0$, we find that $r_1,r_2$ remain constant and the third coordinate goes to zero along the disc. (In the moment polytope this corresponds to a path, depending on $\beta$, from $A$ to a facet. And $\xi_1,\xi_2,\eta$ may vary along the path.) We count each disc over the possible $\beta$, weighted by area, in Theorem \ref{theorem:disc_count_s}.

This concludes the proof that there is only one $J_0$-holomorphic disc in each $\beta$ class.
\end{proof}


\subsection{Existence of regular choices to define compact moduli spaces}\label{defn}

The existence of regular choices is theory that is known in numerous cases; for completion we include the proofs because they are in the setting of Lagrangian boundary condition, versus the case of no Lagrangian boundary condition used in references cited here. Also, the proof of compactification is specific to the geometry of our set-up, so details are provided there as well.

\begin{definition}
We say that an almost complex structure $J \in \End(TY)$ is \emph{regular} if, for all $J$-holomorphic maps $u: (\Sigma, j) \to (Y,J)$ on the complex curve $\Sigma$, the linearization (or derivative) of the Cauchy-Riemann operator $\ol{\dd}_J$ is surjective.
\end{definition}

\begin{remark}
``Regularization" refers to perturbing the $\ol{\dd}_J$ operator to be equivariantly transverse to the zero section of a Fredholm bundle which we can build so that the operator is a section of the bundle. ``Geometric regularization" means the perturbations are obtained by perturbing the almost complex structure $J$, so are geometric in nature. Namely the perturbations of $\ol{\dd}_J$ are $\ol{\dd}_{J'} - \ol{\dd}_J$ as $J'$ varies. In this setting equivariance will be automatic, as described below. Note that later on, we will need to use a non-regular $J$ for computations and in that case we will use ``abstract regularization" by adding abstract perturbations $p$ which are sections of an ``obstruction" bundle built from the non-surjectivity of $D_u$. They are not necessarily of the form $\ol{\dd}_{J'} - \ol{\dd}_J$. As intuition one can recall that the preimage of a regular value of the function $f(\bm{x}) = x_1^2+x_2^2+x_3^2$ is a manifold, namely a 2-sphere. Here we consider an infinite-dimensional analogue.
\end{remark}


\begin{center}
\fbox{\bf Finite rank geometric regularization intro from \cite[Lecture 1A]{wehr_reg}}
\end{center}

For a section of a finite rank bundle over a finite dimensional manifold, its zero set is automatically compact. However, we don't necessarily have that over an infinite-dimensional manifold. In the finite dimensional case, there exists a space of perturbations $\mc{P}$ so that $s+p$ is transverse to the zero section for all $p \in \mc{P}$, i.e.~the derivative $D_u(s+p)$ is surjective for all $u \in (s+p)^{-1}(0)$. We can guarantee that $\mc{P}$ is non-empty, the perturbed $s$ is transverse to the zero section, $(s+p)^{-1}(0)$ is compact, and we can construct cobordisms between zero sets for different choices of $p$. 

These properties allow us to define a fundamental class for $\ol{\mc{M}} := s^{-1}(0)$, denoted $[\ol{\mc{M}}]$, even if it is singular from lack of transversality to the zero section (which would have ensured it is a manifold). We view the fundamental class of the singular moduli space as the intersection of nested open sets $\mc{W}_k$, which are subsets of points in the base of the bundle that are $1/k$ away from $s^{-1}(0)$. In particular, by non-triviality of $\mc{P}$, we can take a sequence of perturbations so $(s+p_k)^{-1}(0) \subset \mc{W}_k$ stays the same. A fundamental class is necessary to count the zero set of $s$, i.e.~compute an intersection of $s$ with the zero section. We build a fundamental class using that \v{C}ech homology of the singular manifold is isomorphic to an inverse limit of rational \v{C}ech homology of the nested open sets. Each $\mc{W}_k$ has a fundamental class in the top degree of \v{C}ech homology, which will limit to the fundamental class we are looking for.

Fortunately, Banach spaces still come equipped with all the necessary tools to obtain smooth structures analogous to the finite-dimensional case. We follow the arguments of \cite[Chapters 3--10]{jholbk}, adapting their $S=\mb{CP}^1$ setting to our $S=\mb{D}$ setting with Lagrangian boundary conditions; this is also discussed in \cite{seidel} and \cite{math257b}.  

%

\begin{example}An example of the geometric regularization theory is implemented in \cite{wehr_reg0} for Gromov non-squeezing, which involves illustrating how to show a family of moduli spaces varying $J_t$ is 1 dimensional (Fredholm), a manifold (transversality/regularity), compact (Gromov compactness) and has boundary.
\end{example}

\begin{remark}
$J$-curves have some analogous properties as complex curves, such as the Carlemann similarity principle and unique continuation (if two $J$-holomorphic maps agree on an open set, or all derivatives agree at a point, then the maps are equal). Other properties include that there are only finitely many points in the preimage of a point and only finitely many critical points. Furthermore simple curves (or their analogue in the Lagrangian boundary condition case, somewhere injective curves) have an open dense set of injective points.
\end{remark}


\begin{center}
\fbox{{\bf Existence of regular $\bm{J}$ in 4 steps}}
\end{center}

Let $U$ be the open neighborhood of the singular fiber of $v_0$ defined in the paragraph before Definition \ref{defn:maps}. Recall
\begin{equation}\label{eq:set_of_J}
    \mc{J}_\omega (Y,U) := \{J \in \Omega^0(Y,\End(TY)) \mid J^2 = -\bm{1}, \omega(\cdot, J \cdot) \mbox{is a metric}, J\mid_{Y \backslash U} \equiv J_0 \}
\end{equation}
In particular, this set is non-empty because it contains $J_0$, and is contractible by the same argument as in the case of no boundary. This set of $J$ is what's needed to prove geometric regularization in the boundary case, see \cite[Remark 3.2.3]{jholbk}. Furthermore, let $\gamma_i$ denote curves in the base of $v_0$ and $\bigcup_{\gamma_i}\ell_i$ for parallel transport of $\ell_i$ over $\gamma_i$. Recall from Definition \ref{defn:fuk_cat} that $L_i$ is the Lagrangian given by parallel transporting $\ell_i$ from the $-1$-fiber in a U-shape. Then define the notation
\begin{equation}\label{eq:lagr_bdry_notation}
    L|_\gamma := \bigcup_{i \in I \subset \mb{Z}} \bigcup_{\gamma_i}\ell_i
\end{equation}
Since all maps to $Y$ we consider are polygons when projected to the base of $v_0$, all discs pass through the zero fiber.

\begin{remark}[Notation] The notation $\dot{J}$ does not mean we can only vary $J$ in one direction as is usually the case with the dot notation. We use the notation as a symbolic way to denote tangent vectors to the space of complex structures; it's denoted $Y$ in \cite{jholbk}.
\end{remark}

\begin{lemma}[Geometric regularization]\label{geom_reg_lemma} There exists a dense set $J \in \mc{J}_{reg}^1 \subset \mc{J}_\omega (Y,U)$ such that, for all $J$-holomorphic maps $u: (\mb{D},\dd\mb{D})\backslash \{z_1,\ldots,z_{|I|}\} \to (Y, L|_{\gamma})$ as in Definition \ref{defn:fuk_cat}, the linearized $\ol{\dd}$-operator $D_u$ is surjective.\end{lemma}

\begin{remark} We use the superscript 1 because later we will want existence of the slightly smaller set $\mc{J}_{reg}^2$ of $J$ regular for a disc attached to a sphere with similar Lagrangian boundary conditions. These will require not only surjectivity of the linearized operator but also compatible behavior when evaluating at the intersection point.\end{remark}

\begin{proof}[Roadmap adapting the 2nd edition book {\cite[pg 55, proof of Theorem 3.1.6 (ii)]{jholbk}}] We adapt McDuff-Salamon to the setting with Lagrangian boundary conditions. Note Theorem C.1.10 in \cite{jholbk} already proves we have a Fredholm problem for the case of boundary. The background for this was also learned from \cite[Lecture 9]{wehr_reg}. 

Take a homology class $\beta \in \pi_2(Y, L|_{\gamma})$. We make the following definitions
\begin{equation}
\begin{aligned}
    \mc{B}_\beta^{1,p}&:=\{u \in W^{1,p}((\mb{D},\dd \mb{D}), (Y, L|_{\gamma})) \mid [u] = \beta, \lim_{z \to z_i} u(z) = p_i\}\\
        \mc{J}^\ell_\omega (Y,U)& := \{J \in C^\ell(Y,\End(TY)) \mid J^2 = -\bm{1}, \omega(\cdot, J \cdot) \mbox{ is a metric}, J\mid_{Y \backslash U} \equiv J_0 \}\\
        \hat{\mathring{\mc{M}}}_\mc{J}&:=\{(u, J)\mid J \in \mc{J}^{\ell}_\omega(Y, U),  u \in  \mc{B}_\beta^{1,p}, \ol{\dd}_J(u) = 0\} 
    \end{aligned}
    \end{equation}
where the hat indicates we haven't yet quotiented by automorphisms of the source curve and the ring indicates we haven't compactified yet. Then we claim that 
\begin{equation}
\mc{B}_\beta^{1,p}\times \mc{J}^\ell_\omega(Y,U) \ni (u,J) \mapsto \ol{\dd}_J(u) \in L^{p}(\mb{D}, \Lambda^{0,1}\otimes u^*TY)
\end{equation}
is a Fredholm section $s$ of a Banach bundle, with surjective derivative, hence it also has a right inverse and we can invoke the Inverse Function Theorem to deduce that $ \hat{\mathring{\mc{M}}}_\mc{J}= s^{-1}(0)$ is a $C^{\ell-1}$ Banach submanifold of the base. We will now justify this. We use $L$ to denote $L|_{\gamma}$ for ease of notation.

\begin{center} \fbox{Regular $J$ step 1 of 4: Banach manifold structure on $\mc{B}^{1,p}$}
\end{center}
A map $u \in \mc{B}^{1,p}$ has a local Banach chart from exponentiating $T \Gamma(u^*TY, u^*TL) \ni \xi$ via $u \mapsto \exp_u \xi$. The map $\exp_u \xi$ still has the correct boundary condition; there exists a metric so that one Lagrangian is totally geodesic \cite[Lemma 4.3.4]{jholbk} which can be adapted to the argument for two transversely-intersecting Lagrangians as is done in \cite[Lemma 6.8]{milnor_hcobord}, \cite{miln_notes} for submanifolds of complementary dimensions. Namely, take a convex combination of the two metrics defined for each Lagrangian separately, and by uniqueness of geodesics given a starting point and direction, as well as the definition of totally geodesic, we see that the result still holds in a neighborhood of an intersection point. And at any point considered there are at most two Lagrangians intersecting, so this suffices. Once we have the metric for which both Lagrangians are totally geodesic, then for $p \in\dd \mb{D}$ we see that $\exp_{u(p)}\xi$ is a point on the geodesic which starts in $L$ and travels in the direction of a vector tangent to $L$, so it must remain in $L$. 

\begin{center}\fbox{{Regular $J$ step 2 of 4: Banach manifold structure on $\mc{J}^{\ell}_\omega(Y, U)$}}
\end{center}

The second factor on the base of the Banach bundle we are constructing, $\mc{J}^{\ell}_\omega(Y, U)$, has Banach charts around elements $J$ as described in \cite[\textsection 3.2]{jholbk}. We linearize the conditions on $J$ to obtain the conditions for vectors $\dot J$ in the tangent space, and local charts can then be recovered from the tangent space by exponentiating. Three conditions on $J$ become linearized: 1) $J|_{Y \backslash U} \equiv J_0$ implies $\dot J|_{Y \backslash U} \equiv 0$, 2) $J^2 = -\bm{1}$ implies $\dot{J}J + J \dot{J} = 0$, and 3) $\omega(\dot{J}\cdot , \cdot)+\omega(\cdot , \dot{J} \cdot) =0$ arises from $\omega$-compatibility. Equivalently, the second and third conditions impose that $\dot{J} = J \dot{J} J$ and $\dot{J}$ is self-adjoint with respect to metric $\omega(\cdot, J \cdot)$. So a chart centered at $J$ is constructed via $\dot{J} \mapsto J \exp(-\dot{J} J)$. 

A fiber of the bundle will not have boundary conditions because it is given by the space where $(du)^{0,1} = \frac{1}{2}(du + J \circ du \circ j)$ lands in, namely $L^{p}(\mb{D}, \Lambda^{0,1}\otimes u^*TY)$, and that doesn't concern the boundary. Note that $J \circ du \circ j$ does not a priori have the same behavior as $du$ in the direction tangent to the boundary, so $(du)^{0,1}$ does not satisfy any particular boundary condition. Hence the structure of the Banach bundle over open sets in the base will mimic the case of \cite[Proof of Proposition 3.2.1]{jholbk} of no Lagrangian boundary conditions. We again use the exponential map to trivialize the bundle over a neighborhood $\mc{N}(u)$ in the first factor of $\mc{B}^{1,p} \times\mc{J}^{\ell}_\omega(Y, U)$, and we use parallel transport to trivialize over a neighborhood $\mc{N}(J)$. This trivializes the bundle over a neighborhood in the base, and a composition of these gives transition maps that satisfy conditions for a Banach bundle, \cite[Proof of Proposition 3.2.1]{jholbk} and \cite[\textsection (9k)]{seidel}.


\begin{center}\fbox{Regular $J$ step 3 of 4: Key regularity argument}\label{reg argument}\end{center}

We have that $s(u,J) := \ol{\dd}_J(u)$ is a Fredholm section of this Banach bundle by \cite[Theorem C.1.10]{jholbk}. Furthermore, its derivative at any $(u,J)$ such that $s(u,J)=0$ and $u$ is somewhere injective (guaranteed by $u$ being $J_0$-holomorphic outside of $U$, so a section of $v_0$), is surjective, as follows. See \cite[Proof of Proposition 3.2.1]{jholbk}.  
\begin{equation}
(d_{(u,J)}s)(\xi, \dot{J}) =  D_u\xi + \frac{1}{2}\dot{J} du j: W^{1,p}(\mb{D}, u^*TY) \times C^{\ell}(Y, \End(TY, J, \omega)) \to L^{p}(\mb{D}, \Lambda^{0,1} \otimes_j u^*TY)   
\end{equation}

where
\begin{equation}
\begin{aligned}
    D_u:=d\mc{F}_u(0): W^{1,p}(\mb{D},\dd \mb{D}; u^*TY, u^*TL)& \to L^{p}(\mb{D},\Lambda^{1,0} \otimes_j u^*TY)\\
    \mc{F}_u: W^{1,p}(\mb{D},\dd \mb{D}; u^*TY, u^*TL)& \to L^{p}(\mb{D},\Lambda^{1,0} \otimes_j u^*TY)\\
\xi & \mapsto \ol{\dd}_J(\exp_u \xi)
\end{aligned}
\end{equation}
see \cite[Proposition 3.1.1]{jholbk}. In words, $D_u$ is defined by, for a nearby $u'$ in the exp neighborhood of $u$, parallel transport back to the origin of the chart at $u$, take $\ol{\dd}_J$, then map forward again on the fiber under parallel transport; the linearized operator $D_u\xi$ will be the derivative of this operation at the point 0. This is well-defined because of the totally geodesic condition above. The $D_u$ term is only from varying $u$. Varying $J$ as well we get (e.g.~see \cite[Lecture 9]{wehr_reg}):
\begin{equation}\label{eq:deriv_univ}
(d_{(u,J)}s)(\xi, \dot{J}) =  D_u\xi + \frac{1}{2}\dot{J} du j   
\end{equation}

Given that a $J$-holomorphic map $u$ must be a section of $v_0$ outside of $U$ (since $J$ is $J_0$ there), we see that $u$ cannot wrap more than once around the boundary. Hence the curve is somewhere injective. We claim that the operator $ds$ is surjective with continuous right inverse, so is regular. Suppose by contradiction the image is not dense. Then we can construct a nonzero linear functional that is orthogonal to $im(ds)$ and locally lies in a fiber of the Banach bundle, namely a non-zero $L^q$ form $\eta$ that annihilates $D_u\xi + \frac{1}{2}\dot{J} du j$ over all $W^{1,p}$ tangent vectors $\xi$ and $\dot{J}$, in particular $\dot{J}=0$. So $D_u\xi=0$ for all $\xi$ by Equation \ref{eq:deriv_univ}. This implies that $\eta$ has $W^{1,p}_{loc}$ regularity by the elliptic bootstrapping result of \cite{jholbk}, proven for Lagrangian boundary conditions. 

Once we have regularity, we can integrate and prove $D_u^* \eta=0$. Via integration by parts (with evaluation on the 1-form and inner product on the bundles) we have $D^*\eta = 0$, because taking the adjoint leaves a boundary term $d\left<\eta, J \xi\right>_\omega$. Then using Stokes' theorem and that the test vectors $\xi$ are tangent to the Lagrangian at the boundary, this equals zero.

However since $\eta \neq 0$, using bump functions we may construct a perturbation $\dot{J}$ as in \cite[page 65]{jholbk} so $\eta$ integrated on $\dot{J}$ is nonzero, contradicting that $D_u^*\eta=0$. The construction in \cite[page 65]{jholbk} still works in the Lagrangian boundary setting because the constructed $\dot{J}$ is supported in a small neighborhood around a somewhere injective point, so will be zero near the boundary as required. It is as follows. We've assumed $\eta \neq 0$ so pick point $p$ so that $\eta|_p \neq 0$. Somewhere injective points are dense so find a neighborhood of them around $p$. Use bump functions to construct a $\dot{J}$ so that $\int_{\mb{D}} \eta(\dot{J} du j) >0$. This is a contradiction. So $\eta$ vanishes on the open set of injective points, hence vanishes identically by unique continuation \cite[Theorem 2.3.2]{jholbk}. This is again a contradiction since $\eta \neq 0$. So the annihilator of $ds$ is zero and the Hahn-Banach theorem implies that the image of $ds$ is dense. So combining that property with the image being closed from the Fredholm property of the operator, we find that the operator surjects onto $L^p$. This will allow us enough freedom to find the vectors $\dot{J}$.  


\begin{center}\fbox{Regular $J$ step 4 of 4: Implicit and inverse function theorem find dense set of regular $J$}
\end{center}

Since $D_u$ is Fredholm by \cite[Theorem C.1.10]{jholbk} and $ds= D_u \oplus B$ for bounded linear operator $B=\frac{1}{2}\dot{J} du j$, is surjective, \cite[Lemma A.3.6]{jholbk} implies that $ds$ has a right inverse. Thus 0 is a regular value of $s(u,J)=\ol{\dd}_J(u)$ and by the Implicit Function Theorem \cite[Theorem A.3.3]{jholbk} $s^{-1}(0) = \hat{\mathring{\mc{M}}}_\mc{J}$ is a $C^{\ell-1}$-Banach submanifold of $\mc{B}_\beta^{1,p} \times \mc{J}^\ell_\omega(Y,U)$. Separability of $\hat{\mathring{\mc{M}}}_\mc{J}$ is inherited.  

Now consider the projection $\pi: \hat{\mathring{\mc{M}}}_\mc{J} \to \mc{J}^\ell$ given by $(u,J) \mapsto J$. This is Fredholm because it has the same kernel and cokernel as $D_u$ from \cite[Lemma A.3.6]{jholbk}. Also its linearization is surjective since it's a projection. So we have regularity of $\pi$ at $(u,J)$ whenever $J$ is regular. Hence we have the hypothesis of \cite[Theorem A.5.1 (Sard-Smale Theorem)]{jholbk} (which relies on the infinite-dimensional inverse function theorem, \cite[Theorem A.3.1 (Inverse Function Theorem)]{jholbk}). The result of Sard-Smale implies that these regular $J$ values are dense, i.e.~$\mc{J}^1_{reg}$ is dense in $\mc{J}_\omega(Y,U)$ as we wanted. So in particular we have existence of regular $J$.

This concludes the proof of existence of regular $J$.
\end{proof}

This type of problem shows up often so it has a name.

\begin{definition}
A \emph{Fredholm problem} concerns the zero set of a \emph{Fredholm section} of a Banach bundle. I.e.~a section whose linearization is a Fredholm operator, namely $\dim \ker$ - $\dim \coker$ is finite and whose image is closed. It's the set-up for the moduli spaces in question (the main part or fiber products of moduli spaces that show up when Gromov compactifying) as the zero sets of Fredholm sections. Regular values of a Fredholm problem put additional structure on the moduli spaces, enabling us to count them.
\end{definition}

\begin{lemma} The set of parametrized (i.e.~before quotienting) $J$-holomorphic discs $u:(\mb{D},\dd \mb{D}) \to (Y,L|_\gamma)$ for $J \in \mc{J}^1_{reg}$ is a manifold of finite dimension given by index$(D_u)$.\end{lemma}

\begin{proof}[Proof from {\cite[Theorem 3.1.6 (i)]{jholbk}}]
Charts are given by
\begin{align*}
\mc{F}_u: W^{1,p}(\mb{D},\dd \mb{D};u^*TY, u^*TL) & \to L^{p}(\mb{D},\Lambda^{1,0} \otimes_j u^*TY)\\
\xi & \mapsto \ol{\dd}_J(\exp_u \xi)
\end{align*}
using the exponential map to obtain a diffeomorphism of an open set around 0 in $\mc{F}_u^{-1}(0)$ to a neighborhood of $u$ in the space of parametrized $J$-holomorphic discs. Regularity of $J$ implies $d\mc{F}_u(0) = D_u$ is surjective. The implicit function theorem \cite[Theorem A.3.3]{jholbk} implies these are smooth manifold charts after restricting to potentially smaller open sets. This does not depend on $p$ because $J$-holomorphic maps $u$ are smooth by elliptic regularity \cite[Proposition 3.1.10]{jholbk}, when $J$ is smooth. Note that \cite[Appendix B (Elliptic Regularity)]{jholbk} covers the necessary background with \emph{totally real boundary conditions} e.g.~the Lagrangian boundary condition case here. See also \cite[Lectures 4--6]{wehr_reg0}. The dimension statement follows from the tangent space of the moduli space being given by $\ker(D_u)$, and also that $\coker(D_u)=0$, so their difference i.e.~ the Fredholm index, is also the dimension of the moduli space. \end{proof}

\begin{remark} Even for non-regular $J$, we can still construct a Fredholm problem by \cite[Theorem C.1.10 (Riemann-Roch)]{jholbk}, which is proven in the case of Lagrangian boundary condition. How we get the smooth structure will be a different matter though, because $J$ is not regular. This is where abstract perturbations of the $\ol{\dd}_J$-operator are used.
\end{remark}

\begin{center}\fbox{\bf Compactification}\label{dim_claim}\end{center}

We will show that the moduli space, which now has a smooth manifold structure, is already compact. Then we take the zero dimensional part, which as a compact 0-manifold is now something we can count. The more general case of compactifying and putting on a smooth structure is by gluing, e.g.~\cite[\textsection 3.4, Proposition 6.2.8]{jholbk},  \cite{wehr_reg0}, and \cite{fuk_lectures}. In particular, we have compactness up to sphere bubbling in the setting of moduli spaces in this paper by Example \ref{ex:compactness}.


The intuition for why there are no sphere bubbles is as follows: the union of all points in a zero-dimensional family of spheres is two, and that of discs in a one-dimensional family is three, so generically these two don't intersect in a six dimensional manifold. Here by `generic $J$' we mean that transversality should hold for evaluation maps at marked points on discs and spheres. 


So we excluded disc bubbling and strip breaking. We now show that, for regular $J$, the moduli space of a somewhere injective disc union a simple sphere is a manifold of negative dimension, meaning it is empty and can be excluded.

\begin{lemma}[Excluding bubbling in our setting] There exists a dense set $\mc{J}^2_{reg}(Y, \dd Y; \mb{D} \cup \mb{P}^1)$ of $J$ regular for the moduli space of maps with domain a simple sphere attached to a somewhere injective disc with one boundary marked point. The maps are section-like hence somewhere injective. \label{Jreg2}
\end{lemma}

\begin{cor}  The moduli space of stable configurations consisting of a disc with one marked boundary point identified at its center to one or more sphere bubbles for regular $J$ has negative dimension, which is empty. In particular, the moduli space of any somewhere injective disc passing through the open set $U$ union any configuration of multiply-covered and simple spheres can be excluded.\label{exclude_cor} \end{cor}

\begin{proof}[Proof of Corollary \ref{exclude_cor}] The Riemann-Roch theorem \cite[Appendix]{jholbk} implies the dimension of the manifold cut out by the regular $J$ is of negative dimension, specifically dimension $-2$. Lazzarini's result \cite{lazz2} implies any disc can be decomposed into simple discs and his paper \cite{lazz1} shows that any $J$-holomorphic disc contains a simple $J$-holomorphic disc. Thus if we had a nonempty configuration as in the statement of the corollary, we would have a non-constant map in the case of a simple disc union a simple sphere, by factoring through the multiple covers and taking one simple disc that goes through the sphere. But this is a contradiction, so there couldn't have been any such nonempty moduli spaces to begin with.
\end{proof}

\begin{proof}[Proof of Lemma \ref{Jreg2}] The proof is similar to the previous proof of existence of regular $J$, however we have an additional constraint which is the point of attachment between the disc and the sphere. We have dense sets of $J$ regular for each component (the disc and the sphere); the disc was described earlier and the sphere situation is done in \cite[Chapter 3]{jholbk}. This proof will involve checking that there is still a dense set of $J$ in the intersection of these two dense sets which interact well at the point where the disc and sphere intersect. 

Intersect the dense sets of regular $J$ for the sphere and disc separately. Consider $\mc{U} \subset \bigcup_{J \in \hat{\mathring{\mc{M}}}(Y,\dd Y)} \mc{M}(A_{\mb{D}}; J) \times \hat{\mathring{\mc{M}}}(A_{\mb{P}^1};J)$ where $A$ denotes the respective homology classes and $u_\mb{D}(\mb{D}) \nsubseteq u_{\mb{P}^1} (\mb{P}^1)$ (in contrast with the case of just spheres where we require that the images not be equal). Namely, $\mc{U}$ is the subset of pairs of maps where the disc image isn't contained in the $\mb{CP}^1$ image. We have the pointwise constraint that the sphere and disc are attached at a point. So we need transversality of the evaluation map $\mc{U} \to Y \times Y$. More specifically, the disc and sphere must intersect at a marked point which we place at 0 in the domain. We also fix a point on the boundary of the disc so there are no nontrivial automorphisms. 

Using Sard-Smale we can deduce that $\mc{U}$ is a manifold after an additional check at the intersection point. This is from \cite[Chapter 6]{jholbk}. In order to  show the evaluation map at 0 is transverse, we want the linearized evaluation map to be surjective. So we select any two tangent vectors in the codomain at $(0_{\mb{D}},0_{\mb{P}^1})$, and then construct two $\dot{J}$ supported in two disjoint small balls, one on each component, to ensure surjectivity. Each ball should not intersect the other component. This is possible because the Lagrangian boundary condition implies $\dot{J}$ is zero near the boundary and so if it's only supported on a small ball in the interior it is of this form; then we can extend each $\dot{J}$ by zero and simply add them. 

We now construct the $\dot{J}$. We follow \cite[\textsection 3.4]{jholbk}, then add the additional information from \cite[\textsection 6]{jholbk} for transversality of the evaluation map. Note that the reference considers the case of a sphere, and the argument works for the case of a disc because we have less to test due to more geometric constraints. Recall the construction of $\eta$ on page \pageref{reg argument} in {Step 3: Key Regularity argument.} We can construct such an $\eta$ in both cases of sphere or disc, and also require that $\xi$ now vanish on the point of intersection of the sphere and disc. We are still working on a single component, $\mb{P}^1$ or $\mb{D}$. This surjection tells us that we can find a $\xi$ pointing in a specified direction at a specified point \emph{and} tangent to its respective disc or sphere moduli space. Moreover one can do this in a small neighborhood around a specified point that is not the intersection point. Hence we can add vectors that work on \emph{different components} using bump functions to extend each one by zero.

The implicit function theorem puts the structure of a Banach submanifold on 
$$\bigcup_{J \in \mc{J}(Y,\dd Y)} \hat{\mathring{\mc{M}}}(A_{\mb{D}}; J) \times \hat{\mathring{\mc{M}}}(A_{\mb{P}^1};J)$$ 
We want to prove that $\mc{U}$ (the preimage under the evaluation map at 0 of the diagonal) is a submanifold. Elliptic bootstrapping and the implicit function theorem in Appendices of \cite{jholbk} apply because they are proven for totally real boundary conditions, such as Lagrangian boundary conditions. Since the linearized evaluation map is surjective on vectors as shown in the previous paragraphs, by Sard-Smale the subset of the universal space where maps respect the pointwise constraint is a manifold and we have existence of a dense set of regular $J$ for the disc and sphere so that the evaluation map at their intersection is transverse. Note also that we have omitted discussion of the asymptotic behavior at strip-like ends in our sketch of the functional analysis setup; the function spaces we consider and their topology need to be appropriately modified in the strip-like ends to enforce the asymptotic conditions, see \cite{seidel}. This concludes the proof of Lemma \ref{Jreg2}.
\end{proof}

Now that we have a regular $J$ for the disc union sphere configuration, we have that the moduli space of such configurations is a manifold. In particular, the manifold has dimension given by the Fredholm index, which is $< 0$ (a sphere bubble is codimension 2 and a disc with one marked boundary point is in a zero dimensional space). So the moduli space is empty. We have now excluded all types of bubbling behavior. So the moduli space $\hat{\mathring{\mc{M}}}_J((\mb{D},\dd\mb{D}), (Y,L); J, \beta, pt)$ of discs with one marked boundary point is already \emph{rigid} i.e.~dimension zero by the point constraint so we don't need to quotient by automorphisms, and it is also already compact as we've excluded limit behavior lying outside this moduli space. Hence we can write $\hat{\mathring{\mc{M}}}_J((\mb{D},\dd\mb{D}), (Y,L); J, \beta, pt)$ as $\mc{M}_J((\mb{D},\dd\mb{D}), (Y,L); J, \beta, pt)$.

\subsection{Quasi-invariance of the Fukaya category on regular choices}\label{section:q_invc_choices}
 
 We pre-face descriptions with ``quasi" when the descriptions hold on the cohomology level. 

\begin{lemma}  Let $J_1$ and $J_2$ be two regular almost complex structures. Then they define quasi-equivalent Fukaya $A_\infty$-categories, i.e.~isomorphic Donaldson-Fukaya categories.\end{lemma}

\begin{proof}[References for proof] We need to show that the Donaldson-Fukaya categories have isomorphic object and morphism spaces. We also need to show that there exists a functor between them, namely that it respects composition. The Lagrangians depend only on the symplectic form, so remain unchanged upon changing the almost complex structure. Likewise for the Floer complexes, which are generated by their intersection points. Note that Seidel in \cite[\textsection (10c)]{seidel} discusses upgrading this equivalence to the $A_\infty$-category, in particular for the product or composition map. See \cite[\textsection 8, \textsection (10c)]{seidel}. 

To show the morphism groups are isomorphic, we show that the differential $\dd$ is the same for each $J_1$ and $J_2$. This will follow from the use of a continuation map. (In general this argument only shows that the Floer complexes with the Floer differentials for $J_1$ and $J_2$ are quasi-isomorphic, hence have isomorphic cohomology.) This is the moduli space from solutions of a single PDE that is the usual Cauchy-Riemann equation however instead of $J$ we use $J_t$ where $J_t$ at time 0 is $J_1$ and $J_t$ at time 1 is $J_2$. This defines what the functor does on morphisms. In particular, this will require the existence of a path of regular $J$ in the space of all almost complex structures. This is discussed in \href{https://people.math.ethz.ch/~salamon/PREPRINTS/floer.pdf}{Lectures on Floer Homology} by D.~Salamon.
 
That existence of $J_t$ holds follows from a Sard-Smale argument, as in \cite[Theorem 3.1.8]{jholbk} in their second edition book. The difference here is that our Riemann surface has boundary (a disc with $k$ punctures on the boundary corresponding to the moduli space in defining the structure map $\mu^{k-1}$). So the base and fiber of the Banach bundle will be the same as in \cite[pg 55]{jholbk} however the spaces of almost complex structures will restrict to ones that are identically $J_0$ outside of the open set $U$ from Equation (\ref{defn: of_U_set}) around the origin and moduli spaces will consist of maps on discs instead of spheres. The Sard-Smale theorem and elliptic regularity proven in the Appendices of \cite{jholbk} already incorporate Lagrangian boundary conditions since they assume totally real boundary conditions.  
\end{proof}

\section{Computing the differential on $(Y,v_0)$}\label{section:differential}

The main result of this section is the computation of the differential.  We will use capital $M^k$ to denote structure maps on the total space of $Y$ and lowercase $\mu^k$ to denote structure maps on the torus fiber. We take two steps to reduce the calculation of $M^1$ to something that is computable: first in Lemma \ref{lemma: seidel_htpy_cob} we construct a cobordism between $M^1: CF(\ell_{i+1}, t_x) \to CF(\ell_i, t_x)$ and another disc count, then Corollary \ref{cor: use_c} allows us to calculate that differential in Lemma \ref{lemma:final_diffl_computation}, and lastly in Lemma \ref{leib_lemma} we will prove that $M^1: CF(\ell_{i+1},\ell_j) \to CF(\ell_i,\ell_j)$ can be computed from the data of $M^1: CF(\ell_{i+1}, t_x) \to CF(\ell_i, t_x)$ over all $x \in V$. This will finally allow us to prove the main theorem.

\subsection{Cobordism between generic choice and specific choice for computation}\label{section:cobord}

In this section we will discuss obstruction bundles, whose definition can be found in \cite[\textsection 7.2]{jholbk}. Consider the homotopy of Lagrangians in the total space, given by an isotopy of curves in the base of $v_0$, depicted in Figure \ref{seidel_htpy}. The left side, with a generic regular $J$, indicates the $M^1(p_{\infty,i+1}')$ we want to calculate. The right side with the standard $J_0=i$ indicates something we can compute. Note that since the set of $J$ regular for all configurations (discs and discs union spheres) is nonempty and dense, and the maps are $J_0$ sections away from a neighborhood of zero we can claim that $J_0$ is a limit of such $J$ by the denseness of the regular $J$. I.e.~we have a path of $J$'s limiting to $J_0$, so these $J$'s perturb $J_0$ near the zero fiber. Then with the construction of a cobordism, we can compute the left by computing the right.

\begin{figure}
\begin{center}
\includegraphics[scale=0.45]{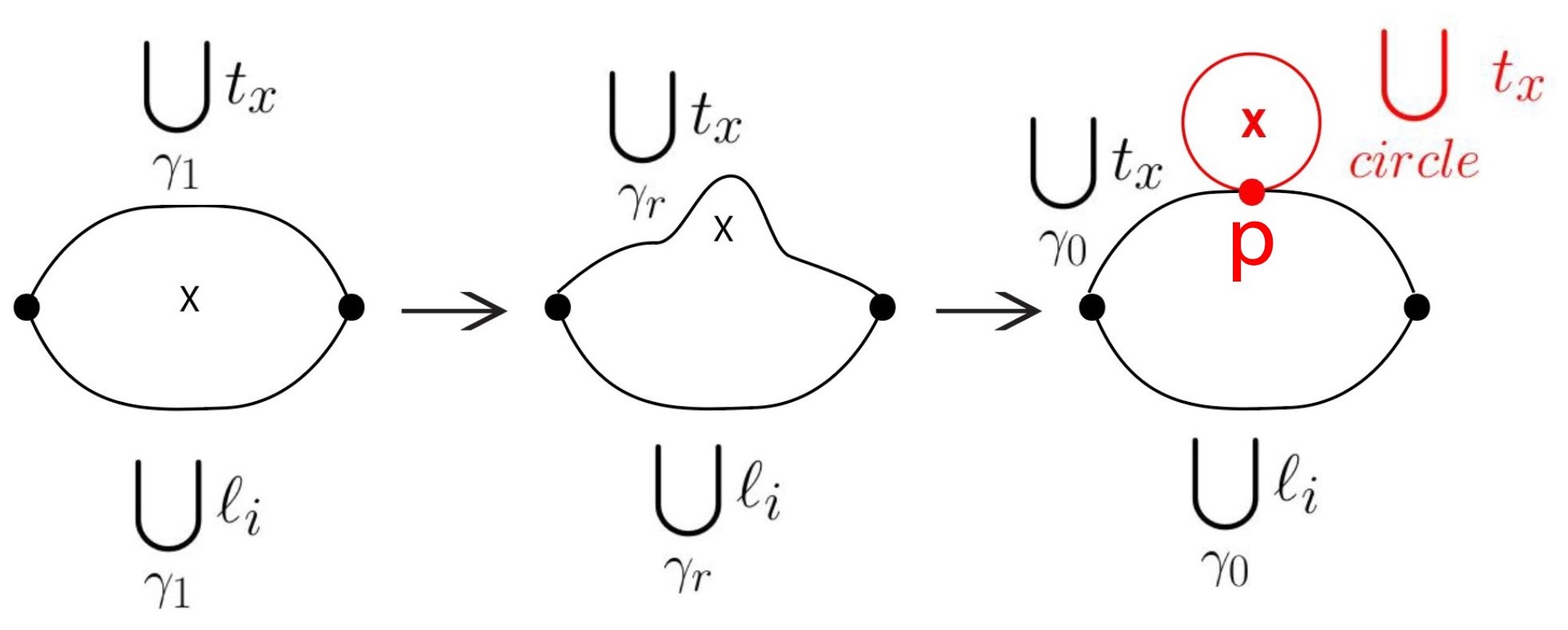}
\end{center}
\caption{The homotopy between $\dd$ (left) and the count of discs we compute (right)}
\label{seidel_htpy}
\end{figure}

\begin{figure}
\begin{center}
\includegraphics[scale=0.8]{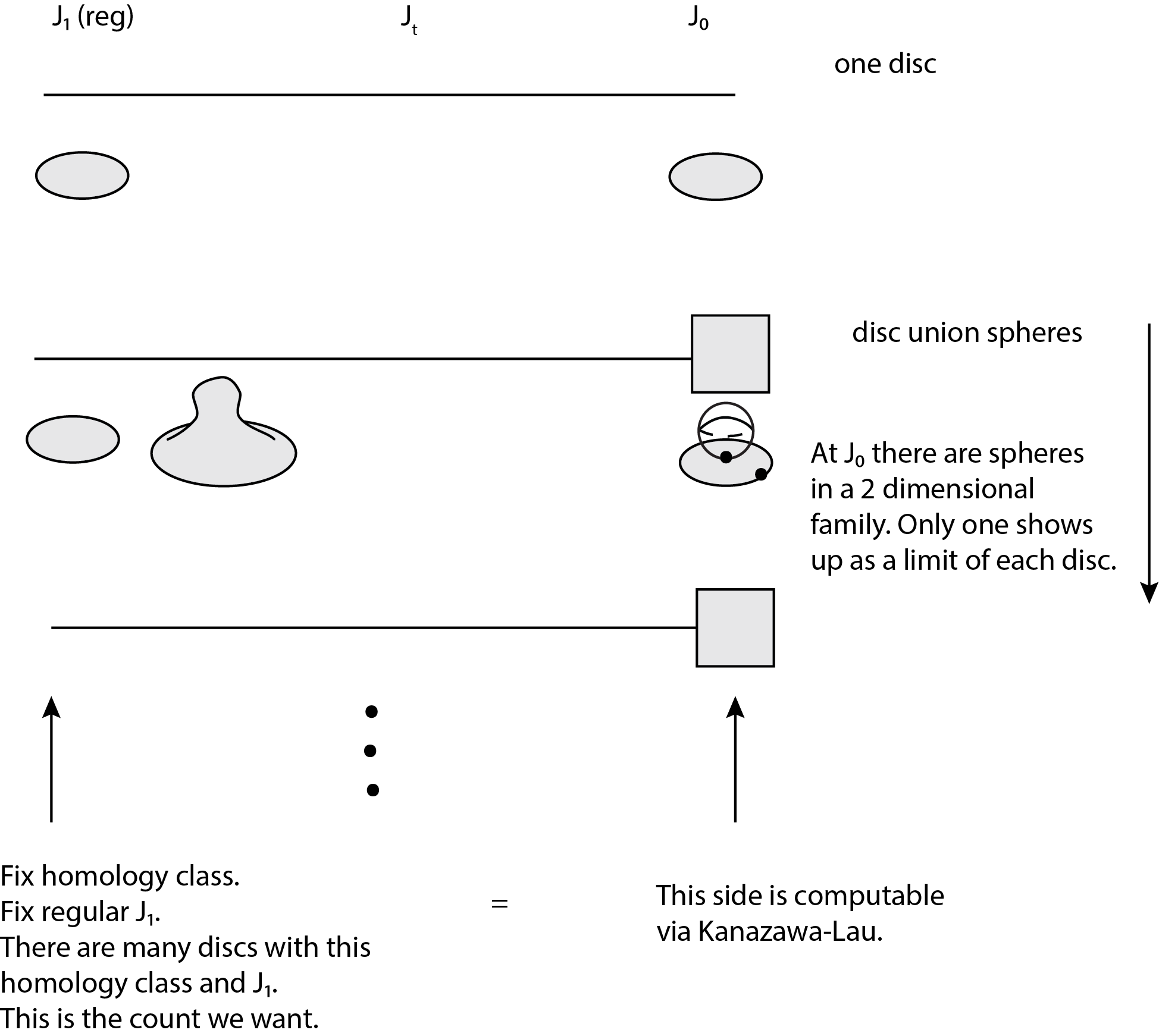}
\caption{Gromov compactification}
\label{hms_computation}
\end{center}
\end{figure}

There is existing theory for computing open Gromov-Witten invariants of $J_0$-holomorphic discs with boundary on a moment map fiber, one marked point, and passing through the singular fiber of $Y$ once (as is the case here because they are 1-1 in general and $v_0$-sections with $J_0$). This is the setting with $\bigcup_{circle} t_x$ that we see on the right side of Figure \ref{seidel_htpy} when we allow $J$ to vary to $J_0$ at the right end. However, $M^1$ as is counts bigons through the singular fiber with boundary over a bigon as in the left of Figure \ref{seidel_htpy}, instead of a circle. So following \cite[\textsection 17g]{seidel} we deform $M^1$ to $M^2(c p, \cdot)$ where $c$ counts $J_0$-holomorphic discs with boundary on $\bigcup_{circle} t_x$ and marked point $p$. This deformation constructs the homotopy. Note that in the book, he deforms the fibrations. However in this setting, the fibration stays the same, while the Lagrangian boundary conditions are deformed via an automorphism of $Y$ relative to the boundary. So we will need a gluing argument. 

See Figure \ref{hms_computation} for a pictorial depiction of the analytic and algebraic steps involved. In particular, we must abstractly perturb $\ol{\dd}_{J_0}$ in order to see the configurations we want to count on the right hand side as sitting in a moduli space. At the moment, the moduli space with $J_0$ has too many elements, a two-dimensional family of elements for each configuration that we only want to count once. Let $\bigcup_{\gamma_r} \ell$ denote the Lagrangian boundary condition at time $r$ depicted in Figure \ref{seidel_htpy}.

\begin{lemma}\label{lemma: seidel_htpy_cob} Choose $\beta_0$ and let $\beta_r={\phi_r}_*(\beta_0)$ where $\phi_r: (Y, \bigcup_{\gamma_1} \ell) \xrightarrow{\cong} (Y, \bigcup_{\gamma_r} \ell)$ is a diffeomorpism inducing an isomorphism ${\phi_r}_*$ on homology for $0< r \leq 1$. Then for a suitable family $J_r$ described in the proof,
$$\overline{\bigcup_{r \in (0,1]} \hat{\mathcal{M}}((Y,\bigcup_{\gamma_r} \ell);\beta_r; J_{r})/Aut}$$ 
where $Aut$ denotes strip-translation, has the structure of a compact topological 1-manifold. In particular, the signed count of its boundary is 0.
\end{lemma}

\begin{proof}
{\bf The $(0,1]$ manifold structure.} We use background from {\cite[Lecture 9]{wehr_reg}}, {\cite[Lecture 14]{wehr_reg0}}, and {\cite[\textsection 10.9]{jholbk}}. Define
$$\tilde{\mc{B}}^{k,p} :=\{(\phi_r^{-1} \circ u, r) \mid u: \mb{R} \times [0,1] \to (Y,\bigcup_{\gamma_r} \ell) \in W^{k,p}\}$$

We can classify the tangent space to a path in $\tilde{B}$ as one where the derivative of the path at the boundary is a vector that is a sum of a vector in the tangent space to the Lagrangian boundary and a vector corresponding to the flow of the isotopy $\phi_r$. As in \cite{jholbk}, this will give us a 1-manifold structure on the set of maps $u$ so that $\ol{\dd}_{J_r}(\phi_r^{-1} \circ u)=0$. Namely, given a Fredholm problem $E \to B$ describing the $r=1$ moduli space, we can pull back:
\begin{diagram}
(\phi_r \circ pr_1)^*E & \rTo & E\\
\dTo & & \dTo\\
\tilde{\mc{B}}^{k,p} & \rTo^{\phi_r \circ pr_1} & B
\end{diagram}

The bundle of the right vertical arrow admits a regular $J$ by Lemma \ref{geom_reg_lemma}. We can choose a path of perturbations given by sections of the obstruction bundle, since at $J_0$ the cokernel of $D_u$ is greater than zero but has constant dimension, see Claim \ref{spheres_not_regular}. See \cite[\textsection 7.2]{jholbk}. In particular, we choose the path so at $r=1$ the perturbation is 0 and at $r=0$ the perturbation $p_0$ is such that $(\ol{\dd}_{J_0} + p_0)^{-1}(0)$ counts the limiting curves on the right in Figure \ref{hms_computation}. This path gives us a section of the Fredholm problem in the left downward arrow. We then also abstractly perturb this section, with a perturbation which vanishes on the boundary of $\tilde{\mc{B}}^{k,p}$, analogous to \cite[Theorem 5.8]{hwz} for the polyfold setting. Continuous families of perturbations (CF-perturbation) are discussed in \cite[\textsection 7]{kur_strs}. Now we look at the zero set of this perturbed operator. Since there is no disc bubbling for each $r$, the boundary of the zero-set is just the $r=0$ end (and $r=1$ end once we compactify). So we obtain a 1-manifold structure on ${\bigcup_{r \in (0,1]} \hat{\mathcal{M}}((Y,\bigcup_{\gamma_r} \ell);\beta_r; J_{r})/Aut}$. Now we equip the $r=1$ end with a topology. 


{\bf Topological 1-manifold structure.} The next step will be to Gromov compactify at the $r=0$ end. Note that $\phi_r$ does not have a limit at $r=0$, as it becomes very degenerate and is not a diffeomorphism. So instead we consider elements as $u$ in this moduli space instead of $\phi_r^{-1}\circ u$. In order to preglue, which gives the topology at the $r=0$ end, we trivialize the normal bundle in a neighborhood of where we want to glue, and then interpolate linearly between the two maps. See \cite[lec 3, 1 hr]{wehr_reg}. Note that we take the gluing parameter to be $e^{-l}$ which goes to zero as the gluing length $l$ goes to infinity, where we have the configuration of two discs in the right side of Figure \ref{seidel_htpy}. (Note that even without trivializing, there are scaling functions on the normal bundle.  E.g.~the gluing parameter for the two discs in the base could be a cross ratio of four points around the belt that is getting pinched to a point.) 

{\bf Preglue the domains.} We remove a neighborhood of the puncture first. In the $(s,t)$ coordinates on strip-like ends, we glue $(s-l,t)$ to $(s,t)$. That is, we  place an amount $l$ in the $\mb{R}$ direction on one strip overlapping onto the other strip. The two parts separately give the $r=0$ case and the two parts glued together is the $r=\eps>0$ case. The embedding that gives the strip-like end embedding is as follows: 1) map $(-\infty,0] \times [0,1] \to (-\infty,0] \times [0,\pi]$ by $\cdot \pi$. Then map to the lower half of an annulus by $e^{-z}$, and then lastly to the right half of a disc with a puncture at 1 by $\frac{z+i}{z-i}$. See Figure \ref{str_le}. The reason why the preglued map is close to the $J_r$-holomorphic glued map one would obtain by Newton iteration is because by continuity $\ol{\dd}_{J_r}$ of the glued map is still small; if it were constant on the glued part then it would actually be holomorphic. We interpolate slowly so it is still close to constant.  


{\bf Preglue the maps.} We define a new preglued map $u^0 \#^{R} u^\infty$ on the preglued domain defined above, where $u^0$ and $u^\infty$ denote the two maps on discs at the $r=0$ end. We know how to interpolate in the base $v_0 \in \mb{C}$ coordinate using $(1-\rho)u^0 + \rho u^\infty -i$. Then we apply this same linear interpolation in the moment map coordinates $(\xi_1,\xi_2,\eta,\theta_1,\theta_2, \theta_\eta)$ to preglue maps $u$ to the total space. This choice of interpolation for the pregluing ensures that, when pregluing the disc bounded by $\bigcup_{circle} t_x$ and a strip with boundary on $\bigcup_{\gamma_0}t_x$ and $\bigcup_{\gamma_0} \ell_i$, the resulting preglued map has boundary on $\bigcup_{\gamma_r}\ell_i\cup t_x$ as in Figure \ref{seidel_htpy}. For in the fiber direction, the values of $(\xi_1,\xi_2)$ on the two components agree along the boundary and the interpolation preserves them. In the base direction, we can choose the family of paths $\gamma_r$ to be the family of paths obtained from $\gamma_0$ and the circle centered at the origin by our interpolation procedure.

{\bf Gromov compactness, maps limit to preglued maps.} With $J_0$ and no $\alpha$ spheres in the homology class $\beta_1$, we know the moduli space has one disc, by Claim \ref{one_disc_only}. All discs by themselves are regular for $J_0$. Then we look at a limit of $J_r$-holomorphic discs $u_r$ as $r$ goes to $0$, namely they solve the Cauchy-Riemann equation with $J_r$. After possibly passing to a subsequence, then $\lim_{r \to 0} u_r  =: u^0 \#^{R} u^\infty$ because of the exclusion of disc bubbling and strip breaking for a fixed Lagrangian by the geometry of $(Y,v_0)$. This is also discussed in \cite[Proposition 4.30]{cll}. Note that there are more pseudo-holomorphic discs with $J_1$ than $J_0$. The latter only has one in each homology class. This is because with $J_1$ some of the discs must converge to a disc union bubbles as $r \to 0$ for a path from $J_1$ to $J_0$. See Figure \ref{hms_computation}.

{\bf Deducing result.} We have constructed a cobordism between $\ol{\dd}_{J_1}^{-1}(0)$ and $(\ol{\dd}_{J_0}+ p_0)^{-1}(0)$ for an admissible perturbation $p_0$ given by e.g.~a section of the obstruction bundle (equivalently, a Kuranishi structure with one chart since the cokernel has constant dimension). So their counts are equal, by taking the signed boundary of this topological 1-manifold, which will be zero, and also the difference of these two counts. This completes the proof. 
\end{proof}

\begin{definition}\label{def: c} Let $c$ denote this open Gromov-Witten invariant of $J_0$-holomorphic curves with boundary on $\bigcup_{circle} t_x$ and marked point $p$, for the K\"ahler parameters $q_1=q_2=q_3=\tau$ of this paper.
\end{definition}

\begin{cor}\label{cor: use_c} Taking the boundary of the topological 1-manifold constructed in Lemma \ref{lemma: seidel_htpy_cob}, we may calculate  $M^1: CF(\ell_{i+1}, t_x) \to CF(\ell_i, t_x)$ by calculating $M^2(c p, \cdot)$ instead, where $c$ is as defined in Definition \ref{def: c}.

\end{cor}

\begin{proof}
This is a corollary of Lemma \ref{lemma: seidel_htpy_cob}, from which we deduced that $\#\ol{\dd}_{J_1}^{-1}(0)=\# (\ol{\dd}_{J_0}+ p_0)^{-1}(0) $. In particular, one can count the moduli space $(\ol{\dd}_{J_0}+ p_0)^{-1}(0)$ by taking the Euler number of a section of its obstruction bundle. This is done in \cite{kl} where they use the result of \cite{chan}, who shows that one can add an additional ray to the fan $\Sigma_{\tilde{Y}}$ to compactify a configuration of disc union sphere to a configuration of only spheres, and the Kuranishi chart on this closed Gromov-Witten invariant is isomorphic to that on the original open Gromov-Witten invariant. This now-closed curve count, given by the Euler number, can be counted by the Picard-Fuchs equation from algebraic geometry. This, in turn, can be done using the mirror theorem of Givental. See Figure \ref{kl_flowchart} for an outline of these steps with references. 
\end{proof}

\subsection{Count of discs regular for $J_0$}
In this section we take $J=J_0$ and consider moduli spaces of discs only, for which $J_0$ is regular. In other words, we only consider homology classes $\beta$ that arise from discs. The homology classes in $\pi_2(Y,L)$ that we consider cover a disc in the base of $v_0$ around 0, and pass through the central fiber in one point. Equivalently, taking their real part, they can be depicted in $\Delta_{\tilde{Y}}$ as a line from a fixed point on the interior of the polytope to a facet. Varying the facet allows one to enumerate all the homology classes, done in \cite{cho_oh}. This will finish the disc-only count since recall there is only one disk in each homology class by Claim \ref{one_disc_only}. 

\begin{theorem}\label{theorem:disc_count_s} Assume the set-up in the previous paragraph. Then the disc count equals the defining theta function $$s(x) =  \sum_{n \in \mb{Z}^2} x_1^{-n_1} x_2^{-n_2}\tau^{\frac{1}{2} n^t\left(\begin{matrix} 2 &1\\ 1& 2\end{matrix}\right)n}$$ 
up to a coordinate change.\end{theorem}

\begin{proof}
The count in this setting of discs in a toric variety is given in \cite{cho_oh}. Recall that we weight by $\tau^{-\int \omega}$ in the count. (Note that $\tau$ corresponds to the complex structure on the genus 2 curve, so it corresponds to the symplectic structure on the mirror. The complex structure on the mirror was in terms of $T$, see Figure \ref{coords in tiling}.) If $x_i$ are coordinates on the complex side on $V$ and $|x_i|:=\tau^{\xi_i}$ where the point in the polytope we measure from is $(a_1,a_2,a_3)=(\xi_1,\xi_2,\eta)$, then from \cite{cho_oh} the area of the disc intersecting the $(m_1,m_2)$ facet is $2\pi \left( \left<\underline{a},\nu(F_{m_1,m_2}\right>  + \alpha(F_{m_1,m_2})\right)$. 

The facet equations are determined from Equation (\ref{eq:polytope_def_Y}). In particular, $(\xi_1,\xi_2,\eta) = (0,0,0)$ is a point on the facet in the $\eta=0$ plane. Denote the facets by $F_{m_1,m_2}$ and let $\nu(F_{m_1,m_2})$ and $\alpha(F_{m_1,m_2})$ denote the normal and constant defining the plane the facet lies in. Recall that $\eta \geq \vp(\xi)$ where $\vp(\xi+\gamma) = \vp(\xi) - \kappa(\gamma) + \left<\la(\gamma), \xi\right>$. Suppose $(\xi_1,\xi_2,\eta) \in F_{m_1,m_2}$. We know from Claim (\ref{gamma_acts}) that $\Gamma_B$ acts on the moment map coordinates in the following way:
\begingroup \allowdisplaybreaks 
\begin{equation}
\begin{aligned}
& (-m_1\gamma' - m_2 \gamma'') \cdot (\xi_1,\xi_2,\eta) \\
& = ( \xi_1 - 2m_1-m_2, \xi_2-m_1-2m_2, \eta - \kappa(m_1\gamma' + m_2 \gamma'') - m \cdot \xi)\\
& =  ( \xi_1 - 2m_1-m_2, \xi_2-m_1-2m_2, \eta + m_1^2 + m_1m_2 +m_2^2 - m_1\xi_1-m_2\xi_2)
\end{aligned}
\end{equation} 
\endgroup
In particular, this point must be in $F_{0,0}$. We also know $(\xi_1,\xi_2,\eta) \in F_{m_1,m_2}$. Plugging each point into the equation of the corresponding facet, we find that:
 \begingroup \allowdisplaybreaks
\begin{align*}
 &\left<\nu(F_{m_1,m_2}),\left(\begin{matrix} \xi_1 \\ \xi_2 \\ \eta \end{matrix}\right) \right> + \alpha(F_{m_1,m_2}) =0\\
  \implies & \left<\nu(F_{0,0}),\left(\begin{matrix} \xi_1 - 2m_1-m_2 \\ \xi_2 -m_1-2m_2 \\ \eta + m_1^2 + m_1m_2 +m_2^2 - m_1\xi_1-m_2\xi_2 \end{matrix}\right) \right> + \alpha(F_{0,0})\stepcounter{equation}\tag{\theequation} \\
 &=  \eta + m_1^2 + m_1m_2 +m_2^2 - m_1\xi_1-m_2\xi_2 =0\\
  \implies & \nu(F_{m_1,m_2})=(-m_1,-m_2,1)^t, \;\; \alpha(F_{m_1,m_2})=m_1^2 + m_1m_2 +m_2^2
\end{align*}
\endgroup

So comparing the series from counting discs weighted by area for the differential, and that of the theta function, we find that
\begingroup \allowdisplaybreaks
\begin{align*}
\mbox{$\theta$-function $=$ } & \sum_{n \in \mb{Z}^2} x_1^{-n_1} x_2^{-n_2}\tau^{\frac{1}{2} n^t\left(\begin{matrix} 2 &1\\ 1& 2\end{matrix}\right)n}\stepcounter{equation}\tag{\theequation}\\
\mbox{disc count by area $=$ } & \tau^\eta \sum_{n \in \mb{Z}^2} |x_1|^{n_1}|x_2|^{n_2} \tau^{\frac{1}{2} n^t\left(\begin{matrix} 2 &1\\ 1& 2\end{matrix}\right)n}\\
\end{align*}
\endgroup
which agree up to a change of coordinates when we include local systems on the Lagrangians, (which removes the absolute value signs in the disc count). 
\end{proof}

\subsection{Count of spheres not regular for $J_0$}

\begin{claim}\label{spheres_not_regular} $J_0$ is not a regular almost complex structure for disc + sphere configurations.\end{claim}
\begin{proof} The standard $J_0$ is multiplication by $i$ in the toric coordinates. Recall that $H_2(v_0^{-1}(0))\neq 0$ and $v_0$ is holomorphic with respect to $J_0$, so submanifolds representing classes in $H_2$ of the central fiber are holomorphic. Then nonzero Dolbeault cohomology implies the spheres are not regular; this follows by the Riemann-Roch theorem and the fact that the cokernel of the $\ol{\dd}_{J_0}$ operator is Dolbeault cohomology. Since the cokernel is nonzero we find that $D_u$ is not surjective for maps $u$ arising from these holomorphic spheres.
\end{proof}




\begin{remark} \emph{Closed Gromov-Witten} theory counts spheres, for which we can use the algebraic geometry of stacks. Chapter 10 of \cite{cox_katz} gives the stack definition of moduli spaces. A reference for an introduction to stacks is \cite{stacks}.  \emph{Open Gromov-Witten} theory involves counting discs with a Lagrangian boundary condition, and this boundary condition is why we introduce analysis into definitions and use Fredholm problems to count the moduli spaces. 
\end{remark}

\begin{definition}\label{terminology}
We will denote sphere classes in the central $v_0$-fiber by $\alpha$ and the class of the disc passing through the divisor $D_{ij}$ corresponding to the $I:=(i,j)$ facet by $\beta_{ij}$. Let $n_{\beta_I + \alpha}$ denote the count of the following moduli space:
{$${\mc{M}}_{\beta_I + \alpha}(J_0):=\{(u,\ul{v}): (\mb{D}, (S^2)^k) \to Y \mid k \in \mb{N}\cup\{0\}, u(\dd \mb{D}) \subset \cup_{circ} t_x,$$
$$ [u \# \ul{v}] = \beta_I + \alpha, (u,\ul{v}) \in C^\infty, ev_{0}(u) = ev_0(v_1), \mu([D_I])=2, \ol{\dd}_{J_0}(u,\ul{v}) = 0\} \times \{p\}/\Aut(Y, p)$$}
\end{definition}

%
%

\begin{theorem}[Open mirror theorem proved in {\cite[Theorem 3.10]{kl}}]\label{open_mirror}
$\tilde{Y}$ is a toric Calabi-Yau manifold of infinite-type. Then
$$\sum_\alpha n_{\beta_I + \alpha} q^\alpha(\check{q}) = \exp(g_I(\check{q}))$$
where $q$ denotes the K\"ahler parameters, $\check{q}$ the complex parameters, $q(\check{q})$ the mirror map and
$$g_I(\check{q}) := \sum_d \frac{(-1)^{(D_I \cdot d)}(-(D_I \cdot d) - 1)!}{\prod_{I' \neq I} (D_{I'} \cdot d)!} \check{q}^d$$
where the summation over $d$ is taken over all $d \in H^{\mbox{\footnotesize eff}}_2(\tilde{Y}, \mb{Z})$ such that $-K_{\tilde{Y}} \cdot d = 0$, $D_{I} \cdot d < 0$, and $D_{I'} \cdot d \geq 0$ for all $I' \neq I$.
\end{theorem}

\begin{remark}\label{rem: choice of Kahl param} Note that here we consider only a one-parameter family of values of K\"ahler parameters, because we've fixed the symplectic form so that the three toric divisors $x=0$, $y=0$, and $z=0$ have symplectic area 1 (these form the ``banana manifold"). Namely, $q=\tau \in \mb{R}$ and $q^\alpha$, in our notation, is $\tau^{\omega(\alpha)}$.
\end{remark}

\begin{center} {\bfseries Flow chart details} \end{center}

Figure \ref{kl_flowchart} is a flow chart indicating the necessary background for understanding the sphere count in Kanazawa-Lau \cite{kl}. Note that they use $J=J_0$ as we are using here.

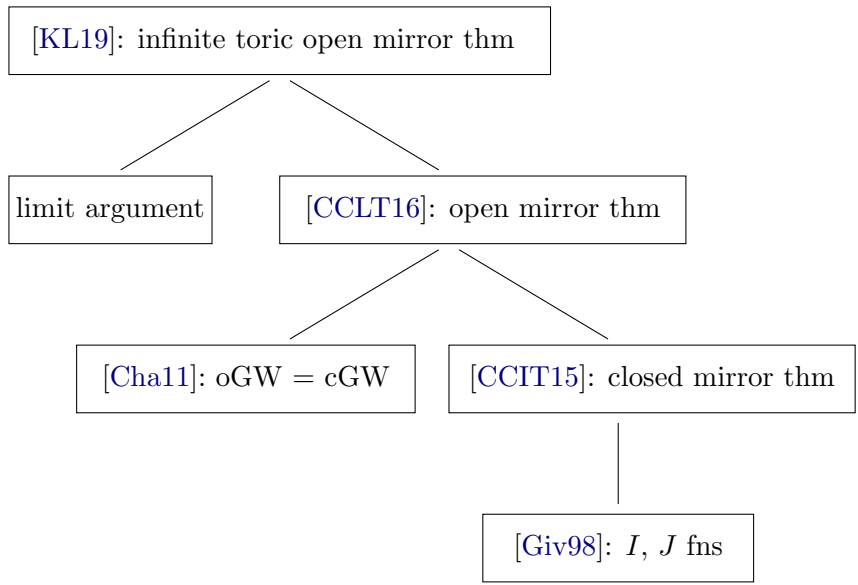
\begin{figure}[h]
\begin{center}
\begin{tikzpicture}[scale=0.9]
{\small \draw  (-7.5,5) rectangle (0.5,4);
\node at (-3.6,4.5) {\cite{kl}: infinite toric open mirror thm};
\node (v1) at (-3.5,4) {};
\node (v2) at (-6,2.5) {};
\draw  (v1) edge (v2);
\draw  (-7.5,2.5) rectangle (-4.5,1.5);
\node at (-6,2) {limit argument};
\node (v3) at (-1,2.5) {};
\draw  (-3.5,2.5) rectangle (2.5,1.5);
\node at (-0.5,2) {\cite{cclt}: open mirror thm};
\draw  (v1) edge (v3);
\node (v6) at (-1,0) {};
\node (v5) at (-3.5,0) {};
\node (v7) at (1.5,0) {};
\draw  (-6.5,0) rectangle (-1.5,-1);

\draw  (-1,0) rectangle (5,-1);
\node at (-4,-0.5) {{\cite{chan}:~oGW = cGW}};
\node at (2,-0.5) {\cite{ccit}: closed mirror thm};
\node (v4) at (-1,1.5) {};
\draw  (v4) edge (v5);

\draw  (v4) edge (v7);
\node (v8) at (1.5,-1) {};
\node (v9) at (1.5,-2.5) {};
\draw  (-0.5,-2.5) rectangle (3.5,-3.5);
\node at (1.5,-3) {\cite{givental}: $I$, $J$ fns};
\draw  (v8) edge (v9);}\end{tikzpicture}
\caption{Gromov-Witten theory background for mirror symmetry of toric varieties}
\label{kl_flowchart}
\end{center}
\end{figure}

{\bfseries Givental:~\cite{givental}.} The Picard-Fuchs differential equation describes the behavior of periods arising from Hodge structures on the complex side. Givental introduced the $I$ and $J$ functions, where $I$ computes solutions of the Picard-Fuchs equation and $J$ computes the Gromov-Witten invariants. He proved a relation between these two functions, i.e.~a mirror theorem.

{\bfseries Closed mirror theorem:~\cite{ccit}.} The closed mirror theorem relates the $I$ and $J$ function (defined in e.g.~\cite[\textsection 2.6.2]{cox_katz}) and builds on Givental's paper. Varying the complex moduli gives a variation of the Hodge structure on the complex manifold. The $J$ function on the symplectic manifold corresponds to the $I$ function on the complex manifold. These are functions of the K\"ahler and complex moduli, which are isomorphic. The mirror map goes between neighborhoods of a K\"ahler large limit point and a complex large limit point (maximally unipotent monodromy), see \cite[\textsection 6.3, p 151]{cox_katz}. Chapter 7 of \cite{cox_katz} defines GW invariants and Proposition 10.3.4 gives the relation between the $J$-function and the GW potential. 


In particular, \cite[Equation (10.4)]{cox_katz} gives the relation between differentials, intersection theory, and Gromov-Witten theory. The Picard-Fuchs equation \cite[\textsection 5.1.2]{cox_katz} is for complex moduli $\check{q}$ near maximally unipotent monodromy (denoted $y_k$ in \cite{cox_katz}). The K\"ahler moduli $q$ is denoted $q_k$ in the same reference.


Givental's mirror theorem for toric complete intersections is described in \cite[\textsection 11.2.5]{cox_katz}. Specific to the toric setting is the GKZ system, see \cite[\textsection 5.5]{cox_katz}.  An example of the mirror theorem is \cite[11.2.1.3]{cox_katz}. In the case of toric varieties, we have an equivariance under the toric action of the moduli spaces, discussed in \cite{ccit}. 

{\bfseries The result of \cite{chan}.} In the Fukaya category one would like to compute open GW invariants, i.e.~discs with Lagrangian boundary conditions, so we would like to be able to count these as well. There is a notion of ``capping off" introduced in \cite{chan} where, for a toric variety $X_\Sigma$, one adds an additional ray to the fan $\Sigma$ to define a partial compactification $\ol{X_\Sigma}$. This is done so the discs in the open GW count are ``capped off" to become spheres. See \cite[\textsection 6.1]{cclt} for a construction in the toric CY setting. The result in \cite{chan} implying that these open and closed GW invariants are equal is that they have isomorphic Kuranishi structures.

\begin{definition}[Kuranishi chart] A Kuranishi neighborhood of $p$ on moduli space $X$ is the following data:
\begin{itemize}[\textbullet]
\item $V_p$ smooth finite-dimensional \emph{manifold}, possibly with corners
\item $V_p \times E_p \to V_p$ is the \emph{obstruction bundle}, where $E_p$ is a finite-dimensional real vector space.
\item $\Gamma_p$ is a finite \emph{group} which acts smoothly and effectively (no non-trivial element acts trivially) on $V_p$, and $E_p$ linearly represents the group.
\item {Kuranishi map} $s_p$ is a smooth \emph{section} of $V_p \times E_p$ (smooth map $V_p \to \Gamma_p$), and is $\Gamma_p$ equivariant.
\item $\psi_p$ is a \emph{topological chart} which is a homeomorphism from the local model $s_p^{-1}(0)/\Gamma_p$ to a neighborhood of $p$ in $X$.
\item $V_p/\Gamma_p$ or $V_p$ may also be referred to as a \emph{Kuranishi neighborhood} (rather than the collection of all these pieces of data). 
\item $o_p$ is a point which the Kuranishi map sends to zero and the chart maps to $p$. 
\item For references in the literature on gluing such charts, see \cite[Fukaya, Tehrani]{new}.
\end{itemize}

%

\end{definition}
%

%


{\bfseries Kanazawa-Lau apply \cite{cclt} to the infinite toric setting.} In \cite{kl} there is a notion of taking a limit to arrive at the infinite toric case, which is our setting as well before we quotient by the $\Gamma_B$-action. They build on the open mirror theorem of \cite{cclt} and compute the sphere count as the coefficient of $1/z$ in the mirror map. This concludes the outline for the flow chart.

\begin{definition}\label{defn: sphere_count_C}
Define the sphere count
\begingroup \allowdisplaybreaks \begin{equation} \label{defn_C_spheres}
 C(x):=  \sum_\alpha n_{\beta_0+\alpha} \tau^{\omega(\alpha)}
\end{equation} \endgroup
where the $n_{\beta_0+\alpha}$ are defined in Definition \ref{terminology} and computed by Theorem \ref{open_mirror}, and $\beta_0$ denotes the disk class which projects to the moment polytope in the first two coordinates as a curve from the point 
$$A=(\log_\tau|x_1|, \log_\tau |x_2|, \mu_X(\bm{x},y))$$
(where $\mu_X$ is defined at the end of the paragraph in {\bf How to view $\eta$ as a moment map coordinate} on page \pageref{page_mu}) to the $(0,0)$-th facet.

\end{definition}

\subsection{Computation of intermediary differential with $t_x$}

\begin{lemma}\label{lemma:final_diffl_computation} We have the following set-up:
\begin{compactitem}[\textbullet]
\item Fix $A \in \Delta_{\tilde{Y}}$ and select a point $(x_1,x_2,y) \in X$ such that $A=(\log_\tau |x_1|, \log_\tau |x_2|, \mu_X(\bm{x}, y))$ (for $\mu_X$ defined in \cite[Equation (4.1)]{AAK}).
\item Let the Lagrangian boundary condition be $\bigcup_{circle}t_x$, lying over circle of fixed radius around the origin in the base of $v_0$.
\item Fix a point $pt_{constraint}$ on this Lagrangian.
\end{compactitem}

Then the $c$ from Corollary \ref{cor: use_c}, (which states that $M^1: CF(\ell_{i+1}, t_x) \to CF(\ell_i, t_x)$ equals $M^2(c p, \cdot)$), is given by
\begingroup \allowdisplaybreaks \begin{equation}
c = C(x)\cdot \left( \sum_\gamma \tau^{\omega(\gamma_* \beta_0)}\right)
\end{equation} \endgroup
where $C(x)$ is the sphere count from Definition \ref{defn: sphere_count_C}.
\end{lemma}

\begin{proof}
Define a $\Gamma_B \times (\mb{C}^*)^3$-action on $\tilde{Y}$ where $ (\gamma, c_{\gamma})$ acts by $c_{\gamma} \circ \gamma $ and $c_{\gamma}$ is complex multiplication. Specifically, $\gamma \in \Gamma_B$ sends $D_{ij}$ to some $D_{i'j'}$, and $c_{\gamma}$ is defined by requiring the point $\gamma_*(A)$ in the moment polytope to map back to $A$ via the $(\mb{C}^*)^3$ toric action (while fixing the divisors). Thus $\Gamma_B \times (\mb{C}^*)^3$ acts on moduli spaces of curves with a fixed Lagrangian boundary condition, varying homology classes, and a marked point $pt$, by post-composition. Fix a disc homology class $\beta_{ij}=[D_{ij}]$ in $H^2(\tilde{Y}, \cup_{circ} t_x)$. Then we have an isomorphism of moduli spaces:
{\footnotesize \begingroup \allowdisplaybreaks \begin{equation}\label{invariance_argument_final}
 \begin{aligned}
 & \{(u,pt), u: \mb{D} \to Y, pt\in \dd \mb{D} | u(\dd \mb{D}) \subset \bigcup_{circ} t_x, u(pt) = pt_{\tiny{\mbox{constraint}}}, [u] =\beta_{ij}, \ol{\dd}_{J_{reg}}(u) = 0\}  \cong \\
&\{(u,pt), u: \mb{D} \to Y, pt\in \dd \mb{D} | u(\dd \mb{D}) \subset \bigcup_{circ} t_x, u(pt) = c_{\gamma} \circ \gamma(pt_{\tiny{\mbox{constraint}}}), [u] =(c_{\gamma} \circ \gamma)_*\beta_{ij}, \ol{\dd}_{(c_{\gamma} \circ \gamma)^*J_{reg}}(u) = 0\}
\end{aligned}
\end{equation} \endgroup}
In particular, for a regular $J_{reg}$ as exists by Lemma \ref{geom_reg_lemma}, (so moduli spaces are manifolds) and introducing the point constraint (so they are zero dimensional), counting points in these moduli spaces produces an infinite series of discs and no sphere bubbles (because we excluded them). By denseness, we can choose $J_{reg}$ sufficiently close to $J_0$ so a limit of regular $J$'s limits to $J_0$. The disc will either converge to a disc or to a disc with a sphere bubble configuration as we saw in Figure \ref{hms_computation}. This count of discs for $J_{reg}$ is hence proportional to the differential in the Fukaya category. 

By \cite[Proposition 4.30]{cll}, we know that the only homology clases that can appear in the compactification are stable trees of the form $\beta_{ij} + \sum_i n_i \alpha_i$ for some integers $n_i$ and spheres $\alpha_i$.  Note that $(c_{\gamma} \circ \gamma)^*J_0=J_0$ since multiplication by scalars is a holomorphic map. The claim we want to prove is that the defined moduli spaces are isomorphic as we vary the homology classes. Applying these group actions should produce isomorphic moduli spaces, and we know then that the curve count for a particular homology class $D_{ij} + \alpha$  will be the same for all others and the counts will be the same so we can factor out the common factor. Namely the counts $n_{\beta + \alpha}$ do not depend on the disc class $\beta$ since there is a 1-1 bijection between moduli spaces of sphere configurations showing up with $D_{ij}$ and with any other $D_{i'j'}$, via the map $c_{\gamma}\circ \gamma$. That is because it has an inverse $(c_{\gamma}\circ \gamma)^{-1}$ given by multiplication by the inverse scalars. 

Suppose we write for an arbitrary disc and sphere configuration $\beta + \alpha' =: \gamma_*(\beta_0+\alpha)$ for fixed $\beta_0$ with a suitable $\gamma$ which then determines $\alpha$. Then we can denote all $n_{\beta + \alpha'}$ independent of $\beta$ and only depending on $\gamma, \alpha$ as $n_{\beta+\alpha'} = n_{\gamma_*(\beta_0+\alpha)}=n_{\beta_0+\alpha}$. The last equality is true as follows. We streamline notation below and use $\gamma$ to incorporate both actions of $\gamma$ and $c_{\gamma}$. We choose $J_{reg}$ to also be regular for the homology class $\gamma_*(\beta_I + \alpha)$ so $\gamma^* J_{reg}$ is regular for the class $\beta+\alpha$. Invariance on regular $J$ by a continuation map argument (see Section \ref{section:q_invc_choices}) then implies
$$\mc{M}(\beta_I +\alpha, J_{reg}) \cong \mc{M}(\gamma(\beta_I+\alpha),\gamma^*J_{reg})\cong \mc{M}(\gamma(\beta_I+\alpha),J_{reg})$$
\begin{equation}
\begin{aligned}
c&= \sum_{\beta, \alpha} n_{\beta + \alpha'} \tau^{\omega(\beta) + \omega(\alpha')}  = \sum_{\alpha, \gamma \in \Gamma_B} n_{\gamma_*(\beta_0 + \alpha)} \tau^{\omega(\gamma_*\beta_0) + \omega(\alpha)}\\
& = \sum_{\gamma \in \Gamma_B} \left(\sum_\alpha n_{\beta_0+\alpha} \tau^{\omega(\alpha)} \right) \cdot \tau^{\omega(\gamma_* \beta_0)} = \left(\sum_\alpha n_{\beta_0+\alpha} \tau^{\omega(\alpha)} \right) \cdot \left( \sum_{\gamma \in \Gamma_B} \tau^{\omega(\gamma_* \beta_0)}\right)
\end{aligned}
\end{equation}
The first factor is $C(x)$ which we can put in front, and the second is the multi-theta function described above in the computation of $J_0$-discs. Note that $n_{\beta_0} = 1$ by Claim \ref{one_disc_only}, hence $C(x)=1+$ (higher order terms) and is invertible.
\end{proof}

\subsection{Computation of the differential in general, using the Leibniz rule} Now we put everything together. Figure \ref{Leibniz} serves as a pictorial depiction of how we use the Leibniz rule to compute the differential $M^1$. As illustrated in Figure \ref{Leibniz}, $\hom_Y(L_i,L_j)$ decomposes into two hom groups on the fiber, $\hom_{\text{\tiny right}}(\ell_{i+1},\ell_j)[-1] \oplus \hom_{\text{\tiny left}}(\ell_i,\ell_j) = CF(\ell_{i+1},\ell_j)[-1] \oplus CF(\ell_i,\ell_j)$. In particular, $M^1$ will map from $\hom_{\text{\tiny right}}$ to $\hom_{\text{\tiny left}}$ in the Floer differential.

\begin{figure}[h]
\begin{center}
\includegraphics[scale=0.2]{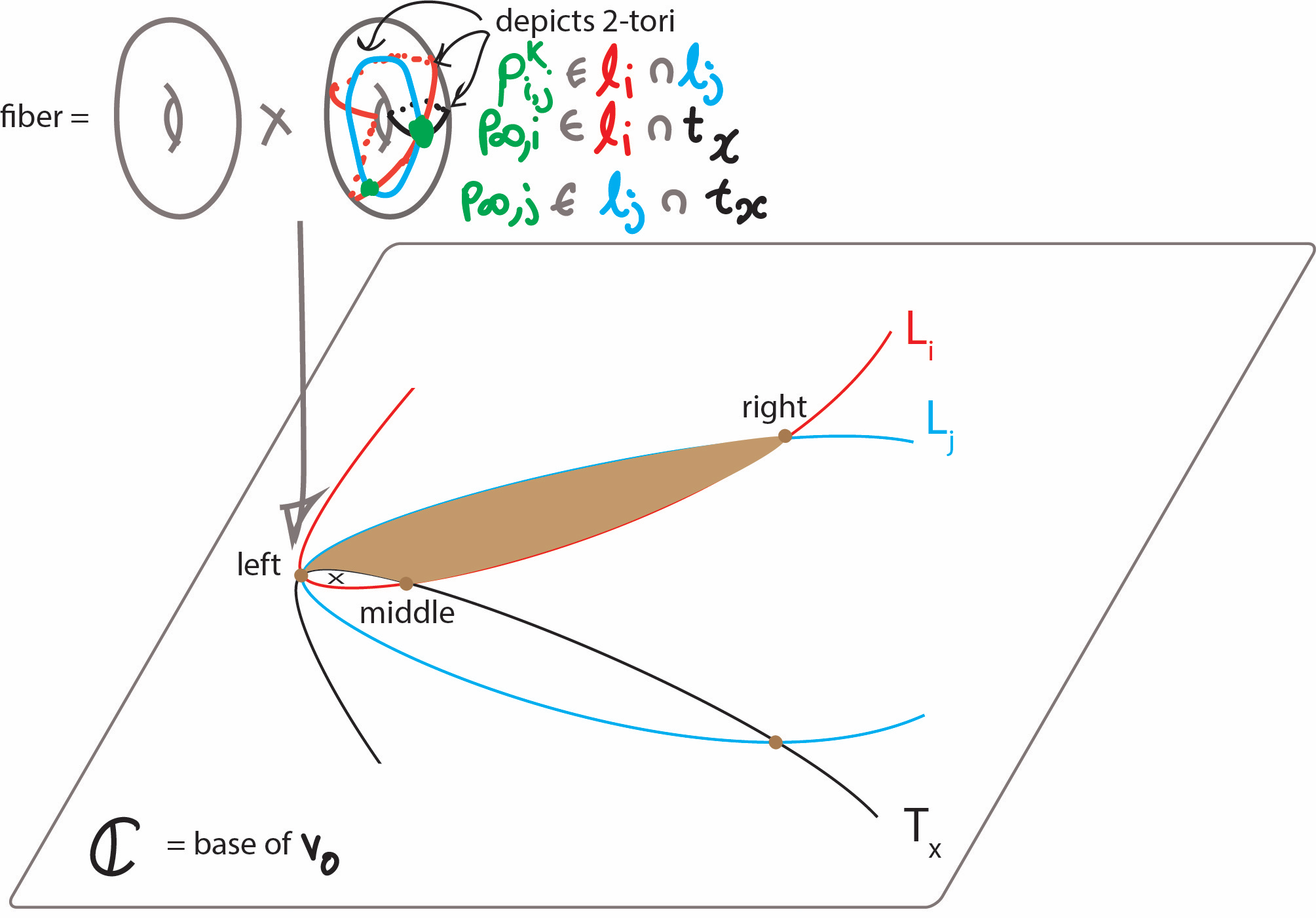}
\caption{Leibniz rule}
\label{Leibniz}
\end{center}
\end{figure}

The reason we use the Leibniz rule is because we would like to use $\bigcup_{circle} t_x$ as the Lagrangian boundary condition and not $\bigcup_{circle} \ell_j$, as this will allow us to count discs with boundary in the preimage of a moment map, as in \cite{cho_oh}. Note that 
$$M^1: CF(\ell_i, t_x) \to CF(\phi^H_{2\pi}(\ell_i), t_x) \xrightarrow{CF((\phi^H_{2\pi})^{-1})} CF(\ell_i, (\phi^H_{2\pi})^{-1}(t_x)) = CF(\ell_i, t_x) \cong \mb{C}$$ 
where $\phi^H_{2\pi}$ is the monodromy from Section \ref{mon}. We've used that applying the diffeomorphism $(\phi^H_{2\pi})^{-1}$ gives a bijection between intersection points, and  $t_x = \{(\xi_1,\xi_2,\theta_1,\theta_2)\})_{\theta_1,\theta_2 \in [0,2\pi)}$ is invariant under parallel transport as it only rotates angles. So we land in $CF(\ell_i, t_x)$, which has only one intersection point.

\begin{lemma}[Leibniz rule]\label{leib_lemma} \label{eval_at_pt} $M^1: CF(\ell_{i+1},\ell_j) \to CF(\ell_i,\ell_j)$ can be computed from the data of $M^1: CF(\ell_{i+1}, t_x) \to CF(\ell_i, t_x)$ over all $x \in V$. \end{lemma}

\begin{proof} First note that although $M^1$ is an automorphism of $CF(L_i, L_j) = CF(\ell_{i+1},\ell_j)[-1] \oplus CF(\ell_i,\ell_j)$, for degree reasons and geometric considerations the only nonzero contribution is from the the first summand to the second summand.

Let $p_{\infty, j}$ denote the intersection point of $\ell_j$ and $t_x$, and let $p^{k}_{i,j}$ denote the $k$th intersection point of $\ell_i \cap \ell_j$, where $0 \leq k < (i-j)^2$ as described in the definition of the Fukaya-Seidel category above. We now use the diagram in Equation (\ref{leibniz}). The crux of this argument relies on the Leibniz rule, one of the associativity relations in an $A_\infty$-category. The homs in the first diagram, Equation (\ref{leibniz}) are in the total space $Y$. We can reduce some calculations to the fiber, because each hom on the total space is a direct sum of Floer groups over two fibers, namely over the points of intersection of the curves given by their projection to the base. This is illustrated in Figure \ref{Leibniz}. 

Thus we can simplify the diagram of Equation (\ref{leibniz}) to involve homs only in the fiber, as done in the diagram of Equation (\ref{simplify}). Though the domains for $M^1$ and $M^2$ have more terms, we can just list the ones not mapping to zero. The rest map to zero because there are no bigons on two Lagrangians in the fiber, hence the only bigons must be over one in the base. Also note $t_x$ is invariant under the monodromy, since the latter affects only the angles and $t_x$ consists of all angles $(\theta_1,\theta_2)$ for a fixed $(\xi_1,\xi_2)$. Note that the subscripts in the simplified diagram indicate which fiber the hom intersection points lie in. We refer to them in Figure \ref{Leibniz} as left, middle, and right respectively, going from left to right. Lastly, $p_{\infty, j}$ refers to intersection points in the left fiber and $p'_{\infty,j}$ refers to intersection points in the middle fiber.
\begin{figure}[h]
\begingroup \allowdisplaybreaks \begin{equation}\label{leibniz}
\begin{aligned}
\begin{diagram}
\hom_Y(L_j, T_x) \otimes \hom_{Y}(L_i, L_j)  & \rTo^{M^2} & \hom_{Y}(L_i, T_x)\\
\dTo^{M^1 \otimes \bm{1} + \bm{1} \otimes M^1}  & & \dTo^{M^1}\\
\hom_Y(L_j, T_x)  \otimes \hom_{Y}(L_i, L_j) &  \rTo^{M^2} & \hom_Y(L_i, T_x)
\end{diagram}
\end{aligned}
\end{equation} \endgroup
\caption{Diagram illustrating Leibniz rule}
\end{figure}

\begin{figure}[h]
\begingroup \allowdisplaybreaks \begin{equation}\label{simplify}
\begin{aligned}
\begin{diagram}
\hom_{\text{\tiny{left}}}(\ell_j, t_x) \otimes \hom_{\text{\tiny right}}(\ell_{i+1},\ell_j) \ni p_{\infty, j} \otimes p^k_{i+1,j}  & \rTo^{M^2} & \hom_{\text{\tiny middle}}(\ell_{i+1},t_x)\\
\dTo^{M^1 \otimes \bm{1} + \bm{1} \otimes M^1}  & & \dTo^{M^1}\\
\hom_{\text{\tiny left}}(\ell_j, t_x)  \otimes \hom_{\text{\tiny left}}(\ell_i,\ell_j) &  \rTo^{M^2=\mu^2_{\text{\tiny left}}} & \hom_{\text{\tiny left}}(\ell_{i},t_x)
\end{diagram}
\end{aligned}
\end{equation} \endgroup
\caption{Simplified diagram on fibers}
\end{figure}

Differentiating the product $M^2$, the Leibniz rule implies (where $l=j-i$ and so $p^k_{i,j}$ indexed over $k$ can instead be written as $p_{e,l}$ indexed over $e$)
\begingroup \allowdisplaybreaks \begin{equation}
M^1(M^2(p_{\infty, j}, p_{e,l-1})) =  M^2(M^1(p_{\infty, j}), p_{e,l-1}) + M^2(p_{\infty, j}, M^1(p_{e,l-1})) = \mu_{\text{\tiny left}}^2(p_{\infty, j}, M^1(p_{e,l-1}))
\label{leib_eqn}
\end{equation} \endgroup
where the last step follows because $M^1(p_{\infty, j}) = 0$ as $p_{\infty, j}$ is of degree 0 and there is nothing in degree 1 at the other intersection points of the two Lagrangians. Note that $p_{\infty, i}$ is of degree $-1$.

Here are the steps which allow us to arrive at the conclusion of the lemma. 

\begin{enumerate}[(i)]
\item RHS of Equation (\ref{leib_eqn}): $M^1(p_{e,l-1}) := \sum_{\tilde{e}} \alpha_{\tilde{e}}(e,l)  p_{\tilde{e},l}$ where $\alpha_{\tilde{e}}(e,l)$ is the count of bigons between $p_{e,l-1}$ and $p_{\tilde{e},l}$ weighted by their area. (In particular, since $\tau$ was the complex parameter on the $B$-side, it is the Novikov parameter here on the symplectic side.) This is the differential we are looking for. It goes from the right fiber to the left fiber.

\item Then: $\mu_{\text{\tiny left}}^2(p_{\infty, j}, M^1(p_{{e},l-1})) =\sum_{\tilde{e}} \alpha_{\tilde{e}}(e,l) \mu_{\text{\tiny left}}^2(p_{\infty, j},  p_{\tilde{e},l})=\left(\sum_{\tilde{e}} \alpha_{\tilde{e}}(e,l) n_{\tilde{e}}(l)\right) p_{\infty, i}$ where $n_{\tilde{e}}(l)$ is the weighted count of triangles (in the left fiber, for degree reasons) with vertices at $p_{\infty, i}, p_{\tilde{e},l}$, and $p_{\infty, j}$. Since it's in the fiber, we can compute $n_{\tilde{e}}(l)$ directly as in Lemma \ref{functor_ok}, as follows. 

\begin{figure}[h]
\begin{tikzpicture}[every node/.style={inner sep=0,outer sep=0, scale = 0.8}]

\draw (-2.5,0.5) node (v1) {} -- (-0.5,-1.5) -- (1.5,-1.5) -- (v1) -- cycle;
\node at (-1,-2) {$p^{}_{\tilde{e},l}$};
\node at (2,-2) {$p_{\infty, i}$};
\node at (-3,0.6) {$p_{\infty, j}$};
\node at (0.5,-2) {$\ell_i$};
\node at (0,0) {$t_x$};
\node at (-2,-1) {$\ell_j$};
\end{tikzpicture}
    \caption{A triangle in $V^\vee$ contributing to $\mu^2$, viewed in $\xi_1,\xi_2$ plane in the universal cover $\mb{R}^4$}
    \label{fig: triangle}
\end{figure}

\noindent The three vertices of the triangle are on lifts of $\ell_i \cap t_x\ni p_{\infty, i}$, $\ell_j \cap t_x \ni p_{\infty, j}$, and $\ell_i \cap \ell_j \ni p_{\tilde{e},l}$. Translate $p_{\tilde{e},l}$ so that it lies in the fundamental domain for the $\Gamma_B$-action. 
$$p_{\tilde{e},l} = \left( \frac{\gamma_{i \cap j}}{j-i}, -i\la\left(\frac{\gamma_{i \cap j}}{j-i} \right) \right)$$

\noindent Along $t_x$, the $\xi$ stay constant at some translate of $a$, say $a+\gamma$. Thus the sum of the $\xi$ going around the triangle equaling zero implies the amount we add to the $\xi$ coordinate, moving along the $\ell_i$ and $\ell_j$ directions from $p_{\tilde{e},l}$, are equal. Thus:
\begin{equation}
    \begin{aligned}
    a + \gamma & = \frac{\gamma_{j \cap i}}{j-i} + \xi\\
        p_{\infty, i} & = \left( \frac{\gamma_{i \cap j}}{j-i}, -i\la\left(\frac{\gamma_{i \cap j}}{j-i} \right) \right) + (\xi, -i\la(\xi))\\
        p_{\infty, j} & = \left( \frac{\gamma_{i \cap j}}{j-i}, -i\la\left(\frac{\gamma_{i \cap j}}{j-i} \right) \right) + (\xi, -j\la(\xi))
    \end{aligned}
\end{equation}

\noindent Thus the triangle is half of the parallelogram spanned by:
\begin{equation}
        \left<a + \gamma - \frac{\gamma_{i \cap j}}{j-i} =:\xi, -i\la(\xi)\right> \mbox{ and }\left< \xi, -j\la(\xi )\right>
\end{equation}

\noindent This is in the 2-plane spanned by $\left<\xi,0 \right>$ and $\left<0, \la(\xi)\right>$. As before, the area of the triangle is half of $\left<\xi, \la(\xi)\right>$ times 
$$\det \left( \begin{matrix} 1 & -i \\1 & -j \end{matrix}\right)= (j-i).$$

\noindent That is
\begin{equation}\label{eq:n_s_count_equal}
    n_{\tilde{e}}(l) \equiv n_{\tilde{e}}(l, x) = \sum_{\gamma \in \Gamma_B} \tau^{-(j-i)\kappa\left( \log_\tau |x| + \gamma - \frac{\gamma_{i \cap j}}{j-i} \right)}=s_{\tilde{e},l}[\log_\tau |x|](1)
\end{equation}
 
\noindent where $[\log_\tau |x|]$ denotes a shift in the $\Gamma_B$-lattice by $\xi=\log_\tau |x|$.
 
\item LHS of Equation (\ref{leib_eqn}): $M^2(p_{\infty, j}, p_{e,l-1}) = n_e(l-1)\cdot p'_{\infty, i+1}$, noting that the fibration is trivializable in the beige region of Figure \ref{Leibniz}. So this $M^2$ inputs points in the left and right fibers, and outputs points in the middle fiber.

\item\label{item_where_count_used} Then going from the middle fiber to the left fiber counts bigons over the region we deformed using Seidel's homotopy. Namely, recall $M^1(p'_{\infty, i+1})= C(x)s(x) \cdot p_{\infty,i}$ by Corollary \ref{cor: use_c} (deforming the differential by a cobordism to a disk count) and Lemma \ref{lemma:final_diffl_computation} (the disk count).
\end{enumerate}

Thus Equation (\ref{leib_eqn}) becomes
$$C(x)s(x) \cdot n_e(l-1) \cdot p_{\infty,i} = M^2(M^1(p_{e,l-1},p_{\infty, j}))= \left(\sum_{\tilde{e}} \alpha_{\tilde{e}}(e,l) n_{\tilde{e}}(l)\right) p_{\infty, i}$$
so comparing the coefficients on $p_{\infty, i}$ we find that
\begin{equation}\label{eq: t_x to l_i}
    C(x)s(x) \cdot n_e(l-1) =\left(\sum_{\tilde{e}} \alpha_{\tilde{e}}(e,l)n_{\tilde{e}}(l)\right)
\end{equation}
In conclusion, we see that the differential is given by
\begin{equation}
    \begin{aligned}
    M^1: CF(\ell_{i+1},\ell_j) & \to CF(\ell_i,\ell_j)\\
    M^1(p_{e,l-1})& =\sum_{\tilde{e}} \alpha_{\tilde{e}} p_{\tilde{e},l}
    \end{aligned}
\end{equation}
where $\alpha_{\tilde{e}}$ is calculated by Equation (\ref{eq: t_x to l_i}), which also indicates how it depends on $x$. This proves the lemma.
\end{proof}

Now finally we can compute the differential explicitly, which was the goal of this chapter.

\begin{lemma}\label{lem: the real final diffl calcn} Consider the following two maps:
\begin{equation}
    \begin{aligned}
    \Hom_V(\mc{L}^{i+1},\mc{L}^j)& \xrightarrow{s \otimes} \Hom_V(\mc{L}^i,\mc{L}^j)\\
    \Hom_{V^\vee}(\ell_{i+1}, \ell_j) & \xrightarrow{\dd} \Hom_{V^\vee}(\ell_i, \ell_j) 
\end{aligned}
\end{equation}
Then $\dd = C(x)\tau^{\kappa(\log_\tau x)} \cdot s\otimes$ as linear maps on vector spaces, where recall 
$$s(x)=\sum_{n \in \mb{Z}^2} x_1^{-n_1} x_2^{-n_2}\tau^{\frac{1}{2} n^t\left(\begin{matrix} 2 &1\\ 1& 2\end{matrix}\right)n}$$ 
is the theta function defining the genus 2 curve in the abelian variety $V$, and $C(x)$ was defined in Definition \ref{defn: sphere_count_C}.
\end{lemma}

\begin{proof}
Let $\mc{B}$ denote the basis $s_{e,l}$ and $p_{e,l}$ from \textsection \ref{section:fff_T4}, for some choice of $(i,j)$. We know from Equations (\ref{eq:n_s_count_equal}) and (\ref{eq: t_x to l_i}) is that
$n_{\tilde{e}}(l, x) =s_{\tilde{e},l}[\log_\tau |x|](1)$ and for all $x \in (\mb{C}^*)^2/\Gamma_B$,
\begin{equation}\label{eq:known_leib_trick}
C(x)s(x) \cdot s_{l-1,{e}}[\log_\tau |x|](1) =\left(\sum_{\tilde{e}} \alpha_{\tilde{e}}(e,l)s_{\tilde{e},l}[\log_\tau |x|](1)\right).    
\end{equation}
We can re-arrange some terms as follows:
    \begin{allowdisplaybreaks}
    \begin{align*}
        s_{e,l}[\log_\tau |x|](1) & = \sum_{\gamma \in \Gamma_B} \tau^{-l\kappa(\gamma + \log_\tau |x| - \gamma_{e,l}/l)}\\
        & = \sum_{\gamma \in \Gamma_B} \tau^{-l\kappa(\gamma- \gamma_{e,l}/l) -l \kappa(\log_\tau |x|) +l\left<\log_\tau |x|, \la(\gamma - \gamma_{e,l}/l) \right>}\\
        &= \tau^{-l \kappa(\log_\tau |x|)} \sum_{-\gamma \in \Gamma_B}\tau^{-l \kappa(-\gamma + \gamma_{e,l}/l)} |x|^{-l \la(-\gamma + \gamma_{e,l}/l)}\stepcounter{equation}\tag{\theequation}\\
        &=\tau^{-l \kappa(\log_\tau |x|)} \sum_{\gamma \in \Gamma_B}\tau^{-l \kappa(\gamma + \gamma_{e,l}/l)} |x|^{-l \la(\gamma + \gamma_{e,l}/l)}\\
        &=\tau^{-l \kappa(\log_\tau |x|)} s_{e,l}(|x|)
    \end{align*}
        \end{allowdisplaybreaks}
Hence from Equation (\ref{eq:theta_prod}) for $x$ positive real we obtain  (again we remove this restriction by adding local systems on the Lagrangians):
\begin{equation}
    \begin{aligned}
        s(x)\cdot s_{l-1,e}[\log_\tau x](1) & = s_{1,1}(x) \cdot \tau^{-(l-1) \kappa(\log_\tau x)} s_{l-1,e}(x)\\
        &= \tau^{-(l-1) \kappa(\log_\tau x)} \sum_{\tilde{e}} \left( \sum_\eta \tau^{-\left( 1 + \frac{1}{l-1}\right) \kappa \left(\eta - \frac{(l-1)\gamma_{\tilde{e}}}{l}\right)}\right)s_{\tilde{e},l}.
    \end{aligned}
\end{equation}
Then multiplying by $C(x)$ and using Equation (\ref{eq:known_leib_trick}) the LHS also equals
\begin{equation}
        \sum_{\tilde{e}} \alpha_{\tilde{e}}(e,l) s_{\tilde{e},l}[\log_\tau x](1)= \tau^{-l \kappa(\log_\tau x)}\sum_{\tilde{e}} \alpha_{\tilde{e}}(e,l)s_{\tilde{e},l}.
\end{equation}
Thus 
\begin{equation}
    C(x)\tau^{-(l-1) \kappa(\log_\tau x)} \sum_{\tilde{e}} \left( \sum_\eta \tau^{-\left( 1 + \frac{1}{l-1}\right) \kappa \left(\eta - \frac{(l-1)\gamma_{\tilde{e}}}{l}\right)}\right)s_{\tilde{e},l}=\tau^{-l \kappa(\log_\tau x)}\sum_{\tilde{e}} \alpha_{\tilde{e}}(e,l)s_{\tilde{e},l}
\end{equation}
and equating coefficients on the B-side basis gives
\begin{equation}
    C(x)\tau^{-(l-1) \kappa(\log_\tau x)} \left( \sum_\eta \tau^{-\left( 1 + \frac{1}{l-1}\right) \kappa \left(\eta - \frac{(l-1)\gamma_{\tilde{e}}}{l}\right)}\right)=\tau^{-l \kappa(\log_\tau x)} \alpha_{\tilde{e}}(e,l).
\end{equation}
This implies what we originally wanted:
\begin{allowdisplaybreaks}
    \begin{align*}
        \dd(p_{e,l}) = \sum_{\tilde{e}}\alpha_{\tilde{e}}(e,l)p_{\tilde{e},l}&=C(x) \tau^{\kappa(\log_\tau x)}\sum_{\tilde{e}}\left( \sum_\eta \tau^{-\left( 1 + \frac{1}{l-1}\right) \kappa \left(\eta - \frac{(l-1)\gamma_{\tilde{e}}}{l}\right)}\right)p_{\tilde{e},l}\stepcounter{equation}\tag{\theequation}\label{eq:diffl_count}\\
        s\otimes s_{e,l}& = \sum_{\tilde{e}}\left( \sum_\eta \tau^{-\left( 1 + \frac{1}{l-1}\right) \kappa \left(\eta - \frac{(l-1)\gamma_{\tilde{e}}}{l}\right)}\right)s_{\tilde{e},l}
    \end{align*}
\end{allowdisplaybreaks}
which completes the proof that $\dd = C(x)\tau^{\kappa(\log_\tau x)} \cdot s(x)\otimes $.
\end{proof}

\section{Right arrow of main theorem: fully-faithful embedding $D^b_{\mc{L}}Coh(H) \into H^0FS(Y,v_0)$}\label{section: proof}


In this section we prove that the right vertical arrow of Theorem \ref{theorem: me} is a fully faithful embedding, namely that the arrow is indeed a functor and that the morphism groups between objects and images of those objects are isomorphic.


On objects we map $\mc{L}|_H^{\otimes i} \mapsto L_i$. If $\phi^H_{2\pi}$ is the monodromy of the symplectic fibration $v_0:Y \to \mb{C}$ around the origin, then the symmetry of our definition of $\omega$ ensures that $\phi^H_{2\pi}(\ell_i)$ is Hamiltonian isotopic to $\ell_{i+1}$ by Lemma \ref{lem:ham_isotop}. Since Floer cohomology is invariant under Hamiltonian isotopy, we can consider linear Lagrangians in the fibers. This allows us to obtain the bottom row of the following diagram, whenever $j \geq i+2$. When $j< i+2$ there are also Ext groups to consider and we get a long exact sequence instead, namely the last horizontal map is not surjective anymore.
\begin{figure}[h]
\begin{diagram}
& Hom(\mc{L}^{i+1},\mc{L}^{j}) &  \xrightarrow{\otimes s} & Hom(\mc{L}^i,\mc{L}^j) & \to & Hom(\mc{L}^i,\mc{L}^{j}\otimes \iota_* \mc{O}_H) &  \to & 0\\
& \dTo_{\cong}^{\cdot \left(\tau^{\kappa(\log_\tau x)}C(x)\right)^{-1}} & & \dTo^{\cong} & & \dTo \\
& CF(\ell_{i+1},\ell_j) & \xrightarrow{\dd} & CF(\ell_i,\ell_j) & \to & HF(L_i,L_j) & \to & 0
\end{diagram}
\caption{Proof of Main Theorem}
\label{proof_diag}
\end{figure}

\noindent Ext groups here are computed from injective resolutions:

\begin{diagram}
( 0 & \rTo & \mc{L}^{-1} & \rTo & \mc{O}_V & \rTo & \mc{O}_H & \rTo & 0 ) \otimes \mc{L}^{j-i}\\
0 & \rTo & \mc{L}^{j-i-1} & \rTo & \mc{L}^{j-i} & \rTo & \mc{L}_H^{j-i} & \rTo & 0
\end{diagram}
so taking the cohomology long exact sequence we obtain:
\begin{equation}
    \begin{aligned}
    0 &\to H^0(\mc{L}^{j-i-1}) \to H^0(\mc{L}^{j-i}) \to H^0(\mc{L}^{j-i}\mid_H) \to H^1(\mc{L}^{j-i-1}) \to H^1(\mc{L}^{j-i}) \\
    &\to H^1(\mc{L}^{j-i}\mid_H) \to H^2(\mc{L}^{j-i-1}) \to H^2(\mc{L}^{j-i}) \to H^2(\mc{L}^{j-i}\mid_H)
    \end{aligned}
\end{equation}

The crux of the argument is that the left-side square in the diagram in Figure \ref{proof_diag} commutes, which then implies that the rightmost vertical arrow is an isomorphism as well. This follows from Lemma \ref{lem: the real final diffl calcn}.




 More precisely, we've shown  in Lemma \ref{lem: the real final diffl calcn} that under the chosen isomorphisms of Lemma \ref{functor_ok}, the Floer differential $\dd$ agrees with multiplication by $s \in H^0(V,\mc{L})$ up to the multiplicative factor $\tau^{\kappa(\log_\tau x)}C(x)$, where $C(x)$ was the open Gromov-Witten invariant in \cite{kl} for the particular choice of K\"ahler parameter in Remark \ref{rem: choice of Kahl param}. So scaling the first arrow on morphisms by $\left(\tau^{\kappa(\log_\tau x)}C\right)^{-1}$ gives a commutative diagram on the left, which we can do since $C(x)=1+\ldots$. 

Furthermore, this map on objects induces a functor because it respects composition. This is because the product structures on the groups in the right-hand vertical arrow of the diagram in Figure \ref{proof_diag} are those naturally induced on the quotients, and the left two vertical maps are functorial by Lemma \ref{functor_ok}. So the induced isomorphisms on the cokernels of the horizontal maps are also functorial, and composition on the complex side versus the Floer product on $HF$ match under the constructed isomorphisms. 

Hence this provides the desired isomorphism between the morphisms groups for the functor $D^b_{\mc{L}}Coh(H) \to H^0(FS(Y,v_0))$ and completes the proof of the main Theorem \ref{theorem: me}.

\section{Appendix A: Enough space to bound derivatives}
\begin{lemma}[Estimates on bump function derivatives] The definition corresponding to Figure \ref{delineation_pic} provides enough space to make $\log_T$-derivatives as small as needed.
\end{lemma}

\begin{proof} {\bf Region VII} contains the point in its center where $r_x=r_y=r_z=T^{l/3}$ and its boundary circle is $d=T^{l/4}$ for whichever region $d$ we are using. So $\Delta \log_T r_x = O(l)$ is of order $l$.

{\bf Region I} is delineated by a rectangle in $(d_I,\theta_I)$ coordinates of length $2T^{-l/p}$ by $T^{-l/p}$ for a variable $p$ which we constrain below. Namely, 
\begin{equation}
    \mbox{Region I} := \{(d_I, \theta_I) \mid (d_I, \theta_I) \in [T^{l/4}, T^{l/4 - l/p}] \times [-T^{l/p}, T^{l/p}]\}
\end{equation}
Thus $\Delta \log_T d_I = O(l/p)$. The angular coordinate requires a bit more to analyze since it goes to zero when $r_y=r_z$.

The curves delineating region II are $\theta_{II}$ constant for the radially outward lines, and $d_{IIA}$ or $d_{IIB}$ constant along angular curves. Note that $d_{IIA}$ constant here is approximately $r_x$ constant.

Label the points of Figure \ref{delineation_pic} as follows.

{\bf Region II} Radial lines are $d_{IIx}$ constant, for $x \in \{A,B,C\}$. Two curves around the origin through $P_1, Q_1$ and $P_2, Q_2$ respectively are $\theta_{II}$ constant.

\begin{itemize}
\item $P_1=B \approx ({l/4}, {3l/8-l/p}, {3l/8 + l/p})$

\item $P_2 \approx ({l/4 - l/p}, {3l/8 - 2l/p}, {3l/8 + 3l/p})$. 
This follows because the sliver from $P_1$ to $P_2$ is $\theta_{II} = \log_T r_y - \log_T r_x$ constant, while $P_2$ is obtained by moving $E$ up along constant $d_{I}   \approx (Tr_x)^2$ so $r_x(P_2) \approx T^{l/4 - l/p}$. Then constant $\theta_{II}$ implies:
\begin{align*}
\theta_{II}(P_1)&\approx 3l/8 - l/p - l/4 = l/8 - l/p\\
=\theta_{II}(P_2) &\approx \log_T r_y - (l/4 - l/p) \\
 \implies \log_T r_y& \approx 3l/8 - 2l/p
 \end{align*}

\item $Q_1 \approx ({l/4} , {3l/8 - 2l/p}, {3l/8 + 2l/p})$
From $P_1$ to $Q_1$ along $d_{IIA}$ constant we increase $r_y$ by a factor of $T^{-l/p}$ keeping $r_x$  approximately constant. 

\item $Q_2 \approx ({l/4 - l/p}, {3l/8 - 3l/p}, {3l/8 + 4l/p})$
From $Q_1$ to $Q_2$ we have $\theta_{II}$ is constant and at $Q_1$ we have $\theta_{II}(Q_1) \approx 3l/8 -2l/p -l/4  = l/8 - 2l/p$. From $P_2$ to $Q_2$ we have $r_x$ approximately constant hence $\log_T r_x$ at $Q_2$ is approximately $l/4-l/p$ so
$$\log_T r_y \approx (l/8 - 2l/p) + (l/4 - l/p) = 3l/8 -3l/p$$
The $r_z$ coordinate is determined by $r_xr_yr_z = T^l$.

\end{itemize}

Now finally we get a condition on $p$. We want $r_x >> r_y$ everywhere in region IIA so that contour lines for $d_{IIA}$ look roughly as they are drawn and approximations for $d_{IIA}$ are valid.  Looking at $Q_1$ this means $T^{l/4} >> T^{3l/8 - 2l/p}$ and $Q_2$ gives  the same constraint. Hence we need $1/8 - 2/p > 0$ or $p>16$. E.g.~take $p = 17$.  

Recall Figure \ref{delineation_pic} and that from $A$ to $B$ $r_y$ moves through $l/p$ orders of $T$ magnitude while from $D$ to $E$ it moves through $3l/2p$ which is a lot more for small $T$. 

{\bf Region I: $\alpha_3, \alpha_5$.} Since $\alpha_3$ is a function of $d_I \approx T^2[r_x^2 - \frac{1}{2}(r_y^2 + r_z^2)] \approx (Tr_x)^2$ in region I because $r_x$ is many orders of magnitude bigger than $r_y$ and $r_z$ get in that region. We need to see how $\log_T (Tr_x)^2$ changes. Recall that $r_x$ changed by $l/p$ orders of magnitude, so $\log_T (Tr_x)^2$ changes by approximately $2l/p$. Thus the log derivative can be made as small as possible, as explained above.

For $\alpha_3^{'' \log}$ we want to know the change in slope over $\log d$ time. The derivative goes from 0 to $\frac{1}{3l}$ in approximately $l$ log time. Think of a bump function from $2/3$ to 1. Hence all terms involving a derivative of $\alpha_3$ are bounded by a constant times $T^2/l$.

{\bf Region I: $\alpha_4$.} However, in the calculation for $\alpha_4$, we end up dividing by $\theta$. So $\alpha_4$ will need to be constant for a short while in the middle, so we don't divide by zero. In the calculations above, we have cut out a region where $T < (r_y/r_z)^2 < T^{-1}$. In this region $\alpha_4 \equiv 0$. To do this, we need to make sure that we have at least one order of magnitude difference between $r_y$ and $r_z$. This is fine in region I because we have $r_y$ and $r_z$ many orders of magnitude apart at $B$ and $C$, with even more discrepancy as we move out to $E$ and $F$. (Note however, the reverse would have happened in region II. In other words, $r_z$ gets smaller as we move out, so $r_x$ and $r_y$ get bigger, and constant $r_y^2-r_x^2$ means they will get closer and closer together as we move out.) 

This one is a function of $\theta_I \approx T^2[r_y^2 - r_z^2]$ or its negative. Furthermore, we're taking out a sliver around the axis where $T < (r_y/r_z)^2 < 1/T$. So we need to check there's enough space left. At the bottom where $C$ and $F$ are, $\theta_I \approx -(Tr_z)^2$ which is on the order of $-T^{2\left(\frac{3l}{8} - \frac{l}{p}+1\right)}$ as seen above. Then we stop when $r_z/r_y = 1/\sqrt{T}$. At this point $r_z$ and $r_y$ are basically equal, since they are only about one $T$-order of magnitude apart and both really small. At some fixed $d_I$, we have $r_x \approx T^{l/4 - tl/p}$ where $t$ is fixed at some number between 0 and 1. So $r_z = \sqrt{T^{-1}} r_y$ and:
$$r_xr_yr_z = T^{l} \implies T^{l/4 - tl/p} \cdot r_y \cdot \sqrt{T^{-1}} r_y = T^{l} \implies r_y^2 = {\sqrt{T}} \cdot T^{3l/4  + tl/p}.$$
Hence 
$$\theta_I \approx T^2[r_y^2 - r_z^2] = T^2r_y^2 \left(1 - \left(\frac{r_z}{r_y}\right)^2\right) \approx T^2{\sqrt{T}} \cdot T^{3l/4  + tl/p}(1 - T^{-1})$$
$$ \approx -T^{3l/4  + tl/p+3/2}\because T<<1$$

We are checking how $\theta_I$ changes when we cut out this sliver to make sure $\alpha_4$ has enough space for the log derivatives. We find $\theta_I$ decreases from order of magnitude $3l/4 - 2l/p+2$ to order of magnitude $3l/4  + tl/p+3/2$. This means a net change of 
$$(3l/4 - 2l/p+2) - (3l/4  + tl/p+3/2)  = -l \left( \frac{2}{p}+\frac{t}{p} \right) + 1/2$$
with a multiple of $l$, so we're still okay for the log derivative of $\alpha_4$ because this is the denominator of the log derivative which can be made very large because of the $l$. Note that we only care about $\alpha_4$ in region I since it's constant at $1/2$ when we exit the region and move into region II. 

{\bf Region II.} Also, there is enough space in $\theta_{II}$ and $\log d_{IIA}$ for $\alpha_6$ and $\alpha_3, \alpha_5$ respectively. Recall
$$d_{IIA} \approx T^2[r_x^2 - \frac{1}{2}(r_y^2 + r_z^2) + \frac{3}{2}\alpha_6\cdot r_y^2]$$
Since everywhere in region IIA we have $r_x>>r_y$, the leading order term in $d_{II}$ is $(Tr_x)^2$. Likewise since $\theta_{II}$ is linear in $\log_T r_y - \log_T r_x$, for it to stay constant we find that both $\log_T r_y$ and $\log_T r_x$ increase the same amount along contour lines. At $P_1$ and $P_2$, $\theta_{II} \approx \log_T( T^{l/8 - l/p}) $. At $Q_1$ and $Q_2$, $\theta_{II} \approx \log_T(T^{l/8 - 2l/p})$. Thus the change is $l/p$ and $\alpha_6^{'} \approx \Delta \alpha_6/\Delta \theta_{II} \approx p/l$. This can be made small by taking $l$ large.

Now we check $d_{IIA}$ for $\alpha_3$ and $\alpha_5$. At the smaller value we have $d_{IIA} \approx (Tr_x)^2$ at $P_1$ where $r_x^2$ has exponent $l/2$. At the larger value we get $l/2 - 2l/p$. So the change is $2l/p$ and the derivatives in $\alpha_3$ and $\alpha_5$ are approximately $p/2l$ times whatever the full range of $\alpha_i$ is. Again this can be made small.

Second-order derivatives also have enough space. In the bump functions $\alpha_3,\ldots,\alpha_6$ the variable we're taking the derivative with respect to has space $k\cdot l$ for some $k>0$ to move while each of the bump functions moves through an amount $1/3$ or 1. By making $l$ really big, we can ensure that while this is happening, the second derivative doesn't get too big. The graph of the bump functions won't be linear because they have to level off at the endpoints of their support. But with enough space, we can make sure they don't turn too quickly from horizontal to linear. 

{\bf Other regions covered by symmetry.} The argument that there is enough space in IIC follows from IIA by swapping $r_x$ and $r_y$ in the calculations. The only variable in region IIB is $d_{IIB}$ and we take the two radial curves to be where $d_{IIB}$ is constant, with the same amount of change in the variable as in $d_{IIA}$. Note that when $r_x>>r_y$ we see that $d_{IIB} \approx T^2[ r_x^2 +r_y^2 - \frac{1}{2}r_z^2] \approx (Tr_x)^2$. So initially, constant $d_{IIB}$ is approximately the same thing as constant $r_x$. At some point we increase $r_y$ enough to equal $r_x$. Likewise, coming from region IIC we have that constant $d_{IIB}$ means, initially, approximately the same thing as constant $r_y$. So in the middle the curve interpolates between vertical (constant $r_x$) and horizontal (constant $r_y$). The derivatives for functions of $d_{IIB}$ have enough space because $d_{IIB}$ goes through the same amount as $d_{IIA}$ and $d_{IIC}$.  

{\bf Outside $\mb{C}^3$ patch.} Note that we will have enough space for log derivatives because $\theta \approx  (1+|Tz|^2)T^2(r_x^2 - r_y^2)$ for $r_x,r_y$ very small and this was already checked earlier when $\theta \approx T^2(r_x^2-r_y^2)$.\end{proof}

\section{Appendix B: Negligible terms in defining the symplectic form}

\subsection{Region I}
{\bf Negligible terms.} The terms that produce derivatives of the bump functions are $\alpha_3(d_I)d_I$ and $\alpha_4(\theta_I) \alpha_5(d_I)\theta_I$. Note that $d$ in region $I$ close to  where $r_x = r_y = r_z$ is approximately linear in $(Tr_x)^2,(Tr_y)^2$, and $(Tr_z)^2$.

Some notation: Let $d \equiv d_I$ in this section. We want to allow the variable we're taking the derivative with respect to to vary among $\{r_x,r_y,r_z\}$. So we denote those variables to be $\{r_\star, r_\bullet\} \in \{r_x,r_y,r_z\}$. Furthermore, $' \log$ means we take the log derivative $d\alpha_3/d (\log(d_I)) = \alpha_3' \cdot d$. The following calculation for the second derivative applies to $\alpha_5$ as well, and $\alpha_4$ if we replace $d$ with $\theta$.
\begin{align*}
\frac{d^2 \alpha_3}{d (\log d)^2} & = (\alpha_3' \cdot d)' \cdot d = (\alpha_3'' \cdot d + \alpha_3')d\\
\implies \alpha_3'' \cdot d + \alpha_3' & = \frac{1}{d} \cdot \frac{d^2 \alpha_3}{d (\log d)^2}\\
\implies \alpha_3'' & = \frac{1}{d^2} \cdot \left(\frac{d^2 \alpha_3}{d (\log d)^2} - \frac{d \alpha_3}{d (\log d)}\right)
\end{align*}

{\bfseries Diagonal terms for $\alpha_3(d) \cdot d$:} 
{\small
\begin{allowdisplaybreaks}
\begin{align*}
d & \approx T^2[r_x^2   - \frac{1}{2} \left( r_y^2 + r_z^2  \right)] \implies d_{r_*}  \approx \la T^2 r_*, \;\; d_{r_* r_*} \approx \la T^2,\;\; \la \in \{2, -1\}\\
 \left(\frac{1}{r_*}\dd_{r_*} + \dd^2_{r_* r_*}\right)(\alpha_3(d) \cdot d) &= \frac{1}{r_*}( \alpha_3' d_{r_*} d + \alpha_3(d) \cdot d_{r_*})\\
 & + (\alpha_3''d_{r_*}^2 d + \alpha_3' d_{r_* r_*} d + 2 \alpha_3'(d) d_{r_*}^2 + \alpha_3(d) \cdot d_{r_* r_*})\\
& \approx   \frac{1}{r_*}( \alpha_3'  \la T^2 r_* d + \alpha_3(d) \cdot  \la T^2 r_*) \\
& + (\alpha_3''( \la T^2 r_*)^2 d + \alpha_3' \la T^2 d + 2 \alpha_3'(d) ( \la T^2 r_*)^2 + \alpha_3(d) \cdot \la T^2)\\
& = \la T^2[\alpha_3' d + \alpha_3 + \la T^2\alpha_3''r_*^2 \cdot d + \alpha_3' d + 2\la T^2\alpha_3' r_*^2 +\alpha_3(d)]\\
& = \lambda T^2 [ \alpha'_3 ( 2d + \lambda T^2 r_*^2) + \lambda T^2 r_*^2(\alpha''_3 d + \alpha'_3)] + 2\lambda T^2 \alpha_3 \\
& = \la T^2[ \alpha_3^{' \log}(2 + \frac{\la T^2r_*^2}{d} ) + \la\frac{T^2r_*^2}{d}\alpha_3^{'' \log}]+2\la T^2 \alpha_3 
\end{align*}
\end{allowdisplaybreaks}

\begin{align*}
\alpha_3^{' \log} & \approx \frac{\Delta \alpha_3}{\Delta \log d} \approx \propto \frac{1}{l}\\
\frac{T^2 r_*^2}{d}& \approx \frac{T^2 r_*^2}{T^2[r_x^2   - \frac{1}{2} \left( r_y^2 + r_z^2  \right)] } = \frac{r_*^2}{r_x^2   - \frac{1}{2} \left( r_y^2 + r_z^2  \right)} = \frac{r_*^2/r_x^2}{1-\frac{1}{2}\left(\frac{r_y^2}{r_x^2} + \frac{r_z^2}{r_x^2}\right)}\approx 1 \mbox{ or } 0 \because r_x>>r_y, r_z\\
\alpha_3^{'' \log} & \approx \propto \frac{1}{l^2}
\end{align*}
}

Thus all terms involving bump function derivatives can be made much smaller than $2\la T^2 \alpha_3$ for $l$ sufficiently large.

{\bfseries Off-diagonal terms for $\alpha_3(d)\cdot d$ where $* \neq \star$:}
\begin{align*}
d_{r_*r_\star} & = 0\\
\dd^2_{r_* r_\star} (\alpha_3 \cdot d) & = \dd_{r_*} (\alpha_3' d_{r_\star} d + \alpha_3(d) \cdot d_{r_\star})\\
&=\alpha_3''d_{r_*}d_{r_\star}d + \alpha_3'd_{r_\star r_*}d + 2\alpha_3'd_{r_\star}d_{r_*} + \alpha_3 \cdot d_{r_\star r_*}\\
& \approx T^4 \la \mu r_*r_\star \alpha_3''d + 0 + 2T^4 \la\mu r_*r_\star \alpha_3'+0,\qquad \la, \mu \in \{2,-1\}\\
& = T^4 \la \mu \frac{r_*r_\star}{d}d(\alpha_3''d + \alpha_3' + \alpha_3')\\
& = T^4 \la \mu \frac{r_*r_\star}{d}(\alpha_3^{'' \log} + \alpha_3^{' \log})\\
 \frac{T^2r_*r_\star}{d}&\approx \frac{r_* r_\star/r_x^2}{1 - \frac{1}{2}((r_y/r_x)^2 + (r_z/r_x)^2)} \approx 0 \mbox{ or } 1
\end{align*}
Again we see that all the derivatives of the bump functions are log derivatives and all the terms are bounded by a constant times $T^2/l$. 

{\bf Diagonal terms $\frac{1}{r_\bullet}\dd_{r_\bullet} + \dd^2_{r_\bullet r_\bullet}$ and off-diagonal terms $\dd^2_{r_\bullet r_\star}$ for $\alpha_4\alpha_5\theta_I$.}
\begin{align*}
\theta & \approx T^2(r_y^2 - r_z^2) \implies \theta_{r_\bullet }  \approx \la_\bullet  T^2 r_\bullet , \qquad \la_\bullet  \in \{0,\pm 2\} \\
d& \approx T^2(r_x^2 - \frac{1}{2}(r_y^2+r_z^2)) \implies d_{r_\bullet } \approx \mu_\bullet  T^2 r_\bullet , \qquad \mu_\bullet  \in \{2, -1\}\\
\therefore \frac{1}{r_\bullet }\dd_{r_\bullet }(\alpha_4\alpha_5 \theta) & = \frac{1}{r_\bullet }(\alpha_4' \theta_{r_\bullet } \alpha_5 \theta + \alpha_4 \alpha_5' d_{r_\bullet } \theta + \alpha_4 \alpha_5 \theta_{r_\bullet }) \approx T^2(\la_\bullet  \alpha_4^{' \log} \alpha_5 + \mu_\bullet  \alpha_4 \alpha_5^{' \log} \frac{\theta}{d} + \la_\bullet  \alpha_4 \alpha_5)
\end{align*}

{\scriptsize \begin{align*}
\dd^2_{r_\star r_\bullet }(\alpha_4\alpha_5\theta) & \approx T^2\dd_{r_\star}[\textcolor{red}{r_\bullet \la_\bullet  \alpha_4' \theta \alpha_5 }+ {r_\bullet \mu_\bullet  \alpha_4 \alpha_5'\theta} + \textcolor{blue}{r_\bullet \la_\bullet  \alpha_4 \alpha_5} ]\\
& = T^2[\textcolor{red}{\dd_{r_\star} (r_\bullet)  \la_\bullet  \alpha_4'\theta \alpha_5 +  r_\bullet \la_\bullet  \alpha_4'' \theta_{r_\star} \theta \alpha_5 + r_\bullet \la_\bullet  \alpha_4'\theta_{r_\star}  \alpha_5 + r_\bullet  \la_\bullet  \alpha_4' \theta \alpha_5' d_{r_\star}} \\
&+{ \dd_{r_\star}(r_\bullet )\alpha_4 \alpha_5' \theta    + r_\bullet  \mu_\bullet  \alpha_4' \theta_{r_\star} \alpha_5' \theta + r_\bullet  \mu_\bullet  \alpha_4 \alpha_5'' d_{r_\star} \theta + r_\bullet  \mu_\bullet  \alpha_4 \alpha_5' \theta_{r_\star}}\\
 &+\textcolor{blue}{ \dd_{r_\star}(r_\bullet ) \la_\bullet  \alpha_4 \alpha_5 +   r_\bullet  \la_\bullet  \alpha_4' \theta_{r_\star} \alpha_5 + r_\bullet  \la_\bullet  \alpha_4  \alpha_5' d_{r_\star} }  ]\\
 & \approx T^2 [\textcolor{red}{\delta_{\bullet \star}\la_\bullet  \alpha_4^{' \log} \alpha_5 + T^2[\la_\bullet \la_\star \frac{r_\bullet r_\star}{\theta}(\alpha_4^{'' \log} - \alpha_4^{' \log}  ) \alpha_5 + \la_\bullet \la_\star \frac{r_\bullet  r_\star}{\theta}\alpha_4^{' \log} \alpha_5 + \la_\bullet  \mu_\star \frac{ r_\bullet r_\star}{d}\alpha_4^{' \log}\alpha_5^{' \log}]}\\
 &+\delta_{\star \bullet } \frac{\theta}{d}\alpha_4 \alpha_5^{' \log}  + \mu_\bullet \la_\star \frac{T^2r_\bullet  r_\star}{d} \alpha_4^{' \log} \alpha_5^{' \log} + \mu_\bullet \mu_*\frac{T^2r_\bullet r_\star \theta}{d^2} \alpha_4( \alpha_5^{'' \log} - \alpha_5^{' \log}) + \mu_\bullet  \la_\star \frac{T^2 r_\bullet  r_\star }{d}\alpha_5^{' \log}\alpha_4\\
 & +\textcolor{blue}{ \delta_{\star \bullet } \la_\bullet  \alpha_4 \alpha_5 + \la_\bullet  \la_\star \frac{T^2r_\bullet r_\star}{\theta}\alpha_4^{' \log} \alpha_5 + \la_\bullet  \mu_\star \frac{T^2r_\bullet  r_\star}{d} \alpha_4\alpha_5^{'\log}     }     ]
\end{align*}}

We've already seen above that $T^2\displaystyle{\frac{r_\star r_\bullet}{d}}$ and hence $\displaystyle{\frac{\theta}{d}}$ are bounded. So it remains to check that $\displaystyle{T^2\frac{r_\star r_\bullet}{\theta}}$ is bounded. Also, we only need to consider the cases where the numerator does not involve $r_x$ by the comment above that we get zero otherwise and that second-order partial derivatives are symmetric. So we have to bound the following expressions:
$$\frac{r_y^2}{r_y^2-r_z^2},\frac{r_z^2}{r_y^2-r_z^2},\frac{r_yr_z}{r_y^2-r_z^2}.$$
We are considering the top half of region I, where $r_y>r_z$. We declare that $\alpha_4$ is constant in the region $1<\left( \frac{r_y}{r_z} \right)^2 < \frac{1}{T}$. So the support of $\alpha_4$ is where $\left( \frac{r_y}{r_z} \right)^2 > \frac{1}{T}$. In particular, we see that
\begingroup \allowdisplaybreaks \begin{equation}
\label{eqn:bound_loc_constant}\frac{1}{\left( \frac{r_y}{r_z} \right)^2-1} < \frac{1}{T^{-1}-1} = \frac{T}{1-T}\approx T.
\end{equation} \endgroup
So these terms are bounded, which can be seen dividing top and bottom by $r_z^2$. So we've seen $\lambda_\bullet \lambda_* \frac{T^2 r_\bullet r_*}{\theta},  \frac{T^2r_\bullet r_\star}{d}, \frac{\theta}{d}$ are bounded, and multiply log derivatives which can be made sufficiently small for $l$ large.

\subsection{Region IIA}

{\bf Negligible terms.} Bump function terms in $F$ are: $\alpha_6\cdot (Tr_y)^2,\; \alpha_3\cdot d_{IIA},\; \{\alpha_5(Tr_*)^2\}_{*=y,z},\; \alpha_5\alpha_6\cdot (Tr_y)^2$

\begin{center}\fbox{ \bfseries 1st term for region IIA: $\alpha_6(\log(r_y/r_x))\cdot (Tr_y)^2$}\end{center}

{\bfseries First derivative} divided by $r_\star$: 
\begin{align*}
\frac{1}{r_x}\dd_{r_x}[\alpha_6(\log r_y - \log r_x))(Tr_y)^2] & = -\frac{1}{r_x^2} \alpha_6' \cdot (Tr_y)^2 = -T^2\cdot \frac{r_y^2}{r_x^2}\alpha_6', \; \left| \frac{r_y^2}{r_x^2} \right| <1\\
\frac{1}{r_y}\dd_{r_y}[\alpha_6(\log r_y - \log r_x))(Tr_y)^2] & = \frac{1}{r_y^2} \alpha_6' \cdot (Tr_y)^2 + \alpha_6\cdot (2T^2)=T^2 \alpha_6' + 2T^2 \alpha_6\\
\frac{1}{r_z}\dd_{r_z}[\alpha_6(\log r_y - \log r_x)(Tr_y)^2] & = 0\\
\alpha_6' & \approx \propto \frac{1}{l}
\end{align*}
So all bump function derivative terms from first order derivatives of this term can be made small.

{\bfseries Second derivative} $\dd^2_{\star \bullet}$: 
{\footnotesize \begin{align*}
\dd^2_{r_xr_x}(\alpha_6(Tr_y)^2) & = \dd_{r_x}(-\frac{(Tr_y)^2}{r_x} \alpha_6') = \frac{(Tr_y)^2}{r_x^2}(\alpha_6' +\alpha_6'') < T^2(\alpha_6' + \alpha_6'') \approx \propto T^2(\frac{1}{l} + \frac{1}{l^2}), \; \mbox{ small }\\
\dd^2_{r_xr_y}(\alpha_6(Tr_y)^2) & = \dd_{r_y}(-\frac{(Tr_y)^2}{r_x}\alpha_6') = \frac{-T^2r_y}{r_x}(2\alpha_6' +\alpha_6''), \mbox{ norm } < T^2(2\alpha_6'+\alpha_6'')\\
\dd^2_{r_yr_y}(\alpha_6(Tr_y)^2) & = \dd_{r_y}(T^2 \alpha_6'r_y + 2T^2 \alpha_6r_y) = T^2(3\alpha_6' + \alpha_6'' + 2\alpha_6)
\end{align*}}

Note the first derivative $\alpha_6'$ goes from 0 to a maximum of $1/l + \eps$ at the half way point of $l/2$ so the change in slope is roughly $1/l^2$, still small. (Strictly speaking, $(1/l + \eps) (2/l)$.)  Thus the derivatives of $\alpha_6 r_y^2$ can be made small by taking $l$ sufficiently large. 

\begin{center}\label{2nd term} \fbox{\bfseries 2nd term for region IIA: $\alpha_3(d_{IIA}) d_{IIA}$ }\end{center}

{\bfseries First derivative.} $\frac{1}{r_\star}\dd_{r_\star} (\alpha_3d_{IIA}) = (\alpha_3' d_{IIA} + \alpha_3)\cdot \frac{{d_{IIA}}_\star}{r_\star}$. 
Here are the partial derivatives of $d_{IIA}$.
\begingroup \allowdisplaybreaks \begin{equation}\label{d_expressions}
\begin{aligned}
d_{IIA} & \approx T^2[r_x^2 - \frac{1}{2}(r_y^2 + r_z^2)] +\frac{3}{2} \alpha_6(\log r_y - \log r_x) \cdot (Tr_y)^2\\
\frac{(d_{IIA})_x}{r_x} &= \frac{T^2}{r_x}[2r_x + \frac{3}{2}(\alpha_6'\cdot \frac{-1}{r_x}\cdot r_y^2)] = T^2[2 + \frac{3}{2}(\alpha_6' \cdot \frac{-r_y^2}{r_x^2})]\\
\frac{(d_{IIA})_y}{r_y} & = \frac{T^2}{r_y}[-r_y + \frac{3}{2}(\alpha_6'\cdot \frac{1}{r_y} \cdot r_y^2 + 2r_y \alpha_6)]=T^2[-1 + \frac{3}{2}(\alpha_6' + 2\alpha_6)]\\
\frac{(d_{IIA})_z}{r_z} & = \frac{T^2}{r_z}[-r_z] = -T^2
\end{aligned}
\end{equation} \endgroup

Thus derivative terms are either $\alpha_3^{' \log} = \alpha_3'd_{IIA}$ or a regular derivative of $\alpha_6$, which may be multiplied by $(r_y/r_x)^2$ but $r_x>r_y$ in region IIA. So derivative terms of $\alpha_3(d_{IIA}) d_{IIA}$ can be made small for $l$ sufficiently large.

{\bfseries Second derivative terms.} We differentiate each of the first derivative terms. Let's say $P$ is a term in Equation \ref{d_expressions} above. Then we want to differentiate $r_\star P$ because above we divided by $r_\star$. Thus using the product rule with a differential operator $D = \dd_{r_\bullet}$ we get $D(r_\star)P + r_\star D(P)$. The first term gives 0 or 1 times $P$, which we already know is small. So we'll only need to consider $r_\star D(P)$ for 8 choices of $P$. 

\begin{enumerate}
\item $P=\alpha_3'd_{IIA}$ :
{\footnotesize 
\begingroup \allowdisplaybreaks 
\begin{align*}
r_\star\dd_{r_\bullet}(\alpha_3'd_{IIA}) & = r_\star \alpha_3''\cdot (d_{IIA})_{\bullet}\cdot d_{IIA} + r_\star \alpha_3'(d_{IIA})_\bullet  = \frac{r_\star (d_{IIA})_\bullet}{d_{IIA}} \alpha_3^{'' \log} \\
(d_{IIA})_\bullet \mbox{ terms}:\; & \{T^2r_\bullet,\; T^2\alpha_6 r_y,\; T^2\alpha_6' r_y,\;T^2 \alpha_6'  \frac{r_y^2}{r_x}\}< T^2 r_\bullet, \because \alpha_6 \leq 1, \alpha_6' \approx \propto \frac{1}{l}, \frac{r_y}{r_x}<1\\
\frac{T^2 r_\star r_\bullet}{d_{IIA}} & \approx\frac{ r_\star r_\bullet}{r_x^2 - \frac{1}{2}(r_y^2 + r_z^2) + \frac{3}{2}\alpha_6 r_y^2} = \frac{r_\star r_\bullet/r_x^2}{1-\frac{1}{2}((r_y/r_x)^2 + (r_z/r_x)^2) +\frac{3}{2}\alpha_6(r_y/r_x)^2}\stepcounter{equation}\tag{\theequation}\label{anything}\\
& \approx \frac{r_\star r_\bullet}{r_x^2}\approx \in \{0,1\} \because r_x >> r_y, r_z \mbox{ in region IIA}\\
& \therefore r_\star\dd_{r_\bullet}(P) =  \frac{r_\star (d_{IIA})_\bullet}{d_{IIA}} \alpha_3^{'' \log} \approx \mbox{(bounded)}\cdot\frac {1}{l^2}
\end{align*}
\endgroup}

\item\label{first y} $P=\alpha_3'd_{IIA}\alpha_6$ term from differentiating $F$ first wrt $y$, i.e.~$\dd_{r_y}(F)$ (same calculation works replacing $\alpha_6$ with $\alpha_6'$)
\begin{align*}
r_y\dd_{r_\bullet}(\alpha_3'd_{IIA}\alpha_6) & = r_y\dd_{r_\bullet}(\alpha_3'd_{IIA}) \cdot \alpha_6 +r_y (\alpha_3' d_{IIA}) \cdot \dd_{r_\bullet}(\alpha_6)\\
&= (\mbox{above case})\cdot \alpha_6 \pm \alpha_3^{' \log} \alpha_6' \cdot cr_y,\qquad  c \in \{\frac{1}{r_x}, \frac{1}{r_y}, 0\}\\
& = \mbox{small} + \mbox{small}
\end{align*}

\item\label{xterm} $P=\alpha_3'd_{II}\alpha_6' \cdot \left(\frac{r_y}{r_x}\right)^2$ from differentiating first wrt $x$ (same argument for $\alpha_3\alpha_6'(r_y/r_x)^2$ replacing $\alpha_3'd_{II}\alpha_6'$ with $\alpha_3\alpha_6'$)
{\footnotesize \begin{align*}
r_x \dd_{r_\bullet}(\alpha_3'd_{II}\alpha_6' \cdot \left(\frac{r_y}{r_x}\right)^2) & = r_x\dd_{r_\bullet}(\alpha_3'd_{II}\alpha_6' ) \cdot \left(\frac{r_y}{r_x}\right)^2 + r_x(\alpha_3'd_{II}\alpha_6' ) \cdot\dd_{r_\bullet} \left(\frac{r_y}{r_x}\right)^2 \\
& = \frac{r_y}{r_x} \cdot r_y\dd_{r_\bullet}(\alpha_3'd_{II}\alpha_6' )  \pm (\alpha_3^{' \log}\alpha_6' ) \cdot cr_x, \qquad c \in \{2r_y^2/r_x^3, 2r_y/r_x^2, 0\}\\
& = (\mbox{small})(\mbox{previous case}) + (\mbox{small})(r_y/r_x)^{i},\qquad i \in \{1,2\}\\
& = \mbox{small} \because r_x>>r_y
\end{align*}}

\item $P=\alpha_3$: shows up in first derivative of $F$ wrt any variable, $r_\star\dd_{r_\bullet}\alpha_3=r_\star \alpha_3' (d_{IIA})_\bullet=\alpha_3^{'\log} \cdot (r_\star (d_{IIA})_\bullet)/d_{IIA}$, and $(r_\star (d_{IIA})_\bullet)/d_{IIA}$ bounded by Equation (\ref{anything}).

\item $P=\alpha_3\alpha_6$ (same argument for $\alpha_3\alpha_6'$: shows up in first derivatives of $F$ wrt $y$, $r_y\dd_{r_\bullet}(\alpha_3 \alpha_6)=r_y((\alpha_3)_\bullet \alpha_6 + \alpha_3 (\alpha_6)_\bullet)$. First term with $\alpha_3$ derivatives ok by previous item, second term gives $\alpha_3\alpha_6'$ times one or zero, since $r_y/r_x \ll 1$ or $1$. So that term is a bounded term times a small term hence also small.
\end{enumerate}
This concludes our check of the first and second order derivatives of $\alpha_3(d_{IIA})d_{IIA}$.

\begin{center} \fbox{\bfseries 3rd term for region IIA: $\alpha_5 \cdot (Tr_*)^2$ for $* \in \{y,z\}$ }\end{center} 

We run through the same argument as with $\alpha_3 \cdot d_{IIA}$ above, replacing $d_{IIA}$ with $(Tr_*)^2$ in the second term.

{\bfseries First derivative.} $\frac{1}{r_\star}\dd_{r_\star} (\alpha_5(Tr_*)^2) = \frac{1}{r_\star}\alpha_5' (d_{IIA})_\star (Tr_*)^2 + \alpha_5\cdot \frac{((Tr_*)^2)_\star}{r_\star}$. 
The bump function derivative term is $\frac{1}{r_\star}\alpha_5' (d_{IIA})_\star (Tr_*)^2=\alpha_5^{' \log}\frac{T^2r_*^2 (d_{IIA})_\star }{ d_{IIA}r_\star}$. Note that $\frac{T^2r_*^2}{d_{IIA}}$ and $\frac{(d_{IIA})_\star}{r_\star}$ are bounded, the latter by Equation (\ref{d_expressions}) and the former since:
{\footnotesize \begingroup \allowdisplaybreaks \begin{equation}\label{anything2}
\begin{aligned}
\frac{r_*^2}{r_x^2 - \frac{1}{2}(r_y^2+r_z^2) + \frac{3}{2}\alpha_6r_y^2} & = \frac{r_*^2/r_x^2}{1 - \frac{1}{2}((r_y/r_x)^2 + (r_z/r_x)^2) +\frac{3}{2}\alpha_6(r_y/r_x)^2} \ll 1 \because r_x>>r_y, r_z
\end{aligned}
\end{equation} \endgroup}

Note that $* \in \{y,z\}$. So first derivatives are bounded for sufficiently large $l$ in $\alpha_5 \cdot (Tr_*)^2$ since they are either non-bump function derivatives or a bounded quantity multiplied by a small log derivative. 

{\bfseries Second derivatives.} Differentiating first derivatives $\alpha_5' (d_{IIA})_\star (Tr_*)^2$ and $2T^2r_\star \alpha_5$:

1) \vspace{-24pt} \begin{equation*}
\begin{aligned}
r_\star\dd_{r_\bullet}(\alpha_5'(Tr_*)^2) & = r_\star \alpha_5''\cdot (d_{IIA})_{\bullet}\cdot (Tr_*)^2 + r_\star \alpha_5'((Tr_*)^2)_\bullet\\
&  = (\alpha_5^{'' \log} - \alpha_5^{' \log})\frac{r_\star\cdot(d_{IIA})_\bullet \cdot (Tr_*)^2 }{d_{IIA}^2} + \alpha_5^{' \log} \frac{2T^2r_\bullet r_\star}{d_{IIA}}\\
\mbox{Already checked bounded: } & \frac{r_\star (d_{IIA})_\bullet}{d_{IIA}} \because  (\ref{anything}), \; \frac{(Tr_*)^2}{d_{IIA}}\because (\ref{anything2}),\; \frac{T^2 r_\bullet r_\star}{d_{IIA}} \because(\ref{anything})\\
\therefore r_\star\dd_{r_\bullet}(\alpha_5'(Tr_*)^2)  & = \mbox{(small)(bounded)} + \mbox{(small)(bounded)}
\end{aligned}
\end{equation*}

2) $\alpha_5' (Tr_*)^2\alpha_6$ (same for $\alpha_5' (Tr_*)^2\alpha_6'$): term from differentiating $F$ first wrt $y$, i.e.~$\dd_{r_y}(F)$
\begin{align*}
r_y\dd_{r_\bullet}(\alpha_5' (Tr_*)^2\alpha_6) & = r_y\dd_{r_\bullet}(\alpha_5' (Tr_*)^2) \cdot \alpha_6 +r_y (\alpha_5' (Tr_*)^2) \cdot \dd_{r_\bullet}(\alpha_6)\\
&= (\mbox{above case})\cdot\alpha_6 \pm \alpha_5^{' \log}\frac{(Tr_*)^2}{d_{IIA}} \alpha_6' \cdot cr_y,\qquad  c \in \{\frac{1}{r_x}, \frac{1}{r_y}, 0\}\\
& = \mbox{(small)(bounded)} + \mbox{(small)(bounded)}
\end{align*}

3) $\alpha_5' (Tr_*)^2\alpha_6' \cdot \left(\frac{r_y}{r_x}\right)^2$: from differentiating first wrt $x$
{\footnotesize \begin{align*}
r_x \dd_{r_\bullet}(\alpha_5' (Tr_*)^2\alpha_6' \left(\frac{r_y}{r_x}\right)^2) & = r_x\dd_{r_\bullet}(\alpha_5' (Tr_*)^2\alpha_6') \cdot \left(\frac{r_y}{r_x}\right)^2 + r_x(\alpha_5' (Tr_*)^2\alpha_6' ) \cdot\dd_{r_\bullet} \left(\frac{r_y}{r_x}\right)^2 \\
& = \frac{r_y}{r_x} \cdot r_y\dd_{r_\bullet}(\alpha_5' (Tr_*)^2\alpha_6')  \pm (\alpha_5^{' \log}\frac{(Tr_*)^2}{d_{IIA}}\alpha_6' ) \cdot cr_x, \; c \in \{\frac{2r_y^2}{r_x^3}, \frac{2r_y}{r_x^2}, 0\}\\
& = (\mbox{small})(\mbox{previous case}) + (\mbox{small})(r_y/r_x)^{i},\qquad i \in \{1,2\}\\
& = \mbox{small} \because r_x>>r_y
\end{align*}}
4) $2T^2 \alpha_5$: shows up in first derivative of $F$ wrt any variable. Taking another derivative gives $r_\star\dd_{r_\bullet}2T^2 \alpha_5$ =
 $r_\star 2T^2 \alpha_5' (d_{IIA})_\bullet$ = $2T^2\alpha_5^{'\log} \cdot (r_\star (d_{IIA})_\bullet)/d_{IIA}$. So we want $(r_\star (d_{IIA})_\bullet)/d_{IIA}$ to be bounded. This was checked in Equation (\ref{anything}).

This concludes our check of the first and second order derivatives of $\alpha_5(d_{IIA})(Tr_*)^2$ for $* \in \{y,z\}$. We have one more remaining type of term showing up in $F$ to check.

\begin{center} \fbox{\bfseries 4th term for region IIA: $\alpha_5\alpha_6\cdot (Tr_y)^2$ } \end{center}

\begin{align*}
&\frac{1}{r_\star} \dd_{r_\star}(\alpha_5(Tr_y)^2) \cdot \alpha_6 + \alpha_5(Tr_y)^2 \cdot \frac{1}{r_\star} \dd_{r_\star}\alpha_6  = \mbox{(previous)} \cdot \alpha_6 \pm \alpha_5(Tr_y)^2 \cdot \frac{c}{r_\star^2}\alpha_6'\\
&\mbox{2nd term: } \star = z \implies c=0, \;\; \star = x \implies \approx 0 \because r_x >> r_y,\;\; \star = y \implies \frac{r_y^2}{r_\star^2} = 1
\end{align*}

So again first derivatives may be made small by taking $l$ sufficiently large. Finally, we check second order derivatives.
\begin{align*}
\dd_{r_\bullet}[\dd_{r_\star}(\alpha_5(Tr_y)^2) \cdot \alpha_6 + \alpha_5(Tr_y)^2 \cdot \dd_{r_\star}(\alpha_6)] & = \dd^2_{r_\bullet r_\star} (\alpha_5 (Tr_y)^2) \cdot \alpha_6 + \dd_{r_\star}(\alpha_5(Tr_y)^2) \dd_{r_\bullet} \alpha_6\\
& + \dd_{r_\bullet}(\alpha_5(Tr_y)^2) \dd_{r_\star}\alpha_6 + \alpha_5(Tr_y)^2 \cdot \dd_{r_\bullet}\dd_{r_\star}(\alpha_6)
\end{align*}
1st term ok by previous check on $\alpha_5(Tr_y)^2$.
{\footnotesize \begin{align*}
\mbox{Terms 2 \& 3: } \dd_{r_\star}(\alpha_5(Tr_y)^2) \cdot \dd_{r_\bullet } \alpha_6 & = (\alpha_5^{' \log} \frac{(d_{IIA})_{\star}}{d_{IIA}} (Tr_y)^2 + 2\alpha_5 T^2 \delta_{y \star} r_y) \cdot \frac{\pm 1}{r_x \mbox{ or } r_y} \cdot \alpha_6'\\
& = \pm [T^2 \alpha_5^{' \log} \frac{(d_{IIA})_\star r_y}{d_{IIA}} \cdot \frac{r_y}{r_x \mbox{ or } r_y}\cdot \alpha_6' + 2\alpha_5 T^2 \frac{r_y}{r_x \mbox{ or } r_y} \cdot \alpha_6']
\end{align*}}
\begin{align*}
\mbox{4th term: } \alpha_5(Tr_y)^2 \dd_{r_\bullet}(\frac{\alpha_6'}{r_x \mbox{ or } r_y}) & = \alpha_5(Tr_y)^2 \left[ \alpha_6'' \cdot \frac{\pm 1}{r_x \mbox{ or } r_y} \cdot \frac{1}{r_x \mbox{ or } r_y} - \alpha_6' \frac{1}{r_x^2 \mbox{ or } r_y^2} \right]\\
& = T^2 \alpha_5 \alpha_6'' \frac{\pm r_y^2}{r_x^2, r_xr_y, \mbox{ or } r_y^2} - T^2 \alpha_5 \alpha_6' \frac{r_y^2}{r_x^2 \mbox{ or } r_y^2}\\
& = \mbox{small}, \qquad \because r_x >> r_y\\
\implies \dd_{r_\bullet}\dd_{r_\star}(\alpha_5\alpha_6\cdot (Tr_y)^2)  & = \mbox{(small) }
\end{align*}
where the final line follows from the calculations for the 1st term of region IIA, which was $\alpha_6 (Tr_y)^2$, and the third term of region IIA, which was $\alpha_5(Tr_y)^2$. So the upshot is: all terms involving derivatives of bump functions can be made arbitrarily small because they are multiplied by expressions which are bounded (taking either log derivatives of $\alpha_3,\alpha_5$ or regular derivatives of $\alpha_6$.) So we get a positive-definite form for $l$ sufficiently large, because the terms not involving derivatives of bump functions are $O(1)$ so they dominate, and we already showed they give something positive-definite. 

\subsection{Region IIB}

The characteristics in region IIB which we did not have in regions IIA and C are 1) $r_x$ and $r_y$ go from $r_x>>r_y$ to $r_y>>r_x$, passing through $r_x = r_y$ and 2) $\alpha_6 \equiv 1$. All of $r_x,r_y,r_z$ are still small so we still have an approximation for the K\"ahler potential:
\begin{align*}
F & \approx T^2r_z^2 +\alpha_3(d_{IIB})d_{IIB} + \frac{1}{2}\alpha_5(d_{IIB})\cdot(-(Tr_z)^2)\\
d_{IIB} & \approx T^2[r_x^2+r_y^2 - \frac{1}{2}r_z^2]
\end{align*}

Let's repeat the calculations above for $\alpha_3(d_{IIA})\cdot d_{IIA}$ and $\alpha_5 \cdot (Tr_z)^2$ with region IIB, and see if they relied on $r_x>>r_y$. What we know in region IIB is that $r_x, r_y>> r_z$.

\begin{center}\fbox{\bfseries 1st term for region IIB: $\alpha_3(d_{IIB}) d_{IIB}$ }\end{center}

{\bfseries First derivative.} $\frac{1}{r_\star}\dd_{r_\star} (\alpha_3d_{IIB}) = (\alpha_3' d_{IIB} + \alpha_3)\cdot \frac{{d_{IIB}}_\star}{r_\star}$. 
Here are the partial derivatives of $d_{IIB}$.
\begin{align*}
d_{IIB} & \approx T^2[r_x^2+r_y^2 - \frac{1}{2}r_z^2]\implies \frac{(d_{IIB})_x}{r_x} &\approx \frac{T^2}{r_x}(2r_x)=\frac{T^2}{2}, \ \ \frac{(d_{IIB})_y}{r_y}  \approx \frac{T^2}{2}, \ \ \frac{(d_{IIB})_z}{r_z}  \approx - \frac{T^2}{4}
\end{align*}
Hence the terms in $\frac{1}{r_\star}\dd_{r_\star} (\alpha_3d_{IIB}) =(\alpha_3' d_{IIB} + \alpha_3)\cdot \frac{{d_{IIB}}_\star}{r_\star}$ are proportional to $\alpha_3' d_{IIB}$ (a log derivative, so small) and $\alpha_3$ (not a derivative). So first derivatives of $\alpha_3(d_{IIB}) d_{IIB}$ may be made small for $l$ sufficiently large. 

{\bfseries Second derivative terms.}\label{page: Pdefn} We differentiate each of the first derivative terms. Let's say $P$ is a term in the list above. Then we want to differentiate $r_\star P$ because above we divided by $r_\star$. Thus using the product rule with a differential operator $D = \dd_{r_\bullet}$ we get $D(r_\star)P + r_\star D(P)$. The first term gives 0 or 1 times $P$, which we already know is small by the above item for each $P$ on the list. So we'll only need to consider $r_\star D(P)$ for the 2 choices of $P$ listed above. 

\begin{enumerate}
\item $P=\alpha_3'd_{IIB}$. Then this term contributes to $\dd^2_{r_\star r_\bullet}(F)$ via $r_\star\dd_{r_\bullet}(P)$ i.e.
$$r_\star\dd_{r_\bullet}(\alpha_3'd_{IIB}) = r_\star \alpha_3''\cdot (d_{IIB})_{\bullet}\cdot d_{IIB} + r_\star \alpha_3'(d_{IIB})_\bullet  = \frac{r_\star (d_{IIB})_\bullet}{d_{IIB}} \alpha_3^{'' \log}$$
\begin{align*}
(d_{IIB})_\bullet \mbox{ terms}:\; & T^2r_\bullet \\
\frac{T^2 r_\star r_\bullet}{d_{IIB}} & \approx\frac{ r_\star r_\bullet}{r_x^2 +r_y^2- \frac{1}{2} r_z^2 } = \frac{r_\star r_\bullet/r_x^2}{1+(r_y/r_x)^2 -\frac{1}{2}(r_z/r_x)^2} \approx  \frac{r_\star r_\bullet/r_x^2}{1+(r_y/r_x)^2 } \\
(\star, \bullet) & \in \{(r_x,r_x),(r_x,r_z), (r_z,r_z)\} \implies \frac{r_\star r_\bullet}{r_x^2}\in \{1, \mbox{ small}\} \because r_x >> r_z\\
(\star, \bullet) & = (r_y,r_z) \implies \frac{r_y r_z}{r_x^2} < \frac{r_y}{r_x}
\end{align*}
\begin{align*}
(\star, \bullet) & \in \{(r_x,r_y),(r_y,r_y)\} \mbox{ suffices to bound: } \frac{a}{1+a^2}, \frac{a^2}{1+a^2}, \;\; a = r_y/r_x\\
\frac{a^2}{1+a^2} \leq 1,\;\; & a<1 \implies \frac{a}{1+a^2} < 1, \;\; a \geq 1 \implies \frac{a}{1+a^2} \leq \frac{a^2}{1+a^2} \leq 1\\
\therefore r_\star\dd_{r_\bullet}(\alpha_3'd_{IIB}) & = \mbox{(bounded)(small)}
\end{align*}

\item $P=\alpha_3$. Taking another derivative gives $r_\star\dd_{r_\bullet}\alpha_3=r_\star \alpha_3' (d_{IIB})_\bullet=\alpha_3^{'\log} \cdot (r_\star (d_{IIB})_\bullet)/d_{IIB}$. So we want $(r_\star (d_{IIB})_\bullet)/d_{IIB}$ to be bounded. This was just checked above.

\end{enumerate}

\begin{center} \fbox{\bfseries 2nd term for region IIB: $\alpha_5 \cdot (Tr_z)^2$ } \end{center} 

We run through the same argument as with $\alpha_3 \cdot d_{IIB}$ above, replacing $d_{IIB}$ with $(Tr_z)^2$ in the second term.

{\bfseries First derivative.} $\frac{1}{r_\star}\dd_{r_\star} (\alpha_5(Tr_z)^2) = \frac{1}{r_\star}\alpha_5' (d_{IIB})_\star (Tr_z)^2 + \alpha_5\cdot \frac{((Tr_z)^2)_\star}{r_\star}$. 
$$\frac{1}{r_\star}\alpha_5' (d_{IIB})_\star (Tr_z)^2=\alpha_5^{' \log}\frac{T^2r_z^2 (d_{IIB})_\star }{ d_{IIB} \cdot r_\star}$$
Note that $\frac{T^2r_z^2}{d_{IIB}}$ and $\frac{(d_{IIB})_\star}{r_\star}$ are bounded. The latter is approximately constant because $d_{IIB}$ is approximately linear in $r_x^2,r_y^2$ and $r_z^2$. For the former:
\begin{align*}
\frac{r_z^2}{r_x^2 +r_y^2 - \frac{1}{2}r_z^2 } & = \frac{r_z^2/r_x^2}{1 +\frac{r_y^2}{r_x^2}- \frac{1}{2}( (r_z/r_x)^2) } \approx \frac{r_z^2/r_x^2}{1 +\frac{r_y^2}{r_x^2} } \approx 0 \because r_x>> r_z
\end{align*}
So first derivatives of $\alpha_5 \cdot (Tr_z)^2$ may be made small by taking $l$ sufficiently large. 

{\bfseries Second derivative terms.} We differentiate each of the first derivative terms.
They are $\alpha_5' (d_{IIB})_\star (Tr_z)^2$ and $2T^2r_\star \alpha_5$. $P$ is defined as in the second derivative calculation on page \pageref{page: Pdefn}.  
\begin{enumerate}
\item $P=\alpha_5' (Tr_z)^2 \frac{(d_{IIB})_\star}{r_\star}$. It suffices to consider $\alpha_5' (Tr_z)^2$ because $\frac{(d_{IIB})_\star}{r_\star}$ is a constant. 
\begin{align*}
r_\star\dd_{r_\bullet}(\alpha_5'(Tr_z)^2) & = r_\star \alpha_5''\cdot (d_{IIB})_{\bullet}\cdot (Tr_z)^2 + r_\star \alpha_5'((Tr_z)^2)_\bullet\\
&  = (\alpha_5^{'' \log} - \alpha_5^{' \log})\frac{r_\star\cdot(d_{IIB})_\bullet \cdot (Tr_z)^2 }{d_{IIB}^2} + \alpha_5^{' \log} \frac{2T^2r_z r_\star}{d_{IIB}}\\
\mbox{Already checked bounded: } & \frac{r_\star (d_{IIB})_\bullet}{d_{IIB}}, \; \frac{(Tr_z)^2}{d_{IIB}}\because(\mbox{above})\\
 \frac{T^2 r_z r_\star}{d_{IIB}} & \approx \frac{r_z r_\star/r_x^2}{1+(r_y/r_x)^2-\frac{1}{2}(r_z/r_x)^2} \approx \frac{r_z r_\star/r_x^2}{1+(r_y/r_x)^2}< \frac{r_\star/r_x}{1+(r_y/r_x)^2}\\
& \star = x \implies \mbox{bounded}\\
& \star = y  \implies a/(1+a^2) \mbox{ bounded as above}\\
  & \star = z \implies \mbox{small} \\
\therefore r_\star\dd_{r_\bullet}(\alpha_5'(Tr_z)^2)  & = \mbox{(small)(bounded)} + \mbox{(small)(bounded)}
\end{align*}

\item $P=2T^2 \alpha_5$: shows up in first derivative of $F$ wrt any variable. Taking another derivative gives {\small $r_\star\dd_{r_\bullet}2T^2 \alpha_5=r_\star 2T^2 \alpha_5' (d_{IIB})_\bullet=2T^2\alpha_5^{'\log} \cdot (r_\star (d_{IIB})_\bullet)/d_{IIB}$}. So we want $(r_\star (d_{IIB})_\bullet)/d_{IIB}$ to be bounded, which was already checked above.
\end{enumerate}

This completes the calculation of positive definiteness in region IIB.

\subsection{The remainder of the $\mb{C}^3$ patch}

Recall from the construction of coordinate charts on $Y$ (Definition \ref{def:complex_coordinates} and Lemma \ref{symm_lemma}) that the charts $U_{0,g^0}$ with coordinates $(x,y,z)$ and $U_{(-1,0),g^{-1}}$ with coordinates $(x''',y''',z''')$ are glued to each other via $(x''',y''',z''') = (Tv_0 x^{-1}, Tv_0 y^{-1}, T^{-2} z^{-1})$. Also recall from Equation \ref{def_s_form} that along the $z$-axis, where $\alpha_3 = \alpha_5 = 1$, $\omega$ is defined by the K\"ahler potentials
$$F = \frac{1}{2} (g_{xz} + g_{yz}) + \alpha_4(\theta_V) \theta_V,$$ where $\theta_V = \phi_x - \phi_y$, and
$$F'''=\frac{1}{2} (g'''_{xz} + g'''_{yz}) + \alpha_4(\theta'''_V) \theta'''_V,$$ which differs from $F$ by a harmonic term (see proof of Claim \ref{claim:sympl_area_1}).

By symmetry it suffices to check that $\omega$ is positive definite on just half of the $z$-axis. Namely, we can assume that $|z| \leq T^{-1}$, since otherwise $|z'''|=|T^{-2} z^{-1}| \leq T^{-1}$. Although $|z| \leq T^{-1}$, we may assume that $|x|$ and $|y|$ are very small. They are at most of the order of $T^{3l/8-l/p}$, see Figure \ref{delineation_pic}; otherwise $\alpha_4$ is constant and equal to $\pm 1/2$, so $F$ agrees with either $g_{xz}$ or $g_{yz}$. Thus $|T^2xz| \leq |Tx|$ and $|T^2yz| \leq |Ty|$ hence these quantities are very small and we have again small scale approximations for their logarithms. Differentiating $\alpha_4(\theta)\theta$ once and twice:
\begin{align*}
\theta &= \phi_x - \phi_y = \log_T(1+|Tx|^2) - \log_T(1+|T^2yz|^2) -\log_T(1+|Ty|^2) + \log_T(1+|T^2xz|^2)\\
& \approx (1+|Tz|^2)(|Tx|^2 - |Ty|^2)
\end{align*}
\begin{align*}
\frac{1}{r_*}\frac{\dd}{\dd r_*}(\alpha_4\theta) & = (\alpha_4' \theta + \alpha_4)\frac{\theta_*}{r_*}\\
\frac{\dd}{\dd r_\bullet}(\alpha_4' \theta \theta_*) &= \alpha_4'' \theta_\bullet \theta \theta_* + \alpha_4' \theta_\bullet \theta_* + \alpha_4' \theta \theta_{*\bullet} = (\alpha_4'' \theta^2  + \alpha_4' \theta) \cdot \frac{\theta_\bullet \theta_*}{\theta} + \alpha_4^{' \log} \theta_{* \bullet}\\
\frac{\dd}{\dd r_\bullet}(\alpha_4 \theta_*) & = \alpha_4'\theta_\bullet \theta_* + \alpha_4 \theta_{* \bullet} = \alpha_4^{' \log} \frac{\theta_\bullet \theta_*}{\theta} + \alpha_4 \theta_{* \bullet}.
\end{align*}

We see we'll need to bound terms as follows:
\begin{compactitem}[\textbullet]
\item $\frac{\theta_*}{r_*}$ has estimates $ 2(1+|Tz|^2) T^2$ for $* \in \{x,y\}$ and for $*=z$ we have the estimate $\frac{\theta_z}{r_z} \approx T^4(r_x^2-r_y^2) \approx 0$
(using the fact that $r_x, r_y \ll r_z$).
\item $ \frac{\theta_\bullet \theta_*}{\theta}$ implies we'll need to consider $\frac{r_\bullet r_\star}{r_x^2 - r_y^2}$, which is approximately zero unless both $*$ and $\bullet$ are not $z$, because $\theta$ is approximately $(1+T^2r_z^2)T^2(r_x^2-r_y^2)$ and $r_x,r_y \ll r_z$. Otherwise recall from Equation \ref{eqn:bound_loc_constant} that we make $\alpha_4$ constant equal to zero on a sliver $T < (r_x/r_y)^2 < 1/T$. We divide top and bottom by $r_x^2$ or $r_y^2$ depending on which variable is larger, and then the numerator is at most 1 while the denominator is bounded below from the constraint $T < (r_x/r_y)^2 < 1/T$.

\item $\theta_{* \bullet}$ is approximately zero unless $* = \bullet$ in which case it's a constant, so bounded.\newline
\end{compactitem}

This concludes the proof that the bump function derivatives can be made sufficiently small so that the defined $\omega$ is non-degenerate.
\newpage

\section{Notation}
{\bf Chapter 1: Context and main result}
\begin{itemize}
\item $(\gamma_1, \gamma_2)$ = coordinates of $\gamma$ with respect to the standard $\{(1,0), (0,1)\}$ basis
\item $  \gamma':=\left( \begin{matrix} 2 \\ 1 \end{matrix} \right), \gamma'' := \left( \begin{matrix} 1 \\2 \end{matrix} \right)$
    \item $\Gamma_B=\mb{Z}\left<\gamma', \gamma''\right>$
    \item $(n_1,n_2)$ = coordinates of $\gamma$ with respect to $\{\gamma', \gamma''\}$ basis
    \item $V = (\mb{C}^*)^2/\Gamma_B$
    \item $V^\vee$ generic fiber of $(Y,v_0)$ and SYZ mirror abelian variety of $V$
    \item $\tau \in \mb{R}$ parametrizes family of complex structures on the complex genus 2 curve
    \item $\Sigma_2$ = genus 2 curve
    \item $\mc{L} \to V$ is holomorphic line bundle defined in Chapter 2
    \item $H$ is hypersurface defined by section $s:V \to \mc{L}$, a theta function
    \item $D^b_{\mc{L}}Coh$ is generated by powers and shifts of $\mc{L}$
\end{itemize}

{\bf Chapter 2: HMS for abelian variety}
\begin{itemize}
\item $x_1,x_2$ are complex coordinates on $V$
    \item $T_B := \mb{R}^2/\Gamma_B$
    \item $ \la(n_1\gamma' + n_2 \gamma'') := \left(\begin{matrix} n_1 \\ n_2 \end{matrix} \right) = \left(\begin{matrix} 2 & 1\\ 1 & 2 \end{matrix} \right)^{-1}(\gamma)$
    \item $\kappa(\gamma) := -\frac{1}{2}\la(\gamma)^t \left(\begin{matrix} 2 & 1\\ 1 & 2\end{matrix} \right) \la(\gamma)=-\frac{1}{2}\la(\gamma)^t \gamma$
    \item $\ell_k:=\{\left(\xi_1, \xi_2,-k \left( \begin{matrix} 2 & 1 \\ 1 & 2 \end{matrix}\right) ^{-1}\right) \}_{(\xi_1,\xi_2) \in T_B}$ are linear $T^2$-Lagrangians in $V^\vee$ of slope $k$
    \item $t_x:=\{(\log_\tau |x|, \theta)\}_{\theta \in [0,1)^2} \subseteq V^\vee$
    \item $s_{e,l}$ denotes a basis of sections for $H^0(\mc{L}^l)$ of size $l^2$ (see page \pageref{theta_basis})
    \item $p_{e,l}$ denotes a basis of intersection points for $HF_{V^\vee}(\ell_i,\ell_j)$ where $j-i=l$, and there are $l^2$ such points (see page \pageref{def:symp_basis_fiber})
    \item $\gamma_{i\cap j}$ denotes the remainder modulo $\Gamma_B$ for a given intersection point of $\tilde{\ell}_i \cap \tilde{\ell}_j$ (see page \pageref{def:symp_basis_fiber})
    \item $(\xi_1,\xi_2,\eta)$ are the moment map coordinates where $|x_i| = \tau^{\xi_i}$ and $\eta=\mu_X(\bm{x},y)$ from \cite{AAK}, see page \pageref{page_mu}
    \item $\theta_i = \arg(x_i)$ for $i=1,2$
\end{itemize}

{\bf Chapter 3: Construction of symplectic fibration $(Y,v_0)$ and $\omega$}
\begin{itemize}
    \item $T$ is the complex parameter on $Y$/Novikov parameter on genus 2 curve
    \item $\tau$ is the Novikov parameter on $Y$/complex parameter on the genus 2 curve
    \item $(x_1,x_2,y)\in {\mb{C}^*}^2/\Gamma_B \times \mb{C}$
    \item $\Delta_{\tilde{Y}}$ defines the universal cover of $Y$ and a toric variety of infinite type
    \item $x,y,z$ are the complex toric coordinates on $Y$ (different $y$ from $X=Bl_{H \times \{0\}} V \times \mb{C}_y$)
    \item $r_x,r_y,r_z$ denotes their norms
    \item $v_0=xyz$ is the superpotential
    \item $F$ denotes the K\"ahler potential
    \item $(d_I,\theta_I)$ denotes coordinates on delineated region I
    \item $U_{\ul{0}, g^k}$ in Definition \ref{def:complex_coordinates} denotes the complex charts around each hexagon vertex in $\Delta_{\tilde{Y}}$, indexed by $\mb{Z}^6= \left< g \right>$
\end{itemize}

{\bf Chapter 4: Definition of the DFS-type category}
\begin{itemize}
    \item DFS - Donaldson-Fukaya-Seidel
    \item $H$ is the symplectic horizontal distribution
    \item $F$ briefly at the start denotes a fiber $V^\vee$
    \item $\Phi$ is the parallel transport map
    \item $X_{hor}$ denotes the horizontal vector field over a given curve in the base of $v_0$
    \item $\phi_H^t$ is the flow of $X_{hor}$
    \item $\bigcup_\gamma \ell_k$ or $\bigcup_\gamma t_x$ denotes parallel transport of the fiber Lagrangian over $\gamma$ in $(Y,v_0)$
    \item $L_k$ is $\bigcup_\gamma \ell_k$ over U-shaped $\gamma$
    \item $(f_1,f_2)$ denote amount we add to $(\theta_1,\theta_2)$ from monodromy 
    \item $\pi$ denotes $v_0$ when we think of it as a fibration in Section \ref{mon} (versus as a function)
    \item $\phi^H_{2\pi}$ denotes monodromy
    \item $\mc{J}_\omega(Y,U)$ denotes the class of compatible almost complex structures which are identically $J_0$ outside open set $U$ (defined in Equation (\ref{defn: of_U_set})) about the origin in $v_0$ base
    \item $D$ denotes the anti-canonical divisor 
    \item $A=(a_1,a_2,a_3)$ is fixed value for $(\xi_1,\xi_2,\eta)$ in the polytope $\Delta_{\tilde{Y}}$
    \item $p$ is either an abstract perturbation or a fixed point in a Lagrangian
    \item $\mc{J}^1_{reg}$ is the set of $J$ regular for disc configurations
    \item In Key Regularity argument: $\eta$ used briefly as annihilator to image of linearized operator on universal Fredholm problem, $\xi$ for tangent vectors on space of maps, and $\dot{J}$ for tangent vectors on space of almost complex structures
    \item $\mc{J}^2_{reg}$ denotes set of $J$ regular for disc attached to sphere configuration
    \item $\mc{B}$ is base of a Fredholm problem
\end{itemize}

{\bf Chapter 5: Computing the differential}
\begin{itemize}
    \item $M^k$ are structure maps on $Y$
    \item $\mu^k$ are structure maps on fiber $V^\vee$, $\dd = \mu^1$
    \item $p^k_{i,j}, p_{e,l}$ both denote intersection points in $\ell_i \cap \ell_j$ where $l=j-i$, indexed by $k$ and $e$ respectively
    \item $p_{\infty, i} \in \ell_i \cap t_x$
    \item $\alpha_{\tilde{e}}(e,l)$ is weighted count of bigons between $p_{e,l-1}$ and $p_{\tilde{e},l}$
    \item $n_{\tilde{e}}(l)$ is weighted count of triangles between $p_{\infty, i}, p_{\tilde{e},l}$ and $p_{\infty, j}$ where $l=j-i$
    \item $C(x)$ is the sphere bubble count
    \item $c$ is the disc count times the sphere bubble count
    \item $\beta_{i,j} + \alpha$ denotes $(i,j)$-th disc class plus a sphere configuration class $\alpha$
    \item $n_{\beta+\alpha}$ denote curve counts
\end{itemize}

\newpage
\bibliographystyle{halpha.bst}
\bibliography{Cannizzo_thesis.bib}
\end{document}